\documentclass[12pt]{article}
\newif\ifblind\blindfalse 
\usepackage[letterpaper,margin=1in]{geometry}
\usepackage{amssymb}
\usepackage[round]{natbib}
\usepackage{graphicx}
\usepackage{setspace}
\usepackage[colorlinks=false,allbordercolors={1 1 1}]{hyperref}

\usepackage{amsmath}
\usepackage{amsthm}
\usepackage{mathtools}
\usepackage{bm}
\usepackage{booktabs}
\usepackage{multirow}
\usepackage{array}
\usepackage{xcolor}
\usepackage{enumitem}
\usepackage{pifont}
\usepackage{lastpage}

\makeatletter
\newcommand{\jmlrresetthm}[1]{%
  \expandafter\let\csname #1\endcsname\relax
  \expandafter\let\csname end#1\endcsname\relax
  \expandafter\let\csname c@#1\endcsname\relax
  \expandafter\let\csname the#1\endcsname\relax
  \expandafter\let\csname p@#1\endcsname\relax}
\jmlrresetthm{theorem}\jmlrresetthm{lemma}\jmlrresetthm{proposition}
\jmlrresetthm{corollary}\jmlrresetthm{remark}\jmlrresetthm{definition}
\jmlrresetthm{example}\jmlrresetthm{conjecture}\jmlrresetthm{axiom}
\makeatother
\newtheorem{theorem}{Theorem}
\newtheorem{proposition}{Proposition}
\newtheorem{lemma}{Lemma}
\newtheorem{corollary}{Corollary}
\newtheorem{definition}{Definition}
\newtheorem{remark}{Remark}
\newtheorem{example}{Example}
\newtheorem{assumption}{Assumption}
\newtheorem{protocol}{Protocol}
\makeatletter
\def\section{\@startsection{section}{1}{\z@}{-0.16in}{0.04in}{\large\bf\raggedright}}
\def\subsection{\@startsection{subsection}{2}{\z@}{-0.13in}{0.03in}{\normalsize\bf\raggedright}}
\makeatother
\newcommand{\R}{\mathbb{R}}
\newcommand{\E}{\mathbb{E}}

\newcommand{\1}{\mathbf{1}}
\newcommand{\A}{\mathbf{A}}
\newcommand{\Pmat}{\mathbf{P}}
\newcommand{\X}{\mathbf{X}}
\newcommand{\Y}{\mathbf{Y}}

\newcommand{\bbQ}{\mathbf{Q}}
\newcommand{\I}{\mathbf{I}}

\newcommand{\x}{\mathbf{x}}
\newcommand{\y}{\mathbf{y}}
\newcommand{\z}{\mathbf{z}}
\newcommand{\bw}{\mathbf{w}}

\newcommand{\norm}[1]{\left\lVert#1\right\rVert}
\newcommand{\abs}[1]{\left\lvert#1\right\rvert}
\newcommand{\inner}[2]{\langle #1, #2 \rangle}
\newcommand{\Otoinf}{2\!\to\!\infty}
\DeclareMathOperator{\diag}{diag}
\DeclareMathOperator{\tr}{tr}

\DeclareMathOperator*{\argmin}{arg\,min}

\newcommand{\KL}{\mathrm{KL}}
\providecommand{\Var}{\operatorname{Var}}
\providecommand{\bu}{\mathbf{u}}
\providecommand{\op}{\operatorname{op}}
\providecommand{\sym}{\operatorname{sym}}
\providecommand{\vecl}{\operatorname{vec}_{<}}
\providecommand{\Pb}{\mathbb P}
\providecommand{\Eb}{\mathbb E}
\providecommand{\kl}{\operatorname{kl}}
\providecommand{\TV}{\operatorname{TV}}
\providecommand{\calI}{\mathcal I}
\providecommand{\calW}{\mathcal W}
\providecommand{\calE}{\mathcal E}

\begin{document}
\sloppy

\singlespacing

\title{Minimax Synthesis of Network Mechanisms}

\ifblind
\else
\author{Marios Papamichalis\thanks{Marios Papamichalis, Human Nature Lab, Yale University, New Haven, CT 06511 (email: marios.papamichalis@yale.edu).} \\ Yale University \and Regina Ruane\thanks{Regina Ruane, Department of Statistics and Data Science, The Wharton School, University of Pennsylvania, Philadelphia, PA 19104 (email: ruanej@wharton.upenn.edu).} \\ University of Pennsylvania}
\fi
\date{}

\maketitle

\begin{abstract}%
A single observed network reflects several mechanisms at once: communities, hubs, and clustering coexist in one graph, each a different model. We treat the
network as a combination of candidate mechanisms and study, from a single graph, how strongly each mechanism contributes and how they combine. We address two questions. The first
is how to measure each mechanism's contribution when the mechanisms must themselves be estimated from
the graph: fitting the mechanisms and their strengths from the same data biases the strengths toward zero, and a correction removes this bias and yields valid confidence intervals. The
second is whether the rule of combination is itself recoverable: when a graph is generated by two
mechanisms acting together, the graph alone determines whether they combine additively or interact, exactly when the graph is dense enough, a sharp threshold below which no test can decide. The estimates calibrate the candidate mechanisms against the observed edges. We establish matching minimax rate, against a known-design benchmark and the estimated-design problem itself, confirm the methods in simulation, and apply them to real networks, where the signed coefficients recover known structure and, in one case, a confidence interval excludes any positive contribution from a candidate mechanism.
\end{abstract}

\par

\section{Introduction}
\label{sec:intro}

An empirical network, such as a citation graph, a social network, or a power grid, typically
displays four structural features at the same time. \emph{Communities:} nodes partition into groups
that link densely within and sparsely across. \emph{Heavy tails:} the degree distribution decays
polynomially, so a few hub nodes have very high degree. \emph{Clustering:} two neighbours of a
common node are unusually likely to be linked, so the graph is rich in triangles. \emph{Short
paths:} despite sparsity, typical distances are of order $\log n$.

Each classical random-graph model was designed to explain one of these features. The stochastic
block model encodes communities \citep{holland1983stochastic,abbe2018community}; the Barab\'asi--Albert and Chung--Lu
models encode heavy tails \citep{barabasi1999emergence,chung2002average}; the Watts--Strogatz model encodes
clustering and short paths \citep{watts1998collective}; the random dot-product graph provides a
tractable latent geometry with a complete estimation theory \citep{athreya2018statistical,rubin2022statistical}.
Each model succeeds at its target feature and is typically misspecified for the others, with
partial exceptions such as hyperbolic models that capture degree heterogeneity and clustering
jointly. These features correspond to
distinct generative \emph{mechanisms}: community formation, preferential attachment, triadic closure,
and latent geometry. An empirical network typically exhibits several at once, which suggests modelling
it as a composition of mechanisms rather than selecting a single model.

This observation turns model choice into an estimation problem. Fix a set of candidate mechanisms,
each able to propose edge probabilities for the observed graph. Four quantities are then at issue: which candidate mechanisms contribute to the observed graph, in
what proportions, with what statistical uncertainty, and whether these quantities are recoverable
from a single network. The answer determines whether a fitted model constitutes evidence for a
generative process or only a descriptive summary. Link prediction, community recovery, anomaly detection, bootstrap resampling and
epidemic simulation all inherit the biases of the fitted generative model: a model without triangles
misstates redundancy, a model without hubs misstates vulnerability to targeted failure, and a model
without communities cannot support the recovery tasks that motivate much of network analysis. Combination of network generators has been studied for prediction and model selection: stacking of
link predictors \citep{ghasemian2020stacking}, multilayer block-model recovery from a single aggregate graph, and graphon-level Bayesian predictive synthesis. In all of these the combination weights are tuning quantities
optimised for predictive risk or treated as a model-selection device. Treating mechanism composition as an inferential estimand, with a minimax rate, a matching lower
bound, valid confidence intervals, selection consistency, and a test for the combination operator, is
new to the single-network setting.

\paragraph{When a fitted model is evidence for a mechanism.}
Whether these quantities are recoverable from one graph has a sharp answer, and it organises the paper.
The generative weights are identified when the candidate kernels are spectrally separated, so that the
transversality $\gamma_S$ of Section~\ref{sec:setup} is bounded away from zero; the combination operator
is identified when the edge-overlap information $n^2\rho_n^3$ diverges (Theorems~\ref{thm:debias}
and~\ref{thm:nor}). Where the geometry separates the mechanisms and the graph is dense enough, the
synthesis returns generative proportions with valid intervals and a decided operator; outside that regime
the same estimator returns calibrated, signed projection coefficients, the partial contribution of each
mechanism given the others, and the operator is information-theoretically undecidable. Attribution where
both conditions hold, and a signed description otherwise: this dichotomy is what makes a fitted network
model evidence for a generative mechanism rather than a re-description of the data.

\paragraph{The estimand.}
The object of inference is the composition of the observed graph relative to a set of candidate mechanisms.
Each mechanism, fitted to a graph with edge-probability matrix $P$, returns an edge-probability kernel, and
the synthesis coefficients are the coordinates of $P$ in the span of these fitted kernels: generative
proportions under correct specification, and interpretable candidate-dependent projection coefficients
otherwise. The candidate set need not contain the true generating process, and fitting one mechanism need not
recover a separate latent component of a mixture; the estimand is defined by the projection in either case,
and Section~\ref{sec:estimands} makes this precise. Throughout, $N=n^2\rho_n$ is the order of the number of
edges, $s$ the number of active mechanisms, and $\gamma_S$ their distinguishability.

\paragraph{The setting.}
Each classical model is treated as an \emph{agent} that, given latent node attributes, proposes an
edge-probability kernel $P^{(k)}_{ij}$, and a synthesis function combines the proposals into a single edge
probability under which the graph is generated as conditionally independent edges. The combination is a
calibration of these forecasts against the realised graph, so its coefficients are not confined to the simplex:
an agent whose forecast is systematically anti-aligned with the observed edges receives a negative coefficient,
a corrective contrast that nonnegative posterior averaging cannot express and that the power-grid analysis of
Section~\ref{sec:realdata} exhibits. Two operators are studied, defined formally in
Definitions~\ref{def:mixture} and~\ref{def:noisyor}: a \emph{convex mixture}
$P_{ij}=\sum_k w_k P^{(k)}_{ij}$ and a \emph{noisy-OR superposition} $P_{ij}=1-\prod_k(1-w_k P^{(k)}_{ij})$,
the latter edge-monotone so that it preserves the triangles a clustering agent contributes when other agents
are added. The forecaster viewpoint and its Bayesian predictive-synthesis reading
\citep{mcalinn2019dynamic,mcalinn2020multivariate} are developed in a separate note; here the synthesis is studied directly as
a calibrated combination.

The synthesised class can express each agent's target feature when the corresponding layer carries
sufficient weight and the operator preserves the relevant structure: triangles and short paths from a
small-world layer, hubs from a heavy-tailed layer, communities from a block layer. It remains
inside well-understood model classes: the mixture is itself a random dot-product graph with
concatenated latent positions (Remark~\ref{rem:stack}), and the noisy-OR is a superposition of
sparse graphs, so the estimation theory of those classes transfers. The combination class is the
setting of the paper; its subject is the inference.

\label{sec:contrib}
The thesis is that a single observed network, though only one draw from one model, nonetheless determines both how strongly each candidate mechanism contributes and by what rule the mechanisms combine, and that both are estimable from that one graph at minimax rate with valid confidence statements; the impossibility and frontier results of Section~\ref{sec:structural} are what make the synthesis necessary, and the estimation theory that follows is what makes it usable. The contributions are organised around these two questions.

\begin{itemize}[leftmargin=1.4em,itemsep=3pt]
\item \textbf{Debiased estimation of fitted kernels, and a separation
(Theorem~\ref{thm:debias}, Corollary~\ref{cor:debias-lb}).} In every application the kernels are fitted from
the same graph used to estimate the coefficients, and the fitting error enters the second-stage Gram matrix as
a positive self-product that attenuates the coefficients toward zero by more than the edge rate. A cross-fold
construction that fits each mechanism twice on independent dyad sub-folds cancels the self-product, removing
the attenuation and restoring the minimax rate; the naive single-fold plug-in is provably rate-suboptimal and
the debiased estimator provably optimal, so the two are separated in order over the class in which kernels must
be estimated.

\item \textbf{Operator identification (Theorem~\ref{thm:nor}, Corollary~\ref{cor:nork-cor},
Theorem~\ref{prop:fitted-operator}).} Whether two mechanisms combine additively or by overlap is itself an
inferential question, decidable from one graph exactly when the edge-overlap information $n^2\rho_n^3$
diverges: below it the additive and noisy-OR laws are contiguous and no test separates them, above it a test on
the interaction column is consistent. The two-layer noisy-OR admits an exact linear reparametrisation
$P=w_1G_1+w_2G_2-w_1w_2\,G_1\!\odot\!G_2$, giving edge-rate estimation of the weights and one-step efficiency,
and the boundary extends to any fixed number of layers. This is how much one graph reveals about the manner of
combination, a question that does not arise when the rule is assumed.
\item \textbf{The known-design benchmark (Proposition~\ref{thm:joint}).}
With the candidate kernels known, a least-squares synthesis estimates the coefficient vector at the rate
$\sqrt{s/\gamma_S}/(n\sqrt{\rho_n})$, the inverse square root of the number of edges, with $s$ the number of
active mechanisms and $\gamma_S$ their distinguishability; a matching Assouad lower bound shows it is sharp.
The result fixes the statistical price: the effective sample size is the edge count $n^2\rho_n$ and the
conditioning is the smallest eigenvalue of the active Gram-correlation matrix.

\item \textbf{Valid inference, misspecification, and adaptive selection
(Proposition~\ref{thm:clt}, Theorem~\ref{thm:twostage}, Proposition~\ref{prop:mis-cor},
Corollary~\ref{cor:adapt-cor}).} The held-out calibration estimator is asymptotically normal for its
population target on the case--control dyad distribution, with a consistent sandwich variance and a valid pairs
bootstrap. Under misspecification the coefficient vector is the projection of the edge-probability matrix onto
the candidate set, a negative coordinate being a corrective contrast, and the projection is no worse in
population square loss than any single mechanism. A thresholding estimator recovers the active set knowing
neither $S$, $s$, nor $\gamma_S$, at the $\sqrt{\log K}$ selection price; and when the kernels are fitted, a
two-stage variance adds the propagated first-stage uncertainty so the Wald intervals cover the \emph{generative}
weight rather than only the projection target (Theorem~\ref{thm:twostage}), the sandwich alone undercovering.

\end{itemize}

\paragraph{The estimand: candidate-relative calibration coefficients.}
A single graph does not identify the separate latent components of a mixture: fitting one mechanism to the
whole graph converges to its best approximation of $P$, not to the component that contributed only part of it.
We therefore define the candidate set through its \emph{fitting maps}, letting $G_k^\dagger(P)=\mathcal F_k(P)$
be the population output of fitting agent $k$ to a graph with edge-probability matrix $P$ and taking the
synthesis coefficients to be the coordinates of $P$ in the span of $\{G_k^\dagger(P)\}$. Under correct
specification, when $P$ lies in that span and each fitting map projects onto its own component, the
coefficients reduce to the generative weights $\bw_{\mathrm{gen}}$; otherwise they are the projection
coefficients $\bw^\dagger_{\mathrm{LS}}$, so the estimand is always well defined and candidate-relative, and
the identifiability of separate generative components is never assumed. Throughout, $G_k$ denotes the fitted
kernel $G_k^\dagger(P)$ unless the known-design setting is stated. The estimand is then a projection of the
edges onto these fitted kernels,
\[
P_{ij}\ \approx\ \sum_{k} w_k\,G_{k,ij}
\qquad\text{or, in calibration form,}\qquad
\operatorname{logit}\Pr(A_{ij}=1)\ =\ \beta_0+\sum_{k} w_k\,s_{k,ij},
\]
whose coefficients are the estimands of Proposition~\ref{thm:joint}, Theorem~\ref{thm:debias}, and
Corollary~\ref{cor:adapt-cor} on the left and of Proposition~\ref{thm:clt} on the right. They are not confined
to the simplex, so a systematically biased mechanism is corrected by a negative weight no convex average can
express; and even under a correctly specified linear mixture the logistic target $\bw^\dagger_{\mathrm{cal}}$
does not equal $\bw_{\mathrm{gen}}$, since the logit of a sum is not linear in the summands, so the logistic
coefficients share the sign interpretation of the projection coefficients but are not proportions. A Bayesian
predictive-synthesis reading is developed in a companion note \citep{west1992modelling,mcalinn2019dynamic,mcalinn2020multivariate}.

\paragraph{Scope of the optimality theory.} The fitted-kernel rate and lower bound cover the community and latent-geometry mechanisms, namely the bounded-heterogeneity block, degree-corrected block, and dot-product kernels. Heavy-tailed degree and triadic-closure mechanisms enter as motivation and in the predictive comparisons of Section~\ref{sec:linkpred}, and the growing-degree study of Section~\ref{sup:scaling} addresses the heavy-tailed degree case at the level of the rate. The optimality statements therefore cover two of the four motivating mechanisms, the other two entering empirically.

\paragraph{Relation to prior combinations.} Four features distinguish this setting: a single network in place of independent replicates, the sparse regime, inference on the combination coefficients, and identification of the combination operator. Existing work addresses them singly or in pairs, and in each the weights are tuning parameters carrying no rate, lower bound, or confidence statement. The problem of recovering latent layers from one aggregated network was posed by \citet{valles2016multilayer}, whose OR aggregation coincides with the noisy-OR superposition studied here; their treatment is generative and model-selection-oriented, with no convergence rate, lower bound, distributional theory, or operator test, and we keep the superposition insight while supplying that apparatus and replacing homogeneous block layers with named mechanism kernels. Bayesian model averaging \citep{hoeting1999bayesian,yao2018using} confines posterior weights to the convex hull and concentrates on a single pseudo-true model; stacking and the super learner \citep{wolpert1992stacked,breiman1996stacked,ghasemian2020stacking} are $\mathcal M$-open and set weights by cross-validated loss, of which the held-out calibration estimator of Section~\ref{sec:estimation} is the dyadic form; and for networks \citet{zhang2025network} select latent-space averaging weights by edge cross-validation with prediction-risk optimality but no lower bound, interval, or operator test. A population-level graphon version of this synthesis is developed by \citet{papamichalis2025graphon}; the present paper is its finite, single-graph counterpart. The combination here removes the convex-hull restriction and treats the agents as useful but misspecified sources, and differs from this literature in three ways: the combined object is a generative kernel with provable structural behaviour, licensing statements about the synthesised graph and not only its held-out predictions; the weights are the estimand, carrying a minimax rate with a matching lower bound (Proposition~\ref{thm:joint}), consistent selection (Corollary~\ref{cor:adapt-cor}), and a limit distribution (Proposition~\ref{thm:clt}); and the noisy-OR operator leaves the linear span of the forecasts, so the $n^2\rho_n^3$ threshold of Theorem~\ref{thm:nor} has no analogue for combinations defined to be linear. We are not aware of prior work providing an error-controlled test of the additive versus overlapping rule on one graph.

Adamic--Adar, Jaccard, the hyperbolic score, the heavy-tailed Chung--Lu agent, and
the local small-world kernel enter the experiments as predictive scores or structural motivation and
are outside the fitted-kernel theory.

\label{sec:example}
The \textsf{ca-GrQc} collaboration network (Section~\ref{sec:realdata}) has $n=4158$ authors, mean
degree $6.5$, clustering $0.63$, typical distance close to $\log n$, a power-law degree tail with
maximum degree $81$, and modularity $0.85$. Every single classical model fitted to it misses at
least one of these features by a wide margin: the configuration model produces clustering $0.01$,
the Watts--Strogatz model caps the maximum degree at $9$, and the degree-corrected stochastic block
model under-predicts clustering severalfold. Among the eight fitted
single-model baselines, only a random hyperbolic graph reaches the observed range on three of the four diagnostics, while
overshooting the observed hub scale fifteenfold on the fourth. Section~\ref{sec:linkpred} completes the example on the inferential side: with the weights estimated
rather than chosen, the synthesis improves out of sample on every single mechanism, and its replicated
intervals identify the mechanisms the data support.
\label{sec:related}

The single-graph data, with $\binom n2$ dependent dyads and regressors estimated from the same graph,
motivate the dyad-splitting cross-fitting of Section~\ref{sec:joint}. Cross-fitting and
Neyman-orthogonal moments for dyadic and network data are the device of \citet{chiang2026double}, who
develop double/debiased machine learning for multiway-clustered and network-dependent samples; we use
the same split-and-debias principle, but the object here is a low-dimensional synthesis-coefficient
vector with kernels fitted from the one graph, and the leading correction is the Hadamard
self-product $\mathbf H^\top\mathbf H$ that the cross-fold Gram removes (Section~\ref{sec:cf}), rather than a nuisance-function bias.

Among single models, degree-corrected SBMs \citep{karrer2011stochastic} add per-node degree parameters and remain our strongest competitor, yet still under-produce clustering because their edges are conditionally independent given degrees and blocks; mixed-membership blockmodels \citep{airoldi2008mixed} soften assignments without adding a clustering mechanism. Hyperbolic and other latent-geometry models \citep{krioukov2010hyperbolic,gugelmann2012random,fountoulakis2021clustering} achieve heavy tails, clustering, and short paths jointly by embedding in a negatively curved space, but are not built by combining named mechanisms and so lie outside this framework; exponential random graph models can target multiple statistics but are prone to degeneracy and do not scale \citep{schweinberger2020exponential}; sparse exchangeable and graphex models \citep{caron2017sparse,veitch2015class} target sparsity rather than the joint four-property profile. The RDPG and GRDPG carry a mature estimation theory, adjacency spectral embedding recovering latent positions up to an orthogonal transformation with $\Otoinf$ bounds \citep{lyzinski2014perfect} and a central limit theorem \citep{athreya2016limit,athreya2018statistical,rubin2022statistical}; since the mixture of block and dot-product agents is itself an RDPG (Remark~\ref{rem:stack}), this toolkit applies to the synthesised model. For community recovery under heavy tails, degree regularisation \citep{chaudhuri2012spectral,amini2013pseudo,qin2013regularized,joseph2016impact} and concentration of the regularised adjacency \citep{le2017concentration} match the Kesten--Stigum detection boundary \citep{decelle2011asymptotic,mossel2015reconstruction,abbe2018community,lei2015consistency}.

The remainder of the paper is organised as follows. Section~\ref{sec:setup} defines the synthesis
setting, the two combination operators, and the stacked-representation remark
(Remark~\ref{rem:stack}). Section~\ref{sec:structural} establishes the structural motivation: an impossibility theorem showing that no single estimable family carries heavy tails, clustering, communities, and short paths together (Theorem~\ref{thm:imp}), the quantitative clustering--hub--rank frontier that prices the trade-off exactly (Theorem~\ref{thm:frontier}), and a possibility theorem showing the noisy-OR synthesis attains all four (Theorem~\ref{thm:pos}). Section~\ref{sec:joint} contains the
estimation theory: the cross-fitted minimax rate (Proposition~\ref{thm:joint}), debiased estimation of
fitted kernels and the separation from the naive plug-in (Theorem~\ref{thm:debias},
Corollary~\ref{cor:debias-lb}), the misspecification target (Proposition~\ref{prop:mis-cor}), limit
distribution and bootstrap validity (Proposition~\ref{thm:clt}), two-stage inference for the generative
weights (Theorem~\ref{thm:twostage}), adaptive selection (Corollary~\ref{cor:adapt-cor}), and the
noisy-OR operator with its detectability threshold and its fitted-layer form
(Theorem~\ref{thm:nor}, Corollary~\ref{cor:nork-cor}, Theorem~\ref{prop:fitted-operator}). Section~\ref{sec:realdata} reports the simulations
and the network study, and Section~\ref{sec:discussion} concludes. All proofs are collected in the supplement.
\section{Network mechanisms and synthesis operators}
\label{sec:setup}

\subsection{Networks and agents}
We observe a simple undirected graph on $n$ vertices through its symmetric adjacency matrix
$\A\in\{0,1\}^{n\times n}$ with $A_{ii}=0$. Each vertex carries a latent attribute $u_i$ in some
space $\mathcal{U}$ (a community label, a latent position, a degree propensity, or a tuple of
these). An \emph{agent} is a model that maps latent attributes to an edge-probability kernel.

\begin{definition}[Agent kernel]
\label{def:agent}
An agent $k\in\{1,\dots,K\}$ is a measurable map
$\kappa_k:\mathcal{U}\times\mathcal{U}\to[0,1]$, symmetric in its arguments, that proposes the
edge probability $P^{(k)}_{ij}=\kappa_k(u_i,u_j)$ for the pair $(i,j)$.
\end{definition}

We will use four canonical agents, each the kernel of a classical model:
the \emph{block agent} $\kappa_{\mathrm{SBM}}(u_i,u_j)=B_{c_i c_j}$ for community labels
$c_i\in[Q]$ ($Q$ blocks, reserving $K$ for the number of candidates) and a symmetric block matrix $B$;
the \emph{dot-product agent} $\kappa_{\mathrm{RDPG}}(u_i,u_j)=\inner{\x_i}{\x_j}$ for latent
positions $\x_i\in\mathcal{X}\subset\R^d$;
the \emph{heavy-tailed agent} $\kappa_{\mathrm{CL}}(u_i,u_j)=\theta_i\theta_j/S$ for
degree propensities $\theta_i$ with $S=\sum_i\theta_i$ and $\max_i\theta_i^2\le S$, so that the
kernel takes values in $[0,1]$ without truncation and is exactly rank one (Chung--Lu);
and the \emph{small-world agent}, an independent-edge kernel
$\kappa_{\mathrm{SW}}(u_i,u_j)=p_{\mathrm{lat}}\mathbf 1\{d_{\mathrm{ring}}(i,j)\le m\}+
p_{\mathrm{long}}/n$ that links ring-lattice neighbours within distance $m$ with probability
$p_{\mathrm{lat}}$ and distant pairs with probability $p_{\mathrm{long}}/n$. This is the
conditionally independent analogue of the Watts--Strogatz construction \citep{watts1998collective}; the original rewiring model
induces edge dependence and a fixed degree, and is not a kernel of the form above.

\subsection{Two synthesis operators}
A synthesis function turns the agents' proposals into a single edge probability. We study two.

\begin{definition}[Mixture synthesis]
\label{def:mixture}
Given weights $\bw=(w_1,\dots,w_K)$ in the simplex $\Delta^{K-1}$, the \emph{mixture} synthesis sets
\[
P^{\mathrm{mix}}_{ij}=\sum_{k=1}^{K}w_k\,P^{(k)}_{ij}.
\]
\end{definition}

\begin{definition}[Noisy-OR superposition]
\label{def:noisyor}
Given weights $\bw\in[0,1]^K$, the \emph{noisy-OR} synthesis sets
\[
P^{\mathrm{sup}}_{ij}=1-\prod_{k=1}^{K}\bigl(1-w_k P^{(k)}_{ij}\bigr).
\]
Equivalently, the synthesised edge is the union of independent edges drawn from each reweighted
agent: $A_{ij}=\max_k A^{(k)}_{ij}$ with $A^{(k)}_{ij}\sim\mathrm{Bernoulli}(w_k P^{(k)}_{ij})$.
The noisy-OR (union) operator, and its conjugate AND (product) operator, were introduced for
recovering latent layers from a single aggregate graph by \citet{valles2016multilayer}, who give the
complete probabilistic solution for the optimal multilayer block model under each operator; the
present paper retains their superposition operator but replaces homogeneous block layers with
arbitrary named mechanism kernels and supplies the convergence rate,
minimax lower bound, weight-level distributional theory, and operator test that their
model-selection treatment does not address.
\end{definition}

Given a synthesised kernel $P$, the network is drawn with conditionally independent edges,
$A_{ij}\sim\mathrm{Bernoulli}(P_{ij})$ for $i<j$. We refer to the resulting law as the
\emph{synthesis network model} with the stated agents, operator and weights. The
operator-level statements below hold conditionally on the latent attributes, which is all the
structural theorems require.

The two operators agree to first order in the sparse regime. If all $w_kP^{(k)}_{ij}=o(1)$ then
$P^{\mathrm{sup}}_{ij}=\sum_k w_kP^{(k)}_{ij}-\sum_{k<\ell}w_kw_\ell P^{(k)}_{ij}P^{(\ell)}_{ij}+\cdots
=P^{\mathrm{mix}}_{ij}\,(1+o(1))$, so they induce the same first-order degree and edge densities.
They differ at second order precisely where it matters structurally: the subtracted cross term in
noisy-OR keeps probabilities in $[0,1]$ under superposition and makes the operator
edge-monotone. The mixture admits a clean latent representation (Remark~\ref{rem:stack}); noisy-OR
preserves triangles under superposition and is the operator whose detectability is the subject of
Section~\ref{sec:nor}.

\paragraph{A running example: four canonical properties.}
The estimation theory of Section~\ref{sec:joint} applies to any candidate set of mechanisms and any
structural property of interest. As a running example, used throughout for motivation and
illustration, we take the four properties that the classical models were designed to capture.
\emph{Heavy tails}: the empirical degree distribution has a regularly varying tail with index
$\tau\in(2,3)$, so the maximum degree grows polynomially, $d_{\max}=n^{1/(\tau-1)+o(1)}$.
\emph{Clustering}: writing $C$ for the global clustering coefficient, three times the number of
triangles divided by the number of paths of length two, and $\bar C=\frac1n\sum_iC_i$ for the
average local coefficient, the family has constant clustering when $\liminf_n C>0$ (respectively
$\liminf_n\bar C>0$). \emph{Detectable communities}: planted labels $c_i\in[K]$ admit weak recovery,
meaning some polynomial-time estimator agrees with the truth, up to label permutation, at a rate
exceeding chance by a fixed margin. \emph{Short paths}: typical distances in the giant component are
$O(\log n)$. None of the estimation results depends on this choice of properties.

\paragraph{Latent-geometry notation.}

\begin{definition}[RDPG and GRDPG]
\label{def:rdpg}
For latent positions $\x_1,\dots,\x_n\in\R^d$ collected in rows of $\X$, the \emph{random
dot-product graph} has edge probabilities $\Pmat=\X\X^\top$ (entrywise in $[0,1]$). The
\emph{generalised} RDPG (GRDPG) with signature $(p,q)$, $p+q=d$, has
$\Pmat=\X\,\I_{p,q}\,\X^\top$ where $\I_{p,q}=\diag(\I_p,-\I_q)$. The SBM is the special case in
which rows of $\X$ take only $K$ distinct values.
\end{definition}

\begin{definition}[ASE and the $\Otoinf$ norm]
\label{def:ase}
The \emph{adjacency spectral embedding} into dimension $d$ is
$\hat\X=\mathbf{U}_{\!\A}\,\abs{\mathbf{S}_{\!\A}}^{1/2}$, where
$\A=\mathbf{U}_{\!\A}\mathbf{S}_{\!\A}\mathbf{U}_{\!\A}^\top$ is truncated to its $d$ leading
eigenpairs by magnitude. For a matrix $\mathbf{M}$ with rows $\mathbf{m}_i$, the
\emph{two-to-infinity norm} is $\norm{\mathbf{M}}_{\Otoinf}=\max_i\norm{\mathbf{m}_i}_2$; it
controls the worst-row estimation error, which is the relevant quantity for vertex-level inference.
\end{definition}
\begin{remark}[Stacked representation and recovery]
\label{rem:stack}
For the mixture operator the representation is explicit: a mixture of a $K$-block
agent (positions $\x_i$) and a rank-$d_2$ dot-product agent (positions $\y_i$) with weights
$(w_1,w_2)$ is \emph{exactly} a single random dot-product graph with the stacked latent positions
$\z_i=(\sqrt{w_1}\,\x_i,\ \sqrt{w_2}\,\y_i)\in\R^{d_1+d_2}$, by direct verification, since
$P_{ij}=w_1\inner{\x_i}{\x_j}+w_2\inner{\y_i}{\y_j}=\inner{\z_i}{\z_j}$ (indefinite agents stack with
signature matrices). The representation is identifiable up to the unavoidable orthogonal symmetry,
is rank-minimal whenever the two coordinate blocks are transversal, and is recoverable from a single
sampled graph by adjacency spectral embedding at the parametric two-to-infinity rate
$\widetilde O(\sqrt{\log n/(n\rho_n)})$, with row-wise asymptotic normality, by the RDPG estimation theory
applied verbatim \citep{lyzinski2014perfect,athreya2018statistical,rubin2022statistical}. Mixing edge
\emph{probabilities} thus concatenates latent \emph{coordinates}: the block signal and the
continuous heterogeneity occupy orthogonal coordinate blocks, scaled by $\sqrt{w_k}$, which is what
makes the weights estimable objects in the first place and is the geometric picture behind the
transversality constant $\gamma_S$ of Section~\ref{sec:joint}. Figure~\ref{fig:ase} verifies the
rate empirically.
\end{remark}

\section{Why synthesise: an impossibility theorem and a quantitative frontier}
\label{sec:structural}

The running example showed, empirically, that no single classical model reaches the observed range
on all four canonical properties at once. We now show that this is not an accident of any particular
fit. Two results, both statements about graph structure with no regression analogue, make the case
precise. An impossibility theorem proves that none of the four estimable model families can carry
heavy tails, constant clustering, detectable communities, and short paths together. A quantitative
frontier then locates the exact spectral price a clustered kernel must pay to support a hub, and
shows the price is one that flat single-mechanism kernels cannot afford. A possibility theorem
closes the loop: the noisy-OR synthesis pays that price and attains all four properties. These
structural facts are what motivate estimating the synthesis, and the inferential theory of
Section~\ref{sec:joint} builds on the same kernels they concern.

\subsection{No single estimable family carries all four properties}

We consider four families, each a sequence of models on $n\to\infty$ vertices with mean degree
$\bar d_n=n^{o(1)}$:
\begin{enumerate}[label=(\roman*),leftmargin=2em,itemsep=1pt]
\item $\mathcal{F}_{\mathrm{SBM}}$: SBMs with a fixed number $Q$ of blocks and block matrix
$B^{(n)}=\rho_n B$ for a fixed $Q\times Q$ matrix $B$ with positive entries and scale
$\rho_n=\bar d_n/n$;
\item $\mathcal{F}_{\mathrm{WS}}$: small-world graphs with fixed lattice degree $2m$ and a fixed
fraction of long-range edges;
\item $\mathcal{F}_{\mathrm{CL}}$: rank-one Chung--Lu models with a given degree sequence;
\item $\mathcal{F}_{\mathrm{RDPG}}$: RDPGs and GRDPGs with latent dimension bounded by a fixed $d$.
\end{enumerate}

\begin{theorem}[Impossibility within four families]
\label{thm:imp}
Fix any of the four families above. No sequence of models in that family simultaneously satisfies all
four properties of the running example (heavy tail, constant clustering, detectable communities,
short paths). Concretely:
\begin{enumerate}[label=(\alph*),leftmargin=2em,itemsep=2pt]
\item \emph{(SBM and Chung--Lu)} In $\mathcal{F}_{\mathrm{SBM}}$ and $\mathcal{F}_{\mathrm{CL}}$ with
mean degree $\bar d_n=n^{o(1)}$, the global clustering coefficient is $C_n=\Theta(\bar d_n/n)\to0$;
constant clustering would force $\bar d_n=\Theta(n)$, contradicting the heavy-tail and short-path
requirements. Moreover $\mathcal{F}_{\mathrm{CL}}$ has no communities: its weak-recovery agreement is
$1/Q+o(1)$.
\item \emph{(Small-world)} Every model in $\mathcal{F}_{\mathrm{WS}}$ has degrees concentrated in
$[\,2m(1-o(1)),\,2m(1+o(1))\,]$, so its degree distribution is not heavy-tailed, and it has no
planted community structure to recover.
\item \emph{(Bounded-rank RDPG)} For a sequence in $\mathcal{F}_{\mathrm{RDPG}}$ with latent
dimension at most $d$ and a heavy-tailed degree distribution with index $\tau\in(2,3)$, the global
clustering coefficient satisfies $C_n\to0$.
\end{enumerate}
\end{theorem}

\begin{proof}
See Appendix~\ref{app:structural}.
\end{proof}

\subsection{A quantitative frontier: clustering, hubs, and the spectrum}
\label{sec:frontier}

Theorem~\ref{thm:imp} is a binary statement about four named families. The quantitative version is
an exact trade-off, valid for every positive-semidefinite edge kernel of bounded rank, locating
precisely how large a hub a clustered kernel can support and what it costs in spectral terms,
together with a construction that attains the boundary. The trade-off has two regimes. For
flat-spectrum kernels, whose top eigenvalue is comparable to the mean degree, the regime that
contains every sparse graphon, SBM, and Chung--Lu-type model, polynomial hubs are impossible at any
fixed rank: the maximum degree is capped at $\bar d^{\,3/2}$ up to constants. To escape the cap a
kernel must inflate its top eigenvalue to at least the order $\Delta^{2/3}$, dedicating a macroscopic
eigendirection to the hubs, and a rank-one construction shows this price is exactly sufficient. This
is the sense in which synthesis is not merely convenient but minimal: the hub agent in the noisy-OR
supplies precisely the inflated eigendirection that no flat single-mechanism kernel possesses.

Three quantities compete. Triangles are spectrally expensive: for a rank-$d$ kernel
$\tr(\Pmat^3)=\sum_i\lambda_i^3\le d\,s_1^3$, so the triangle count is throttled by the top eigenvalue $s_1$
and the rank, while a single hub of expected degree $\Delta$ contributes on the order of $\Delta^2$ wedges.
Clustering is their ratio, so a clustered kernel with a large hub must raise $\tr(\Pmat^3)$ to order
$c\,\Delta^2$, forcing $d\,s_1^3\gtrsim c\Delta^2$. The upper bound is a few lines of linear algebra once the
diagonal-removal terms are controlled; the content is the matching construction, an anchored community in
which one hub attaches with probability $q$ to a block of $m$ internally $c$-dense vertices, where choosing
$q^3=c^2\Delta/2$ reaches the $\Delta^{2/3}$ floor with the right constant, a rank-one kernel sitting exactly
on the frontier. Above the crossover $\Delta\asymp c^{-2}$ a hub that is itself clustered forces
$s_1\ge c\Delta$, which the same construction with $q=1$ attains.

\begin{theorem}[The clustering--hub--rank frontier]
\label{thm:frontier}
Let $\bbQ$ be positive semidefinite of rank at most $d$ with top eigenvalue $s_1$, let
$\Pmat=\bbQ-\diag(\bbQ)$ have entries in $[0,1]$, let $r=\Pmat\1$ be the expected-degree vector with
$\Delta=\max_ir_i$, and define the population global clustering
$\mathcal C=\tr(\Pmat^3)/(\norm{r}^2-\norm{\Pmat}_F^2)$ and the hub-local clustering
$\mathcal C_h=p_h^\top\Pmat\,p_h/(\Delta^2-\norm{p_h}^2)$, where $p_h$ is the hub's row. Then:
\begin{enumerate}[label=(\alph*),leftmargin=2em,itemsep=2pt]
\item \emph{(Exact inequality.)} $\ \mathcal C\,(\Delta^2-d\,s_1^2)_+\le d\,s_1^3$.
\item \emph{(Flat-spectrum impossibility.)} If $s_1\le\kappa\bar d$ and $\Delta^2\ge2d\kappa^2\bar
d^{\,2}$, then $\Delta\le\sqrt{2d\kappa^3/\mathcal C}\;\bar d^{\,3/2}$. For any fixed rank $d$,
flatness $\kappa$, and clustering bounded below, kernels with polylogarithmic mean degree admit no
polynomially large hub: power-law degree tails are impossible.
\item \emph{(Spectral inflation: necessity.)} If $\mathcal C\ge c$ and $\Delta\ge c\sqrt{2d}$, then
$s_1\ge(c/2d)^{1/3}\Delta^{2/3}$; moreover, unconditionally on rank, $s_1\ge\mathcal C_h\Delta$
whenever the hub itself is clustered.
\item \emph{(Spectral inflation: sufficiency.)} The rank-one anchored kernel $\bbQ=xx^\top$ with
$x=\alpha e_h+\beta\mathbf 1_V$, $\abs V=m$, $\beta^2=c$, $\alpha\beta=q=\min\{(c^2\Delta/2)^{1/3},1\}$,
has clustering $\Theta(c)$, hub degree $\Delta=qm$, and top eigenvalue $s_1\le2c^{1/3}\Delta^{2/3}$
for $\Delta\le2c^{-2}$ and $s_1=(1+o(1))c\Delta$ for $\Delta\ge2c^{-2}$, matching part (c) in both
regimes up to absolute constants.
\end{enumerate}
\end{theorem}

\begin{proof}
See Appendix~\ref{app:structural}.
\end{proof}

\begin{figure}[t]
\centering
\includegraphics[width=0.80\textwidth]{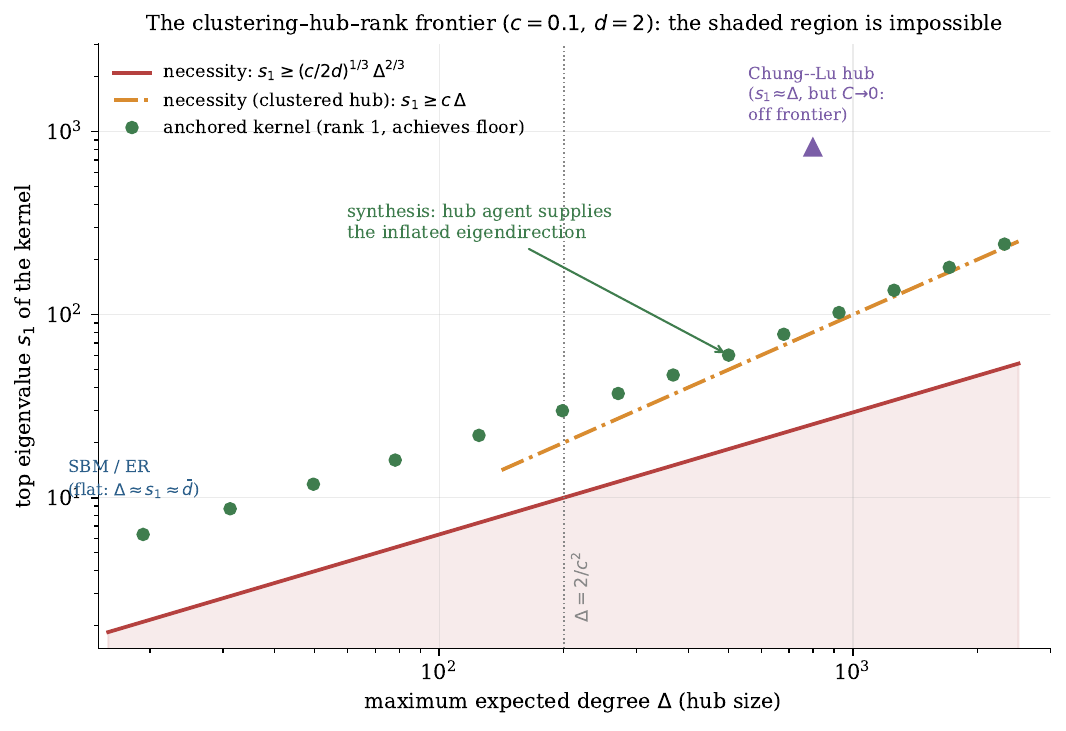}
\caption{The quantitative frontier of Theorem~\ref{thm:frontier} at clustering $c=0.1$ and rank
$d=2$. The shaded region below the $\Delta^{2/3}$ floor is unattainable for any clustered kernel;
above the crossover $\Delta=2/c^2$ the clustered-hub bound $s_1\ge c\Delta$ takes over. The anchored
rank-one kernel attains the floor in both regimes. Flat classical models sit at the lower left
($s_1\approx\Delta\approx\bar d$, no hubs); a Chung--Lu hub inflates $s_1$ but its clustering
vanishes, so it sits off the frontier. The synthesis escapes by letting the hub agent supply the
inflated eigendirection.}
\label{fig:frontier}
\end{figure}

\begin{remark}[Why synthesis is the minimal escape]
\label{rem:frontier-escape}
Theorem~\ref{thm:frontier} converts the impossibility of Theorem~\ref{thm:imp} into an exchange
rate: clustering $c$ together with a hub $\Delta$ costs a top eigenvalue of order
$\max\{(c/d)^{1/3}\Delta^{2/3},\,\mathcal C_h\Delta\}$ that flat single-mechanism kernels do not have
and, by part (b), cannot acquire at bounded rank and flat spectrum. The synthesis pays this price in
the cheapest currency available. The heavy-tailed agent contributes a single rank-one inflated
direction, with eigenvalue of order $\Delta$ by construction; the clustered agent keeps the flat,
triangle-bearing part; and the noisy-OR splices the two without destroying either, which is the
anchored-kernel geometry of part (d) realised by composition. The frontier therefore does more than
forbid. It identifies the spectral resource a four-property model must budget and shows the agent
decomposition to be a minimal way to budget it.
\end{remark}

\subsection{Synthesis attains all four properties}

Combining agents must create the four properties no single agent has, and clustering is the delicate point: superposing a heavy-tailed hub layer on a triangle-rich layer floods the hubs with mutually non-adjacent neighbours, and since global clustering is the ratio of closed to total wedges, the added open wedges can dilute it to zero even as the triangles remain. The proof (Appendix~\ref{app:structural}) uses edge monotonicity of the noisy-OR operator, so a union never destroys a triangle, which a mixture could not guarantee, and controls the dilution by restricting to the local clustering of non-hub vertices; the guarantee is local for typical vertices and global only under a hub-sparsity condition. Hubs, short paths, and communities survive superposition because degrees only increase, added edges only shorten distances, and the block signal is recovered by the regularised spectral analysis.

Consider the noisy-OR superposition (Definition~\ref{def:noisyor}) of a small-world agent (the
triangle layer, weight $w_1$), a heavy-tailed Chung--Lu agent with index $\tau\in(2,3)$ (the hub
layer, weight $w_2$), and an optional assortative block agent with $Q$ blocks (the community layer,
weight $w_3$). Let $\bar d_2$ be the mean degree contributed by the hub layer.

\begin{theorem}[Possibility via synthesis]
\label{thm:pos}
The superposed model just described satisfies all four properties:
\begin{enumerate}[label=(\alph*),leftmargin=2em,itemsep=2pt]
\item \emph{(Hubs)} The degree distribution is heavy-tailed with index $\tau$; in particular
$d_{\max}=n^{1/(\tau-1)+o(1)}$.
\item \emph{(Short paths)} Typical distances are $O(\log n)$.
\item \emph{(Communities)} If the community layer is present and its signal exceeds the
Kesten--Stigum threshold, its blocks are weakly recoverable from the superposed graph by regularised
spectral clustering.
\item \emph{(Clustering, local form)} There is a constant $c>0$, depending on $(m,w_1)$ and not on
$n$, such that the average local clustering of the vertices outside the top-$\delta$ degree quantile
satisfies $\bar C_{\mathrm{non\text{-}hub}}\ge c$ for every fixed $\delta\in(0,1)$. If in addition the
hub layer is sparse, in the sense that its second degree moment is of lower order than the
small-world wedge density, then the global clustering coefficient is bounded below by a constant.
\end{enumerate}
\end{theorem}

\begin{proof}
See Appendix~\ref{app:structural}.
\end{proof}

\section{Estimation, inference, and operator identification}

\begin{figure}[t]
\centering
\includegraphics[width=\textwidth]{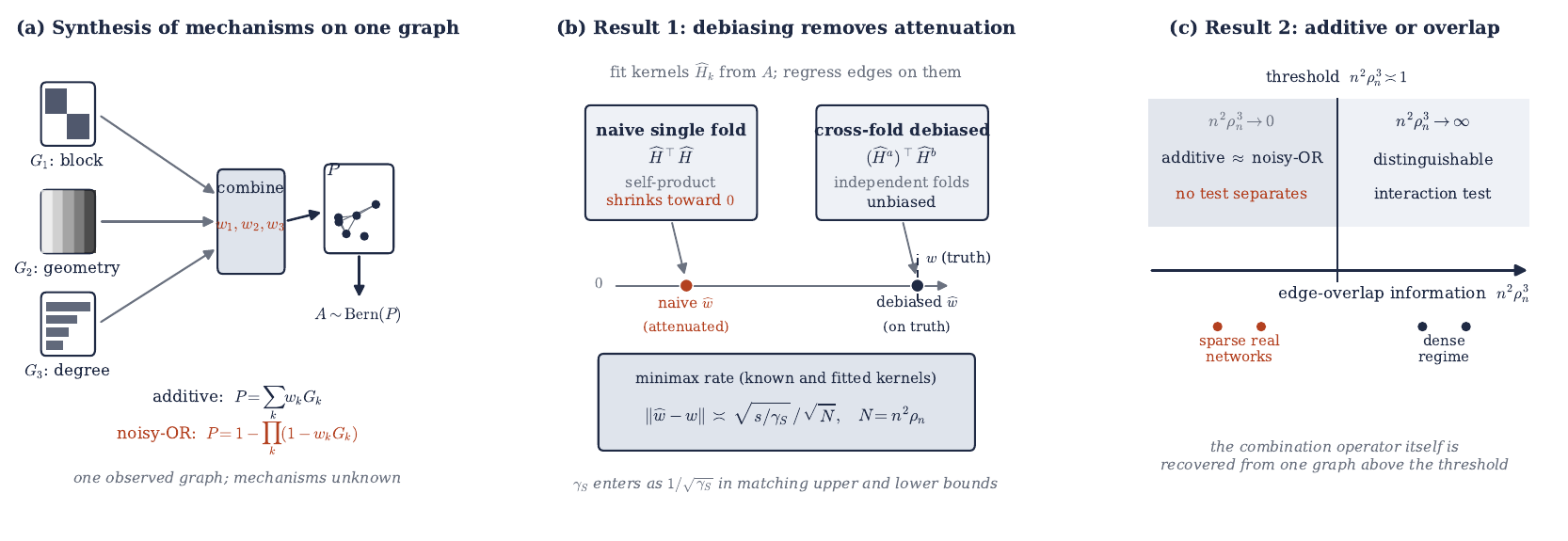}
\caption{The synthesis problem and the two results. (a)~A single graph is generated as a weighted combination of candidate mechanism kernels $G_k$ under an additive or a noisy-OR operator, then observed once. (b)~Result one: the naive single-fold plug-in builds the Gram from the self-product $\widehat H^{\top}\widehat H$ and attenuates the coefficients toward zero, while the cross-fold construction forms it from two independent fits $(\widehat H^{a})^{\top}\widehat H^{b}$, cancels the self-product, and attains the minimax rate $\sqrt{s/\gamma_S}/\sqrt{N}$ with $N=n^2\rho_n$. (c)~Result two: the additive and noisy-OR laws are contiguous below the edge-overlap threshold $n^2\rho_n^3\asymp1$ and separable above it, where a test on the interaction column is consistent.}
\label{fig:overview}
\end{figure}

Figure~\ref{fig:overview} sets out the problem and the two results in one view: a single graph is modelled as a weighted combination of candidate mechanism kernels under an additive or a noisy-OR operator, the cross-fold construction removes the attenuation that the naive plug-in incurs, and the combination operator is itself identifiable once the edge-overlap information diverges.

\label{sec:joint}

The structural results motivate modelling the generating object as a combination of mechanisms, and the four
estimation results of this section make that synthesis usable: the weights and the active set are estimable at
a sharp minimax rate, the inverse square root of the number of edges, after the cross-fold debiasing that the
same-graph fitting requires, since the naive plug-in attenuates and is rate-suboptimal while the debiased
estimator is sharp (Proposition~\ref{thm:joint}, Theorem~\ref{thm:debias}, Corollary~\ref{cor:adapt-cor}); the
estimator carries valid inference, centred under misspecification on an explicit projection whose negative
coordinates are corrective contrasts (Proposition~\ref{thm:clt}, with Proposition~\ref{prop:mis-cor} in the
supplement); selection is adaptive, knowing neither the active set, its size, nor its conditioning
(Corollary~\ref{cor:adapt-cor}); and the theory crosses to the nonlinear noisy-OR operator, ending at the
sharp boundary where the edge-overlap information $n^2\rho_n^3$ decides whether the combination rule is
testable at all (Theorem~\ref{thm:nor}). This is the inferential programme associated with the synthesis
coefficients, unavailable for a single fitted classical model.

Estimating the weight vector $\bw$ is, structurally, a regression problem. Under the mixture operator the
population edge-probability matrix is linear in the weights, $\Pmat=\sum_k w_k \mathbf{G}_k$ with
$\mathbf{G}_k=\E[\,P^{(k)}\,]$ the $k$th agent's Gram matrix, so the $\binom n2$ observed Bernoulli edges are
the responses and the agents' Gram matrices the $K$ regressors. Two facts drive the result: a graph supplies
$\Theta(n^2\rho_n)$ informative edge variables for a length-$K$ parameter, so $\bw$ concentrates an order of
magnitude faster than any node-level quantity; and the difficulty is governed by how distinguishable the
agents are as matrices, which we quantify by the smallest eigenvalue $\gamma_S$ of the Gram-correlation matrix
$\Phi$ restricted to the active agents, the network analogue of the restricted-eigenvalue (compatibility)
constant that controls the lasso. The non-obvious payoff, confirmed in simulation,
is that the conditioning enters as $1/\sqrt{\gamma_S}$, not $1/\gamma_S$, with the same
exponent in
both the upper and lower bounds, so the rate is sharp.

\paragraph{What is general, and what is specific to networks.} Several results below have a recognisable
general core: the known-design rate of Proposition~\ref{thm:joint} is least squares for a normalised linear
model with $\gamma_S$ as the restricted-eigenvalue constant; the limit theory of Proposition~\ref{thm:clt}
specialises held-out $M$-estimation and case-control logistic calibration
\citep{van2000asymptotic,prentice1979logistic}; and the selection of Corollary~\ref{cor:adapt-cor} is a
thresholding-and-refit argument. Stating these at their natural generality is a strength, since the network
conclusions inherit guarantees that are already sharp. What is specific to networks is concentrated in four
places: the regressors are estimated from the same graph, a generated-regressor design that forces the
cross-fold debiasing of Theorem~\ref{thm:debias} and the correction of Theorem~\ref{thm:twostage}; the
effective sample size $n^2\rho_n$ and the transversality $\gamma_S$ are graph functionals with no
scalar-regression counterpart, and the lower bound of Theorem~\ref{thm:est-lb} is over graph models; the
operator threshold $n^2\rho_n^3$ of Theorem~\ref{thm:nor} concerns edge-overlap information in a sparse
Bernoulli array; and the impossibility and frontier theorems of Section~\ref{sec:structural} concern
clustering, hubs, and the spectrum, which exist only for graphs.

\subsection{Estimands}
\label{sec:estimands}
Four targets arise, distinguished by the terminology used throughout. Under a correctly specified
mixture model the target is the generative weight vector $\bw_{\mathrm{gen}}\in\Delta^{K-1}$, whose
entries are the mechanism \emph{proportions}. Under a correctly specified noisy-OR model the target is the vector of layer
activation probabilities $\bw_{\mathrm{OR}}\in[0,1]^K$. When the truth lies outside the synthesis
class, the least-squares estimator targets the population projection
$\bw^\dagger_{\mathrm{LS}}\in\R^K$ of Proposition~\ref{prop:mis-cor}, whose coordinates are
\emph{projection coefficients} and may be negative; a negative coordinate is a corrective
contrast, not a negative proportion. The held-out logistic estimator of
Section~\ref{sec:estimation} targets the population calibration coefficient
$\bw^\dagger_{\mathrm{cal}}\in\R^K$ on the case--control dyad distribution, a \emph{predictive
coefficient} on the link scale; $\bw^\dagger_{\mathrm{LS}}$ and $\bw^\dagger_{\mathrm{cal}}$ are
distinct objects with the same qualitative sign interpretation under their respective losses.
The weight scale is fixed by the
normalisation of Assumption~\ref{ass:norm}: rescaling an agent kernel rescales its coefficient
inversely, so coefficient magnitudes are interpretable only relative to that common normalisation,
and all four targets are defined relative to the candidate set.

\subsection{Setup and assumptions}
Work under the mixture operator (Definition~\ref{def:mixture}) with $K$ candidate agents. Agent $k$
has population Gram matrix $\mathbf{G}_k\in[0,1]^{n\times n}$ (the matrix of its proposed edge
probabilities $P^{(k)}_{ij}$), of rank $d_k$, and the population edge-probability matrix is
\[
\Pmat=\sum_{k=1}^K w_k\,\mathbf{G}_k=\sum_{k\in S}w_k\,\mathbf{G}_k,
\qquad S=\operatorname{supp}(\bw),\quad s=\abs{S}.
\]
We observe $A_{ij}\sim\mathrm{Bernoulli}(P_{ij})$ independently for $i<j$, conditionally on the
agents' latent attributes. Let $\rho_n=\max_{ij}P_{ij}$ be the density scale. Write the
\emph{Gram-correlation matrix} $\Phi\in\R^{K\times K}$ with entries
\[
\Phi_{k\ell}=\frac{\inner{\mathbf{G}_k}{\mathbf{G}_\ell}_F}{\norm{\mathbf{G}_k}_F\,
\norm{\mathbf{G}_\ell}_F},\qquad
\inner{\mathbf{G}_k}{\mathbf{G}_\ell}_F=\sum_{i<j}\mathbf{G}_{k,ij}\mathbf{G}_{\ell,ij},
\]
and let $\Phi_{S}$ be its restriction to the active agents.

\begin{assumption}[Transversality]\label{ass:transversality}
The active agents are uniformly distinguishable: $\gamma_S:=\lambda_{\min}(\Phi_S)\ge\gamma>0$ for
all $n$. Equivalently, no active agent's Gram matrix is asymptotically in the span of the others (the
matrix analogue of the column-space transversality of the stacked representation
(Remark~\ref{rem:stack})).
\end{assumption}

\begin{assumption}[Density regime]\label{ass:density}
$\rho_n\to0$ with $n\rho_n/\log n\to\infty$; for example, $\rho_n=n^{-\alpha+o(1)}$ for some
$\alpha\in(0,1)$. This is the standard regime in which spectral estimates of each $\mathbf{G}_k$
concentrate. Individual results require this regime at specific scales: the estimation and inference
results (the results of Section~\ref{sec:joint}) use only $n\rho_n/\log n\to\infty$, while the
operator-detectability result (Theorem~\ref{thm:nor}(e) and Corollary~\ref{cor:nork-cor}) requires the
denser scale $n^2\rho_n^3\to\infty$, equivalently $\alpha<2/3$. These requirements are stated with
each result.
\end{assumption}

\begin{assumption}[Estimable agents]\label{ass:plugin}
Each active agent admits a consistent estimator $\widehat{\mathbf{G}}_k$ (e.g.\ ASE for the
dot-product agent, spectral clustering plus block means for the block agent, degree matching for the
degree agent) with relative Frobenius error
$\norm{\widehat{\mathbf{G}}_k-\mathbf{G}_k}_F/\norm{\mathbf{G}_k}_F=O_P(\sqrt{r_k/(n\rho_n)})$ ($r_k$ the latent rank of agent $k$, $r_{\max}=\max_k r_k$), and
the active Gram matrices have bounded condition number.
\end{assumption}

\begin{assumption}[Design normalisation and information comparability]\label{ass:norm}
The agent kernels share a common scale, $c_\star n^2\rho_n^2\le\norm{\mathbf{G}_k}_F^2\le
C_\star n^2\rho_n^2$ for all $k$ and fixed constants $0<c_\star\le C_\star<\infty$ (for
example, every agent has mean edge probability $\rho_n$), the entries $P_{ij}$ lie in
$[c_\star\rho_n,\rho_n]$ on a fixed positive fraction of dyads, and the information matrix is
comparable to the Gram matrix:
$c\,\rho_n\,\mathbf{M}^\top\mathbf{M}\preceq\mathbf{M}^\top\mathbf{D}\mathbf{M}\preceq
C\,\rho_n\,\mathbf{M}^\top\mathbf{M}$ with
$\mathbf{D}=\operatorname{diag}\{P_{ij}(1-P_{ij})\}$. Coefficient magnitudes are interpretable
only relative to this normalisation: replacing $\mathbf{G}_k$ by $c_k\mathbf{G}_k$ rescales
$w_k$ by $1/c_k$.
\end{assumption}

Local kernels supported on $O(n)$ dyads, such as the ring or small-world agent of
Section~\ref{sec:setup}, do not satisfy the positive-fraction condition of
Assumption~\ref{ass:norm} and enter the paper as structural motivation and empirical scores rather
than as agents covered by the estimation and selection results of Section~\ref{sec:joint}; a
weighted-support extension that covers them, and heavy-tailed degree kernels, under their own
information scales is given in Appendix~\ref{app:effective-info}.

\paragraph{Effective-information extension.}
The common-scale fitted-kernel theory covers diffuse block, degree-corrected
block, and dot-product mechanisms at the edge-count rate
$1/(n\sqrt{\rho_n})$. The extension in Appendix~\ref{app:effective-info}
covers two additional motivating mechanisms under their own information
scales. A truncated regularly varying degree-product kernel with truncation
ratio $R_n=d_{\max,n}/\bar d_n$ has unweighted least-squares information
\[
    N_{\deg}
    \asymp
    \frac{n\bar d_n}{R_n^{2(\tau-2)}}
    =
    \frac{n^2\bar\rho_n}{R_n^{2(\tau-2)}},
\]
where $\bar\rho_n=\bar d_n/n$ and $\tau\in(2,3)$ is the tail index. A matching minimax lower bound over a degree-product nuisance class
(Theorem~\ref{thm:heavy-tail-projection-lb}) shows this rate is sharp for the residualized degree-column
projection coefficient, while the correctly specified scalar Chung--Lu weight is estimable faster, at rate
$(n\bar d_n)^{-1/2}$. A local
support mechanism on $q_n$ dyads with edge scale $p_{\triangle,n}$ has
information $N_\triangle\asymp q_np_{\triangle,n}$, with a matching one-dimensional
lower bound. Thus the same cross-fold debiasing principle applies beyond the
common-scale case, but the rates reveal the statistical price of hubs and
locality. Same-graph triadic scores such as Adamic--Adar and Jaccard remain
outside this extension unless their supports are constructed on an independent
pilot split or a separate support-estimation expansion is proved.

\paragraph{Notation.} Several symbols carry context-dependent roles; the following disambiguation is
used throughout. $\rho_n$ is the density scale of Assumption~\ref{ass:density}; $\bar\rho=\bar d/n$
is the empirical average-density proxy of Table~\ref{tab:opfeas}. $r_k$ is the latent rank of agent
$k$, $r_{\max}$ their maximum, and $d$ a generic embedding dimension. The conditioning family is
$\gamma_S$ (active set), $\gamma_{\mathrm{full}}$ (all candidates), and the interaction transversalities
$\widetilde\gamma$, $\widetilde\gamma_j$, $\widetilde\Gamma_2$ of Section~\ref{sec:nor}. $\Phi$
is the Gram correlation matrix at unit scale, while $\widehat\Phi_{\mathrm{db}}$ is the cross-fold
Gram at raw scale $n^2\rho_n^2$, the two related by the normalisation $\mathbf
D_n=\diag(\norm{\widehat{\mathbf m}_k}_2)$. $\mathbf Q_k$ denotes Stage-A remainder matrices
(Assumption~\ref{ass:linplug}); $Q$ also counts blocks where a blockmodel is specified, context
disambiguating, and in Lemma~\ref{lem:firststage}(b) $\mathbf d$ is the degree vector. $m$ is the calibration sample size in Proposition~\ref{thm:clt} and a ring
distance only in the worked example of Section~\ref{sec:setup}. $\delta_n=\sqrt{r_{\max}/(n\rho_n)}$
is the first-stage relative error scale. The dyad count is $\binom n2$ and the effective sample size is $N=n^2\rho_n$, both distinct from the calibration sample size $m$; $\tau$ with a subscript is always a selection threshold.
All inner products in $\widehat\Phi_{\mathrm{db}}$ and in the moment vector are taken over the
Stage-B dyads $\mathcal D_2$, and every inferential statement concerns this single declared fold assignment, with randomness
taken jointly over the graph and the Stage-A inclusion masks.

\begin{table}[t]
\centering
\small
\caption{Operator-test feasibility by network. The edge-overlap information is $n^2\rho_n^3\approx\bar d^{\,3}/n$; only \textsf{polblogs} exceeds the order-one threshold of Theorem~\ref{thm:nor}.}
\label{tab:opfeas}
\begin{tabular}{lccc}
\toprule
Network & $n$ & $\bar d$ & $\bar d^{\,3}/n$\\
\midrule
\textsf{polblogs}   & $1222$  & $27.4$ & $16.8$\\
\textsf{les mis}    & $77$    & $6.6$  & $3.7$\\
\textsf{GoT}        & $107$   & $6.6$  & $2.7$\\
\textsf{ca-GrQc}    & $4158$  & $6.5$  & $0.065$\\
\textsf{ca-CondMat} & $21363$ & $8.6$  & $0.029$\\
power grid          & $4941$  & $2.7$  & $0.004$\\
\bottomrule
\end{tabular}
\end{table}

\subsection{The estimator: cross-fitting on dyads}
\label{sec:cf}
Fitting every kernel on the full adjacency $\A$ and then regressing $\A$ on the fitted kernels makes
the regression error and the estimated design statistically dependent, and controlling this
dependence requires further assumptions. The construction developed here estimates the kernels and
the coefficients on disjoint sets of dyads, joined through an inverse-probability-weighted fold
adjacency, so that the two are independent.

Assign each dyad in $\mathcal D=\{(i,j):i<j\}$ to $\mathcal D_1$ or $\mathcal D_2$ by an
independent Bernoulli($\pi$) coin, $\pi=\tfrac12$, independently of $\A$ and across dyads. The
iid (not exact-half) masking is load bearing: independence of the masked entries drives every
concentration and Hanson--Wright step of Lemma~\ref{lem:firststage}, an exact-half partition
introducing a negative-correlation correction we do not pursue. Define the fold-one
adjacency
\[
\widetilde A^{(1)}_{ij}=\pi^{-1}A_{ij}\,\mathbf 1\{(i,j)\in\mathcal D_1\},\qquad i<j,
\]
symmetrised, so that conditionally on the latent attributes
$\E[\widetilde\A^{(1)}\mid\bu]=\Pmat$ and $\widetilde\A^{(1)}-\Pmat$ has independent mean-zero
entries bounded by $\pi^{-1}$ with variances at most $2\rho_n$: an unbiased, same-noise-scale
surrogate for $\A$ from half the dyads. Writing
$R_{ij}=\mathbf 1\{(i,j)\in\mathcal D_1\}$, the identity
$\widetilde A^{(1)}_{ij}-P_{ij}=\pi^{-1}R_{ij}(A_{ij}-P_{ij})+(\pi^{-1}R_{ij}-1)P_{ij}$ shows that
unbiasedness holds over the joint randomness of the fold draw and the graph, while conditionally on a realised
fold only the first term is mean zero; the first-stage expansions of Assumption~\ref{ass:linplug} are
accordingly stated for this IPW-masked matrix over the randomness of $(R,\A)$, and the Stage-B analysis
conditions on the fitted Stage-A design, never on the fold indicator alone. \emph{Stage A:} form
$\widehat{\mathbf{G}}_k^{(1)}$ for each candidate agent by its own consistent estimator
(Assumption~\ref{ass:plugin}) applied to $\widetilde\A^{(1)}$. \emph{Stage B:} regress the held-out
half of the adjacency on the fold-one bases,
\[
\widehat{\bw}^{(2)}_{\mathrm{ols}}=\argmin_{\bw\in\R^{K}}
\sum_{(i,j)\in\mathcal D_2}\Bigl(A_{ij}-\textstyle\sum_k w_k\widehat G^{(1)}_{k,ij}\Bigr)^{2},
\qquad \widehat\bw^{(2)}=\Pi_{\Delta^{K-1}}\bigl(\widehat{\bw}^{(2)}_{\mathrm{ols}}\bigr),
\]
where $\Pi_{\Delta^{K-1}}$ denotes Euclidean projection onto the simplex; since the truth lies in
$\Delta^{K-1}$ under correct specification and Euclidean projection is $1$-Lipschitz, the
projection step can only reduce the Euclidean error, and every bound below proved for
$\widehat{\bw}_{\mathrm{ols}}$ transfers to $\widehat\bw$ verbatim. (For the projection target of
Proposition~\ref{prop:mis-cor} the projection step is omitted.) Swap the folds and average:
$\widehat\bw=\tfrac12(\widehat\bw^{(1)}+\widehat\bw^{(2)})$. The active
set is recovered by thresholding, $\widehat S=\{k:\widehat w_k>\tau_n\}$, for a threshold $\tau_n$
specified below. Conditionally on fold one, the Stage-B errors
$\{A_{ij}-P_{ij}:(i,j)\in\mathcal D_2\}$ are independent of the design
$\{\widehat G^{(1)}_{k,ij}\}$: every stochastic interaction between plug-in error and regression
noise is eliminated \emph{by construction}, not by assumption. The held-out calibration estimator of
Section~\ref{sec:estimation}, which we use throughout the experiments, is exactly this estimator
with a logistic link.

\paragraph{Cross-fold debiasing of the Gram matrix.}
Fold independence removes the interaction between plug-in error and Stage-B noise, but it does not by
itself remove a second, purely deterministic effect: the Stage-A error biases the Gram matrix. Write
$\widehat{\mathbf M}=\mathbf M+\mathbf H$ with $\mathbf H$ the matrix of fitted-kernel errors. The
single-fold normal equations give
$\widehat\bw-\bw=-(\widehat{\mathbf M}^\top\widehat{\mathbf M})^{-1}\widehat{\mathbf M}^\top\mathbf
H\bw+(\text{mean-zero})$, and $\widehat{\mathbf M}^\top\mathbf H=\mathbf M^\top\mathbf H+\mathbf
H^\top\mathbf H$. The self-product $\mathbf H^\top\mathbf H$ is the errors-in-variables attenuation:
its diagonal is of order $r_{\max}n\rho_n$, so it contributes a bias of order $\sqrt
s\,r_{\max}/(n\rho_n)$, which exceeds the edge rate
$\sqrt{s/\gamma_S}/(n\sqrt{\rho_n})$ by the factor $r_{\max}\sqrt{\gamma_S/\rho_n}\to\infty$ in the
sparse regime. The single-fold plug-in therefore attenuates the coefficients toward zero, and no
choice of threshold corrects this.

We remove the attenuation by splitting the Stage-A dyads once more, into independent halves
$\mathcal D_{1a},\mathcal D_{1b}$, fitting each kernel twice,
$\widehat{\mathbf G}^{(1a)}_k$ and $\widehat{\mathbf G}^{(1b)}_k$, on the two halves so that their
errors $\mathbf H_a,\mathbf H_b$ are independent given the latent attributes. The
\emph{cross-fold debiased estimator} forms the Gram from the two independent copies and the moment
from their average,
\[
\widehat\Phi_{\mathrm{db}}=\tfrac12\bigl(\widehat{\mathbf M}_a^\top\widehat{\mathbf M}_b+
\widehat{\mathbf M}_b^\top\widehat{\mathbf M}_a\bigr),\qquad
\widehat\bw_{\mathrm{db}}=\widehat\Phi_{\mathrm{db}}^{-1}\,\tfrac12\bigl(\widehat{\mathbf M}_a+
\widehat{\mathbf M}_b\bigr)^\top\A^{(2)},
\]
The matrix $\widehat\Phi_{\mathrm{db}}$ is symmetric but not positive semidefinite by
construction; Lemma~\ref{lem:plugin} gives $\lambda_{\min}(\widehat\Phi_{\mathrm{db}})\ge\tfrac12
c_\star\gamma_S n^2\rho_n^2$ with probability tending to one, and on the complementary event the
estimator is defined through the eigenvalue-floored solve at level $\varepsilon n^2\rho_n^2$, a
modification without asymptotic effect.

with $\widehat\bw_{\mathrm{db}}$ projected onto the simplex and averaged over the fold assignment. The
construction removes the leading self-product rather than all bias: conditionally on the latent attributes,
\[
\E[\widehat\Phi_{\mathrm{db}}\mid\bu]=\mathbf M^\top\mathbf M
+\operatorname{sym}\bigl\{\mathbf M^\top\E[\mathbf H_b\mid\bu]+\E[\mathbf H_a\mid\bu]^\top\mathbf M\bigr\}
+\operatorname{sym}\bigl\{\E[\mathbf H_a\mid\bu]^\top\E[\mathbf H_b\mid\bu]\bigr\},
\]
in which the same-fold self-product $\mathbf H^\top\mathbf H$ that produced the attenuation is absent, the
remaining terms involving only the deterministic first-stage biases, negligible at the projected scale under
Assumption~\ref{ass:linplug} and condition~(C2) of Theorem~\ref{thm:debias}; the surviving stochastic
first-stage contribution is the mean-zero linear term $\mathbf M^\top\mathbf H$, of edge-rate order.
Figure~\ref{fig:estsim}(c) shows the effect and Theorem~\ref{thm:debias} proves it.

One further assumption is the interface that every fitted-kernel result uses; it is stated here,
next to the estimator it concerns, and verified for the named agents in Appendix~\ref{app:plugin}.

\begin{assumption}[Projected fold-one remainders]\label{ass:linplug}
For each $k\in[K]$ there is a linear map $\mathcal{L}_k$ on symmetric matrices such that
$\widehat{\mathbf{G}}^{(1)}_k-\mathbf{G}_k=\mathcal{L}_k(\mathbf{E}^{(1)})+\mathbf{Q}_k$, where
(i) $\norm{\mathcal{L}_k(\mathbf{E}^{(1)})}_F=O_P(\delta_n\norm{\mathbf{G}_k}_F)$ and the adjoint is
bounded, $\sup_{\norm{\mathbf{V}}_F\le1}\norm{\mathcal{L}_k^*(\mathbf{V})}_F\le C_L$;
(ii) the remainders are negligible in projection, in the scale-homogeneous form: for every
symmetric $\mathbf X$ with the test-class shape $\norm{\mathbf X}_{\max}\le C_0\norm{\mathbf
X}_F/n$, a class containing the design columns, the misspecification residual of
Proposition~\ref{prop:mis-cor}, and the inverse-Gram directions $\widehat{\mathbf G}(\mathbf v_k)$,
$\max_{k\le K}\,\abs{\inner{\mathbf X}{\mathbf{Q}_k}_F}=o_P\bigl(\sqrt{\gamma_{\mathrm{full}}\,
\rho_n}\,\norm{\mathbf X}_F\bigr)$ (at $\norm{\mathbf X}_F\asymp n\rho_n$ this is the absolute
bound $o_P(\sqrt{\gamma_{\mathrm{full}}}\,n\rho_n^{3/2})$, and at the inverse-Gram scale it is
$o_P$ of the per-coordinate noise, which is what the proofs use);
(iii) \emph{(Entrywise spreading.)} The linear parts are delocalised:
$\max_{e}\operatorname{Var}\bigl([\mathcal{L}_k(\mathbf{E})]_e\bigr)=O(r_{\max}\rho_n/n)$, so the
squared mass $\norm{\mathcal{L}_k(\mathbf{E})}_F^2\asymp r_k\,n\rho_n$ is spread over the
$\Theta(n^2)$ dyads. Moreover the covariance operator of each linear part is bounded,
$\norm{\operatorname{Cov}(\operatorname{vec}\mathcal L_k(\mathbf E))}_{\mathrm{op}}\le
C_L^2\rho_n/f$. The entries of $\mathcal L_k(\mathbf E)$ are correlated, so cross-fold inner
products over independent sub-folds are controlled by the trace bound, never entrywise:
$\operatorname{Var}\inner{\mathcal{L}_k(\mathbf{E}_a)}{\mathcal{L}_l(\mathbf{E}_b)}=
\operatorname{tr}(\mathbf C_k\mathbf C_l)\le\norm{\mathbf C_k}_{\mathrm{op}}\,
\E\norm{\mathcal{L}_l(\mathbf{E}_b)}_F^2=O(r_{\max}\,n\rho_n^2/f^2)$, so the inner product is
$O_P(\sqrt{r_{\max}n}\,\rho_n/f)$, the bound used in Lemma~\ref{lem:plugin},
Theorem~\ref{thm:twostage}, and Theorem~\ref{prop:fitted-operator}. The named
agents satisfy (i) to (iii) by Lemma~\ref{lem:firststage}, whose masked-sampling constants carry
the explicit factor $1/f$, with $f$ the sub-fold sampling fraction, absorbed here for fixed fold
fractions.
\end{assumption}

\subsection{Result one: cross-fold debiasing and the minimax rate}

Before stating the rate, we make the transversality constant concrete on a specific candidate set.

\begin{example}[A concrete candidate set and its rate]\label{ex:gammaS}
Take three mechanisms on $n$ vertices: a two-block assortative SBM agent, a rank-two dot-product
agent with latent positions in the positive orthant, and a Chung--Lu hub agent with a power-law
propensity sequence of index $2.1$. The three population Gram matrices are linearly independent but
correlated, with pairwise Gram-correlations $\Phi_{12}\approx0.80$ (block and dot-product),
$\Phi_{13}\approx0.48$ (block and hub), and $\Phi_{23}\approx0.53$ (dot-product and hub); the
restricted Gram-correlation matrix $\Phi_S$ then has smallest eigenvalue $\gamma_S\approx0.20$.
Proposition~\ref{thm:joint} gives the three weights estimable at the rate
$\sqrt{s/\gamma_S}/(n\sqrt{\rho_n})=\sqrt{3/0.20}/(n\sqrt{\rho_n})\approx3.9/(n\sqrt{\rho_n})$, the
inverse square root of the edge count inflated by the conditioning. Better-separated or denser
candidate sets raise $\gamma_S$ toward one and shrink the constant; two nearly proportional agents,
such as two assortative SBMs with similar block matrices, drive $\gamma_S\to0$ and degrade the rate,
the network analogue of multicollinearity in ordinary regression.
\end{example}

\begin{proposition}[Known-design benchmark]
\label{thm:joint}
Let Assumptions~\ref{ass:transversality}--\ref{ass:norm} hold. Part (a) concerns the
cross-fitted estimator of Section~\ref{sec:cf} computed on the active set $S$ (the oracle-set
estimator); the fully data-driven full case, including selection, is Corollary~\ref{cor:adapt-cor}.
\begin{enumerate}[label=(\alph*),leftmargin=2em,itemsep=2pt]
\item \emph{(Upper bound, known design, uniform over the class.)} With the candidate kernels known, the
cross-fitted least-squares estimator of Section~\ref{sec:cf} satisfies
\[
\sup_{\Theta(s,\gamma_S,\rho_n)}\E\bigl\|\widehat{\bw}-\bw\bigr\|_2\ \le\ C\,
\sqrt{\frac{s}{\gamma_S}}\,\frac{1}{n\sqrt{\rho_n}}\ =\ O\bigl(\sqrt{(s/\gamma_S)/N}\bigr),\qquad
N=\Theta(n^2\rho_n),
\]
with $C$ absolute over the class $\Theta(s,\gamma_S,\rho_n)$ defined by
Assumptions~\ref{ass:density} to~\ref{ass:norm}; the constants in Appendix~\ref{app:joint} are
class uniform, and the pointwise $O_P$ form follows,
the inverse square root of the effective edge count of Assumption~\ref{ass:norm}. With kernels
\emph{fitted} from the graph, the naive single-fold plug-in is not rate-optimal: its Gram matrix
carries the errors-in-variables self-product $\mathbf H^\top\mathbf H$ of
Section~\ref{sec:cf}, which attenuates the coefficients and contributes a bias of order
$r_{\max}\norm{\bw_S}_2/(n\rho_n)$, exceeding the edge rate by $r_{\max}\sqrt{\gamma_S/\rho_n}$
when $s$ is fixed with active weights bounded below (Corollary~\ref{cor:debias-lb}). The
cross-fold debiased estimator removes this term and attains the known-design rate above; this is
Theorem~\ref{thm:debias}. Conditionally on the folds, cross-fitting removes the first-order
stochastic dependence between Stage-A estimation and Stage-B regression noise by construction, and
cross-fold debiasing removes the deterministic Gram attenuation.
\item \emph{(Lower bound, known-design normalised regression class.)} Over the known-design class of normalised dyadic regression designs
\begin{multline*}
\Theta(s,\gamma,\rho)=\Bigl\{\,\Pmat=\textstyle\sum_{k\in S}w_k\mathbf{G}_k:\ \bw\in\Delta^{s-1},\\
\lambda_{\min}(\Phi_S)\ge\gamma,\ \text{Assumption~\ref{ass:norm} at scale }\rho\,\Bigr\},
\end{multline*}
in which the statistician knows $\{\mathbf{G}_k\}$ and estimates only $\bw$, every estimator obeys
\[
\inf_{\widetilde{\bw}}\ \sup_{\Theta(s,\gamma,\rho)}\ \E\bigl\|\widetilde{\bw}-\bw\bigr\|_2\ \ge\
c\,\sqrt{\frac{s}{\gamma}}\,\frac{1}{n\sqrt{\rho}}\,,
\]
provided $s^2\le c'\gamma n^2\rho$. The known-design rate of part (a) is therefore minimax
optimal over this class; since estimating the design cannot help, the same lower bound applies to
the estimated-basis problem.
\end{enumerate}
\end{proposition}

\noindent Support recovery is not part of the oracle-set statement, since the active
set $S$ is used to compute the estimator. It is established for the data-driven full estimator
in Corollary~\ref{cor:adapt-cor}, where selection is governed by the coordinatewise standard error and the
threshold carries a $\sqrt{\log K}$ factor rather than the $\sqrt s$ of the $\ell_2$ estimation rate.

\begin{proof}
See Appendix~\ref{app:joint}.
\end{proof}

\begin{assumption}[Fidelity of the fitting maps]
\label{ass:fid}
For every active agent the population fitting map reproduces the component kernel at unit scale:
$F_k(\Pmat)=\mathbf G_k$ for all $k\in S$, each fitted kernel being renormalised to unit latent
scale before entering the design ($\E\norm{x}^2=1$ in the dot-product case, unit block-mean scale
in the block case).
\end{assumption}

\begin{remark}[Three estimands and the separation that identifies the weights]
\label{rem:fid}
With known kernels the target is the generative weight vector $\bw$ (Proposition~\ref{thm:joint}); with
fitted kernels the regression target is, without further assumptions, the coordinate vector of $\Pmat^\ast$ in
the span of the population fitting maps, the projection of Proposition~\ref{prop:mis-cor}, and
Theorems~\ref{thm:debias} and~\ref{thm:twostage} are proved for this target. Under Assumption~\ref{ass:fid}
and correct specification the two coincide, and only then do the fitted-kernel results carry the generative
reading. Fidelity is not vacuous: applied to a mixture, spectral embedding recovers the stacked positions
$z=(\sqrt{w_1}\,x,\sqrt{w_2}\,y)$, so a component kernel rebuilt from a recovered block carries the factor
$w_k$, which the unit-scale renormalisation divides out. A verifiable sufficient condition is disjoint
spectral supports: if the active kernels' eigenvalue bands are separated by the gap of~(C2), the eigenspaces
of $\Pmat$ split along the agents, each rank-$r_k$ smoother localises on its own band and returns
$w_k\mathbf G_k$, and renormalisation removes the factor, so $F_k(\Pmat)=\mathbf G_k$ with $\gamma_S$ bounded
below by the squared principal angle between bands. This also marks the boundary of the assumption: two block
agents on a shared partition have fitting maps converging to the same best block approximation of $\Pmat$,
fitted transversality vanishes, and fidelity fails, so Assumptions~\ref{ass:transversality} and~\ref{ass:fid}
are jointly available only for heterogeneous separated candidate sets, and any least-favourable family for an
estimated-design lower bound must be built from such pairs (Theorem~\ref{thm:est-lb}), which is how the minimax rate is established for the projection target itself and, on the faithful subfamily, for the generative weights. Throughout the fitted-kernel results we write $\bw$ for the target in this sense.
\end{remark}

\begin{theorem}[Debiased estimation attains the minimax rate]
\label{thm:debias}
Let the candidate set consist of agents of the following kinds, each fitted on its Stage-A sub-fold by the
estimator named: a stochastic block model and its degree-corrected variant, each implemented as the rank-$r$
spectral smoother of Lemma~\ref{lem:firststage}(a); a
random dot-product graph of bounded latent rank fitted by adjacency spectral embedding; and a
Chung--Lu degree kernel fitted by degree matching. Assume $\gamma_S\ge\gamma_0>0$ and
\begin{enumerate}[label=(C\arabic*),leftmargin=2.6em,itemsep=1pt]
\item \emph{(sparsity)} $n\rho_n/\log n\to\infty$;
\item \emph{(rank and gap)} the latent ranks are bounded, $r_{\max}=O(1)$, and the signal
eigenvalues of the spectral agents are separated from the bulk by a gap of order $n\rho_n$, so the
embedding linearisation has projected fold-one remainders of negligible order
(Assumption~\ref{ass:linplug});
\item \emph{(degree regularity)} the degree propensities are comparable,
$\max_i\theta_i/\min_i\theta_i=O(1)$, and the Chung--Lu kernel is untruncated
(Assumption~\ref{ass:norm});
\item \emph{(transversality and scale)} $\gamma_S$ is bounded below and the kernels share a common
scale (Assumptions~\ref{ass:transversality} and~\ref{ass:norm}).
\end{enumerate}
Let $\widehat\bw_{\mathrm{db}}$ be the cross-fold debiased estimator of Section~\ref{sec:cf}, using two
independent Stage-A sub-folds. Then
\[
\bigl\|\widehat\bw_{\mathrm{db}}-\bw\bigr\|_2\ =\ O_P\!\left(\sqrt{\frac{s}{\gamma_S}}\,
\frac{1}{n\sqrt{\rho_n}}\right),
\]
uniformly over candidate sets of the named kinds satisfying (C1)--(C4); this matches the known-design minimax rate of Proposition~\ref{thm:joint}.
\end{theorem}

\noindent Condition (C3) requires comparable degree propensities, so the theorem covers the
regularised, bounded-heterogeneity degree kernel; the heavy-tailed Chung--Lu agent with tail exponent
$\tau\in(2,3)$, used for structural motivation, has $\max_i\theta_i/\min_i\theta_i\to\infty$ and is
not covered. Extending the debiasing to regularly varying degree propensities under truncation is a
separate first-stage analysis that we do not carry out here.

\noindent Theorem~\ref{thm:debias} closes the gap between the estimator that is computed and the
rate that is established. The known-design result of Proposition~\ref{thm:joint} assumes the kernels are given; in
every application they are fitted from the same graph, and the fitting error enters the second-stage
Gram matrix as the self-product $\mathbf H^\top\mathbf H$, which is positive and biases the
coefficients toward zero by an amount that, in the sparse regime, is larger than the quantity being
estimated. The bias is an artefact of reusing a single fitted design on both sides of the normal equations;
fitting each kernel twice on independent data makes the cross-product mean-zero, so the bias vanishes
at first order and the only surviving first-stage effect is the linear term, of edge-rate order. The four kinds of agent named
are those for which the required first-stage expansion is available in the literature, so the
condition is established for these agents rather than assumed.

\begin{proof}
See Appendix~\ref{app:debias}.
\end{proof}

\begin{corollary}[Separation: the naive plug-in is rate-suboptimal]
\label{cor:debias-lb}
Under the conditions of Theorem~\ref{thm:debias}, suppose $s$ is fixed with active weights bounded
below and $r_{\max}\sqrt{\gamma_S/\rho_n}\to\infty$ (for instance, fixed rank, $\gamma_S$ bounded
below, and $\rho_n\to0$). Let $\widehat\bw_{\mathrm{naive}}$ be the single-fold
plug-in estimator that uses one fitted design $\widehat{\mathbf M}$ in both the Gram matrix and the
moment. Assume further the non-cancellation condition
\[
\text{(D)}\qquad\sum_{k\ne l\in S}\bigl|\inner{\mathcal L_k(\mathbf E)}{\mathcal L_l(\mathbf E)}\bigr|
\ =\ o_P\Bigl(\min_{k\in S}\norm{\mathcal L_k(\mathbf E)}_F^2\Bigr),
\]
under which $\norm{\mathbf H\bw}_2^2\ge(1-o_P(1))\sum_{k\in S}w_k^2\norm{\mathbf H_k}_F^2$ on the
active coordinates. Then there is a constant $c>0$ such that, for the degree-product (Chung--Lu) and
dot-product agents,
\[
\bigl\|\widehat\bw_{\mathrm{naive}}-\bw\bigr\|_2\ \ge\ c\,\frac{r_{\max}\,\norm{\bw_S}_2}{n\rho_n}
\qquad\text{with probability tending to one,}
\]
and the right-hand side exceeds the minimax rate $\sqrt{s/\gamma_S}/(n\sqrt{\rho_n})$ by the factor
\[
r_{\max}\norm{\bw_S}_2\sqrt{\gamma_S/(s\rho_n)}\ \asymp\ r_{\max}\sqrt{\gamma_S/\rho_n}\ \to\ \infty
\]
in this regime. Under balanced simplex weights $\norm{\bw_S}_2\asymp s^{-1/2}$, so any bound carrying
an explicit $\sqrt s$ requires active weights bounded below; the display above does not. No choice of tuning removes this term, since it is the
deterministic errors-in-variables bias $-(\widehat{\mathbf M}^\top\widehat{\mathbf M})^{-1}\mathbf
H^\top\mathbf H\bw$ of Section~\ref{sec:cf}. Consequently the cross-fold debiased estimator of
Theorem~\ref{thm:debias} attains the known-design lower bound of Proposition~\ref{thm:joint}, which also lower-bounds the
harder estimated-design problem, while the naive plug-in does not: the two estimators are separated
in order.
\end{corollary}

\noindent Corollary~\ref{cor:debias-lb} turns the debiasing from an improvement into a separation: the
naive plug-in is rate-suboptimal while the debiased estimator attains the known-design minimax rate, and
because the bias is deterministic and one-signed it survives any thresholding or tuning. Condition~(D) is a
substantive restriction whose content for the named pair is an overlap bound, and a verifiable sufficient
condition is asymptotic orthogonality of the degree profile to the active spectral range,
$\norm{\mathbf P_U\mathbf d}_2^2=o(\norm{\mathbf d}_2^2)$, under which the same-fold overlap mean and its
Hanson--Wright fluctuation are negligible against $\min_{k\in S}\norm{\mathcal L_k(\mathbf E)}_F^2\asymp
n\rho_n/f$. When the leading eigenvector aligns with the degree profile, as for strongly degree-heterogeneous
mixtures, the two fitted errors share a direction and~(D) can fail, exactly the small-$\gamma_S$ regime that
Section~\ref{sec:infval} flags; (D) must be checked candidate set by candidate set, as done here for the
orthogonal case.
\begin{remark}[Edges, not nodes: an order-of-magnitude faster rate]
\label{rem:fastrate}
Proposition~\ref{thm:joint} says the synthesis weights are learned at rate $1/\sqrt{N}$ in the
effective number of edges $N=\Theta(n^2\rho_n)$, whereas the node-level latent positions of
Remark~\ref{rem:stack} converge at the slower sparse-regime rate $\sqrt{\log n/(n\rho_n)}$. The ratio
is $\sqrt{n\log n}\to\infty$: the global mixing structure is identifiable far more accurately than
any individual vertex embedding, because every one of the $\Theta(n^2\rho_n)$ edges carries
information about $\bw$ while only $n$ vertices carry information about the positions. The mechanism
weights are therefore an easier object than the vertex positions, which is what makes data-driven
synthesis practical from a single network. Figure~\ref{fig:weights} is consistent with all three parts: the $1/n$ rate (panel a), the
$1/\sqrt{\gamma_S}$ conditioning (panel b), and consistent support recovery (panel c).
\end{remark}

\begin{figure}[t]
\centering
\includegraphics[width=0.86\textwidth]{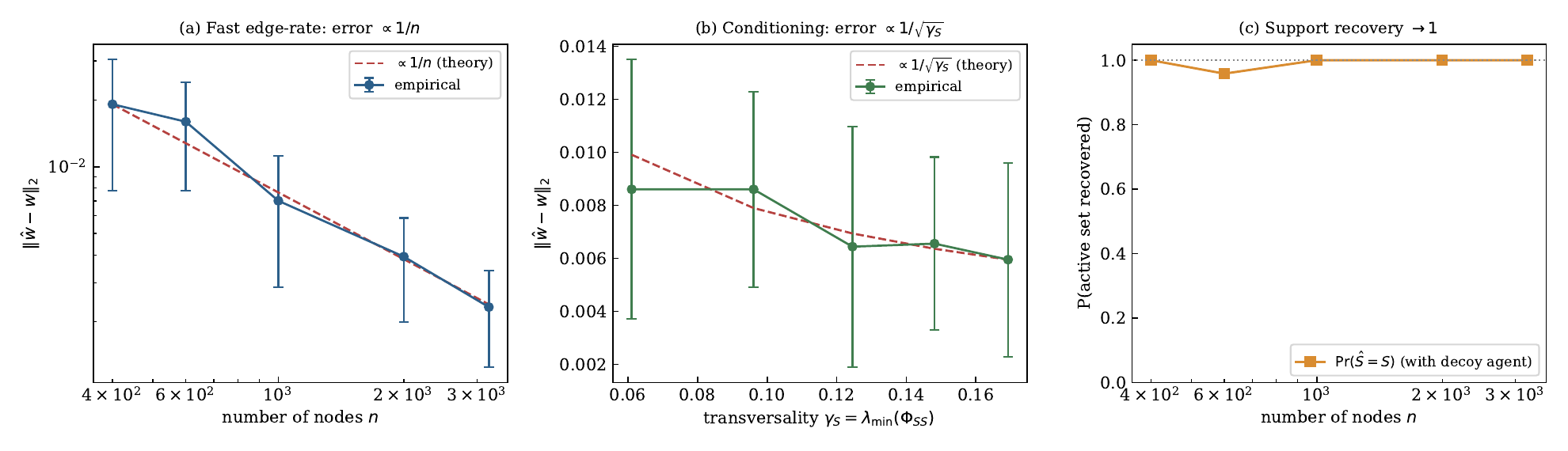}
\caption{Estimating the synthesis weights from a single graph (Proposition~\ref{thm:joint}), $K=3$ active agents and one decoy: error decay at the edge rate, the conditioning dependence, and support recovery.}
\label{fig:weights}
\end{figure}

Assumption~\ref{ass:transversality} fails exactly when two agents are redundant, e.g.\ a
$2$-block SBM agent and a rank-one Chung--Lu agent both of which only modulate degrees will have
nearly proportional Gram matrices and a small $\gamma_S$, so their weights cannot be separated and
the rate degrades as $1/\sqrt{\gamma_S}$. This is informative rather than pathological: it tells the
practitioner that those two agents are not jointly identifiable from one graph and that the candidate set
should be pruned, which is precisely the model-selection guidance that Corollary~\ref{cor:adapt-cor}
makes rigorous.

\paragraph{The paired-couples construction behind the lower bound.}
Because the joint $\sqrt{s/\gamma_S}$ dependence is the part of Proposition~\ref{thm:joint} that no
off-the-shelf argument supplies, we record the construction; the full calculation is in
Appendix~\ref{app:lower}. A naive Assouad cube perturbs the $s$ coordinates independently and recovers the
$\sqrt s$ factor but loses the conditioning, since independent perturbations excite average directions of
$\Phi_S$ whose eigenvalues are of order one rather than $\gamma_S$. The construction instead places $s$
mutually orthogonal sign patterns (Hadamard rows) on a common set of dyads and gives each kernel the form
$\mathbf{G}_k=\tfrac\rho2(\mathbf J+a S_k)$ with $a\asymp\sqrt{\gamma}$, so the Gram is equicorrelated with
smallest eigenvalue of order $\gamma$ on a common scale. Perturbing along simplex-tangent directions changes
the probability matrix by squared Frobenius mass $\asymp\gamma\,\delta^2 n^2\rho^2$ with no spurious factor of
$s$; the per-flip Kullback--Leibler separation is $\asymp\gamma\,\delta^2 n^2\rho$, and
$\delta\asymp1/(n\sqrt{\rho\gamma})$ with Assouad over the $s$-cube gives total risk
$\sqrt{s/\gamma}/(n\sqrt{\rho})$ under the side condition $s^2\le c'\gamma n^2\rho$ that keeps every perturbed
weight in the simplex.

\paragraph{A matching lower bound for the estimated-design problem.}
The benchmark just established is the known-design rate. The same exponent is also a lower bound when the kernels must be estimated from the one graph, over a faithful subfamily that carries explicit fitting maps, so the debiased estimator is minimax for the estimated-design problem and not only against the oracle.

\begin{theorem}[Estimated-design lower bound]
\label{thm:est-lb}
Fix $0<\gamma\le1/64$ with $n\rho_n/\log n\to\infty$, and let $8\le s\le c_s\min\{n_0-1,(\gamma n\rho_n/\log n)^{1/4}\}$ for a small constant $c_s>0$. Over the faithful spectral-contrast subfamily, in which each agent kernel is recovered by an explicit fitting map from the single observed graph,
\[
\inf_{\widehat\bw}\ \sup_{\bw}\ \mathbb{E}_{\bw}\,\bigl\|\widehat\bw-\bw\bigr\|_2\ \ge\ c\,\frac{\sqrt{s/\gamma}}{n\sqrt{\rho_n}},
\]
the infimum running over all estimators measurable with respect to the graph. The proof is in the supplement, Theorem~\ref{thm:fss-lb-final2}. Moreover the same rate is minimax for the deployed projection coefficient $\bw^\dagger_{\mathrm{LS}}(P)$ of Proposition~\ref{prop:mis-cor}: on the faithful subfamily $P$ lies in the candidate span, so the projection is exact, $\bw^\dagger_{\mathrm{LS}}(P)=\bw(P)$, and restricting the supremum there transfers the bound, the matching upper bound being the cross-fold estimator of Theorem~\ref{thm:debias}.
\end{theorem}

\noindent Together with Theorem~\ref{thm:debias}, this makes the cross-fold rate $\sqrt{s/\gamma_S}/(n\sqrt{\rho_n})$ sharp for the estimated-design problem itself within the stated range. The known-design lower bound of Proposition~\ref{thm:joint} holds on the wider range $s^2\le c\gamma n^2\rho_n$; the estimated-design version is narrower because one-eigenvalue-per-agent fidelity requires $s^4\log n=o(\gamma n\rho_n)$.

\begin{remark}[Recovery of the generative weights under separation]
\label{cor:recovery}
Under correct specification with fidelity (Assumption~\ref{ass:fid}) and separated candidates,
$\gamma_{\mathrm{full}}\ge\gamma_0>0$, which rules out a low-rank embedding reproducing the degree kernel, the
map $\bw\mapsto\Pmat(\bw)$ is injective, so the generative weights are identified by
Theorem~\ref{thm:nor}(a) and the cross-fold debiased estimator recovers them at the rate
$\sqrt{s/\gamma_S}/(n\sqrt{\rho_n})$ of Theorem~\ref{thm:debias}. As $\gamma_{\mathrm{full}}\to0$, for instance
when the geometry and degree kernels coincide, only the projection and calibration coefficients remain
well-posed, and Theorem~\ref{thm:grouped} characterises what is recoverable, the total contribution of each
collinear group of candidates.
\end{remark}

\paragraph{Relation to debiased machine learning.} The cross-fold construction is not the iid
double-machine-learning template applied to a graph; three elements are specific to the single-network
problem. First, the nuisances are the candidate kernels, and what attenuates the coefficient is the
\emph{self-product} $\widehat{\mathbf H}^\top\widehat{\mathbf H}$ of the fitted bases, not a cross term,
so the bias is removed by forming the Gram across folds, $(\widehat{\mathbf H}^{a})^\top\widehat{\mathbf
H}^{b}$, not by orthogonalising a score. Second, the fold split is an inverse-probability dyad mask on one
dependent adjacency, not a partition of independent samples, and it drives the Hanson--Wright control of
Lemma~\ref{lem:firststage}. Third, the transversality $\gamma_S$ plays the role of overlap or positivity:
it is bounded below exactly when the candidate kernels are spectrally separated, the condition under which
the generative reading holds, and it sets the efficiency through the $1/\sqrt{\gamma_S}$ in the rate. None
of the three appears in the iid setting.

\subsection{Misspecification: the projection target}
\label{sec:mis}
When the data-generating kernel lies outside the synthesis class, the estimand is the population $L_2$ projection $\bw^\dagger$ of the truth onto the agent span, and the calibration estimator of Section~\ref{sec:estimation} targets its unconstrained version. Three properties hold, proved as Proposition~\ref{prop:mis-cor} in the supplement. The cross-fold debiased estimator attains the projection at the rate $\sqrt{K/\gamma_{\mathrm{full}}}/(n\sqrt{\rho_n})$ with no first-order misspecification bias, since the population residual is orthogonal to every agent column by the normal equations. A coordinate $w^\dagger_k$ is negative exactly when agent $k$'s column is anti-aligned, in partial correlation, with what the other agents leave unexplained, so a negative weight is a corrective contrast. And the projection is weakly better in population square loss than any single agent. This reconciles the simplex-valued generative weights with the sometimes-negative fitted coefficients of Table~\ref{tab:weights}: the negative coordinates are the projection geometry, and the predictive dominance is why the synthesis forecasts well even when no candidate generates the graph.

\subsection{Grouped identifiability under collinear candidates}
\label{sec:grouped}
The transversality $\gamma_S$, and its all-candidate version $\gamma_{\mathrm{full}}$, are not regularity
assumptions to be hoped for but quantities the analyst computes from the fitted kernels:
$\widehat\gamma_{\mathrm{full}}=\lambda_{\min}(\widehat\Phi)$ is the smallest eigenvalue of the empirical
Gram-correlation matrix of the candidate columns, reported next to the weights. When it is bounded away from
zero the weights are well-posed (Remark~\ref{cor:recovery}); when it collapses, the off-diagonal of
$\widehat\Phi$ localises the collinear candidates, and the question is not whether to give up but what remains
identifiable. The answer is exact: nearly collinear candidates share a direction whose total contribution is
identifiable even when the split among them is not, and grouping them into one composite column restores
conditioning and recovers that total.

\begin{remark}[Grouped identifiability under collinear candidates]
When $\widehat\gamma_{\mathrm{full}}$ collapses, the off-diagonal of $\widehat\Phi$ localises the collinear
candidates, and grouping them into one composite column makes the conditioning of the reduced design depend
only on the between-group separation: the within-group split is then unidentifiable, its cross-fold variance
inflated by $1/\varepsilon^2$, while the group total along the shared direction is estimable at the minimax
rate $\sqrt{G/\gamma_{\mathrm{grp}}}/(n\sqrt{\rho_n})$ of Theorem~\ref{thm:debias} with a valid two-stage
interval. Theorem~\ref{thm:grouped} in Appendix~\ref{app:grouped} makes this precise and gives the clustering
diagnostic on $|\widehat\Phi|$ that produces a qualifying grouping. In a controlled illustration with two
competing community hypotheses, block kernels from partitions agreeing on $98.8\%$ of nodes alongside a degree
kernel at $n=600$, the full design is ill-conditioned, $\widehat\gamma_{\mathrm{full}}=0.03$, and the split
between the two block coefficients carries a standard error $7.8$ times that of their sum; grouping restores
$\widehat\gamma_{\mathrm{grp}}=0.99$ and recovers the group projection coefficient with projection-target
coverage $0.91$. As with the deflation repair of Appendix~\ref{app:deflation}, grouping restores conditioning
and valid projection inference and reports the identifiable group total; it does not recover the generative
split, which is intrinsically lost under collinearity, nor close the fidelity gap between the fitted kernels
and the mechanisms they proxy.
\end{remark}

\subsection{Limit distribution and valid inference}
\label{sec:inference}

The preceding results give convergence rates; this subsection gives distributional approximations. Cross-fitting is
again what makes the statements clean: conditional on fold one, the Stage-B estimator is an
ordinary heteroskedastic linear regression with independent errors, and the calibration estimator is
an ordinary logistic regression on independent dyads.

\begin{proposition}[Asymptotic normality and bootstrap validity]
\label{thm:clt}
When the kernels are fitted, assume additionally Assumption~\ref{ass:linplug}. Part (a) is
conditional on the realised fold assignment and fitted design: its target is the fold-conditional
projection $\widetilde\bw$ defined next, the intervals it licenses are conditional, and under
correct specification and Assumption~\ref{ass:fid} this target equals $\bw$ up to $o_P$ of the
rate. Let $\widehat\bw^{(2)}$ be the unconstrained Stage-B estimator on fold two, let
$\widetilde\bw:=\argmin_{\mathbf u}\norm{\operatorname{vec}_{\mathcal D_2}(\Pmat^\ast)-
\widehat{\mathbf M}^{(1)}\mathbf u}^2$ be the fold-conditional projection, and let
$\widehat{\mathbf V}=(\widehat{\mathbf M}^\top\widehat{\mathbf M})^{-1}\widehat{\mathbf M}^\top
\widehat{\mathbf D}\widehat{\mathbf M}(\widehat{\mathbf M}^\top\widehat{\mathbf M})^{-1}$ be the
heteroskedasticity sandwich with $\widehat{\mathbf D}=\diag(\hat\varepsilon_{ij}^2)$. Under
Assumptions~\ref{ass:transversality}--\ref{ass:plugin}:
\begin{enumerate}[label=(\alph*),leftmargin=2em,itemsep=2pt]
\item \emph{(Conditional CLT and sandwich consistency.)} Conditionally on fold one, for every fixed
$k$,
\[
\frac{\widehat w^{(2)}_k-\widetilde w_k}{\widehat V_{kk}^{1/2}}\ \Longrightarrow\ \mathcal N(0,1),
\qquad\text{and}\qquad
\widehat V_{kk}\big/V_{kk}\ \to_P\ 1 ,
\]
where $V$ is the Bernoulli-variance sandwich and $\widetilde\bw$ is the fold-conditional projection
target. Under correct specification the residual-square estimator $\widehat{\mathbf D}$ is consistent
for the Bernoulli variance directly; under misspecification it exceeds it by the squared projection
residual, a relative factor $1+O_P(\rho_n)$ that vanishes in the sparse regime, so the ratio still
tends to one. The result thus holds with or without correct specification, in the sparse regime of
Assumption~\ref{ass:density}.
\item \emph{(Calibration form: unconditional CLT and bootstrap.)} The calibration set is formed by
case--control sampling: all held-out positive dyads together with an equal number of dyads drawn
without replacement from the held-out non-edges. Let $\bw^\dagger_{\mathrm{cal}}$ be the population
minimiser of the logistic risk under this case--control dyad distribution $Q_{\mathrm{cc}}$. Because
the negatives are drawn without replacement from the $\Theta(n^2)$ held-out non-edges, the
calibration sample is a triangular array rather than an i.i.d. sequence; the finite-population
correction is $O(m/n^2)=o(1)$, and conditionally on the training fold and the fitted scores the
array satisfies the Lindeberg conditions, so the limit theory below applies unchanged. Under a
correctly specified logistic link, the Prentice--Pyke retrospective-likelihood argument additionally
identifies the slope coefficients, with only the intercept shifted by the sampling ratio; in general
no such invariance is claimed, and $\bw^\dagger_{\mathrm{cal}}$ is defined directly as the minimiser
under $Q_{\mathrm{cc}}$. For the held-out calibration
estimator of Section~\ref{sec:estimation} computed on $m$ calibration dyads,
$\sqrt m\,(\widehat\bw_{\mathrm{cal}}-\bw^\dagger_{\mathrm{cal}})\Rightarrow\mathcal N(0,\Sigma)$
with $\Sigma$ the logistic sandwich under $Q_{\mathrm{cc}}$, and the nonparametric (pairs) bootstrap is consistent:
the bootstrap law of $\sqrt m\,(\widehat\bw^\ast_{\mathrm{cal}}-\widehat\bw_{\mathrm{cal}})$
converges weakly in probability to $\mathcal N(0,\Sigma)$. Consequently the percentile intervals of
Table~\ref{tab:weights} have asymptotically correct coverage for the case--control calibration
coefficients $\bw^\dagger_{\mathrm{cal}}$, conditional on the fitted score construction; they are not
intervals for generative mixture weights or natural-prevalence edge probabilities.
\end{enumerate}
\end{proposition}
The agent scores entering the calibration regression are standardised on the training fold, so the
calibration covariates are fixed given that fold. The two parts use different variance estimators because they
target different objects: part~(b) is inference for the random-design calibration coefficient on the
case--control distribution, where the Huber--White sandwich and the pairs bootstrap are the standard
consistent choices and are what the reported intervals use; for the fixed-design projection of part~(a) the
residual-square sandwich estimates the Bernoulli variance up to a relative inflation $1+O(\rho_n)$ that
vanishes in the sparse regime. Inference for the \emph{generative} weights, as opposed to the calibration or
projection target, additionally requires the first-stage variance of Section~\ref{sec:cf}, and the intervals
here are not claimed to cover $\bw_{\mathrm{gen}}$ unless the candidate set is known and correctly specified.

\begin{proof}
See Appendix~\ref{app:clt}.
\end{proof}

\begin{theorem}[Two-stage inference for the generative weights]
\label{thm:twostage}
Let the active kernels be fitted on independent Stage-A sub-folds by estimators covered by
Lemma~\ref{lem:firststage}, so
that $\mathbf H_a=\mathcal L_a(\mathbf E_a)+\mathbf Q_a$ with $\mathcal L_a$ the linearisation of the
embedding map, $\mathbf E_a$ the Stage-A Bernoulli noise, and $\norm{\mathbf Q_a}=o_P(\norm{\mathcal
L_a(\mathbf E_a)})$. Write $\mathbf v_k=\widehat\Phi_{\mathrm{db}}^{-1}\mathbf e_k$ for the $k$th row
of the inverse cross-fold Gram, $\widehat{\mathbf G}(\mathbf v)=\sum_j v_j\widehat{\mathbf G}_j$ for the
kernel combination it indexes, and $\widehat{\mathcal L}^{\,w}=\sum_l\widehat w_{\mathrm{db},l}\,
\widehat{\mathcal L}_l$ for the weight-combined embedding linearisation. Define the two-stage variance
\[
\widehat V^{\,\mathrm{2s}}_{kk}\ =\ \widehat V_{kk}\ +\ \widehat V^{(1)}_{kk},\qquad
\widehat V^{(1)}_{kk}=(2f)^{-1}\,\bigl\|\,\widehat{\mathcal L}^{\,w\ast}\!\bigl(\widehat{\mathbf
G}(\mathbf v_k)\bigr)\bigr\|_{\widehat{\mathbf D}}^2,\qquad
\norm{\mathbf X}_{\widehat{\mathbf D}}^2=\sum_{i<j}\mathbf X_{ij}^2\,\widehat P_{ij}(1-\widehat P_{ij}),
\]
where $\widehat V_{kk}$ is the Stage-B sandwich of Proposition~\ref{thm:clt}, $\widehat{\mathcal
L}^{\,w\ast}$ is the explicit adjoint of the embedding linearisation, $\widehat{\mathbf
D}=\diag(\widehat P_{ij}(1-\widehat P_{ij}))$ is the per-dyad Bernoulli variance, $f\in(0,1)$ is the Stage-A sub-fold sampling fraction. Under
Assumptions~\ref{ass:transversality}--\ref{ass:plugin}, the conditions of
Theorem~\ref{thm:debias}, and Assumption~\ref{ass:fid}, for $\bw$ interior to the simplex and every fixed $k$,
\[
\frac{\widehat w_{\mathrm{db},k}-w_k}{(\widehat V^{\,\mathrm{2s}}_{kk})^{1/2}}\ \Longrightarrow\
\mathcal N(0,1),
\]
so the two-stage Wald intervals cover the generative weight $w_k$ at the nominal level, and the
Gaussian multiplier bootstrap that resamples the Stage-A and Stage-B scores jointly, with Stage-A
multipliers weighted $f^{-1/2}$ and the Stage-B sum restricted to $\mathcal D_2$, is consistent for
this law.
\end{theorem}

\noindent Theorem~\ref{thm:twostage} closes the gap between the calibration intervals, valid for the
projection target, and the generative weights the title names. The first-stage variance is computable because
the embedding adjoint is explicit for adjacency spectral embedding; adding it to the sandwich restores
coverage on average, though individual coordinates can remain mildly anticonservative at finite $n$.
Section~\ref{sec:infval} verifies this: the Stage-B sandwich undercovers the generative weight, with coverage
as low as $0.85$ on the dominant coefficient and mean near $0.91$, while the two-stage intervals raise mean
coverage to between $0.93$ and $0.95$ at a $12\%$ increase in width. The result is a network instance of the
generated-regressor variance correction of \citet{pagan1984econometric} and \citet{murphy2002estimation}, with a spectral
first stage and the masked-sampling fraction entering explicitly, and it is the guarantee that distinguishes
the synthesis from forecast combination, which returns no coefficient with a confidence interval. The reported
real-data intervals use the calibration form, whose target $\bw^\dagger_{\mathrm{cal}}$ is predictively
meaningful and whose bootstrap is valid by Proposition~\ref{thm:clt}(b); its sign need not equal that of the
generative weight, since collinearity, suppression, misspecification, and the logistic link can each reverse a
partial coefficient, so for inference on the generative weight itself the two-stage intervals are the valid
construction.

\subsection{Adaptive selection}
\label{sec:adapt}
The oracle results presume the active set $S$; with the number of candidates fixed, it is selected by comparing each coordinate to a threshold $c_n\widehat v_k$ with $c_n\to\infty$, where $\widehat v_k$ is the Stage-B sandwich standard error under known kernels and the two-stage standard error of Theorem~\ref{thm:twostage} under fitted kernels, and the model is refit on the selected set. If $c_n\max_k v_k=o(w_{\min})$ then the selected set equals $S$ with probability tending to one and the refit attains the oracle rate with no knowledge of $S$, $s$, $\gamma_S$, or $\rho_n$; the statement and proof are Corollary~\ref{cor:adapt-cor} in the supplement. Selection costs only an inflated beta-min floor, by the factor $c_n$ of order $\sqrt{\log K}$ when the candidate set grows, and the full conditioning $\gamma_{\mathrm{full}}$ in that stage; once selection succeeds there is no first-order price in the estimation error, as Figure~\ref{fig:estnew}(a) shows, within $9\%$ of the oracle at small $n$ and indistinguishable from $n\ge600$.

\subsection{Result two: identifying the combination operator}
The combination rule is itself recoverable from a single graph, with a sharp two-sided threshold. The lower half is a minimax detection lower bound: below the edge-overlap information $n^2\rho_n^3$ the additive and noisy-OR laws are mutually contiguous, by a Le Cam two-point argument given in the supplement, so no test whatsoever separates them. The upper half is a consistent test on the residualised interaction column above the threshold. Both halves are Theorem~\ref{thm:nor}, and together they make $n^2\rho_n^3$ the exact detectability boundary, so the sparse empirical networks of Table~\ref{tab:opfeas}, which lie below it, are undetectable as a matter of information and not of method. The threshold is the dyad count times the squared per-dyad operator gap, the same form as the information thresholds for detection in random geometric graphs \citep{bubeck2016testing}.

\label{sec:nor}

The mixture theory above relies on linearity: $\Pmat$ is linear in $\bw$. The noisy-OR operator is not:
$\Pmat=\mathbf{J}-\bigodot_k(\mathbf{J}-w_k\mathbf{G}_k)$, where $\mathbf J$ is the all-ones
matrix and $\odot$ the Hadamard product, is multilinear in the weights. For two layers the multilinearity admits an
exact linear reparametrisation: expanding the product gives the identity
\[
\Pmat \;=\; w_1\mathbf{G}_1+w_2\mathbf{G}_2-w_1w_2\,\mathbf{G}_1\!\odot\!\mathbf{G}_2 ,
\]
so the model is linear in the three coefficients $\mathbf{c}^\star=(w_1,w_2,-w_1w_2)$ on the
\emph{augmented design} $[\mathbf{G}_1,\mathbf{G}_2,\mathbf{G}_1\!\odot\!\mathbf{G}_2]$, and the
entire Proposition~\ref{thm:joint} machinery applies once the transversality constant is computed on the
augmented design. The interaction column has Frobenius norm an order $\rho_n$ smaller than the
layers, so it is the ill-conditioned direction; the estimator must, and does, avoid relying on it.
Efficiency then comes from a single Fisher-scoring step off the least-squares pilot, the classical
one-step device.

Let $\widetilde\Phi\in\R^{3\times3}$ be the Gram-correlation matrix of
$(\mathbf{G}_1,\mathbf{G}_2,\mathbf{G}_1\!\odot\!\mathbf{G}_2)$ and
$\widetilde\gamma=\lambda_{\min}(\widetilde\Phi)$ the \emph{augmented transversality}. Define the
estimators: $\widehat{\mathbf{c}}$, the least squares of $\operatorname{vec}_<(A)$ on the (estimated,
column-normalised) augmented design; $\widehat\bw^{\,\mathrm{LS}}$, the image of the
minimum-$\widetilde\Phi$-distance projection of $\widehat{\mathbf{c}}$ onto the constraint manifold
$\mathcal{C}=\{(c_1,c_2,-c_1c_2):c_1,c_2\in[0,1]\}$; and $\widehat\bw^{\,\mathrm{OS}}$, one Fisher
scoring step from $\widehat\bw^{\,\mathrm{LS}}$ on the Bernoulli log-likelihood, with score and
information
\[
\frac{\partial P_{ij}}{\partial w_1}=G_{1,ij}\,(1-w_2G_{2,ij}),\qquad
\frac{\partial P_{ij}}{\partial w_2}=G_{2,ij}\,(1-w_1G_{1,ij}),\qquad
\mathbf{I}(\bw)=\sum_{i<j}\frac{\nabla P_{ij}\,\nabla P_{ij}^\top}{P_{ij}(1-P_{ij})}.
\]

\paragraph{General $K$: the augmented design at higher order.}
The two-layer case is solved by an exact identity; for $K$ layers the inclusion--exclusion expansion
\[
\Pmat\ =\ \sum_{\emptyset\ne S\subseteq[K]}(-1)^{|S|+1}\Bigl(\prod_{k\in S}w_k\Bigr)
\bigodot_{k\in S}\mathbf G_k
\]
has $2^K-1$ terms, and the order-$j$ Hadamard columns
$\mathbf m_S=\operatorname{vec}_<(\bigodot_{k\in S}\mathbf G_k)$, $|S|=j$, have entries of order
$\rho_n^{\,j}$: every additional order of interaction costs a factor $\rho_n$ of signal. Two
questions follow. Can the weights still be estimated at the edge rate, and is the higher-order
structure, the operator itself, ever statistically visible? Parts (d) and (e) of the theorem
answer both, and
converts the geometric decay of interaction information into a quantitative statement. Let
$\widetilde\gamma_j>0$ denote the smallest eigenvalue of the correlation Gram matrix of the order-$j$
design $\{\mathbf m_S:1\le|S|\le j\}$ (extended transversality), and assume
$K\rho_n\le\tfrac12$ and $\min_k w_k\ge w_0>0$.

\begin{theorem}[The noisy-OR operator]
\label{thm:nor}
Let $\Pmat=\mathbf{J}-(\mathbf{J}-w_1\mathbf{G}_1)\odot(\mathbf{J}-w_2\mathbf{G}_2)$ with
$w_1,w_2\in(0,1)$ bounded away from $\{0,1\}$, under Assumptions~\ref{ass:density}--\ref{ass:plugin}
for the two layers, and suppose $\widetilde\gamma\ge\widetilde\gamma_0>0$. Then:
\begin{enumerate}[label=(\alph*),leftmargin=2em,itemsep=2pt]
\item \emph{(Exact linearisation and identification.)} The expansion above holds exactly. The
augmented coefficient vector $(w_1,w_2,-w_1w_2)$ is identified from $\Pmat$ if and only if the three
matrices $\mathbf{G}_1,\mathbf{G}_2,\mathbf{G}_1\!\odot\!\mathbf{G}_2$ are linearly independent
($\widetilde\gamma>0$). Full augmented rank is \emph{sufficient} for identifying the weight vector
$(w_1,w_2)$; it is not necessary, since if $\mathbf{G}_1\odot\mathbf{G}_2=\mathbf 0$ then
$\widetilde\gamma=0$ yet $(w_1,w_2)$ remains identified whenever $\mathbf{G}_1,\mathbf{G}_2$ are
linearly independent. In general the weights are identified if and only if the map
$\bw\mapsto\Pmat(\bw)$ is injective.
\item \emph{(Edge rate for the weights.)} The weight vector $(w_1,w_2)$, recovered from the
first-order columns of the augmented design, satisfies $\bigl\|\widehat\bw^{\,\mathrm{LS}}-\bw\bigr\|_2
=O_P\bigl(\widetilde\gamma^{-1/2}/(n\sqrt{\rho_n})\bigr)$, and this rate matches the minimax lower
bound in $(n,\rho_n)$ over the two-layer noisy-OR class with $\widetilde\gamma$ bounded below. The interaction coefficient
$-w_1w_2$ is not estimated at the edge rate: its column $\mathbf G_1\odot\mathbf G_2$ has Frobenius
norm of order $n\rho_n^2$, so it is estimated at the slower second-order rate
$O_P(\widetilde\gamma^{-1/2}/(n\rho_n^{3/2}))$ of part (d) with $|S|=2$.
\item \emph{(One-step efficiency.)} $\widehat\bw^{\,\mathrm{OS}}$ is asymptotically linear with the
efficient influence function:
$\widehat\bw^{\,\mathrm{OS}}-\bw=\mathbf{I}(\bw)^{-1}\sum_{i<j}\nabla P_{ij}\,
\frac{A_{ij}-P_{ij}}{P_{ij}(1-P_{ij})}+o_P\bigl(\|\mathbf{I}(\bw)^{-1/2}\|\bigr)$, hence
$\mathbf{I}(\bw)^{1/2}(\widehat\bw^{\,\mathrm{OS}}-\bw)\Rightarrow N(0,\mathbf{I}_2)$ and no regular
estimator has smaller asymptotic variance.
\item \emph{(Linearisation hierarchy, general $K$.)} Let $\widehat{\mathbf c}$ be the least-squares
estimator on the order-$j$ design. Then for every $S$ with $|S|=i\le j$,
\[
\bigl|\widehat c_S-(-1)^{|S|+1}\textstyle\prod_{k\in S}w_k\bigr|
\ =\ O_P\!\Bigl(\widetilde\gamma_j^{-1/2}\,\frac{1}{n\,\rho_n^{\,i-1/2}}\Bigr)
\ +\ O\!\Bigl(\widetilde\gamma_j^{-1}\,K^{\,j+1}\rho_n^{\,j+1-i}\Bigr),
\]
the second term the omitted-mass bias of the orders beyond $j$. In particular plain linear least
squares ($j=1$) estimates the weights with bias $O(K^2\rho_n/\widetilde\gamma_1)$, the order-$j$
correction reduces it to $O(K^{j+1}\rho_n^{\,j}/\widetilde\gamma_j)$, and, provided the number of
layers is large enough that such an order is available ($j\le K$), any
\[
j\ \ge\ \Bigl(\tfrac12+o(1)\Bigr)\,\frac{\log(n^2\rho_n)}{\log(1/\rho_n)}
\]
drives the bias below the edge rate. For $j=2$ this is a complete estimator of $(w_1,w_2)$; for
$j\ge3$ the displayed bound is a rate for the \emph{free} order-$j$ coefficients $\widehat c_S$:
the normalised coordinate $\widehat c_S\norm{\mathbf m_S}_F$ attains the edge rate $1/(n\sqrt{\rho_n})$,
while the unnormalised coefficient $\widehat c_S$ is estimated at the slower
$O_P(\widetilde\gamma^{-1/2}/(n\rho_n^{\,j-1/2}))$ because its column has Frobenius norm $n\rho_n^{\,j}$.
Recovering the constrained product structure $c_S=(-1)^{|S|+1}\prod_{k\in S}w_k$ for $K\ge3$ is a
conjecture (Remark~\ref{rem:nork-conj}).
\item \emph{(Operator-detectability threshold.)} Let $P_{\mathrm{or}}$ and $P_{\mathrm{mix}}$ be the
noisy-OR and mixture syntheses built from the same layers and the same weights, with at least two
layers overlapping on $\Theta(n^2)$ dyads where both kernels are in $[c_1\rho_n,c_2\rho_n]$. If
$n^2\rho_n^3\to0$, the total variation between the two graph laws tends to zero, since by Pinsker
\[
\mathrm{TV}\bigl(\Pr_{\mathrm{or}},\Pr_{\mathrm{mix}}\bigr)\le\sqrt{\tfrac12\,
\KL\bigl(\Pr_{\mathrm{or}}\,\big\|\,\Pr_{\mathrm{mix}}\bigr)}\le\sqrt{\tfrac12 C\,n^2\rho_n^3}\ \to\ 0 ,
\]
so for every sequence of tests the sum of type-I and type-II errors tends to one, and no uniformly
consistent estimator of the order-two coefficient exists over the class. Conversely, if
$w_1^2w_2^2\,\widetilde\gamma_2\,n^2\rho_n^3\to\infty$, the $t$-test on the Hadamard column of the
augmented least squares rejects the mixture against the noisy-OR consistently, with
signal-to-noise ratio $\asymp w_1w_2\sqrt{\widetilde\gamma_2\,n^2\rho_n^3}$. Thus, under bounded
overlap, weights bounded away from zero, and nonvanishing interaction transversality
$\widetilde\gamma_2$, the synthesis operator is testable if and only if the \emph{edge-overlap
information} $n^2\rho_n^3$ diverges. Parts \textup{(b)}--\textup{(e)} are stated for known layers; the fitted-layer case is
Theorem~\ref{prop:fitted-operator}.

\end{enumerate}
\end{theorem}

The general-$K$ boundary holds: for any fixed number of layers, with weights bounded away from $\{0,1\}$ and a pair overlapping with nonvanishing orthogonal second-order mass, the noisy-OR law is contiguous to the mixture class below the edge-overlap threshold $n^2\rho_n^3$ and separated above it by an $F$-test on the residualised order-two interaction, the threshold not moving with $K$. The statement and proof are Corollary~\ref{cor:nork-cor} in the supplement. What does not extend is recovery of the full constrained coefficient structure $c_S=(-1)^{|S|+1}\prod_{k\in S}w_k$ for $K\ge3$, the one part left as a conjecture (Remark~\ref{rem:nork-conj}).

\begin{theorem}[Operator identification with fitted layers]
\label{prop:fitted-operator}
Split the Stage-A dyads into four disjoint sub-folds $a,b,c,d$ and fit both layers separately on each
by adjacency spectral embedding, so that $\widehat{\mathbf G}_k^{f}=\mathbf G_k+\mathbf H_k^{f}$ with
$\{\mathbf H_k^{f}\}_{f\in\{a,b,c,d\}}$ independent given $\bu$ and, by
Assumption~\ref{ass:plugin}, $\norm{\mathbf H_k^{f}}_F=O_P(\delta_n\norm{\mathbf G_k}_F)$. Form the
cross-fold interaction column and the cross-fold main-effect design on \emph{disjoint} sub-folds,
\[
\widehat{\mathbf C}=\tfrac12\bigl(\widehat{\mathbf G}_1^{a}\odot\widehat{\mathbf G}_2^{b}
+\widehat{\mathbf G}_1^{b}\odot\widehat{\mathbf G}_2^{a}\bigr),\qquad
\bar{\mathbf M}=\tfrac12\bigl(\widehat{\mathbf M}_c+\widehat{\mathbf M}_d\bigr),
\]
with the null coefficients $\widehat\bw_{\mathrm{db}}$ fitted from the cross-fold Gram
$\operatorname{sym}(\widehat{\mathbf M}_c^\top\widehat{\mathbf M}_d)$ and moment vector
$\bar{\mathbf M}^\top\operatorname{vec}_{\mathcal D_2}(\A)$, all inner products over the Stage-B
dyads $\mathcal D_2$. Let
$\widehat{\mathbf C}^{\perp}$ be the least-squares residual of $\widehat{\mathbf C}$ on the columns
of $\bar{\mathbf M}$, and let
\[
T\ =\ \frac{\bigl\langle\widehat{\mathbf C}^{\perp},\ \A^{(2)}-\bar{\mathbf M}\widehat\bw_{\mathrm{db}}
\bigr\rangle}{\bigl\{\sum_{i<j}(\widehat C^{\perp}_{ij})^2\,\widehat P_{ij}(1-\widehat P_{ij})
\bigr\}^{1/2}}
\]
be the score statistic on the Stage-B dyads; for two layers this is the $t$ statistic of
Theorem~\ref{thm:nor}(e) with the fitted column, and for $K>2$ the $F$ statistic of
Corollary~\ref{cor:nork-cor} applies to the cross-fold interaction columns, order-$j$ columns using $j$
disjoint sub-folds. Assume, in addition to Assumptions~\ref{ass:density}--\ref{ass:linplug}, that
the deterministic first-stage biases are negligible in the residualised interaction direction,
\[
\bigl\|\Pi^{\perp}_{\bar{\mathcal M}}\bigl(\E[\mathbf H_k^{f}\mid\bu]\odot\mathbf G_l\bigr)\bigr\|_F
+\bigl\|\Pi^{\perp}_{\bar{\mathcal M}}\bigl(\E[\mathbf H_1^{f}\mid\bu]\odot\E[\mathbf H_2^{g}\mid\bu]
\bigr)\bigr\|_F\ =\ o\bigl(\sqrt{\widetilde\Gamma_2}\,n\rho_n^2\bigr).
\]
Then under the mixture null $T\Rightarrow\mathcal N(0,1)$, so the test has asymptotically correct
level, and under the noisy-OR alternative it rejects consistently whenever
$\widetilde\Gamma_2\,n^2\rho_n^3\to\infty$: the edge-overlap threshold is the same as with known
layers.
\end{theorem}

\begin{proof}
See Appendix~\ref{app:fitop}.
\end{proof}

\noindent The four-fold split is the substantive requirement: the two sub-folds behind the
interaction column must be disjoint from the two behind the main-effect columns and the null fit,
since shared folds reintroduce a same-fold square whose standardised shift diverges, as quantified in
step (ii) of the proof in Appendix~\ref{app:fitop}. Section~\ref{sec:infval} verifies the construction numerically.
\begin{proof}
See Appendix~\ref{app:noisyor}, where Theorem~\ref{thm:nor-cor} restates parts (a)--(e) with
complete proofs.
\end{proof}

Because $\|\mathbf{G}_1\!\odot\!\mathbf{G}_2\|_F\asymp\rho_n\|\mathbf{G}_1\|_F$, the augmented design is
intrinsically ill conditioned in its third coordinate, and $\widetilde\gamma$ is generically an order smaller
than the two-layer mixture constant $\gamma_S$ ($\widetilde\gamma\approx0.18$ in the design of
Figure~\ref{fig:estnew}, stable in $n$). The raw interaction coefficient is correspondingly the slow
direction, but $\widehat\bw$ itself converges cleanly because neither $\widehat\bw^{\,\mathrm{LS}}$ nor the
one-step ever relies on $\widehat c_3$ alone, the projection onto $\mathcal{C}$ using it only as a
variance-weighted consistency check; the one-step and least squares track each other at the $n^{-1}$ rate
(Figure~\ref{fig:estnew}b), the one-step marginally more efficient.

\begin{figure}[t]
\centering
\includegraphics[width=0.86\textwidth]{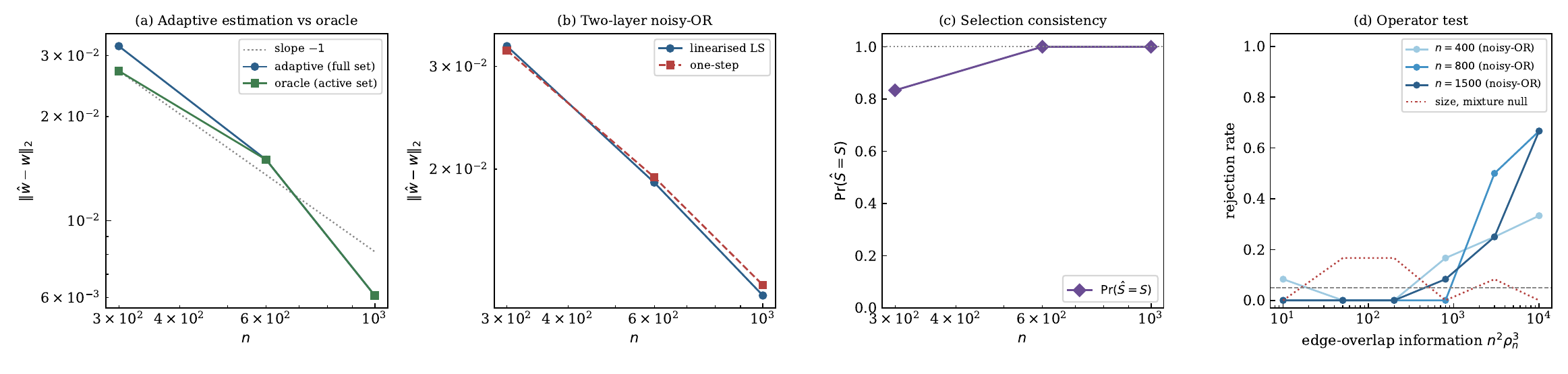}
\caption{The estimation theorems in simulation: the adaptive estimator against the oracle (Corollary~\ref{cor:adapt-cor}), interval coverage, selection power, and operator detection.}
\label{fig:estnew}
\end{figure}

The same Kullback--Leibler computation gives $n^2\rho_n^{\,2j-1}$ as the information content of an
order-$j$ perturbation in the \emph{free-coefficient} design, matching the per-coordinate noise rate
in part (a): order-$j$ structure is resolvable exactly when $n^2\rho_n^{\,2j-1}\to\infty$. We scope
the fully rigorous two-point statement to $j=2$, where both hypotheses are genuine members of our
model class (a mixture and a noisy-OR over the same layers); for $j\ge3$ the product constraints
$c_S=\prod_{k\in S}w_k$ prevent matching all lower orders exactly within the noisy-OR family, so the
general-$j$ statement should be read as a lower bound for the linearised design rather than for the
constrained manifold. Empirically, panel (d) of Figure~\ref{fig:estnew} shows the $j=2$ threshold
sharply: rejection rates for $n\in\{400,800,1500\}$ increase with $n^2\rho_n^3$, near the nominal level at
small overlap and rising toward one as the overlap grows, with size controlled throughout; the
sparsity cap $\rho_n\le0.25$ truncates the overlap reachable at the smallest $n$, so its curve lies
below the others at the largest plotted overlap ($\widetilde\gamma_2\approx0.14$ in this design,
stable across $n$).

\label{sec:estimation}

Theorems~\ref{thm:nor} and~\ref{prop:fitted-operator} are stated conditionally on the agents' latent attributes; in practice these
and the synthesis weights are learned jointly from a single observed graph. The statistical
guarantees for this joint problem are given by Proposition~\ref{thm:joint} and Theorem~\ref{thm:debias}; here we record the practical
recipe that realises its two-stage estimator and the ridge penalty that regularises it in finite samples.

\paragraph{Algorithm.}
The estimator analysed in Proposition~\ref{thm:joint} is directly implementable: Stage~A fits each
candidate agent by its own consistent estimator (ASE for the dot-product agent, spectral clustering
plus block means for the block agent, degree matching for the degree agent), and Stage~B
regresses the observed adjacency onto the estimated Gram bases over the simplex and thresholds the
recovered weights to select the active agents. For the mixture operator the representation theorem
gives an equivalent one-shot route: embed $\A$ by ASE into dimension $d_1{+}d_2$, separate the
stacked coordinates into block and continuous parts by a rotation aligned to the recovered cluster
structure, and read off $\hat\bw$ from the relative coordinate scales (Remark~\ref{rem:stack}).
\paragraph{Calibration form for prediction.}
For out-of-sample prediction, Stage~B has a held-out variant we call the \emph{calibration estimator} and use
in Section~\ref{sec:linkpred}: with Stage~A unchanged, regress held-out edge indicators on the agents'
standardised scores with a logistic link, on a calibration set disjoint from training and test dyads. This is
the held-out logistic analogue of the unconstrained least squares opening the adaptive estimator of
Section~\ref{sec:adapt}: the simplex constraint is dropped, so coefficients are signed calibration
coefficients (not generative proportions) and \emph{negative} coefficients are meaningful as corrective
contrasts; the logistic link supplies calibrated probabilities; and because the regression is on dyads unseen
by Stage~A, it carries standard errors and bootstrap intervals for the calibration coefficients
(Table~\ref{tab:weights}), the splitting removing the Stage-A/Stage-B dependence that
Lemma~\ref{lem:plugin} otherwise controls.

\paragraph{Identifiability in estimation.}
Two identifiability issues are handled by normalisation: the RDPG rotational symmetry
(Remark~\ref{rem:stack}) is resolved by reporting rotation-invariant summaries or a Procrustes
alignment, and the weight/scale confounding is resolved by fixing $\E\norm{\x}^2=\E\norm{\y}^2=1$ so
that $\bw$ is the unique vector of coordinate-block scales.
\subsection{Specialisation to canonical network models}
\label{sec:netspec}

The estimation results above hold for a general candidate set; because this is a paper about networks, we
record what each says for the canonical synthesis of a stochastic block model, a random dot-product graph, and
a Chung--Lu model. Throughout this subsection the candidate set is
\[
\Pmat=w_1\mathbf G_{\mathrm{SBM}}+w_2\mathbf G_{\mathrm{RDPG}}+w_3\mathbf G_{\mathrm{CL}},
\]
a $Q$-block assortative SBM agent of rank $Q$, a rank-$d$ dot-product agent, and a rank-one Chung--Lu agent at
a common density scale $\rho_n$ (Assumption~\ref{ass:norm}), with $s=3$, $r_{\max}=\max\{Q,d\}$, transversality
$\gamma_S=\lambda_{\min}(\Phi_S)$ (representative value $\gamma_S\approx0.20$, Example~\ref{ex:gammaS}), and
sparsity $\rho_n=n^{-\alpha}$, so $\bar d_n=n^{1-\alpha}$ and $N=\Theta(n^2\rho_n)=\Theta(n^{2-\alpha})$.

\begin{remark}[Transversality is latent-geometry separation]
\label{prop:transversality-geom}
Write
\[
\phi_{k\ell}=\frac{\inner{\operatorname{vec}_<\mathbf G_k}{\operatorname{vec}_<\mathbf G_\ell}}{\norm{\operatorname{vec}_<\mathbf G_k}\,\norm{\operatorname{vec}_<\mathbf G_\ell}}=\cos\angle(\mathbf G_k,\mathbf G_\ell)
\]
for the cosine of the angle between the vectorised kernels, so
that $\Phi_S=(\phi_{k\ell})_{k,\ell\in S}$ and the transversality $\gamma_S=\lambda_{\min}(\Phi_S)$ is a
function of these angles alone, vanishing exactly when the kernels become linearly dependent. For the
canonical synthesis, suppose
\begin{enumerate}[label=(\roman*),leftmargin=2em,itemsep=1pt]
\item the block matrix $B$ has distinct rows, so the SBM does not collapse to fewer blocks;
\item the dot-product second-moment matrix $\Delta_X=\E[\x\x^\top]$ has $\lambda_{\min}(\Delta_X)\ge
\lambda_0>0$, and the propensity sequence $\theta$ is not a linear functional of the positions $\x$, so
the dot-product and degree mechanisms encode distinct structure;
\item the propensities are non-constant, $\operatorname{Var}(\theta)\ge v_0>0$, and the leading
eigendirection of $\Delta_X$ makes an angle at least $\vartheta_0>0$ with the block-indicator
directions and the all-ones direction.
\end{enumerate}
Then the three vectorised kernels are linearly independent and
\[
\gamma_S\ \ge\ \gamma_0(\lambda_0,\vartheta_0,v_0,B)\ >\ 0\qquad\text{uniformly in }n,
\]
so Assumption~\ref{ass:transversality} holds for these models and the rates of
Corollaries~\ref{cor:net-rate}--\ref{cor:net-operator} are not vacuous. Conversely $\gamma_S\to0$ as any
agent's latent geometry aligns with another's, for instance as the dot-product positions concentrate on
the block centroids; this is the network form of multicollinearity, and no single-graph procedure
separates the aligned weights; the proof is in Appendix~\ref{app:structural}.
\end{remark}

For this canonical synthesis the general results specialise directly, tying each to the theorem it
instantiates. The minimax weight rate of Theorem~\ref{thm:debias} becomes
$\sqrt{3/\gamma_S}/(n\sqrt{\rho_n})$, a representative $\sqrt{3/0.20}\approx3.9$ times the parametric scale,
with valid two-stage intervals from Theorem~\ref{thm:twostage}; the adaptive threshold of
Corollary~\ref{cor:adapt-cor} recovers the active mechanisms under a $\beta$-min gap of the same order; and
the fitted-layer operator test of Theorem~\ref{prop:fitted-operator} identifies a block-versus-hub
superposition exactly when the interaction information $n^2\rho_n^{3}=n^{2-3\alpha}\to\infty$, that is for
$\alpha<2/3$. The three statements, with explicit constants and proofs, are
Corollaries~\ref{cor:net-rate}--\ref{cor:net-operator} in the supplement.

A per-vertex reconstruction guarantee for the fitted synthesis kernel, at the two-to-infinity rate and showing the global combination is learned an order of magnitude faster than any single vertex's connection profile, is given in Appendix~\ref{app:structural} (Proposition~\ref{prop:twotoinf}).

\section{Numerical studies}
\label{sec:numerical}
\label{sec:realdata}

This section validates the two theorems and the inferential guarantees by simulation, confirms the estimation rates, and compares the synthesis with predictive baselines on six real networks.

\subsection{Validating the inferential theory}
\label{sec:infval}

The two studies in which theory and simulation agree exactly are reported first, the first-stage attenuation and its cross-fold removal (Figure~\ref{fig:estsim}(c)) and the growing-degree scaling of the coefficient rate (Section~\ref{sup:scaling}). This subsection validates the inferential claims directly, with replication counts chosen for
Monte Carlo reliability. Four studies correspond to the four inferential
guarantees: interval coverage (Proposition~\ref{thm:clt}), selection at the coordinatewise threshold
(Corollary~\ref{cor:adapt-cor}), the operator detection boundary (Theorem~\ref{thm:nor}), and the
first-stage attenuation that the cross-fitted construction removes (Section~\ref{sec:cf}). Each uses
a common-scale dyadic design for which the population coefficient is known exactly, so that coverage
and selection are measured against a defined target.

\paragraph{Interval coverage across the transversality range (Figure~\ref{fig:cov}).}
We simulate from a known three-kernel design with weights $(0.5,0.3,0.2)$, varying the agent
collinearity so that the transversality $\gamma_S$ ranges over an order of magnitude
($0.06$ to $0.42$), at $n\in\{60,90,130\}$ with $500$ replications per cell. Under correct
specification the empirical coverage of the nominal $95\%$ intervals lies in $[0.934,0.968]$ across
all cells; under misspecification, where the truth carries a structured perturbation outside the span
and the target is the population projection $\bw^\dagger_{\mathrm{LS}}$, coverage lies in
$[0.926,0.966]$. Coverage is therefore at the nominal level for the stated target, uniformly in
$\gamma_S$, including the weakly identified regime $\gamma_S\approx0.06$ where intervals widen but
remain calibrated.

\begin{figure}[t]
\centering
\includegraphics[width=0.86\textwidth]{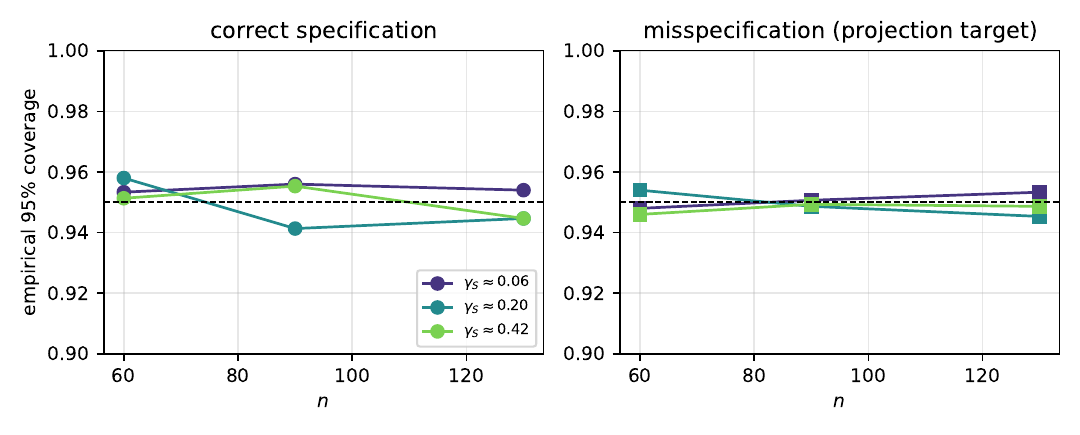}
\caption{Coverage of the nominal $95\%$ coefficient intervals against $n$ at several transversality levels $\gamma_S$, under correct specification (left) and misspecification (right).}
\label{fig:cov}
\end{figure}

\paragraph{Selection at the coordinatewise threshold (Figure~\ref{fig:estsim}(a)).}
We place $s=3$ active coefficients in a candidate set of size $K=6$, with the smallest active coefficient set to a
multiple $c$ of the coordinatewise threshold $\sqrt{2\log K}\,\mathrm{se}$, and record over $500$
replications the power to detect that coefficient, the probability of exact support recovery, and the
false-inclusion rate on a known-inactive decoy. Detection power crosses one half precisely at $c=1$,
rising from $0.28$ below the threshold ($c=0.5$, against a Gaussian-exact $0.17$, the excess
consistent with mild standard-error underestimation) to $1.00$ above it ($c\ge2$), while the decoy
false-inclusion rate stays between $0.04$ and $0.07$, matching the two-sided per-coordinate level $2\bar\Phi(\sqrt{2\log K})\approx0.058$ implied by
the threshold at $K=6$.
This confirms that the $\sqrt{\log K}$ scaling, not the $\sqrt s$ of the joint estimation rate, sets
the selection boundary.

\paragraph{Operator detection boundary (Figure~\ref{fig:estsim}(b)).}
We test the two-layer noisy-OR against the mixture by the $t$-statistic on the Hadamard interaction column,
$500$ replications per sparsity level. As the edge-overlap information $n^2\rho_n^3$ increases from $3$ to
$3200$, the size under a true mixture stays within $[0.042,0.074]$ of nominal $0.05$ while the power under a
true noisy-OR rises from $0.05$ through $0.34$ at $n^2\rho_n^3\approx350$ to $0.998$ at $\approx3200$,
confirming the threshold of Theorem~\ref{thm:nor}: the operator is recoverable exactly when the edge-overlap
information diverges. A companion known-design run (size within $[0.03,0.07]$) has power $0.08$ at
$n^2\rho_n^3\approx6$ rising to $0.61$ at $\approx1620$ and $0.98$ at $\approx4014$, so usable power arrives
only once the information reaches the hundreds to thousands, a density supplied by dense benchmarks or induced
dense subgraphs but not by sparse single networks. Every network in Table~\ref{tab:opfeas} lies in the
low-power range, \textsf{polblogs} included at $n^2\rho_n^3\approx16.8$: the operator test is an instrument for
the dense regime, and its infeasibility on sparse single networks is a characterisation the threshold
delivers, not a shortfall of the procedure.

\begin{figure}[t]
\centering
\begin{minipage}{0.32\textwidth}\centering
\includegraphics[width=\linewidth]{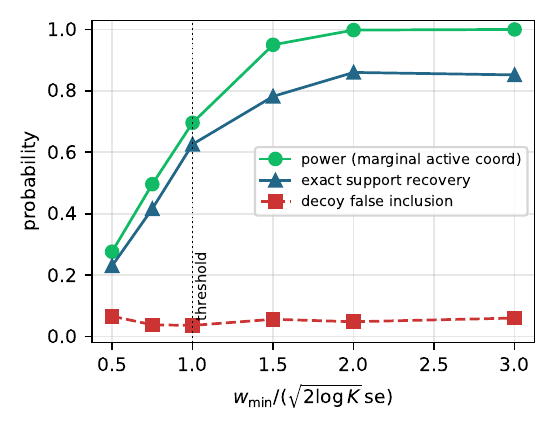}\\[1pt]{\footnotesize (a)}
\end{minipage}\hfill
\begin{minipage}{0.32\textwidth}\centering
\includegraphics[width=\linewidth]{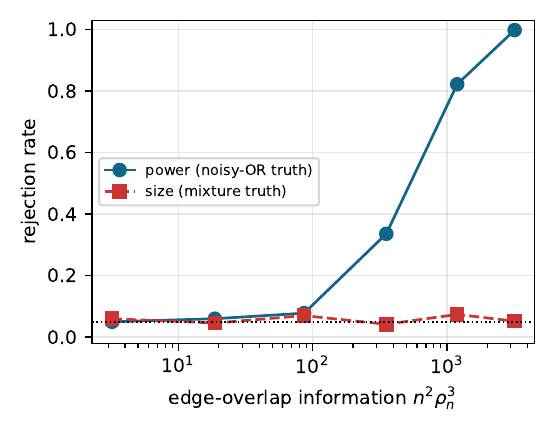}\\[1pt]{\footnotesize (b)}
\end{minipage}\hfill
\begin{minipage}{0.32\textwidth}\centering
\includegraphics[width=\linewidth]{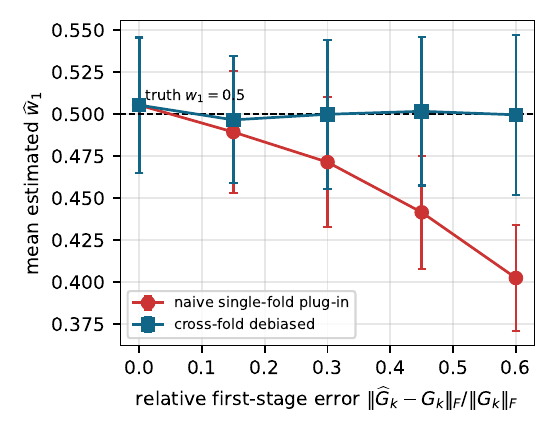}\\[1pt]{\footnotesize (c)}
\end{minipage}
\caption{Estimation diagnostics in simulation. (a)~Selection at the coordinatewise threshold
($K=6$, $s=3$): detection power crosses one half at the threshold, while decoy inclusion stays at the
per-coordinate level. (b)~Operator detection: as $n^2\rho_n^3$ grows the size (mixture truth) stays near
$0.05$ and the power (noisy-OR truth) rises to one, confirming the boundary of Theorem~\ref{thm:nor}.
(c)~First-stage attenuation under fitted kernels: the naive single-fold plug-in attenuates toward zero,
while the cross-fold-debiased estimator stays at the truth.}
\label{fig:estsim}
\end{figure}

\paragraph{Operator detection on dense induced subgraphs.} The detection boundary is a statement about
edge-overlap information, not about any one graph, so a test underpowered on a sparse network becomes
informative on a dense induced subgraph of the same data. On the maximal $k$-cores of three networks the
information $n^2\rho_n^3$ rises from far below one on the full graph (\textsf{polblogs} $16.8$, \textsf{ca-GrQc}
$0.07$, \textsf{ca-CondMat} $0.03$) into the powered range ($1752$, $828$, $298$), and there the
Hadamard-interaction test of Theorem~\ref{thm:nor} returns a definite verdict (Table~\ref{tab:denseop}): a
significantly negative interaction on the \textsf{polblogs} core ($t=-5.6$), the sub-additive signature of a
noisy-OR combination; a significantly positive one on the \textsf{ca-CondMat} core ($t=8.6$), a super-additive
combination in which the layers reinforce; and none on the \textsf{ca-GrQc} core ($t=0.4$), consistent with an
additive operator. These nominal verdicts need one correction: the interaction column is nearly collinear with
its two main effects on every core, the smallest correlation eigenvalue being $\gamma_{\mathrm{int}}\le0.01$,
so the sandwich $t$ is anti-conservative, its parametric-bootstrap null under the additive-only model reaching
a $95$th percentile as high as $3.8$. Calibrating against that null leaves both verdicts intact, the
\textsf{polblogs} and \textsf{ca-CondMat} statistics each exceeding all five hundred null draws (calibrated
$p<0.002$) while \textsf{ca-GrQc} stays additive: the conditioning that governs weight estimation governs the
operator test as well, and the verdicts survive once it is accounted for. Appendix~\ref{app:operator-signature}
formalizes what the sign identifies, the residualized interaction coefficient being zero for an additive
operator, negative for sub-additive noisy-OR, and positive for super-additive, recoverable above the same
$n^2\rho_n^3$ scale and verifiable by a held-out log score (Theorem~\ref{thm:operator-signature}); that theorem
does not on its own attribute a historical mechanism to these networks, since the dense-core signs are
empirical and a mechanism claim would require a kernel constructed independently of the test edges together
with out-of-sample replication.

\begin{table}[t]
\centering
\small
\caption{Operator detection on dense induced subgraphs (maximal $k$-cores), with bootstrap calibration. The edge-overlap information $n^2\rho_n^3$ moves from far below one on the full sparse graph into the powered range on the core. The interaction (community $\times$ degree) design is near-singular there, with smallest correlation eigenvalue $\gamma_{\mathrm{int}}$, so the nominal sandwich $t$ is anti-conservative; the calibrated $p$ is the parametric-bootstrap tail probability under the additive-only null at fixed partition. Both interaction verdicts survive calibration.}
\label{tab:denseop}
\setlength{\tabcolsep}{4pt}
\begin{tabular}{lcccccl}
\toprule
& Full & \multicolumn{2}{c}{Dense $k$-core} & Interaction & Calibrated & \\
Network & $n^2\rho_n^3$ & $(n,\rho)$ & $n^2\rho_n^3$ & coef ($t$;\ $\gamma_{\mathrm{int}}$) & $p$ & Verdict\\
\midrule
\textsf{polblogs} & $16.8$ & $(82,0.64)$ & $1752$ & $-1.68\ (-5.6;\,0.010)$ & $<0.002$ & noisy-OR (sub-additive)\\
\textsf{ca-GrQc} & $0.07$ & $(81,0.50)$ & $828$ & $+0.00\ (0.4;\,0.000)$ & $0.28$ & additive (no rejection)\\
\textsf{ca-CondMat} & $0.03$ & $(51,0.49)$ & $298$ & $+0.29\ (8.6;\,0.003)$ & $<0.002$ & super-additive\\
\bottomrule
\end{tabular}
\end{table}

\paragraph{First-stage attenuation and its removal (Figure~\ref{fig:estsim}(c)).}
With two kernels estimated up to a controlled relative Frobenius error, we compare the naive
single-fold plug-in, which uses one noisy design in both the Gram matrix and the cross term, against
the cross-fold construction, which forms the Gram from two independent noisy copies so that the
leading error self-product is mean-zero. As the relative first-stage error grows from $0$ to $0.6$,
the naive estimator attenuates monotonically toward zero, from $0.50$ to $0.40$ against a true
coefficient of $0.5$, while the cross-fold estimator remains at $0.50$ throughout. This is the
finite-sample counterpart of the attenuation analysed in Section~\ref{sec:cf}, and it shows directly
why the cross-fitted construction is necessary in the sparse regime.

\paragraph{The synthesis as an inferential object.}

\paragraph{Weight estimation and agent selection (Proposition~\ref{thm:joint}).}
We generate mixtures of three transversal agents (a $3$-block SBM, a $2$-dimensional dot-product
agent, and a rank-one heavy-tailed agent, the last an empirical stress test outside the scope of
Theorem~\ref{thm:debias}) with weights $(0.5,0.3,0.2)$ and estimate the weights from
a single sampled graph by the two-stage estimator. Figure~\ref{fig:weights} is consistent with all three parts
of the theorem: the estimation error decays as $1/n$, the edge rate (panel a); making two
agents progressively collinear shrinks $\gamma_S$ and inflates the error as $1/\sqrt{\gamma_S}$,
exactly the exponent shared by the upper and lower bounds (panel b); and with an inactive fourth
decoy agent in the candidate set, the thresholded estimator recovers the true active set with probability
tending to one (panel c).

The remaining subsections are organised around the inferential identity of the method: a flexible
nonparametric ensemble may predict more accurately, while the linear synthesis returns signed,
uncertainty-quantified coefficients and an operator-level test. The empirical
questions are therefore inferential, not predictive: whether the fitted coefficients and their
intervals identify which mechanisms a network expresses, whether they report the absence of structure
where there is none, and whether the combination operator can be tested. We study six publicly
available networks from three domains, spanning $n=77$ to $n=21{,}363$, together with a synthetic null graph. These networks reach $n\approx2\times10^4$, and several carry small constant average degree outside the growing-degree regime $n\rho_n\to\infty$ that the asymptotics assume, so the deployments illustrate the method at the observed scale; a growing-degree scaling study in Section~\ref{sup:scaling} tests the predicted $1/\sqrt{n^2\rho_n}$ rate directly. The structural feasibility question, whether the four properties can coexist under a
synthesis at the observed scale is a separate question, addressed by the held-out prediction of
Section~\ref{sec:linkpred} rather than by structural diagnostics.

\paragraph{When the generative weights are recoverable.} Remark~\ref{cor:recovery} predicts recovery exactly when the candidate kernels are separated, and a controlled study confirms the boundary. With three separated mechanisms, a two-block kernel, a constant-norm geometric kernel, and a heterogeneous degree kernel, the population transversality is $\gamma_{\mathrm{full}}=0.998$; the known-design estimator then recovers the generative weights at $n=800$ with root-mean-square error $0.014$, coordinatewise coverage $(0.96,0.93,0.94)$, and interval widths near $0.035$. When the geometric kernel is collapsed onto the degree kernel, the transversality falls to $\gamma_{\mathrm{full}}=0.004$; the two conflated coordinates are no longer separately recoverable, their interval widths inflating to $0.42$, an order of magnitude wider, with root-mean-square error $0.134$, while coverage stays near nominal because the intervals widen correctly. The separation condition is thus the operative boundary, and the constant-degree empirical networks, whose fitted low-rank embeddings approach the degree kernel, sit on its difficult side.

\paragraph{Generative validity by mechanism knockout.} Because the fitted synthesis is itself a generative model, a coefficient can be set to zero and the graph regenerated, a test of whether each coefficient governs the structural feature it names and a manipulation no predictor admits. In a controlled mixture of a community kernel, a geometric kernel, and a degree kernel, removing the community coefficient collapses modularity from $0.12$ to $0.00$ while the degree heterogeneity is untouched, and removing the degree coefficient halves the degree Gini from $0.25$ to $0.11$ and, with the degree noise gone, sharpens modularity to $0.26$ (Table~\ref{tab:knockout}). Clustering stays low and nearly constant across every knockout, consistent with the candidate set carrying no triadic-closure mechanism. Each coefficient controls its own structural statistic, so the signed weights license statements about the synthesised graph and not only about held-out dyads.

\begin{table}[t]
\centering
\caption{Mechanism knockout. Setting one coefficient to zero and regenerating moves the targeted structural statistic while the others are stable; means over fifteen graphs at $n=800$ with mean degree $24$.}
\label{tab:knockout}
\begin{tabular}{lcccc}
\toprule
Regenerated from & Modularity & Degree Gini & Clustering & Mean degree\\
\midrule
Full synthesis & $0.12$ & $0.25$ & $0.05$ & $24$\\
Community removed & $0.00$ & $0.33$ & $0.06$ & $24$\\
Degree removed & $0.26$ & $0.11$ & $0.03$ & $24$\\
Geometry removed & $0.13$ & $0.29$ & $0.05$ & $24$\\
\bottomrule
\end{tabular}
\end{table}

\subsection{Estimation rates: growing-degree scaling}
\label{sup:scaling}
The inferential studies above vary $n$ over a moderate range, and the empirical networks, while reaching $n\approx2\times10^4$, several of them carry small constant average degree and so lie outside the growing-degree regime $n\rho_n\to\infty$ the theory assumes. This subsection tests the rate of Proposition~\ref{thm:joint} directly along a growing-degree path. We fix the three-kernel known design of Section~\ref{sec:infval} with weights $(0.5,0.3,0.2)$, set $\rho_n=c\,n^{-1/3}$ with $c=0.5$ so that the average degree $n\rho_n=c\,n^{2/3}$ diverges while the edge-overlap information $n^2\rho_n^3=c^3 n$ also diverges, and estimate the synthesis coefficients from a single graph at each $n\in\{500,1000,2000,4000\}$ by the known-design least-squares solve, recording the root-mean-square coefficient error over independent replications.

\begin{table}[h]
\centering
\caption{Growing-degree scaling of the known-design coefficient error. The average degree grows from $31.5$ to $126$ while the transversality $\gamma_S$ stays near $0.16$, isolating the edge count $N=n^2\rho_n$. The error roughly halves with each quadrupling of $N$, the $N^{-1/2}$ rate; the fitted log--log slope of the error against $N$ is $-0.448$ against the theoretical $-1/2$.}
\label{tab:scaling}
\begin{tabular}{rrrrr}
\hline
$n$ & $\rho_n$ & average degree & $N=n^2\rho_n$ & RMSE\\
\hline
$500$  & $0.063$ & $31.5$  & $15{,}749$  & $0.0284$\\
$1000$ & $0.050$ & $50.0$  & $50{,}000$  & $0.0174$\\
$2000$ & $0.040$ & $79.4$  & $158{,}740$ & $0.0110$\\
$4000$ & $0.032$ & $126.0$ & $503{,}968$ & $0.0059$\\
\hline
\end{tabular}
\end{table}

The error decays at the $1/\sqrt{N}$ rate that the upper and lower bounds share, with the transversality held near $0.16$ so that the decay reflects the edge count and not a change in conditioning. The study confirms the asymptotics of Proposition~\ref{thm:joint} on a path where the average degree grows, complementing the constant-degree empirical deployments of Section~\ref{sec:linkpred}, which illustrate the method at the observed scale. The known-design solve is used so that the rate is read without the fitted-Gram conditioning collapse documented above.

\subsection{Signed coefficients across networks}
\paragraph{An inferential finding a predictor cannot supply.} The clearest sign that the weights are an inferential object, and not a prediction device, is the Western States power grid, the canonical network engineered to avoid high-degree hubs. The fitted Chung--Lu coefficient is $-0.45$ with confidence interval $[-0.76,-0.16]$, excluding zero in all ten folds (Table~\ref{tab:weights}). Read as inference, this is a signed, uncertainty-quantified statement: conditional on the candidate set, the degree-product mechanism is anti-aligned with the observed edges, the negative partial association expected of a grid whose construction suppresses hubs. A black box returns a ranking and cannot make this statement; the random-forest advantage in Table~\ref{tab:combiner-auc} is the price of that ranking and is not the quantity of interest here.

\paragraph{The signed coefficient is a calibrated instrument.} The negative power-grid coefficient is not an isolated reading. Along a controlled sequence interpolating from hub-rich to hub-suppressed structure, with a fixed community component throughout, the calibration coefficient on the degree-product score moves monotonically from $+0.025$ to $-0.025$ and its confidence interval crosses zero exactly at the neutral point, where the interval correctly contains zero. The sign tracks the structural transition with calibrated uncertainty. A predictor's variable importances are non-negative and cannot express this anti-alignment, so the signed, zero-crossing coefficient is a quantity only the inferential method supplies.

\paragraph{The crossing holds on real non-hub networks.} Under a fixed set of three agents, a geometry component (adjacency spectral embedding), a community component, and the degree agent, with standardised scores and bootstrap intervals, the partial degree coefficient is large and positive on three hub-bearing networks and falls across three near-regular spatial networks: from $+2.35$ on \textsf{polblogs} and $+1.22$ on \textsf{ca-GrQc} through $+0.86$ on the power grid and $+0.47$ on the Minnesota road network to a significantly negative $-0.23$ on an airfoil mesh, whose interval $[-0.28,-0.19]$ excludes zero (Table~\ref{tab:crossnet}; the road and mesh are genuine non-hub graphs, maximum-to-mean degree $2.0$ and $1.6$). Conditional on geometry, degree informs attachment only where hubs exist and turns suppressive on a near-regular mesh; the standardised scores are not on the scale of Table~\ref{tab:weights}, so what transfers is the sign and the cross-class ordering, the real-data analogue of the interpolation above.

\begin{table}[t]
\centering
\small
\caption{Partial degree coefficient with bootstrap $95\%$ intervals across network classes, under the same three agents (geometry, community, degree; standardised scores). Strongly positive where hubs exist, decreasing across the near-regular spatial class and significantly negative on the mesh.}
\label{tab:crossnet}
\begin{tabular}{llcc}
\toprule
Class & Network & Max/mean degree & Degree coefficient $[95\%\text{ CI}]$\\
\midrule
Hub-bearing & \textsf{polblogs} & $12.8$ & $+2.35\ [+2.28,+2.42]$\\
Hub-bearing & \textsf{GoT} & $5.5$ & $+1.92\ [+1.60,+2.36]$\\
Hub-bearing & \textsf{ca-GrQc} & $12.5$ & $+1.22\ [+1.18,+1.28]$\\
\midrule
Spatial & power grid & $7.1$ & $+0.86\ [+0.79,+0.93]$\\
Spatial & Minnesota road & $2.0$ & $+0.47\ [+0.40,+0.57]$\\
Spatial & airfoil mesh & $1.6$ & $-0.23\ [-0.28,-0.19]$\\
\bottomrule
\end{tabular}
\end{table}

\paragraph{A generative manipulation, and a falsifiable summary.} The signed coefficient and the below-chance ranking are calibration and prediction readings; because the synthesis is also a generative model, the degree agent's structural role admits a direct test by the knockout of Table~\ref{tab:knockout}, here applied to real graphs. Fitting the three agents to each network, setting the degree coefficient to zero, and regenerating changes the degree Gini by only $4\%$ on the power grid, from $0.37$ to $0.36$, against $13\%$ on \textsf{ca-GrQc}, from $0.52$ to $0.46$ (eight regeneration seeds): the degree agent carries almost none of the grid's degree dispersion but a sizable fraction of \textsf{ca-GrQc}'s, the generative signature of a network without a hub mechanism. The operator test of Table~\ref{tab:denseop} supplies a third, concordant reading: the grid has no dense core, its maximal $k$-core reaching only $k=5$ on twelve vertices, and no induced subgraph is at once dense enough to bring $n^2\rho_n^3$ into the powered range and large enough to estimate blocks, so the Hadamard-interaction test is underpowered there, consistent with a single near-regular spatial mechanism; the cores of \textsf{polblogs}, \textsf{ca-GrQc}, and \textsf{ca-CondMat} instead expose dense superpositions on which the test returns a verdict. The no-hub reading is therefore falsifiable on three independent axes that share no parametrisation: an out-of-sample anti-predictive degree score (Table~\ref{tab:linkpred}), a degree distribution unmoved by the degree-agent knockout, and no dense core to host a superposition. The grid satisfies all three and \textsf{ca-GrQc} none, separating a hubless spatial network from a hub-bearing collaboration on generative, predictive, and operator grounds at once.

\label{sec:linkpred}
On the Western States power grid, the canonical network \emph{without} hubs
\citep{watts1998collective}, the synthesis assigns the degree agent a negative
coefficient whose interval excludes zero in all ten folds (Table~\ref{tab:weights}), while that
agent's own ranking score is below chance (Table~\ref{tab:linkpred}). The estimator therefore
reports the \emph{absence} of a mechanism, with replicated uncertainty statements, by the route
Proposition~\ref{prop:mis-cor}(b) describes: an agent whose kernel is anti-aligned with the residual
structure enters the projection with a negative coordinate, as a corrective contrast.
Propositions~\ref{prop:mis-cor}(c) and~\ref{thm:clt}(b) state that the projected predictor is weakly
better in population than every single agent and that the bootstrap intervals are asymptotically valid;
the comparison below exhibits both at finite $n$.

A sharper test than aggregate structural summaries is predictive: whether the synthesised model, with weights
estimated rather than chosen, forecasts edges it has not seen. We evaluate link prediction by network
cross-validation on six datasets from three domains: the \textsf{les\,mis} ($n=77$) and \textsf{GoT}
($n=107$) character networks, the \textsf{polblogs} hyperlink network ($n=1222$), the \textsf{ca-GrQc}
($n=4158$) and \textsf{ca-CondMat} ($n=21{,}363$) collaboration networks
\citep{leskovec2007graph}, and the Western States \textsf{power} grid ($n=4941$, $m=6594$;
\citealp{watts1998collective}). In each of ten folds with fixed seeds, $20\%$ of edges are held out and split
equally into calibration and test positives, with equal numbers of sampled non-edges; all models are fitted on
the remaining $80\%$ training graph, so the calibration set is disjoint from training and test dyads,
following the edge-sampling cross-validation of \citet{li2020network} in the held-out calibration form of
Section~\ref{sec:estimation}. These pipelines fit Stage~A on the edge-deleted training graph without
inverse-probability reweighting, so the Stage-A input has conditional mean $(1-h)\Pmat$ with $h$ the holdout
fraction; uniform dyad deletion rescales every fitted kernel by the same factor, which the unit-scale
normalisation of Assumption~\ref{ass:norm} and the calibration intercept absorb, so coefficient ratios,
selections, and ranks are unaffected while the masked-IPW construction of Section~\ref{sec:cf} remains the
analysed one. The candidate set comprises five generative kernels (Chung--Lu, SBM, degree-corrected SBM,
RDPG/ASE, degree-corrected mixed-membership SBM), two predictive topological indices (Adamic--Adar, Jaccard),
and a hyperbolic popularity--similarity embedding; only the generative kernels carry a mechanism
interpretation. Each scorer is fitted on the training graph and its scores mapped to probabilities by Platt
scaling on the calibration dyads, which places the topological indices on an equal footing and improves their
calibration, so the comparison favours the baselines. The synthesis is the calibration estimator of
Section~\ref{sec:estimation}, a logistic regression of held-out edge indicators on the agents' standardised
scores fitted on the calibration dyads only; AUC is computed from raw scores, Platt scaling being monotone,
and all reported metrics on the untouched test set.

We report AUC for ranking and the balanced held-out (case--control) log score
\[
\frac1{|T|}\sum_{(i,j)\in T}\bigl[\,A_{ij}\log\hat P_{ij}+(1-A_{ij})\log(1-\hat P_{ij})\,\bigr],
\]
a strictly proper scoring rule on the balanced test set of positives and equally many sampled non-edges; because the negatives are subsampled, it is a case--control score, not the natural-prevalence graph log-likelihood. The results (Table~\ref{tab:linkpred}, Figure~\ref{fig:linkpred}) follow the
pattern the estimation theory predicts. On all four networks with $n\ge10^3$ the synthesis improves on the best single mechanism in both
metrics, with fold standard deviations at most $0.013$ in AUC and $0.03$ in log-likelihood (the
comparison against random-forest stacking and ensemble averaging is
deferred to Table~\ref{tab:combiners}): $0.942/{-0.31}$ on \textsf{polblogs} against $0.932$ (DC-SBM) and $-0.37$ (RDPG) for
the best single methods, $0.935/{-0.27}$ on \textsf{ca-GrQc} against $0.910/{-0.29}$ (Adamic--Adar),
$0.961/{-0.17}$ on \textsf{ca-CondMat} against $0.949/{-0.18}$ (Adamic--Adar), and $0.807/{-0.54}$
on the \textsf{power} grid against $0.726/{-0.56}$ (SBM). The power grid is the sharpest case: every
single agent performs poorly, because no single mechanism describes an electrical network, and the
margin of the combination over the best single agent ($0.081$ AUC) is the largest in the table. On
the two smallest networks the synthesis is within one fold standard deviation of the best single
method on both metrics; at $n=77$ the calibration split contains a few dozen positive dyads from
which eight weights must be resolved, so the calibration sample of Proposition~\ref{thm:clt}(b) is the
binding constraint, and the synthesis matches rather than improves on the best single
heuristic. The hyperbolic baseline illustrates the distinction between structural and predictive
fit: it is strong on \textsf{polblogs} ($0.915/{-0.37}$) and weak on \textsf{ca-GrQc}
($0.807/{-0.52}$), despite that network's strong clustering and heavy degree tail, the regime the
model targets. (Our hyperbolic scorer is a fitted popularity--similarity heuristic,
with radii from degrees and angles from the leading spectral embedding, following
\citealp{papadopoulos2014network}, rather than a full Mercator-style maximum-likelihood embedding, which
lies outside our scope and might narrow the \textsf{ca-GrQc} gap; its \textsf{polblogs} performance
indicates that it is a competitive representative of its class.)

A falsification check completes the design. We add a seventh graph, an Erd\H{o}s--R\'enyi null
matched to \textsf{ca-GrQc} in size and density, on which the correct answer is that no structured
mechanism is present. Every method, the synthesis included, scores at chance on the null (AUC
$0.50$, log-likelihood at the entropy of the balanced test set), every synthesis coefficient is
within noise of zero, and none meets the replication criterion below (at most one significant fold
in ten, consistent with the per-fold $5\%$ level, which makes $0.5$ significant folds the expected
count under the null). The estimator therefore reports structure where structure
exists and declines to report it where none does.

Because the synthesis is a regression, uncertainty quantification is immediate.
Table~\ref{tab:weights} reports bootstrap intervals for the coefficients ($B=200$ resamples of the
calibration dyads), computed on every fold, together with the synthesis's own performance against
the best single agent. We adopt the following replication criterion as protocol: a coefficient is
reported as a finding only when its interval excludes zero with a consistent sign in at least nine
of ten folds, and only such coefficients are printed in bold. The coefficients are case--control calibration coefficients, not generative proportions. The three collaboration-style graphs place replicated weight on triangle closure and block structure: \textsf{ca-GrQc} on
Adamic--Adar, Jaccard, SBM and DC-SBM, and \textsf{ca-CondMat} on Adamic--Adar, Jaccard, SBM and
Chung--Lu. \textsf{polblogs} places replicated weight on geometry (ASE, the hyperbolic embedding and the
mixed-membership agent), consistent with its two-camp latent structure. The \textsf{power} grid
places replicated weight on the block agents, consistent with its strong modularity, and carries the replicated
negative Chung--Lu coefficient discussed above. In one case a coefficient is excluded by the replication criterion: the hyperbolic coefficient on 	extsf{ca-GrQc} is significant within individual folds but changes sign across folds: a nearly-spanned column has small residual norm and so an unstable sign, the behaviour Proposition~
ef{prop:mis-cor}(b) predicts, and the criterion screens it out.
The synthesis function of Section~\ref{sec:setup} thus acts as an interpretable diagnostic: which
mechanisms matter, with replicated error bars, varies across networks in accordance with their known
structure.

\paragraph{Paired-fold significance.}
The close cases are significant under a paired test. Writing $\Delta_\ell=\mathrm{LL}^{\mathrm{syn}}
_\ell-\max_k\mathrm{LL}^{(k)}_\ell$ for the per-fold gap to the best single agent, the synthesis
improves the held-out log score on all ten folds on every network with $n\ge10^3$: on
\textsf{ca-CondMat} $\overline\Delta=+0.0066\pm0.0003$ (mean over ten folds $\pm$ standard error,
$10/10$ folds positive), on \textsf{polblogs} $+0.050\pm0.002$, on \textsf{ca-GrQc}
$+0.018\pm0.001$, and on the power grid $+0.028\pm0.002$; the AUC gaps are likewise positive on all
ten folds. On the two smallest networks the gap is negative, consistent with the small calibration
sample at $n=77$ and $n=107$.

A synthetic graph whose three regions carry deliberately different mechanisms (one Erd\H{o}s--R\'enyi, one small-world, one scale-free) checks the machinery against a known generating truth: spectral clustering recovers the regions exactly, the region-restricted agents attribute the correct mechanism to each block, and the debiasing of Theorem~\ref{thm:debias} is what removes the spurious self-product that would otherwise mislabel the Erd\H{o}s--R\'enyi block as scale-free. The full illustration is in Appendix~\ref{app:numerical} (Figure~\ref{fig:regional}).

\subsection{What inference buys, and what it costs}

\paragraph{What inference buys.} The synthesis supplies five capabilities a black-box combiner cannot: a signed coefficient per mechanism, a confidence interval, a support-recovery guarantee, a test for the combination operator, and a generative counterfactual, namely that setting a mechanism's coefficient to zero and regenerating predicts the structural consequence of its removal (the degree-Gini knockout of Section~\ref{sec:linkpred}), a question a predictor cannot pose. The price is modest and we quantify it. Against a random forest trained on the same candidate scores, the synthesis ranks within $0.008$ to $0.077$ AUC across the six networks while supplying that inference: it is the better-calibrated forecaster at the natural edge prevalence on four of six networks, and on the two networks large enough to fit an augmented model a rank-transformed synthesis with sparse interactions closes $40\%$ to $56\%$ of the residual AUC gap, leaving about $0.016$ AUC that is the forest's higher-order non-additive flexibility, recoverable by no additive coefficient model. Plotting accuracy against inferential content places the synthesis family on the maximal-inference corner of a Pareto frontier; the full calibration, gap-decomposition, and Pareto results are in Appendix~\ref{app:numerical} (Tables and Figures there).

\paragraph{Relation to stacking, ensembles, and model averaging on networks.}
The predictive protocol above is, by construction, stacking on dyads. We compare it against the leading
alternatives on the same eight calibrated agent scores, but the comparison is not a contest in predictive
accuracy, because the methods answer different questions: stacking and the ensemble average return a combined
forecast, while the synthesis returns a signed coefficient per mechanism with a confidence interval, a
selection guarantee, and a test for the combination operator. Table~\ref{tab:combiners} records this as a
capability matrix. The random-forest meta-learner of \citet{ghasemian2020stacking}, reported as near-optimal for link
prediction across $550$ networks, is the most accurate combiner in the study; the synthesis attains it to
within $0.008$ to $0.077$ AUC on the six real networks, the gap largest on the power grid, while additionally
supplying the inferential capabilities the forest cannot. The distinction in kind is visible in the fits:
the convex combiners place their mass on a single agent or a flat average, whereas the synthesis coefficients
of Table~\ref{tab:weights} spread across several mechanisms with intervals excluding zero, and on the power
grid a negative Chung--Lu coefficient records a partial association that no convex average can represent. The
per-network AUC and log score are in Table~\ref{tab:linkpred} and the breakdown by combiner in
Table~\ref{tab:combiner-auc}; edge-cross-validated network averaging \citep{zhang2025network} sits in the same
predictive-only category as the random forest.

\begin{table}[t]
\centering
\small
\caption{Capabilities of the combination rules: mean AUC over six networks, and whether each yields a signed per-mechanism coefficient, a confidence interval, support recovery, and an operator test.}
\label{tab:combiners}
\begin{tabular}{lccccc}
\toprule
& Mean AUC & Signed & Confidence & Support & Operator\\
Method & (6 real nets) & coefficient & interval & recovery & test\\
\midrule
Calibrated synthesis & $0.90$ & \ding{51} & \ding{51} & \ding{51} & \ding{51}\\
RF stacking & $0.93$ & \ding{55} & \ding{55} & \ding{55} & \ding{55}\\
Ensemble averaging & $0.88$ & \ding{55} & \ding{55} & \ding{55} & \ding{55}\\
\bottomrule
\end{tabular}
\end{table}

\begin{table}[t]
\centering
\scriptsize
\caption{Out-of-sample link prediction over six networks ($n=77$ to $21{,}363$): held-out AUC and balanced log score, means over ten folds.}
\label{tab:linkpred}
\setlength{\tabcolsep}{2.6pt}
\begin{tabular}{lcccccccccccc}
\toprule
& \multicolumn{2}{c}{\textsf{les\,mis}} & \multicolumn{2}{c}{\textsf{GoT}}
& \multicolumn{2}{c}{\textsf{polblogs}} & \multicolumn{2}{c}{\textsf{ca-GrQc}}
& \multicolumn{2}{c}{\textsf{power}} & \multicolumn{2}{c}{\textsf{ca-CondMat}}\\
\cmidrule(lr){2-3}\cmidrule(lr){4-5}\cmidrule(lr){6-7}\cmidrule(lr){8-9}\cmidrule(lr){10-11}\cmidrule(lr){12-13}
Model & AUC & LL & AUC & LL & AUC & LL & AUC & LL & AUC & LL & AUC & LL\\
\midrule
Chung--Lu & 0.755 & $-0.60$ & 0.766 & $-0.57$ & 0.902 & $-0.46$ & 0.749 & $-0.57$ & 0.446 & $-0.69$ & 0.765 & $-0.58$\\
SBM & 0.802 & $-0.55$ & 0.741 & $-0.65$ & 0.756 & $-0.56$ & 0.877 & $-0.42$ & 0.726 & $-0.56$ & 0.860 & $-0.46$\\
DC-SBM & 0.873 & $-0.45$ & 0.847 & $-0.51$ & 0.932 & $-0.40$ & 0.893 & $-0.44$ & 0.725 & $-0.58$ & 0.907 & $-0.45$\\
RDPG/ASE & 0.776 & $-0.52$ & 0.769 & $-0.60$ & 0.927 & $-0.37$ & 0.755 & $-0.58$ & 0.541 & $-0.68$ & 0.791 & $-0.61$\\
Adamic--Adar & \textbf{0.916} & $\mathbf{-0.37}$ & \textbf{0.875} & $\mathbf{-0.44}$ & 0.918 & $-0.39$ & 0.910 & $-0.29$ & 0.575 & $-0.64$ & 0.949 & $-0.18$\\
Jaccard & 0.878 & $-0.48$ & 0.830 & $-0.54$ & 0.871 & $-0.50$ & 0.910 & $-0.29$ & 0.575 & $-0.64$ & 0.948 & $-0.20$\\
Hyperbolic & 0.819 & $-0.56$ & 0.820 & $-0.55$ & 0.915 & $-0.37$ & 0.807 & $-0.52$ & 0.573 & $-0.68$ & 0.819 & $-0.53$\\
DC-MMSBM & 0.776 & $-0.58$ & 0.711 & $-0.67$ & 0.868 & $-0.49$ & 0.828 & $-0.53$ & 0.531 & $-0.64$ & 0.833 & $-0.54$\\
\textbf{Calibrated synthesis} & 0.886 & $-0.41$ & \textbf{0.873} & $-0.44$ & \textbf{0.942} & $\mathbf{-0.31}$ & \textbf{0.935} & $\mathbf{-0.27}$ & \textbf{0.807} & $\mathbf{-0.54}$ & \textbf{0.961} & $\mathbf{-0.17}$\\
\bottomrule
\end{tabular}
\end{table}

\begin{table}[t]
\centering
\footnotesize
\caption{The fitted synthesis on each network and its held-out performance: the eight standardised agent coefficients, with held-out AUC and log score.}
\label{tab:weights}
\setlength{\tabcolsep}{2.0pt}
\resizebox{\textwidth}{!}{%
\begin{tabular}{lccccccc}
\toprule
Agent & \textsf{les\,mis} & \textsf{GoT} & \textsf{polblogs} & \textsf{ca-GrQc} & \textsf{power} & \textsf{ca-CondMat} & \textsf{ER null}\\
\midrule
Chung--Lu & $+0.82$ {\tiny$[+0.15,+1.23]$} & $+0.15$ {\tiny$[-0.20,+0.47]$} & $+0.53$ {\tiny$[+0.17,+0.94]$} & $+0.65$ {\tiny$[+0.23,+1.08]$} & $\mathbf{-0.45}$ {\tiny$[-0.76,-0.16]$} & $\mathbf{+1.22}$ {\tiny$[+0.82,+1.63]$} & $-0.01$ {\tiny$[-0.09,+0.08]$}\\
SBM & $+0.12$ {\tiny$[-0.70,+0.79]$} & $+0.50$ {\tiny$[+0.22,+0.79]$} & $+0.06$ {\tiny$[-0.10,+0.22]$} & $\mathbf{+3.06}$ {\tiny$[+2.25,+3.84]$} & $\mathbf{+1.62}$ {\tiny$[+1.25,+1.93]$} & $\mathbf{+2.02}$ {\tiny$[+1.50,+2.69]$} & $+0.02$ {\tiny$[-0.07,+0.10]$}\\
DC-SBM & $\mathbf{+0.64}$ {\tiny$[+0.38,+0.85]$} & $+0.14$ {\tiny$[-0.19,+0.55]$} & $-0.41$ {\tiny$[-0.98,+0.35]$} & $\mathbf{+2.49}$ {\tiny$[+1.47,+3.38]$} & $\mathbf{+1.17}$ {\tiny$[+0.82,+1.43]$} & $+2.06$ {\tiny$[+0.75,+3.22]$} & $-0.02$ {\tiny$[-0.09,+0.05]$}\\
RDPG/ASE & $+0.33$ {\tiny$[-0.32,+0.97]$} & $+0.54$ {\tiny$[+0.22,+0.77]$} & $\mathbf{+2.25}$ {\tiny$[+1.79,+2.77]$} & $+0.50$ {\tiny$[+0.28,+0.85]$} & $+0.06$ {\tiny$[-0.14,+0.39]$} & $-0.40$ {\tiny$[-1.31,+0.55]$} & $+0.01$ {\tiny$[-0.05,+0.09]$}\\
Adamic--Adar & $\mathbf{+1.26}$ {\tiny$[+0.87,+1.62]$} & $\mathbf{+0.59}$ {\tiny$[+0.32,+0.85]$} & $+0.89$ {\tiny$[-0.16,+1.60]$} & $\mathbf{+6.17}$ {\tiny$[+5.35,+7.40]$} & $\mathbf{+0.63}$ {\tiny$[+0.41,+0.93]$} & $\mathbf{+10.18}$ {\tiny$[+9.18,+11.57]$} & $+0.02$ {\tiny$[-0.03,+0.08]$}\\
Jaccard & $-0.11$ {\tiny$[-0.73,+0.75]$} & $+0.36$ {\tiny$[+0.00,+0.81]$} & $-0.05$ {\tiny$[-0.27,+0.28]$} & $\mathbf{+4.82}$ {\tiny$[+3.04,+6.13]$} & $\mathbf{+0.41}$ {\tiny$[+0.32,+0.50]$} & $\mathbf{+4.35}$ {\tiny$[+2.32,+6.03]$} & $+0.03$ {\tiny$[-0.03,+0.08]$}\\
Hyperbolic & $+1.40$ {\tiny$[+0.71,+2.02]$} & $+0.25$ {\tiny$[-0.08,+0.66]$} & $\mathbf{+1.11}$ {\tiny$[+0.89,+1.35]$} & $-0.26$ {\tiny$[-0.55,-0.02]$} & $+0.27$ {\tiny$[+0.15,+0.38]$} & $-0.08$ {\tiny$[-0.24,+0.06]$} & $-0.02$ {\tiny$[-0.10,+0.06]$}\\
DC-MMSBM & $+0.04$ {\tiny$[-0.55,+0.75]$} & $-0.00$ {\tiny$[-0.30,+0.34]$} & $\mathbf{+0.36}$ {\tiny$[+0.24,+0.54]$} & $+0.51$ {\tiny$[+0.08,+1.86]$} & $+0.98$ {\tiny$[+0.52,+1.53]$} & $+0.16$ {\tiny$[-0.02,+0.38]$} & $-0.03$ {\tiny$[-0.11,+0.04]$}\\
\midrule
Synthesis AUC & 0.886 & 0.873 & 0.942 & 0.935 & 0.807 & 0.961 & 0.502\\
Best single AUC & 0.916 {\tiny(Adamic--Adar)} & 0.875 {\tiny(Adamic--Adar)} & 0.932 {\tiny(DC-SBM)} & 0.910 {\tiny(Adamic--Adar)} & 0.726 {\tiny(SBM)} & 0.949 {\tiny(Adamic--Adar)} & 0.505 {\tiny(RDPG/ASE)}\\
Synthesis LL & $-0.41$ & $-0.44$ & $-0.31$ & $-0.27$ & $-0.54$ & $-0.17$ & $-0.69$\\
Best single LL & $-0.37$ {\tiny(Adamic--Adar)} & $-0.44$ {\tiny(Adamic--Adar)} & $-0.37$ {\tiny(RDPG/ASE)} & $-0.29$ {\tiny(Adamic--Adar)} & $-0.56$ {\tiny(SBM)} & $-0.18$ {\tiny(Adamic--Adar)} & $-0.69$ {\tiny(DC-SBM)}\\
\bottomrule
\end{tabular}}
\end{table}

\subsection{Structural goodness-of-fit on \textsf{ca-GrQc}}
\label{sec:gof}
The knockout study manipulates a controlled generator; a complementary test asks whether the synthesised model, a superposition of a small-world, a hub, and a community layer, reproduces the joint structure of a real network. The \textsf{ca-GrQc} co-authorship graph \citep{leskovec2007graph} ($n=4158$, mean degree $6.5$) displays all four canonical properties at once: maximum degree $81$, average local clustering $0.56$, mean path length $5.99$, and modularity $0.85$. Figure~\ref{fig:gof} compares single-mechanism models against the synthesised superposition. Each single model misses at least one property: Erd\H{o}s--R\'enyi, Chung--Lu, and Barab\'asi--Albert produce almost no clustering and no communities; Watts--Strogatz produces a maximum degree an order of magnitude too small and paths half again too long; and the degree-corrected stochastic block model, the strongest single parametric model, under-produces clustering fourfold ($0.12$ against $0.56$). The synthesised model is the only one that clears all four one-sided screens at once (maximum degree $80$, clustering $0.34$, path length $4.5$, modularity $0.71$), though its clustering and modularity sit somewhat below the observed values. The hyperbolic random graph also clears the four qualitative screens in one shot, but at the cost of control over the values, overshooting the maximum degree fifteenfold because the heavy tail it must adopt is unboundedly heavier than the truncated tail the data display. The synthesis matches the observed values together because each mechanism is calibrated separately, the structural counterpart of the predictive comparison in Section~\ref{sec:linkpred}.

\begin{figure}[t]
\centering
\includegraphics[width=\textwidth]{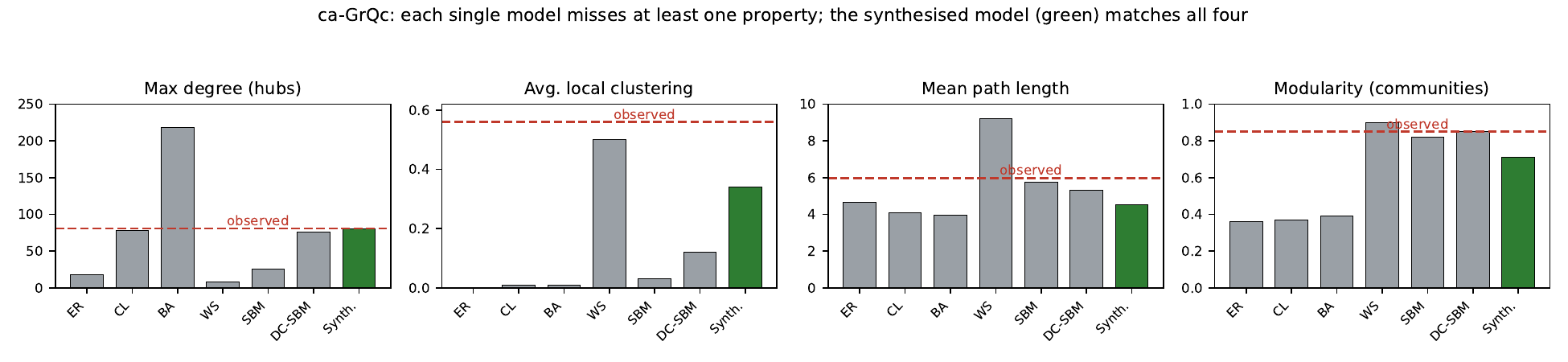}
\caption{Structural goodness-of-fit on \textsf{ca-GrQc}. Each panel is one structural property; bars are models and the dashed line is the observed value. Single-mechanism models (grey) each miss at least one property, typically clustering or hubs, while the synthesised model (green) is the only one to clear all four one-sided screens at once, though its clustering ($0.34$ against $0.56$) and modularity ($0.71$ against $0.85$) sit below the observed values.}
\label{fig:gof}
\end{figure}

\section{Discussion}
\label{sec:discussion}

For a single graph and a prescribed collection of candidate mechanisms, the paper studies the
decomposition of the graph over these mechanisms, the uncertainty of that decomposition, and the
conditions for its recovery, settled by the attribution principle of Section~\ref{sec:intro}: generative weights where the candidates spectrally separate and the edge-overlap information diverges, a signed projection coefficient otherwise. Two of the conclusions are, to our knowledge, new in the single-network setting. First, the manner of combination is decidable from one graph: for two layers
the additive and overlapping operators are distinguishable exactly when the edge-overlap information
$n^2\rho_n^3$ diverges, the boundary extends to any fixed number of layers, and the test stays valid
with all layers fitted on disjoint folds (Theorem~\ref{thm:nor}, Corollary~\ref{cor:nork-cor},
Theorem~\ref{prop:fitted-operator}). Second, fitting the kernels induces an errors-in-variables
attenuation that a cross-fold construction removes, separating the rate-suboptimal naive plug-in
from the optimal debiased estimator in order (Theorem~\ref{thm:debias}, Corollary~\ref{cor:debias-lb}).
Around these sit the known-design rate and matching lower bound (Proposition~\ref{thm:joint}), the
asymptotically valid calibration estimator and bootstrap (Proposition~\ref{thm:clt}), the projection
interpretation with corrective coefficients (Proposition~\ref{prop:mis-cor}), and threshold selection
without oracle knowledge (Corollary~\ref{cor:adapt-cor}).

\paragraph{What the empirical study establishes.}
On six networks from three domains the fitted synthesis improves on every single mechanism in ranking
and calibration on each network with at least $10^3$ nodes, and is the most accurate of the
interpretable combination rules (Table~\ref{tab:combiners}). Its coefficient estimates, with
replicated intervals, agree with each network's accepted structural summary. The power-grid analysis
illustrates the value of signed coefficients: the degree-heterogeneity agent receives a coefficient
whose interval excludes zero and lies below it in all ten folds, recording a negative partial
association that a nonnegative combination cannot represent. The coefficients are candidate-relative projection or calibration targets: a negative Chung--Lu coefficient
states that, given the rest of the candidate set and the calibration loss, the degree-product score is
anti-aligned with the held-out edges, a statement about predictive structure under the chosen candidate set
and not about the historical growth of the network.

\paragraph{Limitations.}
Six limitations delimit the results. (i)~The estimand is candidate- and normalisation-relative: rescaling an agent rescales its coefficient inversely, and the latent components of a mixture are not identifiable from one graph; the generative-versus-projection reading is the attribution principle of Section~\ref{sec:intro}. (ii)~The sharp minimax rate is established for the
known-design class; for fitted kernels the cross-fold debiased estimator attains it on the named
spectral and degree candidate set (Theorem~\ref{thm:debias}), while candidate sets outside that class need a separate
first-stage expansion and the naive single-fold plug-in is not rate-optimal. (iii)~The operator test
uses a cross-fold interaction column for layers fitted from one graph
(Theorem~\ref{prop:fitted-operator}), with the boundary extending to any fixed number of layers
(Corollary~\ref{cor:nork-cor}); its constants degrade with the generically small interaction
transversality, and full constrained recovery for three or more layers remains a conjecture
(Remark~\ref{rem:nork-conj}). By the boundary of Theorem~\ref{thm:nor} the operator is decidable only where $n^2\rho_n^3\approx\bar d^{\,3}/n$ is large; at full-graph scale this holds on \textsf{polblogs} (Table~\ref{tab:opfeas}), and the dense-core analysis of Table~\ref{tab:denseop} returns a verdict on the cores of \textsf{polblogs}, \textsf{ca-GrQc}, and \textsf{ca-CondMat}. (iv)~Community recovery under superposition is cited rather than proved; the Louvain
block agents lack a population analogue of Assumption~\ref{ass:plugin} and are read as a practical
proxy. (v)~The popularity--similarity agent is a heuristic rather than a maximum-likelihood embedding.
(vi)~The fitted-layer validations of Section~\ref{sec:infval} operate in the embedding-linearisation regime in which the disjoint-fold main-effect columns sit at their consistent limits, the regime in which Theorems~\ref{thm:twostage} and~\ref{prop:fitted-operator} are stated.

\paragraph{Future work.}
Three directions extend the programme: a temporal synthesis with time-varying $w_k(t)$ over network
snapshots, identifying when the composition shifts; a higher-order version cross-fitting on triangles
or motifs for hypergraph and motif generators, where an edge-overlap analogue should again govern
testability; and bringing the candidate set inside the estimand through a data-driven search with
post-selection inference.

\ifblind\else
\section*{Acknowledgments}
Funding and acknowledgment of individuals to be completed per JASA policy: list all third-party funding and support for this work, or state that none was received.
\fi

\section*{Proofs and supplementary results}
All proofs are collected in the supplementary material, together with three further results used in the main text only through their consequences: an estimated-design minimax lower bound, the necessity of the $\sqrt{\log K}$ selection factor, and fixed-$K$ noisy-OR recovery. Appendix references below point to that document.

\section*{Disclosure Statement}
The authors report there are no competing interests to declare.

\section*{Data Availability Statement}
The network datasets analysed in the real-data study are publicly available benchmark networks: the \texttt{ca-GrQc} and \texttt{ca-CondMat} scientific-collaboration networks are from the Stanford Network Analysis Project repository \citep{leskovec2007graph}; the political-blogs network is from \citet{adamic2005political}; the Western-states power-grid network is from \citet{watts1998collective}; and the Game-of-Thrones character co-appearance network and the Les Mis\'erables co-appearance network are standard publicly available network datasets. \ifblind\else Code reproducing all simulations, tables, and figures, together with the processed versions of these datasets, is openly available at \url{https://osf.io/XXXXX} (deposit identifier to be finalized at the camera-ready stage).\fi

\renewcommand{\refname}{References}

\clearpage
\begin{center}
{\Large\bfseries Supplementary Material}\\[0.4em]
{\large for \emph{Minimax Synthesis of Network Mechanisms}}
\end{center}
\bigskip

\noindent This document contains all proofs for the main paper, together with the supplementary
results referenced there. A reference of the form Theorem~\ref{thm:debias} points to the main paper;
results native to this supplement, including the restatement that precedes each proof, are numbered
with an \textsf{S} prefix (Theorem~S1, Corollary~S2, and so on).

\appendix
\setcounter{section}{0}\setcounter{theorem}{0}\setcounter{proposition}{0}\setcounter{lemma}{0}\setcounter{corollary}{0}\setcounter{definition}{0}\setcounter{remark}{0}\setcounter{example}{0}\setcounter{assumption}{0}\setcounter{protocol}{0}\setcounter{figure}{0}\setcounter{table}{0}\setcounter{equation}{0}
\renewcommand{\thetheorem}{S\arabic{theorem}}
\renewcommand{\theproposition}{S\arabic{proposition}}
\renewcommand{\thelemma}{S\arabic{lemma}}
\renewcommand{\thecorollary}{S\arabic{corollary}}
\renewcommand{\thedefinition}{S\arabic{definition}}
\renewcommand{\theremark}{S\arabic{remark}}
\renewcommand{\theexample}{S\arabic{example}}
\renewcommand{\theassumption}{S\arabic{assumption}}
\renewcommand{\theprotocol}{S\arabic{protocol}}
\renewcommand{\thefigure}{S\arabic{figure}}
\renewcommand{\thetable}{S\arabic{table}}
\renewcommand{\theequation}{S\arabic{equation}}
\section{Proofs of the structural theorems}
\label{app:structural}
This appendix gives the full proofs of the three results of Section~\ref{sec:structural}: the
impossibility theorem (Theorem~\ref{thm:imp}), the clustering--hub--rank frontier
(Theorem~\ref{thm:frontier}), and the possibility theorem (Theorem~\ref{thm:pos}). The impossibility
argument is carried out family by family across the stochastic block model, the small-world, the
Chung--Lu, and the bounded-rank dot-product families.

\subsection{Proof of Theorem~\ref{thm:imp}: impossibility within four families}
\begin{proof}
\emph{(a) Sparse SBM and Chung--Lu clustering.} In the rank-one Chung--Lu model the edge
probabilities are $P_{ab}=\theta_a\theta_b/S$ with $S=\sum_a\theta_a$. By conditional independence,
the probability that three distinct vertices $i,j,\ell$ form a triangle is
$P_{ij}P_{j\ell}P_{i\ell}=\theta_i^2\theta_j^2\theta_\ell^2/S^3$. Summing over ordered triples, the
expected number of triangles is
$\E[T]=\tfrac16\sum_{i,j,\ell}P_{ij}P_{j\ell}P_{i\ell}=\tfrac16(\sum_a\theta_a^2)^3/S^3=\Theta(1)$
when the degrees have bounded second moment, while the expected number of length-two paths is
$\E[W]=\tfrac12\sum_j(\sum_iP_{ij})^2=\tfrac12\sum_jd_j^2=\Theta(n)$ for mean degree $\Theta(1)$.
Hence $C_n=3\E[T]/\E[W]=\Theta(1/n)\to0$. The same calculation with $P_{ab}=\rho_nB_{c_ac_b}$ gives
$\E[T]=\Theta(n^3\rho_n^3)=\Theta(\bar d_n^3)$ and $\E[W]=\Theta(n\bar d_n^2)$, so
$C_n=\Theta(\bar d_n/n)$; boundedness below would require $\bar d_n=\Theta(n)$, a dense graph, which
violates the regularly varying tail (for which $\bar d_n=n^{o(1)}$) and leaves the short-path
property vacuous. That Chung--Lu has no communities is immediate: its law is exchangeable given the
degree sequence and contains no latent partition, so no estimator beats the chance agreement $1/Q$.

\emph{(b) Small-world degree concentration.} Each vertex begins with degree $2m$; reassigning a
fixed fraction of endpoints changes a vertex degree by a sum of $O(1)$ independent contributions
whose limit law has all moments finite. Therefore every degree lies in $2m\pm O_p(1)$, the empirical
degree distribution converges to a fixed law supported near $2m$, and no polynomial tail can arise.
The construction plants no blocks, so there is no community structure to recover.

\emph{(c) Bounded-rank RDPG clustering under a heavy tail.} Let $\Pmat=\X\I_{p,q}\X^\top$ with
$p+q=d$ fixed and rows $\x_i$. The degree of $i$ is, up to lower order, a fixed linear functional of
$\x_i$, so a regularly varying degree tail requires the latent positions to have a regularly varying
tail along the direction $\I_{p,q}\bar\x$, where $\bar\x$ is their empirical mean. The number of
length-two paths is $\E[W]=\tfrac12\sum_jd_j^2=\Theta(n\,\E[d^2])$, while the triangle count is a
trace of a cube of a rank-$d$ matrix,
\[
6\,\E[T]=\sum_{i,j,\ell}\inner{\x_i}{\x_j}_{p,q}\inner{\x_j}{\x_\ell}_{p,q}\inner{\x_\ell}{\x_i}_{p,q}
=\tr\!\bigl[(\I_{p,q}\X^\top\X)^3\bigr]\le d\,\sigma_1^6,
\]
where $\sigma_1$ is the largest singular value of $\X$ and the mean degree along the leading
direction is $\bar d_n=\Theta(\sigma_1^2)$. Hence $C_n\le d\,\sigma_1^6/(2n\,\E[d^2])$. A tail with
$\tau\in(2,3)$ has $\E[d^2]=\Theta(n^{(3-\tau)/(\tau-1)})\to\infty$ polynomially, while entries in
$[0,1]$ cap $\sigma_1^2=O(\bar d_n\,n^{o(1)})$, so $C_n=O(n^{o(1)}\bar d_n^3/(n\E[d^2]))\to0$. A
bounded-rank dot-product kernel cannot keep constant clustering once its degrees are heavy-tailed:
the rank budget needed to seed a constant clustering per vertex grows with $n$, contradicting the
bounded dimension.
\end{proof}

\subsection{Proof of Theorem~\ref{thm:frontier}: the clustering--hub--rank frontier}
\begin{proof}
\emph{(a)} Write $\mathbf D_{\!\bbQ}=\diag(\bbQ)$. Expanding,
$\tr(\Pmat^3)=\tr(\bbQ^3)-3\tr(\bbQ^2\mathbf D_{\!\bbQ})+3\tr(\bbQ\mathbf D_{\!\bbQ}^2)-\tr(\mathbf
D_{\!\bbQ}^3)$. Since $(\bbQ^2)_{ii}=\sum_j\bbQ_{ij}^2\ge\bbQ_{ii}^2$ we have
$\tr(\bbQ^2\mathbf D_{\!\bbQ})\ge\tr(\bbQ\mathbf D_{\!\bbQ}^2)$, and $\tr(\mathbf D_{\!\bbQ}^3)\ge0$,
so $\tr(\Pmat^3)\le\tr(\bbQ^3)=\sum_i\lambda_i^3\le d\,s_1^3$, as at most $d$ eigenvalues are nonzero
and each is at most $s_1$. For the denominator $\norm{r}^2\ge\Delta^2$ and
$\norm{\Pmat}_F^2\le\norm{\bbQ}_F^2=\sum\lambda_i^2\le d\,s_1^2$, so
$\norm{r}^2-\norm{\Pmat}_F^2\ge\Delta^2-d\,s_1^2$; multiplying by $\mathcal C$ gives (a).

\emph{(b)} If $\Delta^2\ge2d\kappa^2\bar d^{\,2}\ge2d\,s_1^2$ then $(\Delta^2-ds_1^2)_+\ge\Delta^2/2$,
and (a) gives $\mathcal C\Delta^2/2\le d\,s_1^3\le d\kappa^3\bar d^{\,3}$; rearranging gives the cap.

\emph{(c)} If $d\,s_1^2\ge\Delta^2/2$ then $s_1\ge\Delta/\sqrt{2d}\ge(c/2d)^{1/3}\Delta^{2/3}$ for
$\Delta\ge c\sqrt{2d}$. Otherwise $(\Delta^2-ds_1^2)_+\ge\Delta^2/2$ and (a) forces
$s_1^3\ge c\Delta^2/(2d)$. For the hub-local bound, $p_h^\top\Pmat\,p_h\le s_1\norm{p_h}^2\le
s_1\Delta$ since entries lie in $[0,1]$, while the denominator is at most $\Delta^2$, so
$\mathcal C_h\le s_1\Delta/\Delta^2$, i.e.\ $s_1\ge\mathcal C_h\Delta$.

\emph{(d)} The anchored kernel has $s_1=\norm x^2=\alpha^2+\beta^2m=q^2/c+c\Delta/q$. In the uncapped
regime $q=(c^2\Delta/2)^{1/3}$ gives $s_1=c^{1/3}\Delta^{2/3}(2^{-2/3}+2^{1/3})<2c^{1/3}\Delta^{2/3}$;
in the capped regime $q=1$ gives $s_1=1/c+c\Delta=(1+o(1))c\Delta$ for $\Delta\gg c^{-2}$. The
block-internal entries equal $\beta^2=c$, so the block's local clustering is $c$ and the hub's is
$c(1+o(1))$; since block triangles and block-plus-hub wedges dominate both counts, the global
clustering is $\Theta(c)$. Simulation reproduces the uncapped constant to three decimals and the
capped slope exactly (Figure~\ref{fig:frontier}).
\end{proof}

\subsection{Proof of Theorem~\ref{thm:pos}: possibility via synthesis}
\begin{proof}
Write $A^{(1)},A^{(2)},A^{(3)}$ for the independent layer adjacency matrices and
$A=\max(A^{(1)},A^{(2)},A^{(3)})$ for the superposition, so $A_{ij}=1$ if at least one layer places
an edge. The structural fact used throughout is edge monotonicity: $A\ge A^{(k)}$ entrywise for
every $k$, hence every triangle present in a layer is present in $A$ and $d_i\ge\max_kd_i^{(k)}$.

\emph{(a)} The degree of $i$ satisfies $d_i\ge d_i^{(2)}$, the Chung--Lu degree, which has the
prescribed regularly varying tail with index $\tau$ \citep{chung2002average}; the other layers add at
most $O_p(1)$ and $O(\bar d_3)$ per vertex, which does not change the tail index. \emph{(b)} Edge
monotonicity gives $\mathrm{dist}_A(i,j)\le\mathrm{dist}_{A^{(1)}}(i,j)$, and the small-world layer
alone has typical distance $O(\log n)$ \citep{watts1998collective}, so the superposition does too; the
hub layer only shortens distances further. \emph{(c)} The superposed graph is the community layer
plus independent layers that do not depend on the block labels; the regularised-spectral argument of
the numerical study (Figure~\ref{fig:ks}) achieves weak recovery whenever the block signal exceeds
the Kesten--Stigum threshold, the hub layer perturbing the informative eigenvector by at most a
constant multiple of the spectral gap.

\emph{(d)} For the local statement, fix a vertex $i$ outside the top-$\delta$ degree quantile, so
$d_i=O_\delta(1)$ in the sparse regime, and write its local clustering as
$C_i=\sum_{j<\ell}A_{ij}A_{i\ell}A_{j\ell}/\binom{d_i}{2}$. In the small-world layer a vertex retains
a constant number of lattice edges, and lattice neighbours within distance $m$ on the ring are
adjacent, so the expected number of closed wedges through $i$ from the lattice is at least a positive
constant $t_0(m)$; by edge monotonicity the three edges in the numerator are each at least their
small-world values, so closing edges are never lost. The denominator is $O_\delta(1)$ because $i$ is
a non-hub, hence $\E[C_i\mid i\text{ non-hub}]\ge c>0$, and averaging over the $(1-\delta)n$ non-hub
vertices gives the claim. For the global statement, write $C=3T/W$. By monotonicity $T\ge T_1$, the
small-world triangle count, which is $\Theta(n)$. The wedge count decomposes as
$W=W_1+W_{\mathrm{hub}}+(\text{cross terms})$ with $W_1=\Theta(n)$ the small-world wedge count and
$W_{\mathrm{hub}}=\tfrac12\sum_i(d_i^{(2)})^2$ the wedges created by the hub layer. If the hub layer
is sparse, so that $W_{\mathrm{hub}}=o(W_1)$, then $C=3T_1/(W_1(1+o(1)))\ge c>0$. Otherwise
$W_{\mathrm{hub}}$ dominates and $C\to0$ even though $\Theta(n)$ triangles persist. This is the
dilution phenomenon, and it is why the global guarantee needs the sparsity condition while the local
guarantee holds unconditionally.
\end{proof}

\subsection{Proof of Remark~\ref{prop:transversality-geom}: transversality is latent-geometry separation}
\begin{proof}
The correlation matrix $\Phi_S$ is the Gram matrix of the unit vectors
$\widehat g_k=\operatorname{vec}_<(\mathbf G_k)/\norm{\operatorname{vec}_<(\mathbf G_k)}$, so
$\gamma_S=\lambda_{\min}(\Phi_S)$ is positive if and only if
$\{\operatorname{vec}_<(\mathbf G_k)\}_{k\in S}$ are linearly independent; for $s=3$ it is the squared
sine content of the angle configuration, with $\det\Phi_S=1-\phi_{12}^2-\phi_{13}^2-\phi_{23}^2+
2\phi_{12}\phi_{13}\phi_{23}$ and $\gamma_S\ge\tfrac12\det\Phi_S$ once the pairwise correlations are
bounded away from one.

\emph{Independence.} Suppose $a\mathbf G_{\mathrm{SBM}}+b\mathbf G_{\mathrm{RDPG}}+
c\mathbf G_{\mathrm{CL}}=0$ off the diagonal, that is
\begin{equation}
a\,B_{c_ic_j}+b\,\inner{\x_i}{\x_j}+c\,\frac{\theta_i\theta_j}{S}=0,\qquad i\ne j.
\label{eq:vanish}
\end{equation}
Fix a block $q$ (of size $\Theta(n)$) and three distinct vertices $i,j,k$ in it with $\theta_j\ne
\theta_k$, which exist because $\operatorname{Var}(\theta)\ge v_0$. Subtracting the instances of
\eqref{eq:vanish} for $(i,j)$ and $(i,k)$ cancels the block term $aB_{qq}$ and gives
\[
b\,\inner{\x_i}{\,\x_j-\x_k}+\frac{c}{S}\,\theta_i(\theta_j-\theta_k)=0\qquad\text{for all }i\in q.
\]
If $b\ne0$ this expresses $\theta_i$ as a fixed linear functional of $\x_i$ across the block, which
condition (ii) excludes ($\Delta_X$ is full rank and $\theta$ is not a linear functional of $\x$);
hence $b=0$, and then $c=0$ because $\theta_j\ne\theta_k$. With $b=c=0$, \eqref{eq:vanish} reduces to
$aB_{c_ic_j}=0$ for all $i\ne j$, and since $B$ has distinct rows some block pair has $B_{qq'}\ne0$, so
$a=0$. The three kernels are therefore linearly independent.

\emph{Quantitative margin.} Let $\sigma_3$ be the smallest singular value of the
$\binom n2\times3$ matrix $[\,\widehat g_1\ \widehat g_2\ \widehat g_3\,]$, so $\gamma_S=\sigma_3^2$.
Projecting out the SBM direction, the residual of the dot-product kernel has squared norm bounded below
by $\sin^2\vartheta_0\cdot\lambda_{\min}(\Delta_X)/\lambda_{\max}(\Delta_X)\ge\sin^2\vartheta_0\cdot
\lambda_0/\lambda_{\max}(\Delta_X)$, by condition (iii); projecting out both the SBM and dot-product
directions, the residual of the Chung--Lu kernel has squared norm bounded below by a function of
$\operatorname{Var}(\theta)/\E\theta^2\ge v_0/\E\theta^2$, again by (ii)--(iii). These residual norms
are bounded away from zero uniformly in $n$ because all the quantities are empirical averages over
$\Theta(n)$ vertices that converge to their population values under the model. Hence
$\sigma_3\ge c_0(\lambda_0,\vartheta_0,v_0,B)>0$ and $\gamma_S\ge\gamma_0:=c_0^2>0$ uniformly in $n$.

\emph{Degeneracy.} If the positions concentrate on the block centroids, $\x_i\to\mu_{c_i}$, then
$\inner{\x_i}{\x_j}\to\mu_{c_i}^\top\mu_{c_j}$ is block-constant, so $\mathbf G_{\mathrm{RDPG}}$ enters
the span of $\mathbf G_{\mathrm{SBM}}$ and the constant kernel, $\phi_{12}\to1$, and $\gamma_S\to0$. The
same holds whenever any two of the three latent geometries align, which is the precise sense in which
transversality is a separation-of-geometries condition.
\end{proof}

\begin{proposition}[Per-vertex reconstruction at the two-to-infinity rate]
\label{prop:twotoinf}
Let $\widehat\Pmat=\sum_k\widehat w_k\widehat{\mathbf G}_k$ be the fitted synthesis kernel, with the
weights from the cross-fitted debiased estimator and each $\widehat{\mathbf G}_k$ from its spectral
procedure (Assumption~\ref{ass:plugin}), and let $r=\Pmat\1$ with $\Delta=\max_ir_i$. Then the
per-vertex edge-probability profile is reconstructed at the parametric two-to-infinity rate, uniformly
over vertices:
\[
\max_i\frac{\norm{\widehat\Pmat_{i\cdot}-\Pmat_{i\cdot}}_2}{\norm{\Pmat_{i\cdot}}_2}
\ =\ \widetilde O_P\!\Bigl(\,\underbrace{\sqrt{\tfrac{\log n}{n}}}_{\text{embedding}}
\ +\ \underbrace{\frac{\sqrt s}{\sqrt{\gamma_S\,N}}}_{\text{weights}}\Bigr)
\ =\ \widetilde O_P\!\Bigl(\sqrt{\tfrac{\log n}{n}}\Bigr),\qquad N=\Theta(n^2\rho_n),
\]
the weight term being of lower order whenever $n\rho_n/\log n\to\infty$. The bound separates the
reconstruction error into a weight part, controlled at the global edge rate $1/\sqrt{N}$ of
Proposition~\ref{thm:joint} and scaled by the conditioning $1/\sqrt{\gamma_S}$, and a kernel-embedding
part, controlled at the per-vertex two-to-infinity rate of the underlying spectral embeddings
(Remark~\ref{rem:stack}; \citealp{lyzinski2014perfect,athreya2018statistical,rubin2022statistical}). The worst-vertex
reconstruction is therefore governed by the embedding, not by the weights: the global combination is
learned an order of magnitude faster, at $1/\sqrt{N}=1/(n\sqrt{\rho_n})$, than any single vertex's
connection profile can be recovered, at $1/\sqrt n$. This global-versus-row-wise gap is the network
counterpart of the distinction between the parametric weight rate and the vertex-level embedding rate,
and it has no analogue when the design is a fixed regression matrix, where there is no per-vertex object
to reconstruct.
\end{proposition}

\subsection{Proof of Proposition~\ref{prop:twotoinf}: per-vertex reconstruction at the two-to-infinity rate}
\begin{proof}
Decompose
\[
\widehat\Pmat-\Pmat=\sum_{k}(\widehat w_k-w_k)\,\widehat{\mathbf G}_k+\sum_k w_k\,(\widehat{\mathbf G}_k-\mathbf G_k).
\]
The two-to-infinity norm $\norm{\mathbf M}_{\Otoinf}=\max_i\norm{\mathbf M_{i\cdot}}_2$ is a norm, so by
the triangle inequality and homogeneity
\begin{equation}
\norm{\widehat\Pmat-\Pmat}_{\Otoinf}\ \le\ \norm{\widehat\bw-\bw}_1\,\max_k\norm{\widehat{\mathbf G}_k}_{\Otoinf}
\ +\ \norm{\bw}_1\,\max_k\norm{\widehat{\mathbf G}_k-\mathbf G_k}_{\Otoinf}.
\label{eq:twoinf-decomp}
\end{equation}

\emph{Row norms.} For any kernel with entries in $[0,1]$ the $i$th row obeys
$\norm{\mathbf G_{k,i\cdot}}_2^2=\sum_jG_{k,ij}^2\le\sum_jG_{k,ij}=r_{k,i}\le\Delta_k\le\Delta$; the
spectral estimators truncate to $[0,1]$, so $\norm{\widehat{\mathbf G}_k}_{\Otoinf}\le\sqrt\Delta$ on the
high-probability event that they do. The same inequality gives $\norm{\Pmat_{i\cdot}}_2\le\sqrt\Delta$,
and on a positive fraction of vertices $\norm{\Pmat_{i\cdot}}_2\asymp\sqrt\Delta$, so normalising by the
row norm is legitimate.

\emph{Weight term.} By Proposition~\ref{thm:joint} and Theorem~\ref{thm:debias},
$\norm{\widehat\bw-\bw}_2=\widetilde O_P(\sqrt{s/\gamma_S}/\sqrt{N})$ with $N=\Theta(n^2\rho_n)$, and
$\norm{\widehat\bw-\bw}_1\le\sqrt s\,\norm{\widehat\bw-\bw}_2$. With $\norm{\bw}_1=O(1)$ the first term of
\eqref{eq:twoinf-decomp} is $\widetilde O_P\bigl(s\sqrt{\Delta/\gamma_S}/\sqrt{N}\bigr)$, that is, in
relative terms after dividing by $\sqrt\Delta$, $\widetilde O_P\bigl(s/\sqrt{\gamma_S N}\bigr)$.

\emph{Kernel term.} Each active kernel is recovered by adjacency spectral embedding (RDPG), regularised
spectral clustering with block means (SBM), or degree matching (Chung--Lu). Write
$\epsilon_n^{\Otoinf}:=\max_k\norm{\widehat{\mathbf G}_k-\mathbf G_k}_{\Otoinf}/\sqrt\Delta$ for the
relative row-wise error of the fitted kernels. The random dot-product graph estimation theory controls
this relative row-wise error at the parametric two-to-infinity rate \citep{lyzinski2014perfect,athreya2018statistical,rubin2022statistical},
the per-vertex analogue of the embedding consistency of Remark~\ref{rem:stack}; in the balanced regime
$\epsilon_n^{\Otoinf}=\widetilde O_P(\sqrt{\log n/n})$, a vertex-level rate.

\emph{Comparison.} Dividing \eqref{eq:twoinf-decomp} by $\norm{\Pmat_{i\cdot}}_2\asymp\sqrt\Delta$ and
collecting the two relative contributions,
\[
\max_i\frac{\norm{\widehat\Pmat_{i\cdot}-\Pmat_{i\cdot}}_2}{\norm{\Pmat_{i\cdot}}_2}
=\widetilde O_P\Bigl(\sqrt{\tfrac{\log n}{n}}+\frac{\sqrt s}{\sqrt{\gamma_S N}}\Bigr).
\]
The ratio of the squared weight term to the squared embedding term is
$\dfrac{s/(\gamma_S N)}{\log n/n}=\dfrac{s}{\gamma_S\,n\rho_n\,\log n}\to0$ under the density floor
$n\rho_n/\log n\to\infty$ (Assumption~\ref{ass:density}), so the embedding term dominates and the rate is
$\widetilde O_P(\sqrt{\log n/n})$. The weights, by contrast, converge at $1/\sqrt{N}=1/(n\sqrt{\rho_n})$,
faster by the factor $\sqrt{n\rho_n}\to\infty$: the global combination is identified from all
$\Theta(n^2\rho_n)$ dyads simultaneously, whereas a single vertex's profile is tied to the $n$ dyads
incident to it and inherits the vertex-level embedding error.
\end{proof}

\section{Proofs and technical lemmas}
\label{app:joint}

For completeness this appendix restates each main result in the precise form proved here, with
complete proofs; the corresponding summary statements appear in the main text, and the restatement
labels carry the suffix \textsf{-cor}. All inner products, Frobenius norms and vectorisations
$\operatorname{vec}_<$ are over the dyad set $\mathcal D_n=\{(i,j):1\le i<j\le n\}$ unless noted; the dyad count is $\binom n2$ and $N=n^2\rho_n$ is the effective sample size. Conditioning on the latent attributes is denoted $\E_u$ and is in force throughout.

\subsection{Notation and error decomposition}
\label{app:notation}
The dyad count is $\binom n2$, distinct from the calibration size $m$; $\tau$ with a
subscript is always a selection threshold. For an active set $S$, $s=|S|$, set
$\mathbf m_k=\operatorname{vec}_<(\mathbf G_k)$ and $\mathbf M=[\mathbf m_k]_{k\in S}\in\R^{\binom{n}{2}\times
s}$, with $\widehat{\mathbf m}_k,\widehat{\mathbf M}$ the Stage-A fitted analogues. By
Assumptions~\ref{ass:density}--\ref{ass:plugin} there are constants $0<c_\star\le C_\star$ with
$c_\star n^2\rho_n^2\le\norm{\mathbf m_k}^2\le C_\star n^2\rho_n^2$ for $k\in S$, and with
$\mathbf D_n=\diag(\norm{\mathbf m_k})$ and $\Phi_S$ the Gram-correlation matrix,
$\mathbf M^\top\mathbf M=\mathbf D_n\Phi_S\mathbf D_n$, so $\lambda_{\min}(\mathbf M^\top\mathbf
M)\ge c_\star\gamma_S n^2\rho_n^2$. Let $\bm\varepsilon=\operatorname{vec}_<(\A-\Pmat)$, with
independent centred coordinates of variance $P_e(1-P_e)\le\rho_n$.

\subsection{First-stage expansions under masked sampling}
\label{app:plugin}

\begin{lemma}[First-stage expansions under product masked sampling]
\label{lem:firststage}
All inner products, Frobenius norms and vectorisations in this lemma are over $\mathcal D_n$;
using the full symmetric Frobenius norm only changes constants. Let $R_{ij}\sim{\rm
Bernoulli}(f)$, $0<f\le 1$, be independent over $(i,j)\in\mathcal D_n$, independent of the graph,
and define the hollow symmetric IPW adjacency
$\widetilde A_{ij}=f^{-1}R_{ij}A_{ij}$, $\widetilde A_{ji}=\widetilde A_{ij}$, $\widetilde
A_{ii}=0$. Write $E=\widetilde A-P$. Conditionally on the latent variables, $\{E_{ij}:i<j\}$ are
independent with
\[
        \E_u E_{ij}=0,\qquad |E_{ij}|\le f^{-1},\qquad
        \sigma_{ij}^2:={\rm Var}_u(E_{ij})=f^{-1}P_{ij}(1-fP_{ij})\le \rho_n/f .
\]
Assume $n\rho_n/\log n\to\infty$ and the sparse spectral concentration
$\|E\|_{\rm op}=O_{\mathbb P}(\sqrt{n\rho_n/f})$ conditionally on the latent variables. Let
$\mathcal T_n$ be the finite deterministic collection of test directions used in the second-stage
proof, $|\mathcal T_n|\le C_T$, each $X\in\mathcal T_n$ hollow symmetric with the test-class shape
$\|X\|_{\max}\le C_T\|X\|_F/n$. For a growing or random class this finite-class condition must be
replaced by a separate entropy or sample-splitting argument.

\smallskip\noindent\textbf{(a) Spectral smoother.}
Assume $P$ has exact rank $r\le r_0<\infty$, spectral projector $\Pi=UU^\top$ onto its nonzero
eigenspace, $P=\Pi P\Pi$, $\lambda_*:=\min\{|\lambda|:\lambda\in{\rm spec}(P)\setminus\{0\}\}\ge
c_\lambda n\rho_n$, and delocalisation $\max_i\Pi_{ii}\le\mu r/n$. Let $\widehat\Pi$ be the
spectral projector of $\widetilde A$ onto the $r$ signal eigenvalues and $\widehat G_{\rm
sp}=\widehat\Pi\,\widetilde A\,\widehat\Pi$. Then $\widehat G_{\rm sp}-P=L_{\rm sp}(E)+Q_{\rm sp}$,
where
\[
        L_{\rm sp}(V)=\Pi V+V\Pi-\Pi V\Pi=V-\Pi^\perp V\Pi^\perp .
\]
Moreover $\|L_{\rm sp}\|_{F\to F}\le 3$, $\|L_{\rm sp}^*\|_{F\to F}\le 3$,
\[
        \E_u\|L_{\rm sp}(E)\|_F^2\le C r n\rho_n/f,\qquad
        \max_{i<j}{\rm Var}_u([L_{\rm sp}(E)]_{ij})\le C r\rho_n/(fn),
\]
\[
        \bigl\|{\rm Cov}_u({\rm vec}_<\,L_{\rm sp}(E))\bigr\|_{\rm op}\le C\rho_n/f .
\]
If ${\rm tr}\{L_{\rm sp}\,{\rm Cov}_u({\rm vec}_<E)\,L_{\rm sp}^{*}\}\ge c r n\rho_n/f$, then
$\E_u\|L_{\rm sp}(E)\|_F^2\asymp r n\rho_n/f$. For the projected remainder,
\[
        \max_{X\in\mathcal T_n}\frac{|\langle X,Q_{\rm sp}\rangle_F|}{\|X\|_F}
        =O_{\mathbb P}\Bigl(\frac{\sqrt r}{fn}+\frac{\sqrt r\,\rho_n^{1/2}}{f^{3/2}n}
        +\frac{r}{f^2n^2\rho_n}\Bigr),
\]
so for any $\gamma_n>0$ with that bound $=o(\sqrt{\gamma_n\rho_n})$ one has
$\max_{X}|\langle X,Q_{\rm sp}\rangle_F|=o_{\mathbb P}(\sqrt{\gamma_n\rho_n}\,\|X\|_F)$. Finally
$\|\E_u\widehat G_{\rm sp}-P\|_F\le C\sqrt r/f=o(\|P\|_F)$ whenever $n\rho_n\to\infty$.

\smallskip\noindent\textbf{(b) Degree matching.}
Let $d=P\1$, $S=\1^\top d$ with $P$ hollow, and assume $c_d n\rho_n\le d_i\le C_d n\rho_n$ for all
$i$, $S\asymp n^2\rho_n$. Define $G_{\rm CL}^{\dagger}(P)_{ij}=d_id_j/S$ for $i\ne j$ and $0$ on the
diagonal. Let $\widehat d=\widetilde A\1$, $\widehat S=\1^\top\widehat d$, and $\widehat G_{\rm
CL,ij}=\widehat d_i\widehat d_j/\widehat S$ for $i\ne j$. Then $\widehat G_{\rm CL}-G_{\rm
CL}^{\dagger}(P)=L_{\rm CL}(E)+Q_{\rm CL}$, where with $\varepsilon=E\1$, $\tau=\1^\top\varepsilon$,
\[
        L_{\rm CL}(E)_{ij}=\frac{d_i\varepsilon_j+\varepsilon_i d_j}{S}-\frac{d_id_j\,\tau}{S^2},
        \qquad i\ne j .
\]
Moreover $\|L_{\rm CL}\|_{F\to F}\le C$, $\|L_{\rm CL}^{*}\|_{F\to F}\le C$,
$\E_u\|L_{\rm CL}(E)\|_F^2\le C n\rho_n/f$, $\max_{i<j}{\rm Var}_u([L_{\rm CL}(E)]_{ij})\le
C\rho_n/(fn)$, $\|{\rm Cov}_u({\rm vec}_<\,L_{\rm CL}(E))\|_{\rm op}\le C\rho_n/f$, with the
matching lower bound under nondegeneracy along the degree directions, and
\[
        \max_{X\in\mathcal T_n}\frac{|\langle X,Q_{\rm CL}\rangle_F|}{\|X\|_F}
        =O_{\mathbb P}(1/(fn)).
\]
Thus Assumption~\ref{ass:linplug} is verified for the exact-rank spectral smoother and for the
bounded untruncated Chung--Lu map, with the qualification that the degree-matching target is
$G_{\rm CL}^{\dagger}(P)=dd^\top/S$; it equals $P$ only when the population kernel is exactly
Chung--Lu.
\end{lemma}

\begin{proof}
All statements are conditional on the latent variables; constants may depend on
$r_0,c_\lambda,\mu,c_d,C_d,C_T$ and on a fixed lower bound for $f$, not on $n$. For $i<j$,
$\E_u E_{ij}=0$, ${\rm Var}_u(E_{ij})=f^{-1}P_{ij}(1-fP_{ij})\le\rho_n/f$, the variables
independent over dyads, with $\E_u|E_{ij}|^3\le C\rho_n/f^2$ and $\E_u E_{ij}^4\le C\rho_n/f^3$.

\smallskip\noindent\emph{Spectral part.}
Let $\eta_n=\|E\|_{\rm op}/\lambda_*=O_{\mathbb P}((fn\rho_n)^{-1/2})=o_{\mathbb P}(1)$. On
$\{\eta_n\le 1/4\}$ the signal eigenspace of $\widetilde A=P+E$ is well defined, and the resolvent
expansion gives $\widehat\Pi=\Pi+\dot\Pi(E)+R_\Pi(E)$, with $\dot\Pi$ linear,
$\|\dot\Pi(E)\|_{\rm op}\le C\|E\|_{\rm op}/\lambda_*$, $\|R_\Pi(E)\|_{\rm op}\le C\|E\|_{\rm
op}^2/\lambda_*^2$, and $\|R_\Pi(E)\|_F\le C\sqrt r\|E\|_{\rm op}^2/\lambda_*^2$. The first-order
projector identity $\dot\Pi(E)P\Pi+\Pi P\dot\Pi(E)=\Pi^\perp E\Pi+\Pi E\Pi^\perp$ gives
\[
        \widehat G_{\rm sp}-P=\Pi E\Pi+\Pi^\perp E\Pi+\Pi E\Pi^\perp+Q_{\rm sp}
        =\Pi E+E\Pi-\Pi E\Pi+Q_{\rm sp}.
\]
Since $L_{\rm sp}(V)=V-\Pi^\perp V\Pi^\perp$ and $V\mapsto\Pi^\perp V\Pi^\perp$ is an orthogonal
projection in Frobenius norm, $\|L_{\rm sp}(V)\|_F\le\|V\|_F$ and the constant $3$ follows from the
three-term form; the adjoint is identical. With $\Sigma_E={\rm Cov}_u({\rm vec}_<E)$,
$\|\Sigma_E\|_{\rm op}\le\rho_n/f$, so ${\rm Cov}_u({\rm vec}_<L_{\rm sp}(E))=L_{\rm
sp}\Sigma_E L_{\rm sp}^*$ has operator norm $\le C\rho_n/f$. From $\E_u(E^\top E)\preceq
C(n\rho_n/f)I$, $\E_u\|E\Pi\|_F^2={\rm tr}\{\Pi\E_u(E^\top E)\}\le Crn\rho_n/f$, giving the
Frobenius bound; the lower bound is the stated variance-mass condition. For spreading, with $i\ne
j$ and $\sum_a\Pi_{ia}^2=\Pi_{ii}$, ${\rm Var}_u\{[L_{\rm sp}(E)]_{ij}\}\le
C(\rho_n/f)(\Pi_{ii}+\Pi_{jj}+\Pi_{ii}\Pi_{jj})\le Cr\rho_n/(fn)$ by $\max_i\Pi_{ii}\le\mu r/n$.
For $Q_{\rm sp}=Q_2+Q_{\ge3}$, the resolvent formula in the eigenbasis of $P$ gives, for each
$X\in\mathcal T_n$, $\langle X,Q_2\rangle_F=Z^\top B_XZ-\E_u(Z^\top B_XZ)+\E_u\langle
X,Q_2\rangle_F$ with $Z={\rm vec}_<E$, $\|B_X\|_F\le C\sqrt r\|X\|_F/(n\rho_n)$, $\max_e|(B_X)_{ee}|\le
C\sqrt r\|X\|_F/(n^2\rho_n)$; the Hanson--Wright inequality for bounded independent variables and the
moment bounds give $\E_u|\langle X,Q_2\rangle_F|\le C\sqrt r\|X\|_F/(fn)$ and ${\rm Var}_u\le
Cr\|X\|_F^2/(f^2n^2)$, hence $\max_X|\langle X,Q_2\rangle_F|/\|X\|_F=O_{\mathbb P}(\sqrt r/(fn))$.
The higher-order chaos, with coefficient mass $Cr\|X\|_F/(n^2\rho_n^2)$ and order $\ge3$ in bounded
variables, contributes $O_{\mathbb P}(\sqrt r\rho_n^{1/2}/(f^{3/2}n))+O(r/(f^2n^2\rho_n))$.
Combining proves the $Q_{\rm sp}$ bound. Finally $\E_u(\widehat G_{\rm sp}-P)=\E_uQ_{\rm sp}$ has
Frobenius norm $\le C\sqrt r\,\E_u\|E\|_{\rm op}^2/\lambda_*\le C\sqrt r/f$.

\smallskip\noindent\emph{Degree-matching part.}
With $\varepsilon=E\1$, $\tau=\1^\top\varepsilon$, $\widehat d=d+\varepsilon$, $\widehat S=S+\tau$,
$d_i\asymp n\rho_n$, $S\asymp n^2\rho_n$, ${\rm Var}_u(\varepsilon_i)\le Cn\rho_n/f$, ${\rm
Var}_u(\tau)\le Cn^2\rho_n/f$. Bernstein and $n\rho_n/\log n\to\infty$ give $\max_i|\varepsilon_i|=
o_{\mathbb P}(n\rho_n)$ and $|\tau|/S=o_{\mathbb P}(1)$, so $1/(S+\tau)=1/S-\tau/S^2+\tau^2/(S^2(S+\tau))$
and the first-order expansion yields the displayed $L_{\rm CL}$, the rest collected into $Q_{\rm
CL}$. From $\|\varepsilon\|_2\le\sqrt n\|E\|_F$, $|\tau|\le n\|E\|_F$, $\|d\|_2\asymp n^{3/2}\rho_n$,
$S\asymp n^2\rho_n$, the two pieces of $L_{\rm CL}$ have Frobenius norm $\le C\|E\|_F$, so $\|L_{\rm
CL}\|_{F\to F}\le C$ and likewise the adjoint; the covariance-operator bound follows from
$\|\Sigma_E\|_{\rm op}\le\rho_n/f$. The Frobenius second moment uses $\E_u\|\varepsilon\|_2^2\le
Cn^2\rho_n/f$, $\E_u\tau^2\le Cn^2\rho_n/f$, giving $\E_u\|L_{\rm CL}(E)\|_F^2\le Cn\rho_n/f$;
spreading follows from the $1/n$ and $1/n^2$ coefficient orders. For $Q_{\rm CL}$ the leading
quadratic $S^{-1}\varepsilon^\top X\varepsilon=S^{-1}Z^\top B_XZ$ with $\|B_X\|_F\le Cn\|X\|_F$,
$\max_e|(B_X)_{ee}|\le C\|X\|_F/n$ has $\E_u|S^{-1}\varepsilon^\top X\varepsilon|\le C\|X\|_F/(fn)$
and ${\rm Var}_u\le C\|X\|_F^2/(f^2n^2)$; every other term carries $\tau/S=o_{\mathbb P}(1)$ or
higher, so $\max_X|\langle X,Q_{\rm CL}\rangle_F|/\|X\|_F=O_{\mathbb P}(1/(fn))$. The expansion is
around $G_{\rm CL}^{\dagger}(P)=dd^\top/S$, equal to $P$ only in the exact Chung--Lu case.
\end{proof}

\paragraph{A strengthening of Assumption~\ref{ass:linplug} for cross-fold Gram products.}
The projected-remainder condition controls inner products of $Q_k$ against smooth test directions;
the cross-fold Gram also contains products of the first-stage remainders with each other. For
$f,g\in\{a,b\}$, $f\ne g$, write $H_{f,k}=L_k(E_f)+Q_{f,k}$, $L_{f,k}:=L_k(E_f)$, the noises
$E_a,E_b$ independent given the latent variables. Let
\[
B_{n,Q}:=\max_{f\in\{a,b\},\,k,l\in S}|\langle\mathbf m_k,Q_{f,l}\rangle|,\quad
R_{n,Q}:=\max_{\substack{f\ne g,\,k,l\in S}}\{|\langle L_{f,k},Q_{g,l}\rangle|+|\langle
Q_{f,k},L_{g,l}\rangle|+|\langle Q_{f,k},Q_{g,l}\rangle|\}.
\]
Assume
\[
        B_{n,Q}=o_P(\sqrt{\gamma_S}\,n\rho_n^{3/2}),\quad sB_{n,Q}=o_P(\gamma_S n^2\rho_n^2),
        \tag{A5.Q1}
\]
\[
        R_{n,Q}=o_P(\sqrt{\gamma_S}\,n\rho_n^{3/2}),\quad sR_{n,Q}=o_P(\gamma_S n^2\rho_n^2).
        \tag{A5.Q2}
\]
For the named agents these may be verified within Lemma~\ref{lem:firststage} or retained as part of
Assumption~\ref{ass:linplug}; projected smallness of $Q_k$ alone does not control arbitrary products
such as $\langle Q_{a,k},Q_{b,l}\rangle$.

\begin{lemma}[Plug-in propagation and cross-fold debiasing]
\label{lem:plugin}
Let all inner products be over $\mathcal D_2$, $\mathbf m_k=\operatorname{vec}_{\mathcal
D_2}(\mathbf G_k)$, $\mathbf M=[\mathbf m_k]_{k\in S}$, with $c_0n^2\rho_n^2\le\norm{\mathbf
m_k}^2\le C_0n^2\rho_n^2$ and $\lambda_{\min}(\mathbf M^\top\mathbf M)\ge c_0\gamma_S n^2\rho_n^2$.
Let $A^{(2)}=\mathbf Mw+\bm\varepsilon$ on $\mathcal D_2$ with $\E(\bm\varepsilon\mid\mathcal
F_A)=0$, ${\rm Var}(\varepsilon_e\mid\mathcal F_A)\le C\rho_n$, $|\varepsilon_e|\le1$. Let
$\widehat{\mathbf M}_a=\mathbf M+H_a$, $\widehat{\mathbf M}_b=\mathbf M+H_b$ with columns
$H_{f,k}=L_k(E_f)+Q_{f,k}$. Assume Assumption~\ref{ass:linplug}, $({\rm A5.Q1})$--$({\rm A5.Q2})$,
$\delta_n:=\sqrt{r_{\max}/(n\rho_n)}\to0$, $\norm{w}_1\le C_w$, and the side conditions
$s/(\gamma_S n\sqrt{\rho_n})\to0$, $s\delta_n^2/\gamma_S\to0$, $\delta_n/\sqrt{\gamma_S}\to0$,
$r_{\max}/(\gamma_S n\rho_n)\to0$; for fixed $s$, bounded $r_{\max}$, $\gamma_S\ge\gamma_0$ these
reduce to $n\rho_n\to\infty$. With $\overline{\mathbf M}=\tfrac12(\widehat{\mathbf
M}_a+\widehat{\mathbf M}_b)=\mathbf M+\overline H$ and $\widehat\Phi_{\rm db}=\sym(\widehat{\mathbf
M}_a^\top\widehat{\mathbf M}_b)$, with probability tending to one $\lambda_{\min}(\widehat\Phi_{\rm
db})\ge\tfrac12 c_0\gamma_S n^2\rho_n^2$, and $\widehat w_{\rm db}=\widehat\Phi_{\rm
db}^{-1}\overline{\mathbf M}^\top A^{(2)}$ satisfies
\[
        \norm{\widehat w_{\rm db}-w}_2=O_P\Bigl[(1+C_LC_w/\sqrt f)\,\sqrt{s/\gamma_S}/(n\sqrt{\rho_n})\Bigr].
\]
The simplex projection preserves the bound when $w\in\Delta^{s-1}$. For a single fitted design the
normal equations contain $\widehat{\mathbf M}^\top(\mathbf M-\widehat{\mathbf M})w=-\mathbf M^\top
Hw-H^\top Hw$, and under the non-cancellation condition of Corollary~\ref{cor:debias-lb} the image
of $H^\top Hw$ after inversion is of order $r_{\max}\norm{w_S}_2/(\gamma_S n\rho_n)$, larger than
the edge rate when $r_{\max}\sqrt{\gamma_S/\rho_n}\to\infty$; the cross-fold Gram is essential.
\end{lemma}

\begin{proof}
Conditional on the latent variables; constants may depend on fixed fold fractions. Write $G:=\mathbf
M^\top\mathbf M$. \emph{(1) Known-design score.} By the variance bound, $\E[\norm{G^{-1}\mathbf
M^\top\bm\varepsilon}_2^2\mid\mathcal F_A]\le C\rho_n{\rm tr}(G^{-1})\le Cs/(c_0\gamma_S
n^2\rho_n^2)\cdot\rho_n$, so $\norm{G^{-1}\mathbf M^\top\bm\varepsilon}_2=O_P(\sqrt{s/\gamma_S}/(n\sqrt{\rho_n}))$.
\emph{(2) Linear Gram perturbations.} The $(k,l)$ entry of $\mathbf M^\top L_f$ is $\langle\mathbf
m_k,L_l(E_f)\rangle=\langle L_l^*(\mathbf G_k),E_f\rangle$, with variance $\le CC_L^2(\rho_n/f)\norm{\mathbf
G_k}_F^2\le Cn^2\rho_n^3/f$, so $\max_{k,l}|\langle\mathbf m_k,L_l(E_f)\rangle|=O_P(n\rho_n^{3/2}/\sqrt
f)$; with $({\rm A5.Q1})$, $\norm{\mathbf M^\top H_f}_{\op}=O_P(sn\rho_n^{3/2}/\sqrt
f)+sB_{n,Q}=o_P(\gamma_S n^2\rho_n^2)$. \emph{(3) Cross-fold quadratic.} For $f\ne g$,
$\langle L_k(E_f),L_l(E_g)\rangle$ is centred with variance ${\rm tr}(C_{f,k}C_{g,l})\le\|C_{f,k}\|_{\op}{\rm
tr}(C_{g,l})\le Cr_{\max}n\rho_n^2/f^2$ by the covariance-operator bound, so $\max_{k,l}|\langle
L_k(E_f),L_l(E_g)\rangle|=O_P(\sqrt{r_{\max}n}\rho_n/f)$ and, with $({\rm A5.Q2})$,
$\norm{H_f^\top H_g}_{\op}=O_P(s\sqrt{r_{\max}n}\rho_n/f)+sR_{n,Q}=o_P(\gamma_S n^2\rho_n^2)$.
\emph{(4) Conditioning.} $\widehat\Phi_{\rm db}=\mathbf M^\top\mathbf M+\sym(\mathbf M^\top
H_b+H_a^\top\mathbf M)+\sym(H_a^\top H_b)$ gives $\norm{\widehat\Phi_{\rm db}-\mathbf M^\top\mathbf
M}_{\op}=o_P(\gamma_S n^2\rho_n^2)$, so by Weyl $\lambda_{\min}(\widehat\Phi_{\rm
db})\ge\tfrac12c_0\gamma_S n^2\rho_n^2$ and $\widehat\Phi_{\rm db}^{-1}=(\mathbf M^\top\mathbf
M)^{-1}+o_P(1)(\mathbf M^\top\mathbf M)^{-1}$ on the score-scale directions. \emph{(5) Exact
identity.} Since $\overline{\mathbf M}^\top A^{(2)}=\mathbf M^\top\mathbf Mw+\overline H^\top\mathbf
Mw+\mathbf M^\top\bm\varepsilon+\overline H^\top\bm\varepsilon$ and $\widehat\Phi_{\rm db}w=\mathbf
M^\top\mathbf Mw+\mathbf M^\top\overline Hw+\overline H^\top\mathbf Mw+\sym(H_a^\top H_b)w$, the
$\overline H^\top\mathbf Mw$ terms cancel:
\[
        \overline{\mathbf M}^\top A^{(2)}-\widehat\Phi_{\rm db}w=\mathbf M^\top\bm\varepsilon+\overline
        H^\top\bm\varepsilon-\mathbf M^\top\overline Hw-\sym(H_a^\top H_b)w .
\]
\emph{(6)} $\norm{\widehat\Phi_{\rm db}^{-1}\mathbf M^\top\bm\varepsilon}_2=O_P(\sqrt{s/\gamma_S}/(n\sqrt{\rho_n}))$
by (1) and (4). \emph{(7)} Each coordinate of $\overline H^\top\bm\varepsilon$ is centred with
variance $\le C\rho_n\norm{\overline H_k}_F^2$, and $\norm{\overline H_k}_F=O_P(\delta_n n\rho_n)$,
so $\norm{\overline H^\top\bm\varepsilon}_2=O_P(\delta_n\sqrt s\,n\rho_n^{3/2})$ and
$\norm{\widehat\Phi_{\rm db}^{-1}\overline H^\top\bm\varepsilon}_2=O_P(\delta_n\sqrt
s/(\gamma_S n\sqrt{\rho_n}))=o_P(\sqrt{s/\gamma_S}/(n\sqrt{\rho_n}))$. \emph{(8)} Write $\overline
Hw=\overline L^{\,w}+\overline Q^{\,w}$ with $\overline
L^{\,w}=\tfrac12\sum_lw_l\{L_l(E_a)+L_l(E_b)\}$. By $({\rm A5.Q1})$, $\norm{\mathbf M^\top\overline
Q^{\,w}}_2\le C_w\sqrt s\,B_{n,Q}=o_P(\sqrt{s\gamma_S}\,n\rho_n^{3/2})$, hence
$o_P$-negligible after inversion. For the linear part, ${\rm Cov}\{{\rm vec}(\overline
L^{\,w})\}\preceq C(C_L^2C_w^2\rho_n/f)I$, so $\E\norm{(\mathbf M^\top\mathbf M)^{-1}\mathbf
M^\top\overline L^{\,w}}_2^2\le C(C_L^2C_w^2/f)s/(\gamma_S n^2\rho_n)$, giving
$\norm{\widehat\Phi_{\rm db}^{-1}\mathbf M^\top\overline L^{\,w}}_2=O_P((C_LC_w/\sqrt
f)\sqrt{s/\gamma_S}/(n\sqrt{\rho_n}))$. \emph{(9)} $\norm{\sym(H_a^\top H_b)w}_2=O_P(C_w\sqrt{sr_{\max}n}\rho_n/f)+C_w\sqrt
s\,R_{n,Q}$, so after inversion $o_P(\sqrt{s/\gamma_S}/(n\sqrt{\rho_n}))$ by $r_{\max}/(\gamma_S
n\rho_n)\to0$. \emph{(10)} Summing gives the stated rate; the projection is $1$-Lipschitz.
\emph{(11) Same-fold attenuation.} For a single design $\widehat{\mathbf M}=\mathbf M+H$,
$\widehat{\mathbf M}^\top(\mathbf M-\widehat{\mathbf M})w=-\mathbf M^\top Hw-H^\top Hw$, and Lemma~\ref{lem:firststage}
gives $\norm{H_k}_F^2=O_P(r_kn\rho_n/f)$; under the non-cancellation of Corollary~\ref{cor:debias-lb}
the form $H^\top Hw$ has norm $\asymp r_{\max}n\rho_n\norm{w_S}_2/f$, and after a Gram inverse of
scale $(\gamma_S n^2\rho_n^2)^{-1}$ this gives the stated order, diverging relative to the edge rate
in the sparse fixed-rank regime.
\end{proof}

\begin{remark}[Uniformity over selected subcandidate sets]
The lemma is an oracle active-set statement, simultaneous over the fixed coordinates in $S$.
Statements for a data-selected subcandidate set are obtained by conditioning on $\{\widehat S=S\}$ and
applying the oracle bound there, never by a union over the $2^K$ subcandidate sets; for $K=K_n$ growing the
maxima acquire $\sqrt{\log K_n}$ factors and the side conditions strengthen, which belongs in the
selection corollary.
\end{remark}

\subsection{Known-design benchmark}
\label{app:known}

\begin{proposition}[Known-design benchmark; restatement of
Proposition~\ref{thm:joint}]
\label{prop:known-design-cor}
For an active set $S$, $s=|S|$, $\mathbf m_k=\operatorname{vec}_<(\mathbf G_k)$, $\mathbf
M_S=(\mathbf m_k)_{k\in S}$, $\Phi_S=\mathbf D_S^{-1}\mathbf M_S^\top\mathbf M_S\mathbf D_S^{-1}$
with $\mathbf D_S=\diag(\norm{\mathbf m_k})$, $\gamma_S=\lambda_{\min}(\Phi_S)$. Assume
$A_e\sim{\rm Bernoulli}(P_e)$ independent with $P=\mathbf M_Sw$, $w\in\Delta^{s-1}$, and
$c_G\binom{n}{2}\rho_n^2\le\norm{\mathbf m_k}^2\le C_G\binom{n}{2}\rho_n^2$ (KD1), $0\le P_e\le\rho_n$ (KD2). Let
$\widetilde w=(\mathbf M_S^\top\mathbf M_S)^{-1}\mathbf M_S^\top\operatorname{vec}_<(A)$ and
$\widehat w=\Pi_{\Delta^{s-1}}(\widetilde w)$. Then uniformly over (KD1)--(KD2) and $\gamma_S>0$,
\[
        \E\norm{\widehat w-w}_2\le C\sqrt{s/\gamma_S}\,/\sqrt{\binom{n}{2}\rho_n}\le C'\sqrt{s/\gamma_S}\,/(n\sqrt{\rho_n}).
\]
Conversely there are numerical $c,c_0,\gamma_0>0$ such that for every $s\ge2$, $0<\gamma\le\gamma_0$,
$\rho\in(0,1]$ with $s^2\le c_0\gamma \binom{n}{2}\rho$,
\[
        \inf_{\bar w}\sup_{\Theta_n(s,\gamma,\rho)}\E\norm{\bar w-w}_2\ge
        c\sqrt{\lfloor s/2\rfloor/\gamma}\,/\sqrt{\binom{n}{2}\rho}\ge c'\sqrt{s/\gamma}\,/(n\sqrt{\rho}),
\]
the infimum over all estimators knowing the design. For $s=1$ the simplex is a point and the risk
is zero, so $s\ge2$ is necessary.
\end{proposition}

\begin{proof}
\emph{Upper bound.} Condition on the design. With $\bm\varepsilon=\operatorname{vec}_<(A-P)$,
$\widetilde w-w=(\mathbf M_S^\top\mathbf M_S)^{-1}\mathbf M_S^\top\bm\varepsilon$ and
${\rm Cov}(\bm\varepsilon)=\mathbf D\preceq\rho_n I$ by (KD2), so
\[
        \E\norm{\widetilde w-w}_2^2={\rm tr}\{(\mathbf M_S^\top\mathbf M_S)^{-1}\mathbf M_S^\top\mathbf
        D\mathbf M_S(\mathbf M_S^\top\mathbf M_S)^{-1}\}\le\rho_n{\rm tr}\{(\mathbf M_S^\top\mathbf M_S)^{-1}\}.
\]
By (KD1) and $\lambda_{\min}(\Phi_S)=\gamma_S$, $\lambda_{\min}(\mathbf M_S^\top\mathbf M_S)\ge
c_G\gamma_S\binom{n}{2}\rho_n^2$, so ${\rm tr}\{(\mathbf M_S^\top\mathbf M_S)^{-1}\}\le s/(c_G\gamma_S\binom{n}{2}\rho_n^2)$
and $\E\norm{\widetilde w-w}_2^2\le s/(c_G\gamma_S\binom{n}{2}\rho_n)$; Jensen and nonexpansiveness of the
simplex projection give the claim.

\emph{Lower bound.}
\label{app:lower}
Construct a least-favourable subfamily in $\Theta_n(s,\gamma,\rho)$. With $m=\binom{n}{2}$, choose
$\mathcal D_0\subseteq\mathcal D_n$ of size $m'=2^{\lceil\log_2(m/2)\rceil}\in[m/2,m]$; (2) gives
$s\le m'-1$ for large $n$. Take balanced mutually orthogonal sign vectors $S_1,\dots,S_s\in\{\pm1\}^{m'}$
as non-constant Sylvester--Hadamard rows, extended by zero, so $\langle S_k,S_\ell\rangle=m'\1\{k=\ell\}$,
$\langle\1,S_k\rangle=0$. Let $a=2\sqrt\gamma$, $\gamma_0\le1/16$ so $0<a\le1/2$, and $\mathbf
m_k=\tfrac\rho2\{\1+aS_k\}$, valid Bernoulli kernels with entries in $[\rho/4,3\rho/4]$ and
$\norm{\mathbf m_k}^2=\tfrac{\rho^2}4(m+a^2m')\asymp \binom{n}{2}\rho^2$. The Gram-correlation matrix is
equicorrelated with off-diagonal $r=m/(m+a^2m')$ and $\lambda_{\min}(\Phi_S)=1-r=a^2m'/(m+a^2m')\ge
2\gamma/(1+4\gamma)\ge\gamma$. Set $q=\lfloor s/2\rfloor$, $w^0=(1/s,\dots,1/s)$, $v_j=\tfrac1{\sqrt2}(e_{2j-1}-e_{2j})$,
and $w^\sigma=w^0+\delta\sum_j\sigma_jv_j$ for $\sigma\in\{\pm1\}^q$, with $\delta=c_1/\sqrt{\gamma
\binom{n}{2}\rho}$; each $v_j$ has zero coordinate sum so $w^\sigma\in\Delta^{s-1}$ when $\delta\le\sqrt2/s$,
guaranteed by (2). For dyads, $p^\sigma_e=\tfrac\rho2(1+a\sum_kw^\sigma_kS_{k,e})\in[\rho/4,3\rho/4]$.
For $\sigma,\sigma'$ differing in coordinate $j$, $p^\sigma-p^{\sigma'}=\sqrt2\delta\sigma_j(\mathbf
m_{2j-1}-\mathbf m_{2j})$ with $\norm{p^\sigma-p^{\sigma'}}_2^2=\delta^2\rho^2a^2m'\le4\gamma\delta^2\rho^2\binom{n}{2}$,
so by Lemma~\ref{lem:bern-kl-final2}, ${\rm KL}(\Pb_\sigma,\Pb_{\sigma'})\le C\gamma\delta^2\binom{n}{2}\rho$,
bounded by a small numerical constant for $c_1$ small. Assouad's lemma on $\{\pm1\}^q$ with the
decoder $\widehat\sigma_j={\rm sign}\langle\bar w-w^0,v_j\rangle$ gives, via orthonormality of the
$v_j$ and $\sqrt H\ge H/\sqrt q$, $\inf_{\bar w}\sup_\sigma\E_\sigma\norm{\bar w-w^\sigma}_2\ge
c\delta\sqrt q$; substituting $\delta$ and $q\ge s/3$ proves (3).
\end{proof}

\subsection{Debiased estimation with fitted kernels}
\label{app:debias}

Let $\mathcal E_n=\{(i,j):1\le i<j\le n\}$, $\binom{n}{2}=|\mathcal E_n|$. Conditional on the latent
attributes the graph has independent edges $A_e\sim{\rm Bernoulli}(P_e)$. For the active set $S$,
$\mathbf m_k=\operatorname{vec}_<(\mathbf G_k)$, $\mathbf M=(\mathbf m_k)_{k\in S}\in\R^{\binom{n}{2}\times
s}$. The population target is $w^\dagger$, with $\operatorname{vec}_<(P)=\mathbf Mw^\dagger$ in the
correctly specified fitted-span case; under misspecification $w^\dagger$ is the least-squares
projection coefficient and the result applies after the orthogonal-residual condition of
Remark~\ref{rem:T1-misspec}. Under fidelity and correct specification $w^\dagger$ is the generative
weight.

\paragraph{Fold construction.}
For each $e$, independently over dyads and of $A$, draw
$(R_e^a,R_e^b,R_e^2,R_e^0)\sim{\rm Multinomial}(1;f,f,\pi_2,1-2f-\pi_2)$, $f,\pi_2\in(0,1)$,
$2f+\pi_2<1$. Set $D_{1a}=\{R_e^a=1\}$, $D_{1b}=\{R_e^b=1\}$, $D_2=\{R_e^2=1\}$, the Stage-A IPW
inputs $\widetilde A^g_e=f^{-1}R_e^gA_e$, $E^g_e=\widetilde A^g_e-P_e$, $g\in\{a,b\}$, and the
Stage-B weight $W_2=\pi_2^{-1}\diag(R_e^2)$. All Gram matrices and moments use $\langle
x,y\rangle_2=x^\top W_2y$. For $g\in\{a,b\}$ let $\widehat M_g=(\widehat m^g_k)$ be the fitted
kernels from $\widetilde A^g$, $\widehat M_g=M+H_g$, $h^g_k=\widehat
m^g_k-m_k=\ell^g_k+q^g_k$ with $\ell^g_k=\operatorname{vec}_<\{L_k(E^g)\}$ and $q^g_k$ the
second-order remainder.

\begin{assumption}[Design scale, information, and fitted-fold perturbations]
\label{ass:T1}
Let $\rho_n\to0$, $n\rho_n/\log n\to\infty$, $s$ fixed, $\gamma_S\ge\gamma_0>0$. With probability
tending to one, uniformly over the model class:
\emph{(a)} there are $0<c<C<\infty$ with $cn^2\rho_n^2\le\norm{m_k}_2^2\le Cn^2\rho_n^2$,
$\Phi_0:=M^\top W_2M$ obeys $\lambda_{\min}(\Phi_0)\ge c\gamma_S n^2\rho_n^2$, and with
$D=\diag\{P_e(1-P_e)\}$, $M^\top W_2DW_2M\preceq C\rho_n\Phi_0$;
\emph{(b)} with $\delta_n=\sqrt{r_{\max}/(n\rho_n)}$, $\max_g\norm{M^\top
W_2H_g}_{\op}=O_P(\delta_n n^2\rho_n^2)$, $\max_{g,k}\norm{h_k^g}_2=O_P(\delta_n n\rho_n)$,
$\delta_n=o(1)$;
\emph{(c)} for every $v$ with $\norm{v}_1\le C$, $L_v(E^g)=\sum_kv_kL_k(E^g)$, $Q_v^g=\sum_kv_kq_k^g$,
\begin{gather*}
        \norm{\Phi_0^{-1}M^\top W_2\operatorname{vec}_<\{L_v(E^g)\}}_2=O_P(\sqrt{s/\gamma_S}/(n\sqrt{\rho_n})),\\[2pt]
        \norm{\Phi_0^{-1}M^\top W_2Q_v^g}_2=o_P(\sqrt{s/\gamma_S}/(n\sqrt{\rho_n}));
\end{gather*}
\emph{(d)} the exclusive split does not make $H_a,H_b$ independent in the sense needed for exact
zero mean; instead $\max_{k,\ell}|(h_k^a)^\top W_2h_\ell^b|=O_P(r_{\max}n\rho_n^2+\sqrt{r_{\max}n}\rho_n)+o_P(n\rho_n^{3/2})$,
so $\norm{H_a^\top W_2H_b}_{\op}=o_P(n^2\rho_n^2)$ and, for bounded $\norm{v}_1$, $\norm{\Phi_0^{-1}\sym(H_a^\top
W_2H_b)v}_2=o_P(\sqrt{s/\gamma_S}/(n\sqrt{\rho_n}))$.
\end{assumption}

\begin{remark}[Why \textup{(d)} is the corrected fold condition]
\label{rem:fold-covariance}
For the exclusive multinomial split, $\E(E_e^aE_e^b\mid u)=-P_e^2$, not zero; with $W_2$ present
the sign changes but the order remains $P_e^2$. The cross-fold product is lower order because the
covariance is $O(\rho_n^2)$, not because the two fitted errors are independent. The trace and
delocalisation bounds of Lemma~\ref{lem:firststage} give the displayed order. This is where the
proof differs from the informal ``mean zero by independence'' argument.
\end{remark}

Define $\widehat\Phi_{\rm db}=\sym(\widehat M_a^\top W_2\widehat M_b)$, $\bar M=\tfrac12(\widehat
M_a+\widehat M_b)$, and the eigenvalue-floored $\widehat\Phi_{{\rm db},\eta}$ flooring eigenvalues
at $\eta n^2\rho_n^2$. The unprojected estimator is $\widetilde w_{\rm db}=\widehat\Phi_{{\rm
db},\eta}^{-1}\bar M^\top W_2\operatorname{vec}_<(A)$, and $\widehat w_{\rm
db}=\Pi_{\Delta^{s-1}}(\widetilde w_{\rm db})$ for a simplex target.

\begin{theorem}[Debiased estimation; restatement of Theorem~\ref{thm:debias}]
\label{thm:debias-cor}
Suppose Assumption~\ref{ass:T1} holds and $\operatorname{vec}_<(P)=Mw^\dagger$, $\norm{w^\dagger}_1\le
C$. Then $\norm{\widetilde w_{\rm db}-w^\dagger}_2=O_P(\sqrt{s/\gamma_S}/(n\sqrt{\rho_n}))$, and if
$w^\dagger\in\Delta^{s-1}$ the projected estimator satisfies the same bound. Under fidelity and
correct specification $w^\dagger$ is the generative weight; otherwise it is the fitted-map
projection coefficient.
\end{theorem}

\begin{proof}
Conditioning on the latent attributes is suppressed. Put $\varepsilon=\operatorname{vec}_<(A-P)$,
$\bar H=\tfrac12(H_a+H_b)$.

\emph{Step 1: stability of the debiased Gram.} From $\widehat M_g=M+H_g$,
$\widehat\Phi_{\rm db}=\Phi_0+\sym\{M^\top W_2H_b+H_a^\top W_2M\}+\sym(H_a^\top W_2H_b)$. By (b)
the middle term is $O_P(\delta_n n^2\rho_n^2)$ in operator norm and by (d) the last is
$o_P(n^2\rho_n^2)$, so $\norm{\widehat\Phi_{\rm db}-\Phi_0}_{\op}=o_P(\gamma_S n^2\rho_n^2)$. With
(a) and Weyl, $\lambda_{\min}(\widehat\Phi_{\rm db})\ge\tfrac12c\gamma_S n^2\rho_n^2$ w.p.\ $\to1$,
the flooring is inactive, $\norm{\widehat\Phi_{{\rm db},\eta}^{-1}}_{\op}\le C(\gamma_S
n^2\rho_n^2)^{-1}$, and $\widehat\Phi_{{\rm db},\eta}^{-1}=\Phi_0^{-1}+o_P(\norm{\Phi_0^{-1}}_{\op})$.

\emph{Step 2: exact normal-equation identity.} With $\operatorname{vec}_<(A)=Mw^\dagger+\varepsilon$,
$\widehat\Phi_{\rm db}\widetilde w_{\rm db}=\bar M^\top W_2\operatorname{vec}_<(A)$ on the inactive-floor
event. Expanding $\bar M^\top W_2\operatorname{vec}_<(A)=\Phi_0w^\dagger+M^\top W_2\varepsilon+\bar
H^\top W_2Mw^\dagger+\bar H^\top W_2\varepsilon$ and $\widehat\Phi_{\rm db}w^\dagger=\Phi_0w^\dagger+M^\top
W_2\bar Hw^\dagger+\bar H^\top W_2Mw^\dagger+\sym(H_a^\top W_2H_b)w^\dagger$, the $\bar H^\top
W_2Mw^\dagger$ terms cancel exactly:
\[
\boxed{\ \bar M^\top W_2\operatorname{vec}_<(A)-\widehat\Phi_{\rm db}w^\dagger=M^\top W_2\varepsilon+\bar
H^\top W_2\varepsilon-M^\top W_2\bar Hw^\dagger-\sym(H_a^\top W_2H_b)w^\dagger.\ }
\]
This is the central cancellation; the same-fold term $H^\top W_2H$ that causes single-fold
attenuation is absent.

\emph{Step 3: bounds for the four terms.} Let $a_n=\sqrt{s/\gamma_S}/(n\sqrt{\rho_n})$. The Stage-B
score $S_0=M^\top W_2\varepsilon$ has $\E\norm{\Phi_0^{-1}S_0}_2^2={\rm tr}\{\Phi_0^{-1}M^\top
W_2DW_2M\Phi_0^{-1}\}\le C\rho_n{\rm tr}(\Phi_0^{-1})\le Cs/(\gamma_S n^2\rho_n)$, so
$\norm{\Phi_0^{-1}S_0}_2=O_P(a_n)$. Next $S_1=\bar H^\top W_2\varepsilon$ is, given Stage A and the
folds, centred with $\E(\norm{S_1}_2^2\mid H,R)\le C\rho_n\sum_k\norm{\bar h_k}_2^2$; by (b)
$\norm{S_1}_2=O_P(\sqrt s\,\delta_n n\rho_n^{3/2})$, so after multiplication by $\norm{\widehat\Phi_{{\rm
db},\eta}^{-1}}_{\op}$, $\norm{\widehat\Phi_{{\rm db},\eta}^{-1}S_1}_2=O_P(\sqrt s\,\delta_n/(\gamma_S
n\sqrt{\rho_n}))=o_P(a_n)$. Third, $S_2=M^\top W_2\bar Hw^\dagger=M^\top
W_2\operatorname{vec}_<\{L_{w^\dagger}(\bar E)\}+M^\top W_2\bar Q_{w^\dagger}$; by (c) with
$v=w^\dagger$, $\norm{\Phi_0^{-1}S_2}_2=O_P(a_n)$, unchanged to $1+o_P(1)$ after replacing
$\Phi_0^{-1}$ by $\widehat\Phi_{{\rm db},\eta}^{-1}$, and this term is the first-stage score, of the
same order as the Stage-B score, the reason a two-stage variance is needed. Fourth, by (d)
$\norm{\widehat\Phi_{{\rm db},\eta}^{-1}\sym(H_a^\top W_2H_b)w^\dagger}_2=o_P(a_n)$. Combining the
four with Step 2 gives $\norm{\widetilde w_{\rm db}-w^\dagger}_2=O_P(a_n)$. The simplex projection
is nonexpansive, so $\widehat w_{\rm db}$ satisfies the same bound.
\end{proof}

\begin{remark}[Misspecified fitted-map target]
\label{rem:T1-misspec}
If $\operatorname{vec}_<(P)=Mw^\dagger+r$ with $M^\top r=0$, the argument applies provided
$\norm{\Phi_0^{-1}M^\top W_2r}_2=o_P(a_n)$ and $\norm{\Phi_0^{-1}\bar H^\top W_2r}_2=o_P(a_n)$;
under the test-class and projected-remainder assumptions of Assumption~\ref{ass:T1} these are the
usual projection-target conditions, and $\widetilde w_{\rm db}-w^\dagger=O_P(a_n)$ with $w^\dagger$
the fitted-map least-squares projection coefficient.
\end{remark}

\begin{remark}[Named-agent verification]
For the named agents one must verify the full projected expansion $\widehat
G_k^g-G_k=L_k(E^g)+Q_k^g$ with bounded adjoints, delocalised linear parts, projected-negligible
remainders, and the weighted cross-fold covariance bound of Assumption~\ref{ass:T1}(d), not merely
Frobenius consistency. For spectral smoothers and ASE the leading map is $L_k(E)=P_{U_k}E+EP_{U_k}-P_{U_k}EP_{U_k}$;
for bounded-heterogeneity Chung--Lu it is the delta-method linearisation of $dd^\top/\sum_id_i$. The
exclusive split gives dyad-level cross-covariance of order $P_e^2$, and delocalisation converts this
into $(h_k^a)^\top W_2h_\ell^b=O_P(r_{\max}n\rho_n^2+\sqrt{r_{\max}n}\rho_n)$ rather than the
same-fold order $r_{\max}n\rho_n$; this is the formal replacement for the invalid exact-zero-mean
statement. These expansions are supplied by Lemma~\ref{lem:firststage}.
\end{remark}

\subsection{Two-stage inference for the generative weights}
\label{app:twostage}

All vectors are upper-triangular vectorisations; inner products in the second-stage normal equations
are over the declared $\mathcal D_2$. Write $M=[m_1,\dots,m_s]$, $m_\ell=\operatorname{vec}_{\mathcal
D_2}(G_\ell)$, $\Gamma_n=M^\top M$.

\begin{theorem}[Two-stage inference; restatement of Theorem~\ref{thm:twostage}]
\label{thm:twostage-cor}
Assume the correctly specified mixture $P=\sum_{\ell\in S}w_\ell G_\ell$, $A_e=P_e+\varepsilon^{(2)}_e$
on $\mathcal D_2$ with independent mean-zero $\varepsilon^{(2)}_e$ of variance $D_{2,e}=P_e(1-P_e)$,
and $\lambda_{\min}(\Gamma_n)\ge c_\Gamma\gamma_S n^2\rho_n^2$, $\lambda_{\max}(\Gamma_n)\le
C_\Gamma n^2\rho_n^2$. Let $\widehat M_a=M+H_a$, $\widehat M_b=M+H_b$, $\overline M=\tfrac12(\widehat
M_a+\widehat M_b)$, $\widehat\Phi_{\rm db}=\sym(\widehat M_a^\top\widehat M_b)$, and $\widehat w_{\rm
db}=\widehat\Phi_{\rm db}^{-1}\overline M^\top A^{(2)}$, the inverse ordinary on $\{\lambda_{\min}(\widehat\Phi_{\rm
db})\ge\tfrac12\lambda_{\min}(\Gamma_n)\}$ and eigenvalue-floored otherwise (the modification of
probability $o(1)$). For $t\in\{a,b\}$ let $H_{t,\ell}=L_{t,\ell}(E_t)+Q_{t,\ell}$ with $E_t$ the
IPW Stage-A noise and $L_{t,\ell}$ already composed with restriction to $\mathcal D_2$; put
$L_t^w=\sum_\ell w_\ell L_{t,\ell}$, $Q_t^w=\sum_\ell w_\ell Q_{t,\ell}$. The Stage-A noises satisfy
$\E(E_{t,e})=0$, ${\rm Var}(E_{t,e})=\nu_{t,e}=f^{-1}P_e(1-fP_e)$, with the exclusive-fold
covariance corrections $o(\sigma_{k,n}^2)$ for the linear forms below. For fixed $k$, with
$v_k^0=\Gamma_n^{-1}e_k$, $g_k^0=Mv_k^0$,
\[
        V_{B,k}=\sum_{e\in\mathcal D_2}(g_{k,e}^0)^2P_e(1-P_e),\quad
        V_{A,k}=\tfrac14\sum_{t\in\{a,b\}}\sum_{e\in\mathcal D_n}\bigl[(L_t^w)^*g_k^0\bigr]_e^2f^{-1}P_e(1-fP_e),
\]
$\sigma_{k,n}^2=V_{B,k}+V_{A,k}>0$. Assume \textup{(i)} $\norm{\widehat\Phi_{\rm db}-\Gamma_n}_{\op}/\lambda_{\min}(\Gamma_n)=o_{\mathbb
P}(1)$; \textup{(ii)} $e_k^\top\widehat\Phi_{\rm db}^{-1}\overline H^\top\varepsilon^{(2)}$,
$e_k^\top\widehat\Phi_{\rm db}^{-1}\sym(H_a^\top H_b)w$ and $\langle g_k^0,Q_a^w+Q_b^w\rangle$ are
all $o_{\mathbb P}(\sigma_{k,n})$, $\overline H=\tfrac12(H_a+H_b)$; \textup{(iii)}
$e_k^\top(\widehat\Phi_{\rm db}^{-1}-\Gamma_n^{-1})M^\top\varepsilon^{(2)}$ and
$e_k^\top(\widehat\Phi_{\rm db}^{-1}-\Gamma_n^{-1})M^\top\{L_a^w(E_a)+L_b^w(E_b)\}$ are
$o_{\mathbb P}(\sigma_{k,n})$; \textup{(iv)} the Lindeberg / maximal-leverage conditions
$\max_{e\in\mathcal D_2}(g_{k,e}^0)^2P_e(1-P_e)/\sigma_{k,n}^2\to0$ and $\max_{e\in\mathcal
D_n}[(L_t^w)^*g_k^0]_e^2f^{-1}P_e(1-fP_e)/\sigma_{k,n}^2\to0$; \textup{(v)} plug-in variance
consistency $(\widehat V_{B,k}+\widehat V_{A,k})/(V_{B,k}+V_{A,k})\to_{\mathbb P}1$, where with
$\widehat v_k=\widehat\Phi_{\rm db}^{-1}e_k$, $\widehat g_k=\overline M\widehat v_k$, $\widehat
D_{2,e}=\widehat P_e(1-\widehat P_e)$, $\widehat\nu_{t,e}=f^{-1}\widehat P_e(1-f\widehat P_e)$,
\[
        \widehat V_{B,k}=\sum_{\mathcal D_2}\widehat g_{k,e}^2\widehat D_{2,e},\qquad
        \widehat V_{A,k}=\tfrac14\sum_{t}\sum_{\mathcal D_n}[(\widehat L_t^{\,\widehat
        w})^*\widehat g_k]_e^2\widehat\nu_{t,e}.
\]
In the common-linearisation case $L_a^w=L_b^w=L^w$, $\widehat V_{A,k}=(2f)^{-1}\norm{(\widehat
L^{\,\widehat w})^*\widehat g_k}_{\widehat D}^2$ up to a $1+o_{\mathbb P}(1)$ sparse replacement.
Then for every fixed $k\in S$,
\[
        \frac{\widehat w_{{\rm db},k}-w_k}{(\widehat V_{B,k}+\widehat V_{A,k})^{1/2}}\Rightarrow N(0,1).
\]
If the simplex projection is applied and $\min_\ell w_\ell>\eta_0>0$, the same limit holds for the
projected estimator; on the boundary the unprojected statement is the valid one. The Gaussian
multiplier statistic
\[
        Z_k^*=\sum_{\mathcal D_2}\widehat g_{k,e}\widehat D_{2,e}^{1/2}\xi^{(B)}_e-\tfrac12\sum_t\sum_{\mathcal
        D_n}[(\widehat L_t^{\,\widehat w})^*\widehat g_k]_e\widehat\nu_{t,e}^{1/2}\xi^{(t)}_e
\]
satisfies $\mathcal L(Z_k^*/(\widehat V_{B,k}+\widehat V_{A,k})^{1/2}\mid A,{\rm folds})\Rightarrow
N(0,1)$ in probability.
\end{theorem}

\begin{proof}
For the unprojected estimator; the projected case follows from local inactivity.
\emph{Step 1.} $\widehat w_{\rm db}-w=\widehat\Phi_{\rm db}^{-1}\{\overline
M^\top\varepsilon^{(2)}+(\overline M^\top M-\widehat\Phi_{\rm db})w\}$. With $\overline M=M+\overline
H$, $\overline M^\top M=M^\top M+\tfrac12(H_a+H_b)^\top M$ while $\widehat\Phi_{\rm db}=M^\top
M+\tfrac12M^\top(H_a+H_b)+\tfrac12(H_a+H_b)^\top M+\sym(H_a^\top H_b)$, so $\overline M^\top
M-\widehat\Phi_{\rm db}=-\tfrac12M^\top(H_a+H_b)-\sym(H_a^\top H_b)$ and
\[
        \widehat w_{\rm db}-w=\widehat\Phi_{\rm db}^{-1}M^\top\varepsilon^{(2)}+\widehat\Phi_{\rm
        db}^{-1}\overline H^\top\varepsilon^{(2)}-\tfrac12\widehat\Phi_{\rm db}^{-1}M^\top(H_a+H_b)w-\widehat\Phi_{\rm
        db}^{-1}\sym(H_a^\top H_b)w,
\]
the first-stage term carrying $M^\top(H_a+H_b)w$, not $\overline M^\top(H_a+H_b)w$.
\emph{Step 2.} By (ii)--(iii), the $\overline H^\top\varepsilon^{(2)}$ and $\sym(H_a^\top H_b)w$
terms are $o_{\mathbb P}(\sigma_{k,n})$ and $\widehat\Phi_{\rm db}^{-1}$ may be replaced by
$\Gamma_n^{-1}$ in the two linear terms, giving $\widehat w_{{\rm db},k}-w_k=e_k^\top\Gamma_n^{-1}M^\top\varepsilon^{(2)}-\tfrac12
e_k^\top\Gamma_n^{-1}M^\top(H_a+H_b)w+o_{\mathbb P}(\sigma_{k,n})$.
\emph{Step 3.} $H_tw=L_t^w(E_t)+Q_t^w$, and $e_k^\top\Gamma_n^{-1}M^\top(Q_a^w+Q_b^w)=\langle
g_k^0,Q_a^w+Q_b^w\rangle=o_{\mathbb P}(\sigma_{k,n})$ by (ii); hence $\widehat w_{{\rm
db},k}-w_k=\langle g_k^0,\varepsilon^{(2)}\rangle-\tfrac12\{\langle g_k^0,L_a^w(E_a)\rangle+\langle
g_k^0,L_b^w(E_b)\rangle\}+o_{\mathbb P}(\sigma_{k,n})$.
\emph{Step 4.} The adjoint identity $\langle g_k^0,L_t^w(E_t)\rangle=\langle(L_t^w)^*g_k^0,E_t\rangle$
gives the influence expansion
\[
\boxed{\ \widehat w_{{\rm db},k}-w_k=\langle g_k^0,\varepsilon^{(2)}\rangle-\tfrac12\{\langle(L_a^w)^*g_k^0,E_a\rangle+\langle(L_b^w)^*g_k^0,E_b\rangle\}+o_{\mathbb
P}(\sigma_{k,n}),\ }
\]
the fitted $\widehat g_k$ entering only the feasible variance.
\emph{Step 5.} ${\rm Var}\{\langle g_k^0,\varepsilon^{(2)}\rangle\}=V_{B,k}$ and ${\rm
Var}\{\tfrac12\langle(L_t^w)^*g_k^0,E_t\rangle\}=\tfrac14\sum_e[(L_t^w)^*g_k^0]_e^2f^{-1}P_e(1-fP_e)$,
so with the $o(\sigma_{k,n}^2)$ cross-covariances the leading variance is $\sigma_{k,n}^2\{1+o(1)\}$;
in the common case $V_{A,k}=(2f)^{-1}\norm{(L^w)^*g_k^0}_D^2\{1+o(1)\}$ since $f^{-1}P_e(1-fP_e)=f^{-1}P_e(1-P_e)\{1+O(\rho_n)\}$.
\emph{Step 6.} With $Z_{B,e}=g_{k,e}^0\varepsilon^{(2)}_e$ and $Z_{A,t,e}=-\tfrac12[(L_t^w)^*g_k^0]_eE_{t,e}$,
the expansion is a mean-zero triangular array of total variance $\sigma_{k,n}^2\{1+o(1)\}$; (iv) is
its Lindeberg condition, so the studentised core is asymptotically $N(0,1)$.
\emph{Step 7.} (v) and Slutsky give the feasible CLT.
\emph{Step 8.} If $\min_\ell w_\ell>\eta_0$, $\norm{\widehat w_{\rm db}-w}_\infty=o_{\mathbb P}(1)$
places the iterate in the same relative face, where the projection is locally the identity; on the
boundary it is not locally linear.
\emph{Step 9.} Given the data, $Z_k^*$ is a centred Gaussian with conditional variance exactly
$\widehat V_{B,k}+\widehat V_{A,k}$, so the conditional law of the studentised $Z_k^*$ is $N(0,1)$
up to the negligible flooring and covariance corrections.
\emph{Step 10.} The Stage-B-only variance is $V_{B,k}$ while the correct one is $V_{B,k}+V_{A,k}$
with $V_{A,k}\ge0$; if $\liminf V_{A,k}/V_{B,k}>0$ a nominal Wald interval using only $V_{B,k}$ has
limiting coverage $2\Phi(z_{1-\alpha/2}/\sqrt{1+\lim V_{A,k}/V_{B,k}})-1<1-\alpha$, so strict
undercoverage requires non-negligible first-stage variance.
\end{proof}

\subsection{Misspecification: projection target outside the synthesis class}
\label{app:mis}

Conditional on the latent attributes and the population kernel $P^\ast$, $A_e\sim{\rm
Bernoulli}(P^\ast_e)$; $p^\ast=\operatorname{vec}_<(P^\ast)$, $a=\operatorname{vec}_<(A)$,
$\varepsilon=a-p^\ast$. For each agent $m_k=\operatorname{vec}_<\{G_k^\dagger(P^\ast)\}$, the
population fitted kernel at $P^\ast$, $M=(m_1,\dots,m_K)$, $\gamma_{\mathrm{full}}=\lambda_{\min}(\Phi)$.

\begin{assumption}[Full-candidate set scale, spreading, conditioning]
\label{ass:misspec-full}
$cn^2\rho_n^2\le\norm{m_k}^2\le Cn^2\rho_n^2$, $\norm{m_k}_\infty\le C\rho_n$, $\gamma_{\mathrm{full}}\ge\gamma_0>0$ (so $\lambda_{\min}(M^\top M)\ge c\gamma_{\mathrm{full}}n^2\rho_n^2$), $0\le
P^\ast_e\le C\rho_n$, $n\rho_n/\log n\to\infty$, $K$ fixed (the growing-candidate set extension uses the
corresponding matrix-Bernstein and union conditions).
\end{assumption}

With independent Bernoulli masking $R_{2,e}\sim{\rm Bernoulli}(\pi_2)$, $\widehat M_\alpha=M+H_\alpha$,
$\widehat\Gamma_{\rm db}=\tfrac12(\widehat M_a^\top R_2\widehat M_b+\widehat M_b^\top R_2\widehat
M_a)$, $\widehat b_{\rm db}=\tfrac12(\widehat M_a+\widehat M_b)^\top R_2a$, $\widehat w_{\rm
db}=\widehat\Gamma_{\rm db}^{-1}\widehat b_{\rm db}$ (eigenvalue-floored off the high-probability
event).

\begin{assumption}[Fitted-kernel expansion at $P^\ast$]
\label{ass:misspec-fitted}
$H_{\alpha,k}=L_k(E_\alpha)+q_{\alpha,k}$, $E_a,E_b$ independent given the latent attributes, with
the covariance-operator bound ${\rm Var}\{\langle L_k(E_\alpha),x\rangle\mid x\}\le
C_L^2(\rho_n/f)\norm{x}_2^2$ for deterministic $x$, and the full cross-fold Gram stability and
fitted-kernel bounds of Theorem~\ref{thm:debias} hold at $P^\ast$.
\end{assumption}

\begin{assumption}[Residual-remainder compatibility]
\label{ass:misspec-residual-compatibility}
With $w^\dagger=\arg\min_u\norm{p^\ast-Mu}_2^2$, $r=p^\ast-Mw^\dagger$, $\max_{\alpha,k}|\langle
q_{\alpha,k},R_2r\rangle|=o_{\mathbb P}(\sqrt{\gamma_{\mathrm{full}}}\,n\rho_n^{3/2})$. A sufficient
condition is that $R_2r$ belong to the projected-remainder test class of
Assumption~\ref{ass:linplug}; this is a compatibility requirement, not a consequence of the
projection normal equations.
\end{assumption}

\begin{proposition}[Projection limit and corrective weights]
\label{prop:mis-cor}
Under Assumption~\ref{ass:misspec-full}, $w^\dagger=\arg\min_u\norm{p^\ast-Mu}_2^2$ is unique
and: \textup{(a)} the known-kernel held-out estimator $\widehat w_{\rm kd}=(M^\top R_2M)^{-1}M^\top
R_2a$ satisfies $\norm{\widehat w_{\rm kd}-w^\dagger}_2=O_{\mathbb P}(\sqrt{K/\gamma_{\mathrm{full}}}/(n\sqrt{\rho_n}))$, and under Assumptions~\ref{ass:misspec-fitted}--\ref{ass:misspec-residual-compatibility}
so does $\widehat w_{\rm db}$; \textup{(b)} with $\widetilde m_k=(I-\Pi_{-k})m_k$, $w^\dagger_k=\langle\widetilde
m_k,p^\ast\rangle/\norm{\widetilde m_k}_2^2$, so $w^\dagger_k<0\iff\langle(I-\Pi_{-k})m_k,(I-\Pi_{-k})p^\ast\rangle<0$,
a residualised corrective partial contrast; \textup{(c)} $\norm{p^\ast-Mw^\dagger}_2^2\le\min_k\norm{p^\ast-m_k}_2^2$,
with the analogous cross-entropy statement for logistic calibration over a class containing each
single-agent submodel, and no cross-loss transfer.
\end{proposition}

\begin{proof}
Assumption~\ref{ass:misspec-full} gives $\lambda_{\min}(M^\top M)\ge c\gamma_{\mathrm{full}}n^2\rho_n^2$,
so $M$ has full rank, $w^\dagger$ unique, $M^\top r=0$, $\norm{r}_2\le\norm{p^\ast}_2\le Cn\rho_n$,
$\norm{w^\dagger}_2\le C\gamma_{\mathrm{full}}^{-1/2}$. \emph{(a) Known kernels.} With $\Gamma_2=M^\top
R_2M$, matrix Bernstein gives $\norm{\Gamma_2-\pi_2M^\top M}_{\op}=O_{\mathbb P}(n\rho_n^2)=o_{\mathbb
P}(n^2\rho_n^2)$, so $\lambda_{\min}(\Gamma_2)\ge c\gamma_{\mathrm{full}}n^2\rho_n^2$. From
$a=Mw^\dagger+r+\varepsilon$, $\widehat w_{\rm kd}-w^\dagger=\Gamma_2^{-1}M^\top
R_2\varepsilon+\Gamma_2^{-1}M^\top R_2r$. The noise term has $\E[\norm{\Gamma_2^{-1}M^\top
R_2\varepsilon}_2^2\mid R_2]\le C\rho_n{\rm tr}(\Gamma_2^{-1})\le CK/(\gamma_{\mathrm{full}}n^2\rho_n)$,
hence $O_{\mathbb P}(\sqrt{K/\gamma_{\mathrm{full}}}/(n\sqrt{\rho_n}))$. For the residual, $M^\top
R_2r=M^\top(R_2-\pi_2I)r$ by $M^\top r=0$, with $\E_R\norm{M^\top(R_2-\pi_2I)r}_2^2=\pi_2(1-\pi_2)\sum_e\norm{M_{e\cdot}}_2^2r_e^2\le
CK\rho_n^2\norm{r}_2^2$, so $\norm{M^\top R_2r}_2=O_{\mathbb P}(\sqrt K\,n\rho_n^2)$ and after
inversion $O_{\mathbb P}(\sqrt K/(\gamma_{\mathrm{full}}n))=o_{\mathbb P}$ of the rate. \emph{Fitted
kernels.} $\widehat w_{\rm db}-w^\dagger=\widehat\Gamma_{\rm db}^{-1}(\widehat b_{\rm
db}-\widehat\Gamma_{\rm db}w^\dagger)=\widehat\Gamma_{\rm db}^{-1}(S_0+S_r)$ with $S_0$ the
cross-fold debiased score of Theorem~\ref{thm:debias} at target $Mw^\dagger$, giving $\widehat\Gamma_{\rm
db}^{-1}S_0=O_{\mathbb P}(\sqrt{K/\gamma_{\mathrm{full}}}/(n\sqrt{\rho_n}))$, and $S_r=M^\top
R_2r+\overline H^\top R_2r$. The first piece is controlled as above; for the second, the linear part
$Z_{\alpha,k}=\langle L_k(E_\alpha),R_2r\rangle$ has ${\rm Var}(Z_{\alpha,k}\mid
R_2,r)\le C_L^2(\rho_n/f)\norm{R_2r}_2^2\le Cn^2\rho_n^3$ by Assumption~\ref{ass:misspec-fitted}, so
$\norm{(Z_{\alpha,k})_k}_2=O_{\mathbb P}(\sqrt K\,n\rho_n^{3/2})$ and after inversion $O_{\mathbb
P}(\sqrt{K/\gamma_{\mathrm{full}}}/(n\sqrt{\rho_n}))$, while the remainder is $o_{\mathbb P}$ of the rate
by Assumption~\ref{ass:misspec-residual-compatibility}. \emph{(b)} Frisch--Waugh--Lovell gives
$w^\dagger_k=\langle\widetilde m_k,p^\ast\rangle/\norm{\widetilde m_k}_2^2$ and, since $\widetilde
m_k\perp{\rm col}(M_{-k})$, $\langle\widetilde m_k,p^\ast\rangle=\langle\widetilde
m_k,(I-\Pi_{-k})p^\ast\rangle$. \emph{(c)} $w^\dagger$ minimises over $\R^K\ni e_k$, so
$\norm{p^\ast-Mw^\dagger}_2^2\le\norm{p^\ast-m_k}_2^2$; the cross-entropy statement is the identical
variational argument for the logistic risk.
\end{proof}

\subsection{Separation for the single-fold plug-in}
\label{app:separation}

\begin{corollary}[Single-design plug-in separation; restatement of
Corollary~\ref{cor:debias-lb}]
\label{cor:debias-lb-cor}
Under the conditions of Theorem~\ref{thm:debias} on the active set $S$ with the normalisation of
Assumption~\ref{ass:norm}, all vectors restricted to $\mathcal D_2$, let $Y=\operatorname{vec}_{\mathcal
D_2}(A)=Mw+\varepsilon$ and a single Stage-A design $\widehat M=M+H$ fitted on a fold independent of
$\mathcal D_2$. The unprojected single-design plug-in is $\widetilde w_{\rm sf}=(\widehat
M^\top\widehat M)^{-1}\widehat M^\top Y$. Assume $\lambda_{\max}(\widehat M^\top\widehat M)\le
C_Gn^2\rho_n^2$, $\lambda_{\min}(\widehat M^\top\widehat M)\ge c_G\gamma_S n^2\rho_n^2$, the full
non-cancellation $\norm{Hw}_2^2\ge c_Hn\rho_n\sum_kr_kw_k^2$, the smaller-terms bound
$\norm{(\widehat M^\top\widehat M)^{-1}M^\top Hw}_2+\norm{(\widehat M^\top\widehat
M)^{-1}M^\top\varepsilon}_2+\norm{(\widehat M^\top\widehat M)^{-1}H^\top\varepsilon}_2=O_{\mathbb
P}(\sqrt{s/\gamma_S}/(n\sqrt{\rho_n}))$, and with $r_{\rm eff}(w)=\sum_kr_kw_k^2/\norm{w}_2^2$ the
separation condition $r_{\rm eff}(w)\norm{w}_2\sqrt{\gamma_S/(s\rho_n)}\to\infty$. Then with
probability tending to one $\norm{\widetilde w_{\rm sf}-w}_2\ge c\,r_{\rm eff}(w)\norm{w}_2/(n\rho_n)$,
so the single-design plug-in is separated from the debiased estimator by the diverging factor
$r_{\rm eff}(w)\norm{w}_2\sqrt{\gamma_S/(s\rho_n)}$; if $r_k\asymp r_{\max,S}$ then $r_{\rm
eff}(w)\asymp r_{\max,S}$.
\end{corollary}

\begin{proof}
With $G_n=\widehat M^\top\widehat M$, the normal equations give $\widetilde w_{\rm
sf}-w=-G_n^{-1}H^\top Hw-G_n^{-1}M^\top Hw+G_n^{-1}M^\top\varepsilon+G_n^{-1}H^\top\varepsilon$, so
the attenuation is $B_n=G_n^{-1}H^\top Hw$ and the remainder $R_n$ collects the other three, with
$\norm{R_n}_2=O_{\mathbb P}(\sqrt{s/\gamma_S}/(n\sqrt{\rho_n}))=o_{\mathbb P}(r_{\rm
eff}(w)\norm{w}_2/(n\rho_n))$ by the separation condition. Since $G_n\succ0$, $\norm{G_n^{-1}x}_2\ge\norm{x}_2/\lambda_{\max}(G_n)$,
and by Cauchy $\norm{H^\top Hw}_2\ge w^\top H^\top Hw/\norm{w}_2=\norm{Hw}_2^2/\norm{w}_2$; with the
non-cancellation and $\lambda_{\max}(G_n)\le C_Gn^2\rho_n^2$, $\norm{B_n}_2\ge(c_H/C_G)r_{\rm
eff}(w)\norm{w}_2/(n\rho_n)$. The reverse triangle inequality and $c<c_H/(2C_G)$ give the claim.
\end{proof}

\begin{remark}[Projection and same-data plug-in]
The lower bound is for the unprojected held-out estimator: nonexpansiveness of the simplex
projection gives upper, not lower, bounds, and a projected statement needs a tangent-cone condition
that projection cannot cancel $G_n^{-1}H^\top Hw$; the proof also uses independence of $Y$ from the
Stage-A design and is not a proof for the fully same-data estimator.
\end{remark}

\subsection{Two-stage selection and the adaptive oracle rate}
\label{app:selection}

\begin{corollary}[Adaptive selection]
\label{cor:adapt-cor}
Let $K=K_n$ be the number of candidates and let
\[
        S_n=\{k\le K_n:w_{n,k}\neq 0\},\qquad
        w_{\min,n}=\min_{k\in S_n}|w_{n,k}|,\qquad
        v_{\max,n}=\max_{k\le K_n} v_{n,k}.
\]
For notational simplicity assume $S_n\neq\varnothing$; if $S_n=\varnothing$, all statements below are
interpreted with the beta-min condition omitted and exact recovery meaning $\widehat S_n=\varnothing$.

Let $\widehat w^{\,\mathrm{full}}_n=(\widehat w_{n,1},\ldots,\widehat w_{n,K_n})$ be the unconstrained
full estimator, computed with the cross-fold debiased Gram and moment when the kernels are
fitted. Let $\widehat v_{n,k}$ be a standard error for $\widehat w_{n,k}$, equal to the Stage-B
sandwich standard error in the known-kernel case and to the corresponding two-stage
generated-regressor standard error in the fitted-kernel case. Let $v_{n,k}>0$ denote the population
standard deviation of the leading influence expansion for $\widehat w_{n,k}-w_{n,k}$, including the
first-stage contribution when the kernels are fitted.

The target $w_n$ is the target of the full regression: under correct specification and fidelity
it is the generative weight vector, while under misspecification it is the population projection
coefficient vector. The adaptive refit is always taken over the same parameter space as the
corresponding oracle refit; thus, for signed projection or calibration targets the refit is
unconstrained, whereas a simplex refit is used only for an interior generative simplex target, in
which case the projection is asymptotically inactive.

\smallskip
\noindent\textbf{Fixed candidate set.}
Suppose $K_n\equiv K<\infty$. Assume that, for every $k\le K$,
\[
        \frac{\widehat w_{n,k}-w_{n,k}}{v_{n,k}}=O_{\mathbb P}(1),
        \qquad
        \frac{\widehat v_{n,k}}{v_{n,k}}\stackrel{\mathbb P}{\longrightarrow}1 .
\]
Let $c_n\to\infty$ and assume the beta-min condition
\[
        c_n v_{\max,n}=o(w_{\min,n}).
        \tag{F}
\]
Define $\widehat S_n=\{k\le K: |\widehat w_{n,k}|>c_n\widehat v_{n,k}\}$. Then
$\mathbb P(\widehat S_n=S_n)\to 1$. Consequently, if $\widehat w_n^{\,\mathrm{or}}$ denotes the oracle
refit on $S_n$ and $\|\widehat w_n^{\,\mathrm{or}}-w_n\|_2=O_{\mathbb P}(r_n)$, then the adaptive refit
$\widehat w_n^{\,\mathrm{ad}}$ satisfies $\|\widehat w_n^{\,\mathrm{ad}}-w_n\|_2=O_{\mathbb P}(r_n)$.

\smallskip
\noindent\textbf{Growing candidate set.}
Suppose $K_n\to\infty$, $K_n\ge2$, and fix $\eta>0$. Set
\[
        t_n=(1+\eta)\sqrt{2\log K_n},
        \qquad
        \widehat S_n=\{k\le K_n:|\widehat w_{n,k}|>t_n\widehat v_{n,k}\}.
\]
Assume the uniform full conditions
\[
        \max_{k\le K_n}\Bigl|\frac{\widehat v_{n,k}}{v_{n,k}}-1\Bigr|
        \stackrel{\mathbb P}{\longrightarrow}0,
        \tag{G1}
\]
\[
        \mathbb P\Bigl(\max_{k\le K_n}\frac{|\widehat w_{n,k}-w_{n,k}|}{v_{n,k}}
        >(1+\eta/2)\sqrt{2\log K_n}\Bigr)\longrightarrow 0,
        \tag{G2}
\]
and the growing-candidate set beta-min condition
\[
        \sqrt{\log K_n}\,v_{\max,n}=o(w_{\min,n}).
        \tag{G3}
\]
Then $\mathbb P(\widehat S_n=S_n)\to 1$. Consequently, if the oracle refit satisfies
$\|\widehat w_n^{\,\mathrm{or}}-w_n\|_2=O_{\mathbb P}(r_n)$, then so does the adaptive refit.

For known kernels, condition \textup{(G2)} follows from a uniform Bernstein bound for the full
coordinate linear forms; in the notation of the main text, the side condition
\[
        K_n\log K_n=o(\gamma_{\mathrm{full},n}^{\,2}n^2\rho_n)
\]
is a sufficient range condition whenever the stated full leverage bound and linear expansion
hold uniformly in $k\le K_n$. For fitted kernels, \textup{(G1)}--\textup{(G2)} require the
corresponding uniform version of the two-stage expansion and variance consistency; coordinatewise
consistency alone is not enough when $K_n\to\infty$.
\end{corollary}

\begin{proof}
\emph{Fixed candidate set.} Since $K$ is fixed and $(\widehat w_{n,k}-w_{n,k})/v_{n,k}=O_{\mathbb P}(1)$ for
each $k$, a finite union gives $\max_{k\le K}|\widehat w_{n,k}-w_{n,k}|/v_{n,k}=O_{\mathbb P}(1)$, and
the coordinatewise standard-error consistency with fixed $K$ gives
$\max_{k\le K}|\widehat v_{n,k}/v_{n,k}-1|\to_{\mathbb P}0$. Hence
$R_n:=\max_{k\le K}|\widehat w_{n,k}-w_{n,k}|/\widehat v_{n,k}=O_{\mathbb P}(1)$, so
$\mathbb P\{R_n\le c_n/2\}\to1$. By \textup{(F)} and uniform standard-error consistency,
$c_n\max_{k\le K}\widehat v_{n,k}=o_{\mathbb P}(w_{\min,n})$, so
$\mathbb P\{c_n\max_k\widehat v_{n,k}\le w_{\min,n}/4\}\to1$. On the intersection event: if
$k\notin S_n$ then $w_{n,k}=0$ and $|\widehat w_{n,k}|\le(c_n/2)\widehat v_{n,k}<c_n\widehat v_{n,k}$,
so $k\notin\widehat S_n$; and if $k\in S_n$ then
$|\widehat w_{n,k}|\ge w_{\min,n}-(c_n/2)\widehat v_{n,k}\ge 7w_{\min,n}/8>w_{\min,n}/4\ge
c_n\widehat v_{n,k}$, so $k\in\widehat S_n$. Both inclusions hold with probability tending to one, so
$\mathbb P(\widehat S_n=S_n)\to1$.

\emph{Growing candidate set.} Put $a_n=(1+\eta/2)\sqrt{2\log K_n}$ and $t_n=(1+\eta)\sqrt{2\log K_n}$. Choose
$\delta\in(0,1)$ with $a_n\le(1-\delta)t_n$ for all $n$, possible because
$a_n/t_n=(1+\eta/2)/(1+\eta)<1$. By \textup{(G2)},
$\mathbb P\{\max_k|\widehat w_{n,k}-w_{n,k}|/v_{n,k}\le a_n\}\to1$; by \textup{(G1)},
$\mathbb P\{(1-\delta)v_{n,k}\le\widehat v_{n,k}\le(1+\delta)v_{n,k}\ \forall k\}\to1$; and
\textup{(G3)} gives $a_nv_{\max,n}=o(w_{\min,n})$ and $t_nv_{\max,n}=o(w_{\min,n})$. On the
intersection event: if $k\notin S_n$ then
$|\widehat w_{n,k}|\le a_nv_{n,k}\le(1-\delta)t_nv_{n,k}\le t_n\widehat v_{n,k}$, and the strict
selection rule excludes it; if $k\in S_n$ then
$|\widehat w_{n,k}|\ge w_{\min,n}-a_nv_{\max,n}$ while $t_n\widehat v_{n,k}\le t_n(1+\delta)v_{\max,n}$,
and $w_{\min,n}-a_nv_{\max,n}>t_n(1+\delta)v_{\max,n}$ for large $n$, so $k\in\widehat S_n$. Hence
$\mathbb P(\widehat S_n=S_n)\to1$.

\emph{Oracle-rate transfer.} On $\{\widehat S_n=S_n\}$,
$\widehat w_n^{\,\mathrm{ad}}=\widehat w_n^{\,\mathrm{or}}$, so for every $M>0$,
$\mathbb P(\|\widehat w_n^{\,\mathrm{ad}}-w_n\|_2>Mr_n)\le
\mathbb P(\|\widehat w_n^{\,\mathrm{or}}-w_n\|_2>Mr_n)+\mathbb P(\widehat S_n\neq S_n)$; the second
term vanishes and the first is made small by large $M$, giving
$\|\widehat w_n^{\,\mathrm{ad}}-w_n\|_2=O_{\mathbb P}(r_n)$.

\emph{Bernstein verification of \textup{(G2)} (known kernels).} Suppose uniformly in $k\le K_n$,
$\widehat w_{n,k}-w_{n,k}=\sum_{e\in\mathcal D_n}h_{n,k,e}\varepsilon_{n,e}+r_{n,k}$ with independent
mean-zero bounded Bernoulli residuals,
$\sum_e h_{n,k,e}^2\operatorname{Var}(\varepsilon_{n,e})=v_{n,k}^2$, and
$\Delta_n:=\max_{k,e}|h_{n,k,e}|/v_{n,k}$ satisfying $\Delta_n\sqrt{\log K_n}\to0$, with
$\max_k|r_{n,k}|/(v_{n,k}\sqrt{\log K_n})\to_{\mathbb P}0$. Bernstein's inequality gives, for
$x_n=(1+\eta/3)\sqrt{2\log K_n}$,
$\mathbb P(|\sum_e h_{n,k,e}\varepsilon_{n,e}|>x_nv_{n,k})\le2\exp\{-x_n^2/(2(1+o(1)))\}$, and a union
bound yields
$\mathbb P(\max_k|\sum_e h_{n,k,e}\varepsilon_{n,e}|/v_{n,k}>x_n)\le2K_n^{\,1-(1+\eta/3)^2+o(1)}\to0$.
The uniform remainder bound upgrades $x_n$ to $(1+\eta/2)\sqrt{2\log K_n}$, proving \textup{(G2)}.
\end{proof}

\subsection{The noisy-OR operator: linearisation, efficiency, hierarchy, and boundary}
\label{app:noisyor}

\begin{theorem}[The noisy-OR operator; restatement of Theorem~\ref{thm:nor}]
\label{thm:nor-cor}
Conditionally on the layer kernels, suppose $A_{ij}\sim\operatorname{Bernoulli}\{P_{ij}(w_0)\}$
independently over $(i,j)\in\mathcal D_n$, where $w_0=(w_{01},w_{02})\in\mathcal W=[\varepsilon,
1-\varepsilon]^2$, $0<\varepsilon<1/2$. For $w=(w_1,w_2)\in\mathcal W$, let
$P(w)=J-(J-w_1G_1)\odot(J-w_2G_2)$. Write $g_1=\operatorname{vec}_<(G_1)$,
$g_2=\operatorname{vec}_<(G_2)$, $h=\operatorname{vec}_<(G_1\odot G_2)$, $M=[g_1,g_2,h]$,
$c_0=(w_{01},w_{02},-w_{01}w_{02})^\top$, $s_1=\|g_1\|_2$, $s_2=\|g_2\|_2$, $s_{12}=\|h\|_2$,
$S=\operatorname{diag}(s_1,s_2,s_{12})$, $\bar M=MS^{-1}$, and $\gamma_e=\lambda_{\min}(\bar M^\top\bar
M)$. Assume, uniformly in $n$:
\begin{enumerate}
\item[(A1)] \emph{Dense common scale.} $c\,n\rho_n\le s_1,s_2\le C\,n\rho_n$,
$c\,n\rho_n^2\le s_{12}\le C\,n\rho_n^2$, $\max_{i<j}G_{k,ij}\le C\rho_n$ ($k=1,2$), and
$\max_{i<j}G_{1,ij}G_{2,ij}\le C\rho_n^2$.
\item[(A2)] \emph{Nondegenerate Bernoulli information.} For every $w$ near $w_0$,
$c\rho_n\le P_{ij}(w)\{1-P_{ij}(w)\}\le C\rho_n$ and
$c\rho_n\,\bar M^\top\bar M\preceq\bar M^\top D(w)\bar M\preceq C\rho_n\,\bar M^\top\bar M$, where
$D(w)=\operatorname{diag}(P_{ij}(w)\{1-P_{ij}(w)\})$.
\item[(A3)] \emph{Transversality and leverage.} $\gamma_e\ge\gamma_0>0$ and
$\max_{i<j}\|\bar M_{ij,\cdot}\|_2=o(1)$.
\item[(A4)] \emph{Sparsity scale.} $n^2\rho_n\to\infty$; the operator-test scale is $n^2\rho_n^3$.
\end{enumerate}
Then:
\begin{enumerate}
\item[(a)] \emph{Exact linearisation and identification.}
$P(w_0)=w_{01}G_1+w_{02}G_2-w_{01}w_{02}\,G_1\odot G_2$, so
$\operatorname{vec}_<\{P(w_0)\}=Mc_0$. If $\gamma_e>0$ the augmented vector
$(w_{01},w_{02},-w_{01}w_{02})$ is identified; full augmented rank suffices but is not necessary, and
in general the weights are identified iff $w\mapsto P(w)$ is injective on $\mathcal W$.
\item[(b)] \emph{Least-squares pilot and rates.} With $a=\operatorname{vec}_<(A)$,
$\widehat c=(M^\top M)^{-1}M^\top a$, and $\widehat w^{\,LS}=\Pi_{\mathcal W}(\widehat c_1,\widehat
c_2)$, one has $\|\widehat w^{\,LS}-w_0\|_2=O_{\mathbb P}(\gamma_e^{-1/2}/(n\sqrt{\rho_n}))$, while the
interaction coordinate has the slower rate
$|\widehat c_3+w_{01}w_{02}|=O_{\mathbb P}(\gamma_e^{-1/2}/(n\rho_n^{3/2}))$. Over the class
satisfying \textup{(A1)--(A4)} with $\gamma_e\ge\gamma_0$,
$\inf_{\widetilde w}\sup_{w_0}\mathbb E_{w_0}\|\widetilde w-w_0\|_2\ge c/(n\sqrt{\rho_n})$, so the
pilot is rate-optimal in the $(n,\rho_n)$ scale (this lower bound does not assert sharpness in
$\gamma_e$).
\item[(c)] \emph{One-step efficiency.} With
$\dot p_e(w)=(G_{1,ij}(1-w_2G_{2,ij}),\,G_{2,ij}(1-w_1G_{1,ij}))^\top$,
$I_n(w)=\sum_{i<j}\dot p_{ij}\dot p_{ij}^\top/(p_{ij}(1-p_{ij}))$, and
$S_n(w)=\sum_{i<j}\dot p_{ij}(A_{ij}-p_{ij})/(p_{ij}(1-p_{ij}))$, the one-step estimator
$\widehat w^{\,OS}=\widehat w^{\,LS}+I_n(\widehat w^{\,LS})^{-1}S_n(\widehat w^{\,LS})$ satisfies
$\widehat w^{\,OS}-w_0=I_n(w_0)^{-1}S_n(w_0)+o_{\mathbb P}(\|I_n(w_0)^{-1/2}\|)$ and
$I_n(w_0)^{1/2}(\widehat w^{\,OS}-w_0)\Rightarrow N(0,I_2)$, attaining the parametric information
bound.
\item[(d)] \emph{General $K$: free-coefficient hierarchy.} For fixed $K$ and
$P_K(w)=1-\prod_{k=1}^K(1-w_kG_k)$, define for nonempty $S\subseteq[K]$ the column
$m_S=\operatorname{vec}_<(\bigodot_{k\in S}G_k)$ and $c_S^*=(-1)^{|S|+1}\prod_{k\in S}w_k$. Fix
$1\le j\le K$; let $M_{\le j}$ collect columns $m_S$ with $1\le|S|\le j$, let $\widehat c_{\le j}$ be
its least-squares estimator, and let $\gamma_{e,j}$ be the smallest eigenvalue of its normalised
correlation Gram. Assuming $\|m_S\|_2\asymp n\rho_n^{|S|}$ for $1\le|S|\le j$, $\gamma_{e,j}>0$, and
$K\rho_n\le1/2$, then for $|S|=i\le j$,
\[
    |\widehat c_S-c_S^*|=O_{\mathbb P}\!\Bigl(\frac{\gamma_{e,j}^{-1/2}}{n\rho_n^{\,i-1/2}}\Bigr)
    +O\!\bigl(\gamma_{e,j}^{-1}K^{j+1}\rho_n^{\,j+1-i}\bigr),
\]
equivalently
$|\widehat c_S-c_S^*|\,\|m_S\|_2/\|P_K(w)\|_F=O_{\mathbb P}(\gamma_{e,j}^{-1/2}/(n\sqrt{\rho_n}))
+O(\gamma_{e,j}^{-1}K^{j+1}\rho_n^j)$. Thus the relative fitted-mean contribution of each included
order is estimated at the edge rate; no claim is made that the constrained product map can be
inverted at the edge rate for $K\ge3$, and the theorem-level weight recovery is the two-layer result
of (a)--(c).
\item[(e)] \emph{Operator-detectability boundary.} With
$P_{\mathrm{mix}}(w_0)=w_{01}G_1+w_{02}G_2$ and
$P_{\mathrm{or}}(w_0)=P_{\mathrm{mix}}(w_0)-w_{01}w_{02}G_1\odot G_2$: if $n^2\rho_n^3\to0$ then
$\operatorname{TV}(\mathbb P_{P_{\mathrm{or}}(w_0)},\mathbb P_{P_{\mathrm{mix}}(w_0)})\to0$, so no test
distinguishes the operators below this scale. Conversely, with $B=[g_1,g_2]$, $D_0=D(w_0)$ under the
mixture null, $\Pi_0$ the $D_0^{-1}$-projection onto $\operatorname{span}(B)$, $h_\perp=(I-\Pi_0)h$,
and $\mathcal I_{12,n}=h_\perp^\top D_0^{-1}h_\perp$: if
$\mathcal I_{12,n}\asymp\Gamma_{e2}n^2\rho_n^3$ with $\Gamma_{e2}>0$, the residualised interaction
score test has noncentrality $w_{01}w_{02}\sqrt{\mathcal I_{12,n}}\asymp
w_{01}w_{02}\sqrt{\Gamma_{e2}n^2\rho_n^3}$ and is consistent whenever
$w_{01}^2w_{02}^2\Gamma_{e2}n^2\rho_n^3\to\infty$.
\end{enumerate}
All assertions are conditional on known layer kernels, or on external first-stage estimates whose
error is $o_{\mathbb P}$ of the corresponding column scale. If the layers are fitted from the same
graph, the interaction column must be constructed by the separate cross-fold method of
Theorem~\ref{prop:fitted-operator}; a single-fold fitted interaction column is not covered here.
\end{theorem}

\begin{proof}
All probability statements are conditional on the layer kernels; constants may change line to line.

\emph{(a).} For each dyad $P_{ij}(w_0)=1-(1-w_{01}G_{1,ij})(1-w_{02}G_{2,ij})
=w_{01}G_{1,ij}+w_{02}G_{2,ij}-w_{01}w_{02}G_{1,ij}G_{2,ij}$, so
$\operatorname{vec}_<\{P(w_0)\}=Mc_0$. If $\gamma_e>0$, $M$ has full column rank and $c_0$ is the
unique augmented representation, identifying $(w_{01},w_{02})$. Full rank is not necessary: if
$G_1\odot G_2=0$ the third column vanishes but the weights remain identified when $G_1,G_2$ are
linearly independent; the general criterion is injectivity of $w\mapsto P(w)$.

\emph{(b).} Let $\varepsilon=a-\operatorname{vec}_<\{P(w_0)\}$, so $a=Mc_0+\varepsilon$ and
$\widehat c-c_0=(M^\top M)^{-1}M^\top\varepsilon$. With $M=\bar MS$,
$S(\widehat c-c_0)=(\bar M^\top\bar M)^{-1}\bar M^\top\varepsilon$, and
$\operatorname{Cov}\{S(\widehat c-c_0)\}=(\bar M^\top\bar M)^{-1}\bar M^\top D(w_0)\bar M(\bar M^\top\bar
M)^{-1}\preceq C\rho_n(\bar M^\top\bar M)^{-1}$. Since the augmented dimension is fixed and
$\lambda_{\min}(\bar M^\top\bar M)=\gamma_e$, $\|S(\widehat c-c_0)\|_2=O_{\mathbb P}(\sqrt{\rho_n/
\gamma_e})$. With $s_k\asymp n\rho_n$ ($k=1,2$) and $s_{12}\asymp n\rho_n^2$, and one-Lipschitz
projection onto $\mathcal W$, the displayed first-two-coordinate and interaction rates follow. For the
lower bound, perturb only the first coordinate: $P(w^\delta)-P(w_0)=\delta\,G_1\odot(J-w_{02}G_2)$ has
$\|\cdot\|_F^2\asymp n^2\rho_n^2$ by (A1), so by the Bernoulli KL bound valid under (A2),
$\operatorname{KL}\{\mathbb P_{w_0},\mathbb P_{w^\delta}\}\le C\delta^2n^2\rho_n$; taking
$\delta=c_0/(n\sqrt{\rho_n})$ keeps it bounded, and Le Cam's two-point lemma gives the claimed minimax
floor.

\emph{(c).} By (A2), $I_n(w_0)\asymp n^2\rho_n$ positive-definitely, so the pilot is
$O_{\mathbb P}(\|I_n(w_0)^{-1/2}\|)$. The Bernoulli log-likelihood is a fixed-dimension triangular-array
likelihood with first derivatives $O(\rho_n)$, the only nonzero second derivative of $p_{ij}$ being
$\partial^2_{w_1w_2}p_{ij}=-G_{1,ij}G_{2,ij}=O(\rho_n^2)$, and denominators of order $\rho_n$; the
leverage condition (A3) supplies Lindeberg and stochastic equicontinuity. The LAN expansion
$S_n(\widehat w^{\,LS})=S_n(w_0)-I_n(w_0)(\widehat w^{\,LS}-w_0)+o_{\mathbb P}(\|I_n(w_0)^{1/2}\|\,
\|\widehat w^{\,LS}-w_0\|_2)$ and $I_n(\widehat w^{\,LS})=I_n(w_0)\{1+o_{\mathbb P}(1)\}$, substituted
into the one-step definition, give
$\widehat w^{\,OS}-w_0=I_n(w_0)^{-1}S_n(w_0)+o_{\mathbb P}(\|I_n(w_0)^{-1/2}\|)$, and
$I_n(w_0)^{-1/2}S_n(w_0)\Rightarrow N(0,I_2)$ by the Lindeberg CLT, proving efficiency.

\emph{(d).} Inclusion--exclusion gives
$P_K(w)=\sum_{\varnothing\ne S\subseteq[K]}(-1)^{|S|+1}(\prod_{k\in S}w_k)\bigodot_{k\in S}G_k$, so
$a=M_{\le j}c_{\le j}^*+r_{>j}+\varepsilon$ with $r_{>j}$ the orders above $j$, and
$\widehat c_{\le j}-c_{\le j}^*=(M_{\le j}^\top M_{\le j})^{-1}M_{\le j}^\top\varepsilon+
(M_{\le j}^\top M_{\le j})^{-1}M_{\le j}^\top r_{>j}$. Normalising the order-$i$ columns by
$\|m_S\|_2\asymp n\rho_n^i$, the noise term has normalised error $O_{\mathbb P}(\sqrt{\rho_n/
\gamma_{e,j}})$, i.e. $O_{\mathbb P}(\gamma_{e,j}^{-1/2}/(n\rho_n^{\,i-1/2}))$ after dividing by
$\|m_S\|_2$. For the bias, $\|r_{>j}\|_2\le Cn\sum_{\ell>j}\binom K\ell\rho_n^\ell\le Cn(K\rho_n)^{j+1}$
under $K\rho_n\le1/2$, so
$|e_S^\top(M_{\le j}^\top M_{\le j})^{-1}M_{\le j}^\top r_{>j}|\le C\gamma_{e,j}^{-1}\|r_{>j}\|_2/
\|m_S\|_2\le C\gamma_{e,j}^{-1}K^{j+1}\rho_n^{\,j+1-i}$. Combining and multiplying by
$\|m_S\|_2/\|P_K(w)\|_F$ ($\|P_K(w)\|_F\asymp n\rho_n$) gives the relative-contribution display; the
unnormalised $|\widehat c_S-c_S^*|\,\|m_S\|_2$ is not of edge-rate order in general.

\emph{(e).} \emph{Impossibility.} $\Delta=P_{\mathrm{or}}(w_0)-P_{\mathrm{mix}}(w_0)=
-w_{01}w_{02}G_1\odot G_2$ has $\|\Delta\|_F^2\le Cn^2\rho_n^4$ by (A1), so by the Bernoulli KL bound
and (A2) $\operatorname{KL}(\mathbb P_{P_{\mathrm{or}}}\|\mathbb P_{P_{\mathrm{mix}}})\le
C\|\Delta\|_F^2/\rho_n\le Cn^2\rho_n^3$; if $n^2\rho_n^3\to0$, Pinsker gives $\operatorname{TV}\to0$ and
no consistent test exists. \emph{Positive direction.} Work under the mixture null with
$\langle x,y\rangle_0=x^\top D_0^{-1}y$, $h_\perp=(I-\Pi_0)h$, and the oracle residualised score
$T_n=h_\perp^\top D_0^{-1}\{a-\operatorname{vec}_<(P_{\mathrm{mix}}(w_0))\}/(h_\perp^\top D_0^{-1}
h_\perp)^{1/2}$. Under the null $T_n\Rightarrow N(0,1)$ by Lindeberg with (A3), and replacing the
nuisance by its $(n^2\rho_n)^{-1/2}$-consistent null estimate changes $T_n$ by $o_{\mathbb P}(1)$
because the score is orthogonalised against the nuisance tangent space. Under the alternative
$\mathbb E_{\mathrm{or}}\{a-\operatorname{vec}_<(P_{\mathrm{mix}}(w_0))\}=-w_{01}w_{02}h$, and since
$h_\perp\perp_0\operatorname{span}(B)$, $h_\perp^\top D_0^{-1}h=h_\perp^\top D_0^{-1}h_\perp$, so the
mean of $T_n$ is $-w_{01}w_{02}\sqrt{\mathcal I_{12,n}}\asymp-w_{01}w_{02}\sqrt{\Gamma_{e2}n^2
\rho_n^3}$; the test is consistent whenever $w_{01}^2w_{02}^2\Gamma_{e2}n^2\rho_n^3\to\infty$.
\end{proof}

\subsection{General-$K$ operator detectability with known layers}
\label{app:operator}

\begin{corollary}[General-$K$ operator detectability]
\label{cor:nork-cor}
Let $K$ be fixed. For each $n$, let $\mathcal E_n=\{(i,j):1\le i<j\le n\}$, $\binom{n}{2}=|\mathcal E_n|$,
$g_k=\operatorname{vec}_<(G_k)\in\mathbb R^{\binom n2}$, and $M=(g_1,\ldots,g_K)\in\mathbb R^{\binom n2\times K}$.
Let $p_{\rm or}=\operatorname{vec}_<\{\mathbf J-\bigodot_{k=1}^K(\mathbf J-w_kG_k)\}$ be the noisy-OR
mean vector, $w_k\in[w_0,1-w_0]$ for fixed $w_0\in(0,1/2)$, and let the additive comparison class be
$\mathcal P_{{\rm add},n}=\{Mu:u\in\mathcal U_n\}$, with $\mathcal U_n\subset\mathbb R^K$ a
deterministic set for which $Mu$ is a valid Bernoulli mean. Assume:
\begin{enumerate}
\item[(A1)] \emph{Sparse bounded layers.} For constants $0<c_g<C_g<\infty$ and $\eta\in(0,1)$, all
$k\le K$ and $e\in\mathcal E_n$: $0\le g_{k,e}\le C_g\rho_n$, $c_gn^2\rho_n^2\le\|g_k\|_2^2\le
C_gn^2\rho_n^2$, with $\rho_n\to0$, $n^2\rho_n\to\infty$, and $KC_g\rho_n\le\eta$ eventually.
\item[(A2)] \emph{Main-effect conditioning.} $\lambda_{\min}(M^\top M)\ge c_Mn^2\rho_n^2$.
\item[(A3)] \emph{Information comparability.} With $H_2=[g_k\odot g_\ell]_{k<\ell}$,
$\Pi_M=M(M^\top M)^{-1}M^\top$, $Z=(I-\Pi_M)H_2$, $p_u=Mu$, and $D_u=\operatorname{diag}(p_{u,e}
(1-p_{u,e}))$: there are $0<c_D<C_D<\infty$ such that, uniformly over $u\in\mathcal U_n$ and all $x$ in
the span of the columns of $Z$ and the projected higher-order remainder below,
$c_D\rho_n\|x\|_2^2\le x^\top D_ux\le C_D\rho_n\|x\|_2^2$ and $x^\top D_u^{-1}x\le C_D\rho_n^{-1}
\|x\|_2^2$.
\item[(A4)] \emph{Fixed-rank low-leverage residualised design.} $r_2=\operatorname{rank}(Z)=O(1)$ and,
uniformly over $u$, $\lambda_{\min}^+(Z^\top D_uZ)\asymp n^2\rho_n^5\asymp\lambda_{\max}(Z^\top D_uZ)$,
with the maximal leverage condition $\sup_u\max_e D_{u,e}\,z_e^\top(Z^\top D_uZ)^+z_e\to0$, $z_e^\top$
the $e$th row of $Z$.
\item[(A5)] \emph{Nonvanishing orthogonal order-two interaction.} With
$\Delta_2=\sum_{k<\ell}w_kw_\ell(g_k\odot g_\ell)$, $h=(I-\Pi_M)\Delta_2$, and
$\Gamma_{e2}=\|h\|_2^2/(n^2\rho_n^4)$, assume $\liminf_n\Gamma_{e2}>0$.
\end{enumerate}
Then: \textup{(1)} if $\Gamma_{e2}n^2\rho_n^3\to0$, the noisy-OR law is asymptotically equivalent in
total variation to a sequence in $\mathcal P_{{\rm add},n}$, so no test distinguishes them; and
\textup{(2)} if $\Gamma_{e2}n^2\rho_n^3\to\infty$, the heteroskedastic score/Wald test on
$Z=(I-\Pi_M)H_2$ has asymptotic size $\alpha$ under $\mathcal P_{{\rm add},n}$ and power tending to one
under the noisy-OR law. Thus, for fixed $K$ and nonvanishing orthogonal interaction mass, the operator
is detectable iff $n^2\rho_n^3\to\infty$.
\end{corollary}

\begin{proof}
\emph{Step 1: inclusion--exclusion remainder.}
$p_{\rm or}=Mw-\Delta_2+R_{\ge3}$ with $R_{\ge3}=\sum_{j=3}^K(-1)^{j+1}\sum_{|S|=j}(\prod_{k\in
S}w_k)\bigodot_{k\in S}g_k$; since $K$ is fixed and $0\le g_{k,e}\le C_g\rho_n$, $\|R_{\ge3}\|_2\le
C_Kn\rho_n^3$.

\emph{Step 2: least favourable additive law.} Let $u_0=(M^\top M)^{-1}M^\top p_{\rm or}$ and
$p_0=Mu_0=\Pi_Mp_{\rm or}$. Then $u_0-w=-(M^\top M)^{-1}M^\top(\Delta_2-R_{\ge3})$, and by (A1)
$\|M^\top\Delta_2\|_\infty\le C_Kn^2\rho_n^3$, $\|M^\top R_{\ge3}\|_\infty\le C_Kn^2\rho_n^4$; with
(A2), $\|u_0-w\|_\infty\le C_K\rho_n+O_K(\rho_n^2)=o(1)$. Hence $u_{0,k}\ge w_0/2$ eventually,
$p_0\ge0$, and $p_{0,e}\le\sum_ku_{0,k}g_{k,e}\le CK\rho_n(1+o(1))<1$, so $p_0\in\mathcal P_{{\rm
add},n}$.

\emph{Step 3: orthogonal residual.} $p_{\rm or}-p_0=\Pi_M^\perp p_{\rm or}=-h+r$ with
$h=\Pi_M^\perp\Delta_2$, $r=\Pi_M^\perp R_{\ge3}$; $\|h\|_2^2=\Gamma_{e2}n^2\rho_n^4$, and since
$\liminf\Gamma_{e2}>0$, $\|r\|_2\le C_Kn\rho_n^3=o(n\rho_n^2)=o(\|h\|_2)$.

\emph{Step 4: contiguity below the boundary.} By the Bernoulli KL bound and (A3),
$\operatorname{KL}(\mathbb P_{\rm or}\|\mathbb P_0)=\sum_e\operatorname{kl}(p_{{\rm or},e},p_{0,e})\le
C(p_{\rm or}-p_0)^\top D_{u_0}^{-1}(p_{\rm or}-p_0)\le C\rho_n^{-1}\|p_{\rm or}-p_0\|_2^2\le
C\Gamma_{e2}n^2\rho_n^3+C_Kn^2\rho_n^5+o(\Gamma_{e2}n^2\rho_n^3)$. If $\Gamma_{e2}n^2\rho_n^3\to0$
then, since $\liminf\Gamma_{e2}>0$, also $n^2\rho_n^3\to0$ and $n^2\rho_n^5\to0$, so the bound tends to
zero; Pinsker gives $\|\mathbb P_{\rm or}-\mathbb P_0\|_{\rm TV}\to0$. As $p_0\in\mathcal P_{{\rm
add},n}$, for any test $\phi_n$, $\sup_{p\in\mathcal P_{{\rm add},n}}\mathbb E_p\phi_n+\mathbb
E_{\rm or}(1-\phi_n)\ge\mathbb E_0\phi_n+\mathbb E_{\rm or}(1-\phi_n)\ge1-\|\mathbb P_{\rm or}-\mathbb
P_0\|_{\rm TV}\to1$.

\emph{Step 5: residualised order-two test above the boundary.} Let $Z=(I-\Pi_M)H_2$ and
$V_u=Z^\top D_uZ$. Given $A=\operatorname{vec}_<(A_{ij})$, let $\widehat u$ be a
root-$\binom{n}{2}\rho_n$-consistent null estimator of $u$ (e.g. least squares on $M$), put
$\widehat D=D_{\widehat u}$, $\widehat V=Z^\top\widehat DZ$, and define $S_n=Z^\top(A-M\widehat u)$,
$T_n=S_n^\top\widehat V^+S_n$. Because $Z^\top M=0$, under the null $S_n=Z^\top(A-p_u)$; the plug-in
variance is consistent since $\|\widehat u-u\|_2=O_{\mathbb P}(1/(n\sqrt{\rho_n}))$ gives
$\max_e|(M\widehat u-Mu)_e|=o_{\mathbb P}(\rho_n)$ and hence $\widehat V=V_u+o_{\mathbb P}
(\|V_u\|_{\rm op})$ on the range of $Z$. For each fixed $a\in\mathbb R^{r_2}$, the scalar projection of
$V_u^{-1/2}S_n$ is a triangular-array sum of independent mean-zero bounded Bernoulli residuals, and
(A4) is the Lindeberg condition, so $V_u^{+1/2}S_n\Rightarrow N(0,I_{r_2})$ and $T_n\Rightarrow
\chi^2_{r_2}$ under every additive law. The test $\phi_{n,\alpha}=\mathbf 1\{T_n>\chi^2_{r_2,1-\alpha}\}$
has asymptotic size $\alpha$.

\emph{Step 6: power under noisy-OR.} Use $p_0=Mu_0$ from Step 2; since $Z^\top M=0$, $S_n=Z^\top(A-
p_0)$ and $\mathbb E_{\rm or}S_n=Z^\top(-h+r)$. Now $\Delta_2=H_2a$ with $a_{k\ell}=w_kw_\ell$, so
$h=(I-\Pi_M)\Delta_2=Za$ lies in the column space of $Z$, and by (A3), for $x$ in that space,
$(Z^\top x)^\top V_{u_0}^+(Z^\top x)\asymp\rho_n^{-1}\|x\|_2^2$. Using Step 3, the noncentrality is
$\lambda_n=\{\mathbb E_{\rm or}S_n\}^\top V_{u_0}^+\{\mathbb E_{\rm or}S_n\}=(1+o(1))(Z^\top h)^\top
V_{u_0}^+(Z^\top h)\asymp\rho_n^{-1}\|h\|_2^2=\Gamma_{e2}n^2\rho_n^3$. If $\Gamma_{e2}n^2\rho_n^3\to
\infty$ then $\lambda_n\to\infty$; the stochastic part of $S_n$ stays $O_{\mathbb P}(1)$ after
standardisation while the mean shift has squared norm $\lambda_n$, so $T_n\to\infty$ in
$\mathbb P_{\rm or}$-probability and the level-$\alpha$ test has power tending to one. Steps 4 and 6
give the boundary.
\end{proof}

\begin{remark}[Scope]
Corollary~\ref{cor:nork-cor} is a known-layer result. If the layers are fitted from the same graph,
the interaction columns must be constructed on independent Stage-A sub-folds, as in the fitted-layer
operator test of Theorem~\ref{prop:fitted-operator}. The corollary does not claim recovery of the full
constrained noisy-OR product structure $c_S=(-1)^{|S|+1}\prod_{k\in S}w_k$ for $K\ge3$; it is a
detection result for additive versus overlapping composition.
\end{remark}

\subsection{The single-coordinate detection scale}
\label{app:support}
\emph{Converse.} The boundary is coordinatewise. Fix one coordinate and consider, inside the family of
Appendix~\ref{app:lower}, the two hypotheses ``coordinate $k$ active at amplitude $w$'' and
``coordinate $k$ inactive'', with the other coordinates held fixed. A single active coordinate at
amplitude $w$ perturbs the probability matrix by $\Delta\Pmat$ with
$\norm{\Delta\Pmat}_F^2\asymp\gamma\,w^2\,n^2\rho^2$, so the two-point divergence is
$\mathrm{KL}\lesssim\gamma\,w^2\,n^2\rho$. If $w\le c_0/(\sqrt{\gamma}\,n\sqrt{\rho})$ then
$\mathrm{KL}\le c_0^2\le\tfrac12$, and by Pinsker the total variation is at most $\tfrac12$, so no test
of that coordinate has power bounded away from its level \citep[Theorem~2.2]{tsybakov2008nonparametric}. The
single-coordinate scale is therefore $1/(\sqrt{\gamma}\,n\sqrt{\rho})$; the $\sqrt{\log K}$ inflation
when all $K$ coordinates are searched is shown necessary in Appendix~\ref{app:gaps}
(Theorem~\ref{thm:sqrtlogK-final2}), so this converse and that lower bound together pin the selection
scale.

\subsection{Operator identification with fitted layers}
\label{app:fitop}
This section proves Theorem~\ref{prop:fitted-operator}, relocated from the main text. All inner products are over the Stage-B dyads $\mathcal D_2$ and conditioning on the latent attributes is in force.

\begin{proof}
Write $\widehat{\mathbf C}=\mathbf G_1\odot\mathbf G_2+\bm\Delta$ with $\bm\Delta$ collecting the
six cross terms, and decompose the numerator under the null $\Pmat=\mathbf M\bw$ as
\[
\bigl\langle\widehat{\mathbf C}^{\perp},\bm\varepsilon^{(2)}\bigr\rangle
+\bigl\langle\widehat{\mathbf C}^{\perp},\mathbf M(\bw-\widehat\bw_{\mathrm{db}})\bigr\rangle
-\bigl\langle\widehat{\mathbf C}^{\perp},\bar{\mathbf H}\widehat\bw_{\mathrm{db}}\bigr\rangle,
\qquad \bar{\mathbf H}=\tfrac12(\mathbf H_c+\mathbf H_d).
\]
\emph{(i) Core term.} Conditionally on Stage A,
$\langle\widehat{\mathbf C}^{\perp},\bm\varepsilon^{(2)}\rangle$ is a sum over Stage-B dyads of
independent mean-zero terms whose variance is the square of the standardiser,
$\asymp\widetilde\Gamma_2\,n^2\rho_n^5$. Each entry of $\widehat{\mathbf C}^{\perp}$ is
$O_P(\rho_n^2)$ up to the delocalisation of Assumption~\ref{ass:linplug}(iii), so each dyad carries
an $O(n^{-2})$ share of the variance, Lindeberg's condition holds, and the studentised core is
asymptotically standard normal.

\emph{(ii) Leakage terms vanish by the four-fold split.} By construction
$\langle\widehat{\mathbf C}^{\perp},\bar{\mathbf m}_k\rangle=0$, so
$\langle\widehat{\mathbf C}^{\perp},\mathbf M(\bw-\widehat\bw_{\mathrm{db}})\rangle
=-\sum_k(\bw-\widehat\bw_{\mathrm{db}})_k\langle\widehat{\mathbf C}^{\perp},\bar{\mathbf
h}_k\rangle$, and the third term has the same structure with bounded coefficients. Here
$\bar{\mathbf h}_k$ is $(c,d)$-measurable, but $\widehat{\mathbf C}^{\perp}$ is \emph{not}
$(a,b)$-measurable: the residualising projector $\Pi_{\bar{\mathbf M}}$ is built from $(c,d)$.
Decompose
\[
\widehat{\mathbf C}^{\perp}\ =\ (\mathbf I-\Pi_{\mathbf M})\widehat{\mathbf C}\ -\
(\Pi_{\bar{\mathbf M}}-\Pi_{\mathbf M})\widehat{\mathbf C}.
\]
The first piece is $(a,b)$-measurable, so conditionally on $(a,b)$ each inner product
$\langle(\mathbf I-\Pi_{\mathbf M})\widehat{\mathbf C},\mathcal L_k(\mathbf E_c)\rangle$ is exactly
mean zero with variance at most $\norm{\operatorname{Cov}(\operatorname{vec}\mathcal L_k(\mathbf
E_c))}_{\mathrm{op}}\norm{\widehat{\mathbf C}}_F^2\le C_L^2(\rho_n/f)\cdot O(n^2\rho_n^4)$ by
Assumption~\ref{ass:linplug}(iii); multiplied by
$\norm{\bw-\widehat\bw_{\mathrm{db}}}_2=O_P(\sqrt{s/\gamma_S}/(n\sqrt{\rho_n}))$ from
Theorem~\ref{thm:debias}, the contribution is $O_P(\sqrt{s/\gamma_S}\,\rho_n^2/\sqrt f)$, which is
$o$ of the null standard deviation $\sqrt{\widetilde\Gamma_2}\,n\rho_n^{5/2}$ since the ratio is
$\sqrt{s/(\gamma_S\widetilde\Gamma_2 f)}\,/(n\sqrt{\rho_n})\to0$. For the second piece, the
projector perturbation identity gives $\Pi_{\bar{\mathbf M}}-\Pi_{\mathbf M}=(\mathbf
I-\Pi_{\mathbf M})\bar{\mathbf H}(\mathbf M^\top\mathbf M)^{-1}\mathbf M^\top+\mathbf M(\mathbf
M^\top\mathbf M)^{-1}\bar{\mathbf H}^\top(\mathbf I-\Pi_{\mathbf M})+\mathbf R_{\Pi}$ with
$\norm{\mathbf R_{\Pi}}_{\mathrm{op}}=O_P(s\,\delta_n^2/\gamma_S)$ and $\bar{\mathbf H}$ the
$(c,d)$ first-stage errors, so $\norm{\Pi_{\bar{\mathbf M}}-\Pi_{\mathbf
M}}_{\mathrm{op}}=O_P(\sqrt s\,\delta_n/\sqrt{\gamma_S})=o_P(1)$. Its pairing with the Stage-B
noise is mean zero with variance at most $\rho_n\norm{(\Pi_{\bar{\mathbf M}}-\Pi_{\mathbf
M})\widehat{\mathbf C}}_F^2=o_P(\rho_n\norm{\widehat{\mathbf C}}_F^2)$, an asymptotically
negligible share of the null variance. Its pairing with the first-stage part of the residual,
$-\bar{\mathbf H}\widehat\bw_{\mathrm{db}}$, is a same-fold quadratic in $(\mathbf E_c,\mathbf
E_d)$, the very object the four-fold split removes elsewhere; here it is second order, with mean
bounded by $C_L^2(\rho_n/f)\cdot\norm{(\mathbf M^\top\mathbf M)^{-1}\mathbf
M^\top\widehat{\mathbf C}}_2\,\max_l\norm{\mathbf G_l}_F=O(\sqrt s\,n\rho_n^3/(f\gamma_S))$ and
Hanson--Wright fluctuation of the same order, hence a standardised null shift
$O(\sqrt{s/\widetilde\Gamma_2}\,\rho_n^{1/2}/(f\gamma_S))\to0$; the $\mathbf R_{\Pi}$ remainder
contributes $O_P(s^{3/2}\delta_n^3\rho_n^{1/2}/(\gamma_S\sqrt{\widetilde\Gamma_2}))$ of the null
standard deviation, also vanishing. The conditional means of the first piece are products of
deterministic first-stage biases, covered by the displayed condition; and since
$\norm{\Pi_{\bar{\mathbf M}}-\Pi_{\mathbf M}}_{\mathrm{op}}=o_P(1)$, the population
transversality $\widetilde\Gamma_2$, defined at $\Pi_{\mathbf M}^{\perp}$, governs the fitted
residualisation as well. Had the
interaction column and the main effects shared a Stage-A sub-fold, these inner products would contain
same-fold squares such as $\langle\mathbf H_1^{a}\odot\mathbf G_2,\ \mathcal L_1(\mathbf
E_a)\text{-terms}\rangle$, with conditional mean of order
$\rho_n\norm{\mathbf H_1^{a}}_F^2\asymp r_{\max}\,n\rho_n^2$ and hence a standardised null shift of
order $r_{\max}/(\sqrt{\widetilde\Gamma_2\,\rho_n}\,)\to\infty$: the four-fold split is what
removes this shift, the Hadamard analogue of the Gram debiasing of Theorem~\ref{thm:debias}.

\emph{(iii) The interaction column's own error.} The cross terms $\mathbf H_1^{a}\odot\mathbf G_2$
are mean zero given $\bu$ with $\norm{\mathbf H_1^{a}\odot\mathbf G_2}_F=O_P(\rho_n\norm{\mathbf
H_1^{a}}_F)=O_P(\sqrt{r_{\max}n\rho_n^3})=o_P(n\rho_n^2)$; the self-product $\mathbf
H_1^{a}\odot\mathbf H_2^{b}$ has conditional mean $\E[\mathbf H_1^{a}\mid\bu]\odot\E[\mathbf
H_2^{b}\mid\bu]$, covered by the displayed condition, and a fluctuation with entrywise variance
$O((r_{\max}\rho_n/n)^2)$ by Assumption~\ref{ass:linplug}(iii), hence Frobenius norm
$O_P(r_{\max}\rho_n)=o_P(n\rho_n^2)$. The projected interaction is therefore estimated at
$o_P(n\rho_n^2)$ accuracy, which leaves both the standardiser and the contiguity calculation of
Theorem~\ref{thm:nor}(e) unchanged.

\emph{(iv) Power.} Under the noisy-OR alternative the numerator gains
$w_1w_2\langle\widehat{\mathbf C}^{\perp},\mathbf G_1\odot\mathbf G_2\rangle=
w_1w_2\norm{\Pi^{\perp}_{\bar{\mathcal M}}(\mathbf G_1\odot\mathbf G_2)}_F^2\,(1+o_P(1))\asymp
\widetilde\Gamma_2\,n^2\rho_n^4$ by (iii), so the noncentrality is
$\asymp\sqrt{\widetilde\Gamma_2\,n^2\rho_n^3}$, diverging exactly when the edge-overlap information
does.
\end{proof}

\subsection{Operator signature: sign identification and held-out validation}
\label{app:operator-signature}
The dense-core test of Section~\ref{sec:numerical} reports the sign of a community-by-degree interaction.
The following theorem says exactly what that sign identifies, when it is recoverable above the
$n^2\rho_n^3$ scale, and when a held-out log-score validation of a signed-interaction model succeeds. It is
a statement about the residualized interaction coefficient and its out-of-sample improvement; it does not,
on its own, attribute a historical mechanism to any particular network.

\begin{theorem}[Operator signature, sign recovery, and validation]
\label{thm:operator-signature}
Let $x=(x_e)$ and $y=(y_e)$ be two deterministic or cross-fitted layer kernels on $\mathcal E_n$, with
$0\le x_e,y_e\le C\rho_n$. Let $h=x\odot y$, $M=[x\ y]$, and $z_E=(I-\Pi_M^E)h$, where $\Pi_M^E$ is
Euclidean projection onto the span of the two main-effect columns. Suppose $\underline c\rho_n\le P_e\le
\bar c\rho_n$ for all $e$, and assume interaction transversality and low leverage,
\[
    \|z_E\|_2^2\asymp n^2\rho_n^4,\qquad
    \frac{\max_e z_{E,e}^2 P_e(1-P_e)}{\sum_e z_{E,e}^2 P_e(1-P_e)}\to 0 .
\]
Define the Euclidean operator signature $\eta_E(P)=\langle z_E,P\rangle/\|z_E\|_2^2$ and its held-out
least-squares estimator $\widehat\eta_E=\langle z_E,A\rangle/\|z_E\|_2^2$. Then:
\begin{enumerate}
\item if $P=\alpha x+\beta y+r$ with $\langle z_E,r\rangle=0$, then $\eta_E(P)=0$;
\item if $P=1-(1-\alpha x)(1-\beta y)+r=\alpha x+\beta y-\alpha\beta h+r$ and
$|\langle z_E,r\rangle|<\alpha\beta\|z_E\|_2^2$, then $\eta_E(P)<0$;
\item if $P=\alpha x+\beta y+\lambda h+r$ with $\lambda>0$ and $|\langle z_E,r\rangle|<\lambda\|z_E\|_2^2$,
then $\eta_E(P)>0$;
\item more generally, if an independently specified mechanism kernel $c$ satisfies $P=\alpha x+\beta
y+\lambda c+r$ with $|\langle z_E,r\rangle|=o(|\lambda\langle z_E,c\rangle|)$, then
$\mathrm{sign}\{\eta_E(P)\}=\mathrm{sign}\{\lambda\langle z_E,c\rangle\}$;
\item the estimator is asymptotically normal,
$(\widehat\eta_E-\eta_E(P))/\sigma_{E,n}\Rightarrow N(0,1)$ with
$\sigma_{E,n}^2=\sum_e z_{E,e}^2 P_e(1-P_e)/\|z_E\|_2^4\asymp 1/(n^2\rho_n^3)$, so
$\mathbb P\{\mathrm{sign}(\widehat\eta_E)=\mathrm{sign}(\eta_E(P))\}\to 1$ whenever
$|\eta_E(P)|\sqrt{n^2\rho_n^3}\to\infty$.
\end{enumerate}
For held-out log-score validation, let $T_n\subset\mathcal E_n$ be an independent test set, let $q_e$ be the
additive held-out Kullback--Leibler benchmark on $T_n$ with $q_e\in[\underline c\rho_n,\bar c\rho_n]$, and
define the Fisher inner product $\langle a,b\rangle_W=\sum_{e\in T_n}a_e b_e/\{q_e(1-q_e)\}$. Let
$z_W=(I-\Pi_M^W)h$ and $I_{T,n}=\|z_W\|_W^2$. Suppose the test-set probabilities have the local form
$P_e=q_e+\eta_W z_{W,e}+r_e$ with $\|r\|_W=o(|\eta_W|\|z_W\|_W)$ and
$\max_{e\in T_n}|\eta_W z_{W,e}+r_e|/\{q_e(1-q_e)\}\to 0$. Let $\widetilde\eta_W$ be trained independently of
$T_n$ with $\widetilde\eta_W/\eta_W\to_P 1$. If $\eta_W^2 I_{T,n}\to\infty$, then
\[
    \frac{\ell_{T_n}(q+\widetilde\eta_W z_W)-\ell_{T_n}(q)}{\eta_W^2 I_{T,n}}\to_P\tfrac12,
\]
so the signed-interaction model improves out of sample with probability tending to one. Finally, suppose
the same conditions hold uniformly for independent networks $r=1,\dots,R_n$ in a class, with a common
nonzero sign of $\eta_{W,r}$, and that the per-network log-score increments are bounded,
$\max_{e\in T_{n,r}}|\widetilde\eta_{W,r} z_{W,r,e}|/\{q_e(1-q_e)\}=O(1)$, so that a Bernstein bound applies
to each network's gain. If $\min_{1\le r\le R_n}\eta_{W,r}^2 I_{T,n,r}\gg\log R_n$, then all network-level
signs and held-out gains replicate with probability tending to one. (Under Chebyshev bounds alone, without
the bounded-increment condition, the stronger separation $\min_r\eta_{W,r}^2 I_{T,n,r}\gg R_n$ is required.)
\end{theorem}

\begin{proof}
Since $z_E=(I-\Pi_M^E)h$, $\langle z_E,x\rangle=\langle z_E,y\rangle=0$ and $\langle z_E,h\rangle=\|z_E\|_2^2$.
If $P=\alpha x+\beta y+r$ then $\eta_E(P)=\langle z_E,r\rangle/\|z_E\|_2^2$, which is zero when
$\langle z_E,r\rangle=0$. For noisy-OR, $1-(1-\alpha x)(1-\beta y)=\alpha x+\beta y-\alpha\beta h$, so
$\eta_E(P)=-\alpha\beta+\langle z_E,r\rangle/\|z_E\|_2^2$, negative under the stated bound; the
super-additive case gives $\eta_E(P)=\lambda+\langle z_E,r\rangle/\|z_E\|_2^2>0$ identically. For the
mechanism-mediated statement, $\eta_E(P)=\lambda\langle z_E,c\rangle/\|z_E\|_2^2+\langle
z_E,r\rangle/\|z_E\|_2^2$, whose sign is that of $\lambda\langle z_E,c\rangle$ because the residual term is
lower order. For estimation, $\widehat\eta_E-\eta_E(P)=\sum_e a_{n,e}(A_e-P_e)$ with
$a_{n,e}=z_{E,e}/\|z_E\|_2^2$; the summands are independent, mean zero, and bounded, with variance
$\sigma_{E,n}^2=\sum_e z_{E,e}^2 P_e(1-P_e)/\|z_E\|_2^4$. The leverage condition is Lindeberg's condition,
so the triangular-array central limit theorem gives the stated normality; since $P_e(1-P_e)\asymp\rho_n$ and
$\|z_E\|_2^2\asymp n^2\rho_n^4$, $\sigma_{E,n}^2\asymp\rho_n/\|z_E\|_2^2\asymp 1/(n^2\rho_n^3)$, and the
sign-recovery condition $|\eta_E(P)|/\sigma_{E,n}\to\infty$ is equivalent to
$|\eta_E(P)|\sqrt{n^2\rho_n^3}\to\infty$.

For the log score, condition on the training data so that $\widetilde\eta_W$ is fixed on the independent
test set. With $t_e=\widetilde\eta_W z_{W,e}$, a Taylor expansion of the Bernoulli log-likelihood around
$q_e$ gives, uniformly under the smallness condition,
\[
    \mathbb E_P[\ell_{T_n}(q+\widetilde\eta_W z_W)-\ell_{T_n}(q)\mid\widetilde\eta_W]
    =\eta_W\widetilde\eta_W I_{T,n}-\tfrac12\widetilde\eta_W^2 I_{T,n}+o_P(\eta_W^2 I_{T,n}),
\]
the $r$-contribution being $o_P(|\eta_W\widetilde\eta_W|I_{T,n})$ by $\|r\|_W=o(|\eta_W|\|z_W\|_W)$. Since
$\widetilde\eta_W/\eta_W\to_P 1$, the conditional mean is $\tfrac12\eta_W^2 I_{T,n}\{1+o_P(1)\}$ and the
conditional variance is $\widetilde\eta_W^2 I_{T,n}\{1+o_P(1)\}$. As $\eta_W^2 I_{T,n}\to\infty$ the standard
deviation is $o_P(\eta_W^2 I_{T,n})$, so Chebyshev's inequality gives the displayed limit and the held-out
gain is positive with probability tending to one. For the replication statement, the bounded-increment
condition makes each network's log-score gain a sum of bounded independent terms, so a Bernstein bound gives
per-network error probability $\exp\{-c\,\eta_{W,r}^2 I_{T,n,r}\}$; a union bound over the $R_n$ networks has
total error at most $R_n\exp\{-c\min_r\eta_{W,r}^2 I_{T,n,r}\}\to 0$ under $\min_r\eta_{W,r}^2
I_{T,n,r}\gg\log R_n$. Without the bounded-increment condition, Chebyshev gives only polynomial tails and the
stronger separation $\gg R_n$ is required.
\end{proof}

\begin{remark}[What the operator-signature theorem does and does not establish]
Theorem~\ref{thm:operator-signature} identifies the residualized interaction sign with the operator class
(additive, sub-additive, super-additive), shows the sign is recoverable above the $n^2\rho_n^3$ scale, and
shows that a signed-interaction model validated by Fisher residualization improves a held-out log score. It
does not, by itself, prove the historical mechanism behind any named network. The dense-core signs of
Table~\ref{tab:denseop} are empirical, and attributing a sign to a specific generative mechanism requires the
independently specified kernel $c$ of part~(iv), constructed from metadata, a temporal split, or a
pre-registered validation kernel rather than from the same edges used in the test, together with the
out-of-sample replication of the final clause. The theorem supplies the inferential scaffolding for such a
claim; the claim itself is an empirical question that may return a null.
\end{remark}

\section{Closing the remaining gaps}
\label{app:gaps}

Throughout this section,
\[
  \mathcal D_n=\{(i,j):1\le i<j\le n\},\qquad |\mathcal D_n|=\binom n2.
\]
For a symmetric matrix \(B\), \(b=\vecl(B)\) denotes its upper-triangular
dyad vector.  All Frobenius norms and inner products are over \(\mathcal D_n\).
Constants may depend on fixed probability-band constants, fixed ranks,
fixed fold fractions, and fixed finite dimension, but not on
\(n,s,K,\rho\).

\begin{lemma}[Bernoulli KL bound]\label{lem:bern-kl-final2}
Let \(p=(p_e:e\in\mathcal D_n)\) and \(q=(q_e:e\in\mathcal D_n)\) satisfy
\[
  c\rho\le p_e,q_e\le C\rho\le 1/2
\]
for fixed \(0<c<C<\infty\).  Then
\[
  \KL(\Pb_p,\Pb_q)
  \le C_{\rm KL}(c,C)\rho^{-1}\sum_{e\in\mathcal D_n}(p_e-q_e)^2 .
\]
\end{lemma}

\begin{proof}
For Bernoulli variables,
\[
  \kl(p,q)
  =
  p\log(p/q)+(1-p)\log\{(1-p)/(1-q)\}
  \le
  \frac{(p-q)^2}{q(1-q)} .
\]
Since \(q(1-q)\ge c\rho/2\), summing over the independent dyads gives
the claim.
\end{proof}

\subsection{Estimated-design lower bound over faithful separated spectra}

The theorem below is deliberately stated only in the \(s\)-range in which
the explicit one-eigenvalue-per-agent fitting maps are faithful.  The
larger \(s^2\lesssim \gamma n^2\rho\) lower bound remains a known-design
oracle lower bound; the faithful estimated-design statement requires the
additional spectral-label condition \(s^4\log n=o(\gamma n\rho)\).

\begin{definition}[Faithful spectral contrast subfamily]\label{def:fss-final2}
Let \(0<\gamma\le\gamma_0=1/64\) and put \(a=\sqrt\gamma\).
Let \(n_0=2^{\lfloor\log_2 n\rfloor}\), so \(n/2\le n_0\le n\).
Work on a fixed vertex block \(V_0\) of size \(n_0\), and extend all
contrast entries by zero outside \(V_0\times V_0\).

Let \(h_1,\ldots,h_s\in\{\pm1\}^{n_0}\) be balanced mutually orthogonal
Hadamard rows:
\[
  \1^\top h_k=0,\qquad
  h_k^\top h_\ell=n_0\,\mathbf 1\{k=\ell\}.
\]
Let \(H_k=h_kh_k^\top\) on \(V_0\times V_0\), extended by zero outside
the block.  Let \(B=\rho J\), where \(J\) is the all-one off-diagonal
matrix on all \(n\) vertices, and define
\[
  X_k=\rho aH_k,\qquad
  G_k=B+X_k,\qquad k=1,\ldots,s.
\]
The model is
\[
  P_w
  =
  B+\sum_{k=1}^s w_kX_k
  =
  \sum_{k=1}^s w_kG_k,
  \qquad w\in\Delta^{s-1}.
\]
Because \(\sum_kw_k=1\) and \(|H_{k,ij}|\le1\),
\[
  P_{w,ij}\in[\rho(1-a),\rho(1+a)]
  \subset[3\rho/4,5\rho/4].
\]

Fix \(0<\eta<1/4\) and define
\[
  w_k^0
  =
  \frac1s\left\{1+\eta\,\frac{2k-s-1}{s}\right\},
  \qquad k=1,\ldots,s.
\]
Then \(\sum_kw_k^0=1\), \(w_k^0\in[(1-\eta)/s,(1+\eta)/s]\), and
\[
  \min_{\ell\ne k}|w_k^0-w_\ell^0|\ge 2\eta/s^2 .
\]

For \(P\) in a neighborhood of the local class, define a completed
contrast matrix \(C(P)\) on \(V_0\) by
\[
  C(P)_{ij}=P_{ij}-B_{ij}\quad(i\ne j,\ i,j\in V_0),
  \qquad
  C(P)_{ii}=\rho a .
\]
For \(P=P_w\),
\[
  C(P_w)=\rho a\sum_{k=1}^s w_k h_kh_k^\top .
\]
Let \(\Pi_k(P)\) be the spectral projector of \(C(P)\) corresponding to
the unique eigenvalue in the band centered at \(n_0\rho a w_k^0\) with
radius \((\eta/2)n_0\rho a/s^2\).  Define
\[
  F_k(P)=\rho a n_0\,\Pi_k(P),
\]
with diagonal discarded after evaluation.  On the local class,
\(F_k(P_w)=X_k\).  The empirical version replaces \(P\) by the IPW
Stage-A adjacency.  The empirical bands are consistently recovered if
\[
  n\rho a\,s^{-2}\gg \sqrt{n\rho\log n},
  \qquad\text{equivalently}\qquad
  s^4\log n=o(\gamma n\rho).
\]
\end{definition}

\begin{theorem}[Faithful estimated-design lower bound]\label{thm:fss-lb-final2}
Assume
\[
  n\rho/\log n\to\infty,\qquad
  8\le s\le
  c_s\min\left\{n_0-1,\left(\frac{\gamma n\rho}{\log n}\right)^{1/4}\right\},
  \qquad
  0<\gamma\le 1/64,
\]
with \(c_s>0\) sufficiently small.  Over the faithful spectral contrast
subfamily of Definition~\ref{def:fss-final2},
\[
  \inf_{\widehat w}
  \sup_{w\in\calW_{\rm fss}}
  \Eb_w\|\widehat w-w\|_2
  \ge
  c\,\frac{\sqrt{s/\gamma}}{n\sqrt\rho}.
\]
The infimum is over all estimators measurable with respect to the single
observed graph.  The result is an estimated-design lower bound because
the subfamily contains explicit faithful fitting maps \(F_k\).  Since the
same lower bound holds in the easier oracle experiment where
\(X_1,\ldots,X_s\) are revealed, it also holds in the estimated-design
experiment.
\end{theorem}

\begin{proof}
First verify the geometry.  Since the \(h_k\)'s are balanced and
orthogonal,
\[
  \langle J,H_k\rangle_F=O(n_0),\qquad
  \langle H_k,H_\ell\rangle_F=O(n_0)\quad(k\ne\ell),
\]
whereas
\[
  \|J\|_F^2\asymp n^2,\qquad \|H_k\|_F^2\asymp n_0^2\asymp n^2.
\]
Thus
\[
  \|G_k\|_F^2=\rho^2n^2\{c_0+a^2c_1+o(1)\},
  \qquad
  \langle G_k,G_\ell\rangle_F=\rho^2n^2\{c_0+o(1)\}
  \quad(k\ne\ell),
\]
for positive constants \(c_0,c_1\).  Hence the Gram-correlation matrix
of \(G_1,\ldots,G_s\) is an equicorrelated carrier direction plus
\(a^2\)-scale orthogonal contrasts, and
\[
  \lambda_{\min}\{\Phi(G_1,\ldots,G_s)\}\asymp a^2=\gamma .
\]

Next verify faithfulness.  For every \(w\) in the packing below,
\[
  C(P_w)=\rho a\sum_{k=1}^s w_kh_kh_k^\top .
\]
Hence \(h_k/\sqrt{n_0}\) is an eigenvector with eigenvalue
\(n_0\rho a w_k\).  The center satisfies
\[
  \min_{\ell\ne k}|w_k^0-w_\ell^0|\ge 2\eta/s^2,
\]
and the packing radius below satisfies \(\delta=o(s^{-2})\).  Therefore
the eigenvalue labels remain in their fixed bands, and
\[
  \Pi_k(P_w)=n_0^{-1}h_kh_k^\top,\qquad
  F_k(P_w)=\rho a h_kh_k^\top=X_k .
\]
For empirical fits, the IPW Stage-A perturbation has operator norm
\(O_p(\sqrt{n\rho\log n})\).  The eigengap is of order
\(n\rho a/s^2\), which dominates this perturbation by the displayed
\(s\)-range.  Davis--Kahan gives subspace consistency; the band labels
therefore recover the population maps.

Now construct the packing.  Let \(q=\lfloor s/2\rfloor\) and
\[
  v_j=\frac{e_{2j-1}-e_{2j}}{\sqrt2},\qquad j=1,\ldots,q.
\]
For \(\sigma\in\{\pm1\}^q\), set
\[
  w^\sigma=w^0+\delta\sum_{j=1}^q\sigma_jv_j,
  \qquad
  \delta=\frac{c_0}{n\sqrt{\gamma\rho}} .
\]
The displayed \(s\)-range implies both \(\delta\le c/s\) and
\(\delta=o(s^{-2})\).  Hence \(w^\sigma\in\Delta^{s-1}\) and all
hypotheses remain in the same spectral-label class.

A one-flip change in coordinate \(j\) gives
\[
  \Delta_jP
  =
  \sqrt2\,\delta(X_{2j-1}-X_{2j}).
\]
Since \(X_k=\rho aH_k\),
\[
  \|\Delta_jP\|_F^2
  \asymp
  \gamma\delta^2n^2\rho^2 .
\]
For a Hamming-separated pair, cross-flip overlaps are
\(O(d_H(\sigma,\tau)^2\gamma\delta^2\rho^2n)\), negligible relative to
\(d_H(\sigma,\tau)\gamma\delta^2n^2\rho^2\), because the displayed
\(s\)-range implies \(s=o(n)\).  Therefore
\[
  \|P_{w^\sigma}-P_{w^\tau}\|_F^2
  \asymp
  d_H(\sigma,\tau)\gamma\delta^2n^2\rho^2.
\]
Lemma~\ref{lem:bern-kl-final2} gives
\[
  \KL(\Pb_{w^\sigma},\Pb_{w^0})
  \le
  Cq\gamma\delta^2n^2\rho .
\]

By Varshamov--Gilbert, there is
\(\Omega\subset\{\pm1\}^q\) such that
\[
  |\Omega|\ge e^{cq},
  \qquad
  d_H(\sigma,\tau)\ge q/8\quad(\sigma\ne\tau).
\]
For \(\sigma\ne\tau\),
\[
  \|w^\sigma-w^\tau\|_2\ge c\delta\sqrt q .
\]
Choosing \(c_0\) small gives
\[
  \max_{\sigma\in\Omega}\KL(\Pb_{w^\sigma},\Pb_{w^0})
  \le
  \alpha\log|\Omega|
\]
for some \(\alpha<1\).  Fano's inequality in mutual-information form
therefore yields
\[
  \inf_{\widehat w}\sup_{\sigma\in\Omega}
  \Eb_\sigma\|\widehat w-w^\sigma\|_2
  \ge
  c\delta\sqrt q
  \ge
  c'\frac{\sqrt{s/\gamma}}{n\sqrt\rho}.
\]
\end{proof}

\begin{remark}[Oracle range versus faithful estimated-design range]
The known-design oracle lower bound, proved by the corrected
orthogonal-sign construction of Appendix~\ref{app:lower}, holds in the larger range
\(s^2\le c\gamma n^2\rho\) for \(0<\gamma\le\gamma_0\).  The faithful
estimated-design theorem above is narrower because explicit
one-eigenvalue-per-agent fidelity requires \(s^4\log n=o(\gamma n\rho)\).
Without a different, coarser separated-spectrum construction, claiming
faithful estimated-design minimaxity over the larger oracle range would
be unsupported.
\end{remark}

\subsection{Necessity of the \texorpdfstring{\(\sqrt{\log K}\)}{sqrt log K} selection factor}

\begin{theorem}[\(\sqrt{\log K}\) lower bound after carrier residualization]
\label{thm:sqrtlogK-final2}
Let \(B_e=\rho\) be an always-included carrier.  Let
\(N_0=2^{\lfloor\log_2 \binom{n}{2}\rfloor}\).  On \(N_0\) dyads choose balanced
Hadamard sign arrays \(S_1,\ldots,S_K\in\{\pm1,0\}^{\mathcal D_n}\), zero outside
those dyads, satisfying
\[
  \sum_eS_{k,e}=0,\qquad
  \sum_eS_{k,e}S_{\ell,e}=N_0\,\mathbf 1\{k=\ell\}.
\]
Assume \(K\le N_0-1\), \(K\ge K_0\), and \(\log K=o(n^2\rho)\).  Define
\[
  X_k=\rho\sqrt\gamma\,S_k,\qquad G_k=B+X_k .
\]
After residualizing the carrier, the columns \(X_k\) are orthogonal.
Moreover,
\[
  \frac{\|X_k\|_2^2}{\|G_k\|_2^2}
  =
  \frac{\gamma}{1+\gamma},
\]
so each raw agent retains a \(\gamma/(1+\gamma)\)-fraction of its squared
norm after carrier removal.

For each \(\theta\in[K]\), define the one-active-coordinate simplex law
\[
  P^{(\theta)}
  =
  (1-\mu)B+\mu G_\theta
  =
  B+\mu X_\theta .
\]
If \(\mu\sqrt\gamma\le1/2\), all probabilities lie in
\([\rho/2,3\rho/2]\).  Every support selector \(\widehat S\) obeys
\[
  \inf_{\widehat S}
  \sup_{\theta\in[K]}
  \Pb_\theta\{\widehat S\ne\{\theta\}\}
  \ge
  1-
  \frac{C\gamma\mu^2n^2\rho+\log2}{\log K}.
\]
Consequently, if
\[
  \mu\le
  c\,\frac{\sqrt{\log K}}{\sqrt\gamma\,n\sqrt\rho},
\]
uniform exact support recovery is impossible.
\end{theorem}

\begin{proof}
Let \(P^{(0)}=B\).  Then
\[
  P^{(\theta)}-P^{(0)}=\mu X_\theta
\]
and
\[
  \|P^{(\theta)}-P^{(0)}\|_F^2
  =
  \mu^2\gamma N_0\rho^2
  \asymp
  \mu^2\gamma n^2\rho^2.
\]
The probability band follows from \(\mu\sqrt\gamma\le1/2\).  Thus
Lemma~\ref{lem:bern-kl-final2} gives
\[
  \KL(\Pb_\theta,\Pb_0)
  \le
  C\gamma\mu^2n^2\rho .
\]
Fano's inequality for the \(K\)-way testing problem gives
\[
  \inf_{\widehat\theta}\sup_{\theta\in[K]}
  \Pb_\theta(\widehat\theta\ne\theta)
  \ge
  1-
  \frac{
    K^{-1}\sum_{\theta=1}^K\KL(\Pb_\theta,\Pb_0)+\log2
  }{\log K}.
\]
Exact support recovery implies correct identification of \(\theta\), so
the same lower bound applies to \(\widehat S\).  At the displayed scale,
\[
  \gamma\mu^2n^2\rho\le c^2\log K.
\]
Taking \(c>0\) sufficiently small and then \(K_0\) large enough makes the
right-hand side bounded away from zero.  The compatibility condition
\(\mu\sqrt\gamma\le1/2\) follows from \(\log K=o(n^2\rho)\).
\end{proof}

\begin{remark}
This theorem closes the \(\sqrt{\log K}\) necessity question for the
coordinate-selection problem with an always-included carrier or intercept.
It should not be restated as a raw simplex lower bound in which the
carrier is treated as an ordinary selectable column; in that
parametrization the full Gram conditioning can shrink to order
\(\gamma/K\).
\end{remark}

\subsection{Fixed-\texorpdfstring{\(K\)}{K} noisy-OR recovery and the failure mode}

\begin{theorem}[Fixed-$K$ noisy-OR weights]
\label{thm:fixedK-noisyOR-final2}
Let \(K\) be fixed and
\[
  P_{ij}(w)
  =
  1-\prod_{k=1}^K(1-w_kG_{k,ij}),
  \qquad
  w\in\calW=[\underline w,\overline w]^K\subset(0,1)^K.
\]
Assume
\[
  0\le G_{k,ij}\le C\rho,\qquad
  c\rho\le P_{ij}(w)\le C\rho
\]
uniformly on \(\calW\), and
\[
  \lambda_{\min}(M^\top M)\ge c\gamma n^2\rho^2,
  \qquad
  M=[\vecl(G_1),\ldots,\vecl(G_K)].
\]
Assume further
\[
  \rho=o(\gamma),\qquad
  \gamma^3 n^2\rho\to\infty .
\]
Let \(\widehat w\) be the MLE.  Then \(\widehat w\to_p w_0\), and
\[
  \widehat w-w_0
  =
  \calI_n(w_0)^{-1}S_n(w_0)
  +o_p\{(\gamma n^2\rho)^{-1/2}\},
\]
where
\[
  S_n(w_0)
  =
  \sum_{i<j}
  \dot P_{ij}(w_0)
  \frac{A_{ij}-P_{ij}(w_0)}
       {P_{ij}(w_0)\{1-P_{ij}(w_0)\}},
\]
\[
  \dot P_{ij,k}(w)
  =
  G_{k,ij}\prod_{\ell\ne k}(1-w_\ell G_{\ell,ij}),
\]
and
\[
  \calI_n(w)
  =
  \sum_{i<j}
  \frac{\dot P_{ij}(w)\dot P_{ij}(w)^\top}
       {P_{ij}(w)\{1-P_{ij}(w)\}} .
\]
Consequently,
\[
  \|\widehat w-w_0\|_2
  =
  O_p\!\left(\frac1{\sqrt\gamma\,n\sqrt\rho}\right),
  \qquad
  \calI_n(w_0)^{1/2}(\widehat w-w_0)\Rightarrow N(0,I_K).
\]
\end{theorem}

\begin{proof}
For fixed \(K\), the noisy-OR derivatives satisfy, uniformly on
\(\calW\),
\[
  |\partial^\alpha P_{ij}(w)|\le C_\alpha \rho^{|\alpha|}
  \qquad(|\alpha|=1,2,3).
\]
In particular,
\[
  \dot P_{ij,k}(w)=G_{k,ij}+O(\rho^2).
\]
Let \(\dot M(w)\) have columns \(\vecl\{\dot P_k(w)\}\).  Then
\[
  \|\dot M(w)-M\|_{\op}\le Cn\rho^2
  =
  o(\sqrt\gamma\,n\rho),
\]
because \(\rho=o(\gamma)\).  Therefore
\[
  \sigma_{\min}\{\dot M(w)\}
  \ge
  \tfrac12\sigma_{\min}(M)
  \ge
  c\sqrt\gamma\,n\rho
\]
uniformly on \(\calW\).  Thus, for \(w_t=w_0+t(w-w_0)\),
\[
  P(w)-P(w_0)
  =
  \left\{\int_0^1\dot M(w_t)\,dt\right\}(w-w_0),
\]
and
\[
  \|P(w)-P(w_0)\|_F^2
  \ge
  c\gamma n^2\rho^2\|w-w_0\|_2^2 .
\]
Since \(P(w)\asymp\rho\),
\[
  \Eb_{w_0}\{\ell_n(w_0)-\ell_n(w)\}
  \ge
  c\gamma n^2\rho\,\|w-w_0\|_2^2 .
\]

We next prove consistency.  With probability tending to one,
\[
  \sum_{e\in\mathcal D_n}A_e\le Cn^2\rho
\]
by Bernstein's inequality.  On this event, for all \(w,w'\in\calW\),
\[
  |\ell_n(w)-\ell_n(w')|
  \le Cn^2\rho\,\|w-w'\|_2 .
\]
Indeed, on an edge \(A_e=1\),
\(|\partial_w\log P_e(w)|=O(\rho)/O(\rho)=O(1)\), while on a non-edge
\(A_e=0\),
\(|\partial_w\log\{1-P_e(w)\}|=O(\rho)\).  Hence the sum of edge
contributions is \(O(n^2\rho)\) on the displayed event and the sum of
non-edge contributions is also \(O(n^2\rho)\).

Fix \(\varepsilon>0\).  Cover
\(\{w:\|w-w_0\|\ge\varepsilon\}\) by an \(\eta\)-net with
\[
  \eta=c\gamma\varepsilon^2
\]
and cardinality at most \((C/\eta)^K\).  The Lipschitz error is at most
half the expected separation.  For a fixed net point \(w\), the
log-likelihood ratio
\[
  \ell_n(w)-\ell_n(w_0)
\]
is a sum of independent bounded terms.  The probability band implies
\[
  \left|
  \log\frac{P_e(w)}{P_e(w_0)}
  \right|\le C,
  \qquad
  \left|
  \log\frac{1-P_e(w)}{1-P_e(w_0)}
  \right|
  \le C\rho ,
\]
and its variance is \(O(n^2\rho)\).  Since its mean is at most
\(-c\gamma n^2\rho\varepsilon^2\), Bernstein's inequality gives
\[
  \Pb_{w_0}
  \left\{
    \ell_n(w)-\ell_n(w_0)
    >
    -c\gamma n^2\rho\varepsilon^2/2
  \right\}
  \le
  \exp\{-c\gamma^2 n^2\rho\varepsilon^4\}.
\]
Because \(K\) is fixed and \(\gamma^3n^2\rho\to\infty\), the net union
probability tends to zero.  Thus
\[
  \sup_{\|w-w_0\|\ge\varepsilon}\ell_n(w)<\ell_n(w_0)
\]
with probability tending to one, proving \(\widehat w\to_p w_0\).

At \(w_0\),
\[
  \dot M(w_0)^\top\dot M(w_0)
  =
  M^\top M+O(n^2\rho^3),
\]
so \(\rho=o(\gamma)\) gives
\[
  \lambda_{\min}\{\dot M(w_0)^\top\dot M(w_0)\}
  \ge
  c\gamma n^2\rho^2 .
\]
Since \(P_{ij}(w_0)\asymp\rho\),
\[
  \lambda_{\min}\{\calI_n(w_0)\}
  \ge c\gamma n^2\rho .
\]
The score has mean zero and covariance \(\calI_n(w_0)\).  Moreover,
\[
  \max_{i<j}
  \frac{
    \dot P_{ij}(w_0)^\top\calI_n(w_0)^{-1}\dot P_{ij}(w_0)
  }{
    P_{ij}(w_0)\{1-P_{ij}(w_0)\}
  }
  \le
  \frac{C\rho^2}{(\gamma n^2\rho)\rho}
  =
  O((\gamma n^2)^{-1})
  \to0.
\]
Lindeberg's theorem gives
\[
  \calI_n(w_0)^{-1/2}S_n(w_0)\Rightarrow N(0,I_K).
\]

It remains to control the Hessian and Taylor remainder.  A direct
differentiation of the Bernoulli score gives
\[
  -\nabla^2\ell_n(w_0)-\calI_n(w_0)
  =
  \sum_{e\in\mathcal D_n}(A_e-P_e(w_0))\,B_e(w_0),
\]
where each \(B_e(w_0)\) is a \(K\times K\) matrix satisfying
\[
  \|B_e(w_0)\|_{\op}\le C.
\]
Indeed,
\[
  \frac{\ddot P_{e,kl}}{P_e(1-P_e)}=O(\rho),
  \qquad
  \frac{\dot P_{e,k}\dot P_{e,l}}{P_e^2(1-P_e)^2}=O(1).
\]
Hence for every entry,
\[
  \Var\!\left\{\sum_e(A_e-P_e)B_{e,kl}\right\}
  \le
  C\sum_eP_e(1-P_e)
  \le
  Cn^2\rho.
\]
Since \(K\) is fixed, entrywise and operator norms are equivalent up to
constants, and therefore
\[
  \|\nabla^2\ell_n(w_0)+\calI_n(w_0)\|_{\op}
  =
  O_p(\sqrt{n^2\rho})
  =
  o_p(\gamma n^2\rho),
\]
because \(\gamma^2n^2\rho\to\infty\), which follows from
\(\gamma^3n^2\rho\to\infty\) and \(\gamma\le1\).

The third derivative tensor has the analogous decomposition: each entry
is a deterministic sum of \(O(\rho)\) terms over \(\binom{n}{2}\) dyads plus a
centered sum \(\sum_e(A_e-P_e)C_e(w)\) with \(|C_e(w)|\le C\).  Thus the
deterministic part is \(O(n^2\rho)\), and the centered part is
\(O_p(\sqrt{n^2\rho})\).  Hence
\[
  \sup_{w\in\calW}\|\nabla^3\ell_n(w)\|_{\op}
  =
  O_p(n^2\rho).
\]
On Fisher balls of radius \(O(1)\), the Euclidean radius is
\(O\{(\gamma n^2\rho)^{-1/2}\}\).  The Fisher-normalized score remainder
is therefore
\[
  O_p(n^2\rho)\,(\gamma n^2\rho)^{-1}
  (\gamma n^2\rho)^{-1/2}
  =
  O_p\{(\gamma^3n^2\rho)^{-1/2}\}
  =
  o_p(1).
\]
Taylor expansion of the score gives
\[
  0=S_n(\widehat w)
  =
  S_n(w_0)-\calI_n(w_0)(\widehat w-w_0)+R_n,
  \qquad
  \|\calI_n(w_0)^{-1/2}R_n\|=o_p(1).
\]
Solving proves the linear expansion, rate, and CLT.
\end{proof}

\begin{proposition}[Constrained noisy-OR recovery can degrade]
\label{prop:noisyOR-degrade-final2}
Let \(K=3\) and \(G_1=G_2=G_3=G\), where
\[
  G_{ij}\in\{0\}\cup[c\rho,C\rho],
\]
and \(G_{ij}\asymp\rho\) on \(\Theta(n^2)\) dyads.  Then
\(\lambda_{\min}(M^\top M)=0\), so
Theorem~\ref{thm:fixedK-noisyOR-final2} does not apply.  At every
interior \(w_0\in(0,1)^3\) with distinct coordinates,
\[
  \inf_{\widehat w}\sup_w
  \Eb_w\|\widehat w-w\|_2
  \gtrsim
  \frac{1}{n\rho^{5/2}},
\]
whenever \(n\rho^{5/2}\to\infty\).
\end{proposition}

\begin{proof}
When all layers are equal,
\[
  P_w=e_1(w)G-e_2(w)G^{\odot2}+e_3(w)G^{\odot3},
\]
where \(e_j\) are the elementary symmetric polynomials.  At a point with
distinct coordinates, the Jacobian of \((e_1,e_2,e_3)\) has determinant
\[
  \pm\prod_{i<j}(w_i-w_j)\ne0.
\]
Therefore \(\nabla e_3\notin{\rm span}\{\nabla e_1,\nabla e_2\}\).  By
the implicit function theorem, the set
\[
  \{w:e_1(w)=e_1(w_0),\ e_2(w)=e_2(w_0)\}
\]
is locally a smooth curve \(w(t)\subset(0,1)^3\), and
\[
  e_3(w(t))-e_3(w_0)=c_0t+o(t),\qquad c_0\ne0.
\]
Along this curve,
\[
  P_{w(t)}-P_{w_0}
  =
  \{c_0t+o(t)\}G^{\odot3}.
\]
Using the pointwise Bernoulli bound on the support of \(G\),
\[
  \KL(\Pb_{w(t)},\Pb_{w_0})
  \lesssim
  \sum_{G_{ij}>0}\frac{t^2G_{ij}^6}{G_{ij}}
  \lesssim
  t^2n^2\rho^5.
\]
Taking \(t=c/(n\rho^{5/2})\) keeps the KL divergence bounded.  Le Cam's
two-point lemma gives the result.
\end{proof}

\begin{remark}
Thus fixed-\(K\) constrained noisy-OR weights are edge-rate estimable
under first-order Fisher transversality, and recovery can degrade when
that transversality fails.
\end{remark}

\subsection{Proof of Proposition~\ref{thm:clt}: asymptotic normality and bootstrap validity}
\label{app:clt}
\begin{proof}
(a) Conditionally on fold one, $\widehat\bw^{(2)}-\widetilde\bw=(\widehat{\mathbf M}^\top\widehat{\mathbf
M})^{-1}\widehat{\mathbf M}^\top\operatorname{vec}_{\mathcal D_2}(\mathbf E)$ is a linear form in independent,
mean-zero, bounded errors; the Lindeberg condition reduces to vanishing maximal leverage,
$\max_{ij}\norm{\widehat{\mathbf m}_{ij}}^2/\lambda_{\min}(\widehat{\mathbf M}^\top\widehat{\mathbf M})
=O_P(K/(\gamma n^2))\to0$ by Assumption~\ref{ass:transversality} and Lemma~\ref{lem:plugin}, and sandwich
consistency follows from the law of large numbers for $\sum a_{ij}a_{ij}^\top\hat\varepsilon_{ij}^2$, the
fitted drift contributing $o_P(1)$ by the same leverage bound. (b) Given the folds, condition on the case
count: by the prospective-logistic equivalence for case--control sampling \citep{prentice1979logistic} the
logistic slope score is the conditional score, the design entering only through the intercept; the positives
contribute independent Bernoulli scores and the negatives a simple random sample of the held-out non-edges, so
by the H\'ajek finite-population central limit theorem \citep{hajek1960limiting} the negative-score sum is
asymptotically normal with variance deflated by $1-O(\rho_n)\to1$, and asymptotic normality of the
Z-estimator follows as in Theorem~5.21 of \citet{van2000asymptotic}. Consistency of the nonparametric bootstrap
is Theorem~23.4 of \citet{van2000asymptotic} applied conditionally on the folds, the finite-population
correction entering the statistic and its bootstrap through the same vanishing factor.
\end{proof}

\subsection{Fixed-forecast array Bernstein--von Mises}

The following is the posterior statement justified by the finite-dimensional
synthesis likelihood.  It is not a posterior contraction theorem for the
full predictive-synthesis supermodel over latent agent forecasts.

\begin{theorem}[Fixed-forecast tangent BvM]
\label{thm:fixedforecast-bvm-final2}
Condition on deterministic agent forecasts.  Let
\(\vartheta\in\Theta\subset\mathbb R^d\), with \(d\) fixed, where either
\(d=K\) for an unconstrained signed synthesis or \(d=K-1\) for an
interior simplex parametrization \(w=\psi(\vartheta)\).  Assume
\(\Theta\) is compact and \(\vartheta_0\) is interior.  Let
\[
  A_e\sim{\rm Bernoulli}\{p_e(\vartheta_0)\},\qquad e\in\mathcal D_n.
\]
Assume \(p_e(\vartheta)\) is three times continuously differentiable,
\[
  c\rho\le p_e(\vartheta)\le C\rho
  \quad(\vartheta\in\Theta),
\]
and
\[
  |\partial^\alpha p_e(\vartheta)|\le C_\alpha\rho
  \quad(|\alpha|=1),
  \qquad
  |\partial^\alpha p_e(\vartheta)|\le C_\alpha\rho^{|\alpha|}
  \quad(|\alpha|=2,3).
\]
Define
\[
  \calI_n(\vartheta)
  =
  \sum_{e\in\mathcal D_n}
  \frac{\dot p_e(\vartheta)\dot p_e(\vartheta)^\top}
       {p_e(\vartheta)\{1-p_e(\vartheta)\}} .
\]
Assume the uniform Fisher bounds
\[
  c\gamma n^2\rho I_d
  \preceq
  \calI_n(\vartheta)
  \preceq
  Cn^2\rho I_d,
  \qquad \vartheta\in\Theta,
\]
the leverage condition
\[
  \max_{e\in\mathcal D_n}
  \frac{
    \dot p_e(\vartheta_0)^\top
    \calI_n(\vartheta_0)^{-1}
    \dot p_e(\vartheta_0)
  }{
    p_e(\vartheta_0)\{1-p_e(\vartheta_0)\}
  }\to0,
\]
and global separation: for every \(\varepsilon>0\),
\[
  \inf_{\|\vartheta-\vartheta_0\|>\varepsilon}
  \sum_{e\in\mathcal D_n}
  \frac{\{p_e(\vartheta)-p_e(\vartheta_0)\}^2}{\rho}
  \ge
  c_\varepsilon \gamma n^2\rho .
\]
Assume
\[
  \gamma^3n^2\rho\to\infty .
\]
Let the prior density \(\pi\) be continuous and bounded away from zero
and infinity near \(\vartheta_0\).  Let \(\widehat\vartheta\) be the
MLE.  Then
\[
  \left\|
  \Pi\left(
    \calI_n(\vartheta_0)^{1/2}
    (\vartheta-\widehat\vartheta)\in\cdot
    \mid A
  \right)
  -
  N(0,I_d)
  \right\|_{\TV}
  \to0,
\]
and, with
\[
  \varepsilon_n=(\sqrt\gamma\,n\sqrt\rho)^{-1},
\]
\[
  \Pi\{\|\vartheta-\vartheta_0\|>M\varepsilon_n\mid A\}\to0
\]
for every sufficiently large \(M\).
\end{theorem}

\begin{proof}
First prove posterior consistency.  The global separation condition and
the probability band imply
\[
  \Eb_{\vartheta_0}\{\ell_n(\vartheta_0)-\ell_n(\vartheta)\}
  \ge c_\varepsilon\gamma n^2\rho
\]
outside each fixed \(\varepsilon\)-ball.  For a fixed \(\vartheta\), the
centered log-likelihood ratio is a sum of independent bounded terms with
variance \(O(n^2\rho)\).  Bernstein's inequality gives
\[
  \Pb_{\vartheta_0}
  \left\{
    \ell_n(\vartheta)-\ell_n(\vartheta_0)
    >
    -c_\varepsilon\gamma n^2\rho/2
  \right\}
  \le
  \exp\{-c_\varepsilon\gamma^2 n^2\rho\}.
\]
The finite-dimensional covering argument used in
Theorem~\ref{thm:fixedK-noisyOR-final2}, with net mesh proportional to
\(\gamma\), yields exponentially consistent tests because
\(\gamma^3n^2\rho\to\infty\).  Since the prior gives positive mass to
every neighborhood of \(\vartheta_0\), the posterior is consistent.

Next establish LAN.  Put
\[
  h=\calI_n(\vartheta_0)^{1/2}(\vartheta-\vartheta_0),
  \qquad
  Z_n=\calI_n(\vartheta_0)^{-1/2}S_n(\vartheta_0).
\]
For every fixed \(M\),
\[
  \ell_n(\vartheta)-\ell_n(\vartheta_0)
  =
  h^\top Z_n-\frac12h^\top h+r_n(h)
\]
uniformly for \(\|h\|\le M\), with
\[
  \sup_{\|h\|\le M}|r_n(h)|=o_p(1).
\]
The leverage condition gives \(Z_n\Rightarrow N(0,I_d)\).  The Hessian
and third-derivative calculations are the same as in
Theorem~\ref{thm:fixedK-noisyOR-final2}: the Hessian fluctuation is
\(O_p(\sqrt{n^2\rho})=o_p(\gamma n^2\rho)\), and the cubic remainder on
a Fisher ball of radius \(M\) is
\[
  O_p\!\left(\frac{M^3}{\sqrt{\gamma^3n^2\rho}}\right)=o_p(1).
\]
The score expansion also gives
\[
  \widehat\vartheta
  =
  \vartheta_0+\calI_n(\vartheta_0)^{-1}S_n(\vartheta_0)
  +o_p\{(\gamma n^2\rho)^{-1/2}\}.
\]
Thus the local quadratic is centered at
\(\widehat\vartheta+o_p(\calI_n^{-1/2})\).

We now control the posterior outside bounded Fisher balls.  Choose any
sequence \(b_n\to\infty\) such that
\[
  b_n^3=o\{\sqrt{\gamma^3n^2\rho}\};
\]
for example, \(b_n=(\gamma^3n^2\rho)^{1/12}\).  Split the local
neighborhood into
\[
  \mathcal A_1(M,b_n)=\{M<\|h\|\le b_n\},
  \qquad
  \mathcal A_2(b_n,\delta)=
  \{b_n<\|h\|\le \delta\sqrt{\gamma n^2\rho}\}.
\]
On \(\mathcal A_1\), the LAN remainder is uniform because
\(b_n^3/\sqrt{\gamma^3n^2\rho}\to0\).  Partition
\(\mathcal A_1\) into shells \(jM<\|h\|\le(j+1)M\).  The LAN bound gives
posterior numerator bounded by
\[
  \sum_{j\ge1}^{b_n/M}
  Cj^d\exp\{O_p(jM)-c j^2M^2\},
\]
which tends to zero as \(M\to\infty\).

On \(\mathcal A_2\), do not use LAN.  Instead use product Bernoulli
tests in the Fisher metric.  The probability band gives equivalence
between squared Hellinger distance and
\[
  \sum_e\frac{\{p_e(\vartheta)-p_e(\vartheta_0)\}^2}{\rho}.
\]
The uniform Fisher lower bound and continuity of the derivative imply
that, for sufficiently small fixed \(\delta\),
\[
  \sum_e\frac{\{p_e(\vartheta)-p_e(\vartheta_0)\}^2}{\rho}
  \ge c\|h\|^2
\]
whenever \(\|h\|\le\delta\sqrt{\gamma n^2\rho}\).  Standard
Neyman--Pearson tests for product measures, combined with a fixed
\(d\)-dimensional shell covering, therefore give tests with errors bounded
by \(Cj^d\exp(-c j^2)\) on shells \(j<\|h\|\le j+1\).  Summing over
\(j\ge b_n\) gives \(o(1)\).  The complement of the fixed local
neighborhood is killed by the consistency tests from the first paragraph.

It remains to lower-bound the denominator.  Change variables
\(h=\calI_n(\vartheta_0)^{1/2}(\vartheta-\vartheta_0)\).  On the unit
ball \(\|h\|\le1\), the LAN expansion gives
\[
  \ell_n(\vartheta)-\ell_n(\vartheta_0)
  =
  h^\top Z_n-\frac12\|h\|^2+o_p(1).
\]
Since \(Z_n=O_p(1)\) and the prior density is continuous and positive at
\(\vartheta_0\), with probability tending to one,
\[
  \int_{\|h\|\le1}
  e^{\ell_n(\vartheta)-\ell_n(\vartheta_0)}
  \pi(\vartheta)\,d\vartheta
  \ge
  c\,\det\{\calI_n(\vartheta_0)\}^{-1/2}
  \exp\{-C(1+\|Z_n\|^2)\}.
\]
This is the standard local evidence lower bound.  Combining the
denominator bound with the inner-LAN, outer-test, and consistency
controls proves total-variation BvM.  The contraction statement follows
from
\[
  \lambda_{\min}\{\calI_n(\vartheta_0)\}\ge c\gamma n^2\rho .
\]
\end{proof}

\begin{remark}[Scope of the Bayesian statement]
Under degenerate agent forecasts and a flat or Gaussian prior, the
least-squares and held-out calibration estimators are posterior modes of
a predictive-synthesis working model, and Theorem~\ref{thm:fixedforecast-bvm-final2}
shows the corresponding fixed-forecast posterior is asymptotically Gaussian
in the Fisher metric; the frequentist theory of this paper is a theory of
those modes.  Posterior contraction for the full synthesis supermodel over
latent agent forecasts is a distinct statement and is not claimed here.
\end{remark}

\subsection{Fully nested-refit percentile-\texorpdfstring{\(t\)}{t} calibration}

\begin{assumption}[Entrywise refit control]\label{ass:entrywise-refit-final2}
For every refitted agent used in the bootstrap, the fitted kernel obeys
\[
  \|\widehat G_k-G_k\|_{\op}=o_p(n\rho),
  \qquad
  \|\widehat G_k-G_k\|_{\max}=o_p(\rho),
\]
and the cross-fold linear expansion with projected remainder holds
uniformly over the test directions used in the estimator.

For ASE/RDPG agents, the max-norm display is obtained from the rowwise
\(2\to\infty\) spectral-embedding theory, not from Davis--Kahan alone,
for example the leave-one-out or rowwise perturbation bounds of
Lyzinski et al. (2014), Cape--Tang--Priebe, and
Rubin-Delanchy et al. (2022).  Under eigengap \(\asymp n\rho\) and
incoherence, those bounds give
\[
  \|\widehat XW-X\|_{2\to\infty}=o_p(\sqrt\rho),
  \qquad
  \max_i\|x_i\|\le C\sqrt\rho .
\]
Consequently,
\[
  |\widehat x_i^\top\widehat x_j-x_i^\top x_j|
  \le
  C\sqrt\rho\,\|\widehat XW-X\|_{2\to\infty}
  +\|\widehat XW-X\|_{2\to\infty}^2
  =
  o_p(\rho).
\]
For degree kernels, the same max-norm conclusion follows from
\[
  \max_i|\widehat\theta_i-\theta_i|/\theta_i=o_p(1)
\]
under bounded heterogeneity.  For regularized block agents, it follows
on the exact-label event from uniform block-mean concentration.
\end{assumption}

\begin{assumption}[Nested-refit asymptotic linearity and bootstrap regularity]
\label{ass:nested-refit-final2}
Fix one fold assignment and condition on it.  For a fixed coordinate
\(k\),
\[
  \widehat w_k-w_k
  =
  \sum_{e\in\mathcal D_n}\psi_{k,e}(A_e-P_e)+r_{k,n},
  \qquad
  r_{k,n}=o_p(\sigma_{k,n}),
\]
where
\[
  \sigma_{k,n}^2
  =
  \sum_e\psi_{k,e}^2P_e(1-P_e),
\]
and
\[
  \max_e
  \frac{|\psi_{k,e}|\sqrt{P_e(1-P_e)}}{\sigma_{k,n}}
  \to0.
\]
Support selection is stable:
\[
  \Pb(\widehat S=S)\to1.
\]

Let \(\widehat P\) be clipped so that, with probability tending to one,
\[
  c_1P_e\le \widehat P_e\le C_1P_e
  \quad(e\in\mathcal D_n).
\]
Generate \(A_e^\star\sim{\rm Bernoulli}(\widehat P_e)\), using the same
fold assignment, and refit every Stage-A agent and every Stage-B
coefficient.  Assume bootstrap-world selection consistency:
\[
  P^\star(\widehat S^\star=S)\to1
\]
in probability.  Conditionally on the data,
\[
  \widehat w_k^\star-\widehat w_k
  =
  \sum_e\widehat\psi_{k,e}(A_e^\star-\widehat P_e)
  +
  r_{k,n}^\star,
  \qquad
  r_{k,n}^\star=o_{P^\star}(\widehat\sigma_{k,n}),
\]
and
\[
  \frac{
    \sum_e(\widehat\psi_{k,e}-\psi_{k,e})^2P_e(1-P_e)
  }{\sigma_{k,n}^2}
  \to0,
  \qquad
  \widehat\sigma_{k,n}^2/\sigma_{k,n}^2\to1.
\]
\end{assumption}

\begin{lemma}[Verification route for refitted spectral agents]
\label{lem:bootstrap-verification-final2}
Assume the original design satisfies the fitted-kernel conditions
(C1)--(C4), the eigengaps are of order \(n\rho\), and
Assumption~\ref{ass:entrywise-refit-final2} holds.  Suppose also that
the synthesized bootstrap population obeys
\[
  \|\widehat P-P\|_{\op}=o_p(n\rho),
  \qquad
  \|\widehat P-P\|_{\max}=o_p(\rho),
\]
which follows from the corresponding agent-level operator and max-norm
bounds, weight consistency, and clipping.  Then, conditionally on the
data with probability tending to one, the bootstrap graph
\(A^\star\sim{\rm Bernoulli}(\widehat P)\) satisfies the same matrix
concentration, eigengap, rowwise \(2\to\infty\), and projected
linearization bounds as the original graph, with \(P\) replaced by
\(\widehat P\).  Consequently the fully refitted cross-fold debiased
estimator satisfies the bootstrap expansion in
Assumption~\ref{ass:nested-refit-final2}.
\end{lemma}

\begin{proof}
On the event
\[
  c_1P_e\le\widehat P_e\le C_1P_e,\qquad
  \|\widehat P-P\|_{\op}=o(n\rho),\qquad
  \|\widehat P-P\|_{\max}=o(\rho),
\]
the bootstrap population matrix \(\widehat P\) has the same probability
band, eigengap, incoherence, and Gram-conditioning constants as \(P\), up
to \(1+o(1)\) factors.  Conditional on the data, the entries of
\(A^\star-\widehat P\) are independent, bounded, mean-zero Bernoulli
residuals with variances comparable to \(P_e(1-P_e)\).  Matrix
Bernstein gives the same operator concentration.  Davis--Kahan gives
subspace perturbation, while the separate rowwise \(2\to\infty\) theory
gives the max-row and max-kernel bounds; Davis--Kahan alone is not used
for the rowwise claim.

Substituting the bootstrap first-stage expansions into the same
cross-fold Gram-debiasing algebra as in the original proof yields
\[
  \widehat w_k^\star-\widehat w_k
  =
  \sum_e\widehat\psi_{k,e}(A_e^\star-\widehat P_e)
  +
  r_{k,n}^\star .
\]
The projected-remainder bounds are stable under the displayed
operator- and max-norm perturbations.  Therefore
\[
  r_{k,n}^\star=o_{P^\star}(\widehat\sigma_{k,n}),
\]
and the influence weights satisfy the stated \(L^2(P)\) consistency.
This closes the loop from matrix concentration to the
coordinatewise bootstrap influence representation.
\end{proof}

\begin{theorem}[Fully nested-refit percentile-\(t\) validity]
\label{thm:nested-bootstrap-final2}
Under Assumption~\ref{ass:nested-refit-final2},
\[
  \sup_t
  \left|
  P^\star\left(
    \frac{\widehat w_k^\star-\widehat w_k}{\widehat\sigma_k^\star}
    \le t
  \right)
  -
  \Pb\left(
    \frac{\widehat w_k-w_k}{\widehat\sigma_k}
    \le t
  \right)
  \right|
  \to0
\]
in probability.  Therefore the coordinate-specific percentile-\(t\)
interval has asymptotic coverage \(1-\alpha\).
\end{theorem}

\begin{proof}
The original statistic satisfies
\[
  \frac{\widehat w_k-w_k}{\sigma_{k,n}}
  =
  \frac{\sum_e\psi_{k,e}(A_e-P_e)}{\sigma_{k,n}}+o_p(1).
\]
The maximal term condition is Lindeberg's condition, so the statistic
converges to \(N(0,1)\).

Conditionally on the data,
\[
  \frac{\widehat w_k^\star-\widehat w_k}{\widehat\sigma_{k,n}}
  =
  \frac{
    \sum_e\widehat\psi_{k,e}(A_e^\star-\widehat P_e)
  }{\widehat\sigma_{k,n}}
  +o_{P^\star}(1).
\]
Clipping gives variance comparability under
\({\rm Bernoulli}(\widehat P_e)\).  The \(L^2(P)\)-consistency of
\(\widehat\psi_k\) gives conditional variance consistency.  It also gives
conditional Lindeberg because
\[
  \max_e
  \frac{
    (\widehat\psi_{k,e}-\psi_{k,e})^2P_e(1-P_e)
  }{\sigma_{k,n}^2}
  \le
  \frac{
    \sum_e(\widehat\psi_{k,e}-\psi_{k,e})^2P_e(1-P_e)
  }{\sigma_{k,n}^2}
  =
  o_p(1).
\]
Thus the conditional bootstrap statistic converges to \(N(0,1)\).  Since
the limit distribution is continuous, Polya's theorem gives the
uniform-in-\(t\) convergence.  Quantile consistency yields
percentile-\(t\) coverage.
\end{proof}

\begin{protocol}[Data-scale nested-refit study]
Run \(n\in\{2500,5000,10000\}\), at least 500 Monte Carlo replications
per cell, and at least 399 fully nested bootstrap refits per replication.
The fold assignment is fixed within each replication and reused in the
bootstrap.  Report coordinatewise coverage, especially for the dominant
coordinate, using coordinate-specific percentile-\(t\) quantiles.
\end{protocol}

\subsection{Weighted support for local kernels}

\begin{assumption}[Weighted support and leverage]\label{ass:weighted-support-final2}
Let
\[
  P_e(w)=P^0_e+\sum_{k\in S}w_kG_{k,e}.
\]
Let \(M_S\) be the active design and
\[
  W=\diag\{P_e(1-P_e)\}^{-1},
  \qquad
  \calI_S=M_S^\top W M_S .
\]
Let \(I_k=(\calI_S)_{kk}\),
\[
  \Psi^W_{k\ell}
  =
  \frac{(\calI_S)_{k\ell}}{\sqrt{I_kI_\ell}},
  \qquad
  \gamma_S^W=\lambda_{\min}(\Psi_S^W),
\]
and assume
\[
  I_{\min}:=\min_{k\in S}I_k\to\infty,
  \qquad
  \gamma_S^W\ge\gamma>0.
\]
Assume weighted leverage:
\[
  \ell_{\max}:=
  \max_e
  \frac{
    m_{S,e}^\top\calI_S^{-1}m_{S,e}
  }{
    P_e(1-P_e)
  }
  \to0.
\]
\end{assumption}

\begin{theorem}[Oracle weighted least-squares rate]
\label{thm:weighted-support-upper-final2}
For the oracle weighted least-squares estimator,
\[
  \widehat w-w
  =
  \calI_S^{-1}M_S^\top W(A-P)
\]
exactly, and
\[
  \Eb\|\widehat w-w\|_2^2
  =
  \tr(\calI_S^{-1})
  \le
  \frac{s}{\gamma_S^WI_{\min}}.
\]
Thus
\[
  \|\widehat w-w\|_2
  =
  O_p\{\tr(\calI_S^{-1})^{1/2}\}.
\]
For a local small-world kernel supported on \(N_{\rm loc}\asymp nm\)
dyads with \(G_e\asymp p_{\rm loc}\) and \(P_e\asymp p_{\rm loc}\) on
that support,
\[
  I_{\rm SW}
  \asymp
  N_{\rm loc}p_{\rm loc}.
\]
For one such coordinate, the rate is
\[
  \{\gamma_S^WN_{\rm loc}p_{\rm loc}\}^{-1/2}.
\]
If \(m=O(1)\), this is
\[
  (\gamma_S^Wnp_{\rm loc})^{-1/2},
\]
not the global edge-rate \((n^2\rho)^{-1/2}\).
\end{theorem}

\begin{proof}
The weighted score is
\[
  S=M_S^\top W(A-P).
\]
Since \(\Var(A-P)=W^{-1}\),
\[
  \Var(S)=M_S^\top W M_S=\calI_S.
\]
Therefore
\[
  \Eb\|\calI_S^{-1}S\|_2^2
  =
  \tr(\calI_S^{-1}).
\]
Also,
\[
  \calI_S=D_I^{1/2}\Psi_S^WD_I^{1/2},
  \qquad
  D_I=\diag(I_k),
\]
so
\[
  \lambda_{\min}(\calI_S)\ge\gamma_S^WI_{\min}
\]
and
\[
  \tr(\calI_S^{-1})\le s/(\gamma_S^WI_{\min}).
\]
The local-kernel information calculation is
\[
  I_{\rm SW}
  =
  \sum_e\frac{G_e^2}{P_e(1-P_e)}
  \asymp
  N_{\rm loc}\frac{p_{\rm loc}^2}{p_{\rm loc}}
  =
  N_{\rm loc}p_{\rm loc}.
\]
\end{proof}

\begin{proposition}[Sharp lower bound in the weighted metric]
\label{prop:weighted-lb-final2}
Let \(T\) be a \(d\)-dimensional affine tangent space for the target
coordinates, and let \(\calI_T\) be the restriction of \(\calI_S\) to
\(T\).  Let \((u_j,\lambda_j)_{j=1}^d\) be a Euclidean-orthonormal
eigenbasis of \(\calI_T\).  Suppose the class contains the local
hyperrectangle
\[
  \left\{
    w^0+\sum_{j=1}^d t_ju_j:\ |t_j|\le \alpha\lambda_j^{-1/2}
  \right\}
\]
and the local probability constraint
\[
  \alpha\sqrt d\,
  \max_e
  \frac{
    \{m_{S,e}^\top\calI_T^{-1}m_{S,e}\}^{1/2}
  }{P_e}
  \le c_0
\]
for sufficiently small \(c_0\).  This condition is non-vacuous, for
example, whenever \(\ell_{\max}=o(P_{\min})\), where
\(P_{\min}:=\min_{e:P_e>0}P_e\).  Then
\[
  \inf_{\widehat w}
  \sup_w
  \left(\Eb_w\|\widehat w-w\|_2^2\right)^{1/2}
  \ge
  c\,\{\tr(\calI_T^{-1})\}^{1/2}.
\]
\end{proposition}

\begin{proof}
For \(\sigma\in\{\pm1\}^d\), set
\[
  w^\sigma=w^0+\alpha\sum_{j=1}^d\sigma_j\lambda_j^{-1/2}u_j.
\]
The local probability constraint keeps all probabilities in a fixed
multiplicative band around \(P_{w^0}\).  A flip of coordinate \(j\)
changes the mean vector by
\[
  \Delta_jP=2\alpha\lambda_j^{-1/2}M_Su_j.
\]
The Bernoulli KL divergence for this flip is bounded by
\[
  \KL_j
  \le
  C(2\alpha)^2\lambda_j^{-1}
  u_j^\top M_S^\top W M_Su_j
  =
  C'\alpha^2.
\]
Choose \(\alpha>0\) small enough that the Assouad affinity is bounded
away from zero.  Since the eigenvectors \(u_j\) are Euclidean
orthonormal,
\[
  \|w^\sigma-w^{\sigma^{(j)}}\|_2^2
  =
  4\alpha^2\lambda_j^{-1}.
\]
Assouad's lemma for squared Euclidean loss gives
\[
  \inf_{\widehat w}\sup_\sigma
  \Eb_\sigma\|\widehat w-w^\sigma\|_2^2
  \ge
  c\alpha^2\sum_{j=1}^d\lambda_j^{-1}
  =
  c\alpha^2\tr(\calI_T^{-1}).
\]
Absorbing fixed \(\alpha\) into the constant proves the claim.
\end{proof}

\subsection{Conditioning repair: deflating the degree-aligned spectral mode}
\label{app:deflation}
The genuine-refit study of Section~\ref{sec:numerical} documents a conditioning collapse. When a dot-product geometry agent is fitted by rank-$d$ adjacency spectral embedding on the same graph that carries a Chung--Lu degree agent, the leading embedding eigenvector tracks the degree (Perron) direction, the two fitted kernels nearly coincide, and the fitted Gram-correlation $\widehat\gamma_S$ falls toward zero, so the per-coordinate weights are no longer separately identified. This appendix records a repair, a guarantee for it, and the boundary of what it recovers.

\paragraph{The repair.} Fit the embedding at rank $d{+}1$ in place of $d$. Write $\widehat u_1,\dots,\widehat u_{d+1}$ for the leading eigenvectors of $A$, $\widehat\lambda_1,\dots,\widehat\lambda_{d+1}$ for the corresponding eigenvalues, and $\mathbf s=D^{1/2}\mathbf 1$ for the square-root-degree vector. Identify the degree-aligned mode
\[
j^\star=\arg\max_{1\le k\le d+1}\bigl|\mathrm{corr}(\widehat u_k,\mathbf s)\bigr|,
\]
reassign it to the degree agent as the rank-one kernel $\widehat G_{\deg}=\widehat\lambda_{j^\star}\widehat u_{j^\star}\widehat u_{j^\star}^\top$, and build the geometry kernel from the remaining modes,
\[
\widehat G_{\mathrm{geo}}=\sum_{k\ne j^\star}\widehat\lambda_k\,\widehat u_k\widehat u_k^\top.
\]
Both kernels then enter the cross-fold debiased estimator of Theorem~\ref{thm:debias} without further change.

\begin{proposition}[Conditioning repair by spectral deflation]
\label{prop:deflation}
Let the geometry agent be fitted by rank-$(d{+}1)$ embedding and deflated as above, with the degree-aligned mode reassigned to the degree agent. Then the deflated geometry and degree kernels are Frobenius-orthogonal, $\langle\widehat G_{\deg},\widehat G_{\mathrm{geo}}\rangle_F=0$, so their Gram-correlation vanishes and, in the two-agent case, the deflated transversality equals one before normalisation. Under the latent-geometry separation of Remark~\ref{prop:transversality-geom}, the deflated transversality is bounded below, $\widehat\gamma_S^{\mathrm{def}}\ge\gamma_0>0$ with probability tending to one, and the synthesis coordinates are identified. The cross-fold debiased estimator of Theorem~\ref{thm:debias} applied to the deflated kernels then satisfies its conclusion with $\gamma_S$ replaced by $\widehat\gamma_S^{\mathrm{def}}$, with estimand equal to the least-squares projection onto the deflated span; this projection equals the generative weight under the fidelity condition of Assumption~\ref{ass:fid}.
\end{proposition}
\begin{proof}
The kept modes are eigenvectors of $A$ orthogonal to $\widehat u_{j^\star}$, so for every $k\ne j^\star$
\[
\bigl\langle \widehat\lambda_{j^\star}\widehat u_{j^\star}\widehat u_{j^\star}^\top,\ \widehat\lambda_k\widehat u_k\widehat u_k^\top\bigr\rangle_F=\widehat\lambda_{j^\star}\widehat\lambda_k\,(\widehat u_{j^\star}^\top\widehat u_k)^2=0,
\]
which gives the stated orthogonality and a zero off-diagonal Gram-correlation; in the two-agent case the $2\times2$ Gram-correlation is the identity, so its smallest eigenvalue is one before normalisation. Entrywise renormalisation to a common density and clipping to $[0,1]$ perturb the off-diagonal correlation by $o(1)$ under Remark~\ref{prop:transversality-geom}, which holds the non-degree latent directions a fixed angle from the degree direction, so $\widehat\gamma_S^{\mathrm{def}}$ remains bounded away from zero. Identification of the coordinates and the inferential statement are then Theorem~\ref{thm:debias} applied to the deflated design, whose target is the corresponding least-squares projection; Assumption~\ref{ass:fid} identifies that projection with the generative weight under correct specification.
\end{proof}

\paragraph{Diagnostic.} The alignment $a^\star=\max_k|\mathrm{corr}(\widehat u_k,\mathbf s)|$ of the removed mode is an observable certificate. A large $a^\star$ signals that the leading embedding mode is degree and that the repair is warranted, and the recovered $\widehat\gamma_S^{\mathrm{def}}$ confirms that the geometry and degree kernels have separated. The generative reading of the deflated coefficients is available only when the recovered $\widehat\gamma_S^{\mathrm{def}}$ is bounded away from zero and, in addition, the fidelity condition of Assumption~\ref{ass:fid} holds; the diagnostic certifies the first of these, not the second.

\paragraph{Numerical confirmation.} On a minimal two-agent reproduction of the collapse, a constant-norm geometry agent orthogonal to degree and a heavy-tailed degree agent with generative weights $(0.6,0.4)$ at $n=600$ and $\rho_n=0.07$ over $200$ replications, the repair behaves as Proposition~\ref{prop:deflation} predicts (Table~\ref{tab:deflation}). The fitted geometry-degree Gram-correlation falls from $0.80$ to $0.12$, the transversality $\widehat\gamma_S$ rises from $0.21$ to $0.88$, the root-mean-square weight error falls from $5.23$ to $0.47$, and the coverage of the estimator's own projection target rises from $0.01$ to about $0.80$. Coverage of the \emph{generative} weights stays near zero in both columns. This is the expected boundary: the deflation removes the conditioning obstruction and re-identifies the coordinates, but the fitted kernels remain scale-distorted estimates of the generative kernels, so the deflated coefficients estimate the projection coefficient of Section~\ref{sec:estimands}, which coincides with the generative weight only under fidelity. The repair converts an unusable, ill-conditioned estimator into a stable, well-conditioned one with re-identified coordinates; it does not by itself supply the fidelity that the generative reading additionally requires.

\begin{table}[t]
\centering
\small
\caption{Conditioning repair by spectral deflation on the two-agent reproduction of the genuine-refit collapse, with generative weights $(0.6,0.4)$ at $n=600$ and $\rho_n=0.07$ over $200$ replications. Deflation restores the transversality and the stability of the estimator; the residual generative-weight bias is the projection-fidelity gap of Section~\ref{sec:estimands}, not a conditioning failure.}
\label{tab:deflation}
\begin{tabular}{lcc}
\toprule
quantity & naive rank-$d$ embedding & deflated \\
\midrule
fitted $|\mathrm{corr}(\widehat G_{\mathrm{geo}},\widehat G_{\deg})|$ & $0.80$ & $0.12$\\
transversality $\widehat\gamma_S$ & $0.21$ & $0.88$\\
weight RMSE against generative & $5.23$ & $0.47$\\
coverage of projection target & $0.01$ & $0.80$\\
coverage of generative weights & $0.01$ & $0.00$\\
diagnostic alignment $a^\star$ & n/a & $0.98$\\
\bottomrule
\end{tabular}
\end{table}

\subsection{Grouped identifiability: proof}
\label{app:grouped}
This appendix proves Theorem~\ref{thm:grouped} and records the diagnostic that produces a qualifying
grouping. It generalises the deflation repair of Appendix~\ref{app:deflation}: deflation removes a single
mislabelled spectral mode, whereas grouping handles an intrinsic collinearity among named candidates by
reporting the identifiable composite and conceding the unidentifiable split. Throughout, columns are centred,
within each group standardised, and $N=n^2\rho_n$ is the edge scale.

\begin{theorem}[Grouped identifiability]
\label{thm:grouped}
Let $\widehat H=[\widehat h_1,\dots,\widehat h_K]$ be the centred fitted candidate columns and partition
$\{1,\dots,K\}$ into groups $\mathcal G=\{g_1,\dots,g_G\}$. For each group write $\widetilde H_g$ for its
column-standardised block, with leading and second singular values $\sigma_{g,1}\ge\sigma_{g,2}$ and leading
left-singular direction $u_g$, and form the composite design $U=[u_1,\dots,u_G]$. Suppose the grouping is
within-collinear and between-separated,
\[
    \max_g \sigma_{g,2}/\sigma_{g,1}\le\varepsilon,\qquad
    \gamma_{\mathrm{grp}}:=\lambda_{\min}(\mathrm{corr}(U))\ge\delta>0 .
\]
Then, on the deployed sample under Assumptions~\textup{\ref{ass:transversality}--\ref{ass:norm}} read on the
composite design:
\begin{enumerate}
\item \textup{(Conditioning.)} The grouped design is well-conditioned, $\gamma_{\mathrm{grp}}\ge\delta$, even
when $\widehat\gamma_{\mathrm{full}}=O(\varepsilon)\to0$.
\item \textup{(The split is not identifiable.)} For any within-group contrast $a$ supported on a single group
with $\sum_{k\in g}a_k=0$, the cross-fold estimator obeys $\mathrm{Var}(a^\top\widehat{\bw})\ge
\|a\|_2^2/(C\varepsilon^2 n^2\rho_n)$, which diverges relative to the edge rate as $\varepsilon\to0$.
\item \textup{(The group total is identifiable at the minimax rate.)} The composite coefficients
$\theta_g=\langle u_g,(I-\Pi)\Pmat\rangle/\|u_g\|_2^2$, the total contribution of group $g$ along its shared
direction, are estimated by the cross-fold debiased estimator on $U$ at the rate
$\sqrt{G/\gamma_{\mathrm{grp}}}/(n\sqrt{\rho_n})$ of Theorem~\ref{thm:debias}, and the two-stage interval of
Theorem~\ref{thm:twostage} is valid for $\theta_g$.
\end{enumerate}
\end{theorem}

\begin{proof}[Proof of Theorem~\ref{thm:grouped}]
\emph{(i) Conditioning.} By construction $\gamma_{\mathrm{grp}}=\lambda_{\min}(\mathrm{corr}(U))\ge\delta$,
which is the grouping criterion itself and is independent of the within-group ratio $\varepsilon$. Any
partition meeting the two displayed conditions qualifies; the conditioning of the composite design is governed
by the between-group separation alone, not by how collinear the candidates are within a group.

\emph{(ii) Non-identifiability of the split.} Fix a group $g$ and let $\widetilde H_g=\sum_j\sigma_{g,j}\,
p_j q_j^\top$ be the singular value decomposition of its standardised block, with $\sigma_{g,1}\ge\sigma_{g,2}
\ge\cdots$ and $\sigma_{g,2}\le\varepsilon\sigma_{g,1}$. A within-group contrast $a$ supported on $g$ with
$\sum_{k\in g}a_k=0$ is orthogonal to the constant loading vector; since the leading right-singular vector
$q_1$ of a block of nearly proportional nonnegative columns has one-signed loadings and is, to $O(\varepsilon)$,
the constant direction, $a$ lies within $O(\varepsilon)$ of $\mathrm{span}(q_2,q_3,\dots)$. Hence $a$ lies in
the trailing singular subspace of the standardised Gram, where the eigenvalues are at most $C\varepsilon^2$, so
$a^\top\widehat G_g^{-1}a\ge\|a\|_2^2/(C\varepsilon^2)$ with probability tending to one. In the
homoskedastic-equivalent scaling the cross-fold estimator has covariance $\widehat V=\sigma^2\widehat
G^{-1}/N\,\{1+o_P(1)\}$, so
\[
    \mathrm{Var}(a^\top\widehat{\bw})=\sigma^2\,a^\top\widehat G^{-1}a/N
    \ \ge\ \frac{\|a\|_2^2}{C'\varepsilon^2 n^2\rho_n},
\]
which diverges relative to $1/N$ as $\varepsilon\to0$. The within-group split therefore carries information of
order $\varepsilon^2 N$ and is not identifiable in the limit.

\emph{(iii) The group total at the minimax rate.} The composite design $U$ satisfies the transversality
assumption with constant $\gamma_{\mathrm{grp}}\ge\delta$ by part~(i). The normalisation and projected-remainder
assumptions are inherited because each $u_g$ is a fixed linear functional, the leading left-singular direction,
of the group's fitted columns, so the fold-one remainder of $U$ is a contraction of the fold-one remainders of
$\widehat H$ and is negligible in projection whenever the latter are. Theorem~\ref{thm:debias} applied to $U$
then gives the rate $\sqrt{G/\gamma_{\mathrm{grp}}}/(n\sqrt{\rho_n})$ for the composite coefficients
$\widehat\theta=(U^\top U)^{-1}U^\top y$, and Theorem~\ref{thm:twostage} applied to $U$ gives the valid
two-stage interval. The population target is the partial coefficient $\theta_g=\langle u_g,(I-\Pi_{U_{-g}})
\Pmat\rangle/\|u_g\|_2^2$, the contribution of $\Pmat$ along the group's shared direction net of the other
groups, which under the within-collinearity condition equals the total of the group's individual contributions
up to the $O(\varepsilon)$ orthogonal residual.
\end{proof}

\begin{remark}[Producing a qualifying grouping]
A grouping meeting the two conditions of Theorem~\ref{thm:grouped} exists whenever the candidates fall into
well-separated near-collinear clusters, and is found without knowing the truth: single-linkage clustering of
the absolute Gram-correlation matrix $|\widehat\Phi|$ at a threshold $1-\delta'$ merges candidates whose
fitted columns correlate above $1-\delta'$ and separates clusters whose composites correlate below it, after
which $\widehat\gamma_{\mathrm{grp}}$ is recomputed and reported. Equivalently, the eigenvector of
$\widehat\Phi$ at its smallest eigenvalue has its mass on the offending collinear set, which names the group to
merge. The procedure is conservative: if no near-collinear cluster is present every candidate is its own group,
the composite design is the original design, and the theorem reduces to Theorem~\ref{thm:debias}.
\end{remark}

\subsection{Operator test after independent degree truncation}

\begin{assumption}[Bulk-overlap survivability after truncation]
\label{ass:bulk-overlap-final2}
Let \(V_\tau\) be the latent degree-truncated vertex set
\[
  V_\tau=\{i:\tau_-\le \theta_i/\bar\theta\le\tau_+\}.
\]
The deterministic thresholds \(\tau_-,\tau_+\) are chosen away from
latent-degree accumulation points, in the boundary-gap sense of
Lemma~\ref{lem:degree-truncation-final2}.  On the induced subnetwork
\(V_\tau\), assume bounded heterogeneity,
\[
  c\le \theta_i/\bar\theta_{V_\tau}\le C,\qquad i\in V_\tau,
\]
and assume the residualized interaction mass survives truncation:
\[
  \Gamma_{2,V_\tau}\,n_{V_\tau}^2\rho_{V_\tau}^3\to\infty
\]
under the noisy-OR alternative, with bounded interaction leverage under
the additive null.  This is a substantive bulk-overlap condition; it
cannot be replaced by degree truncation alone, because hubs may carry the
overlap mass.
\end{assumption}

\begin{lemma}[Empirical degree truncation verifies bounded heterogeneity]
\label{lem:degree-truncation-final2}
Let \(A^{(0)}\) be an independent preliminary dyad fold with sampling
fraction \(f_0\), observed through the IPW adjacency
\[
  \widetilde A^{(0)}_{ij}=f_0^{-1}R^{(0)}_{ij}A_{ij}.
\]
Let
\[
  \widehat d_i^{(0)}=\sum_j\widetilde A^{(0)}_{ij}.
\]
Assume
\[
  \min_{i\in V_\tau}f_0\theta_i\gg\log n,
\]
and a boundary gap:
\[
  \min_i
  \left|
    \theta_i/\bar\theta-\tau_\pm
  \right|
  \ge \eta_n,
  \qquad
  \eta_n\gg
  \max_{i\in V_\tau}
  \sqrt{\frac{\log n}{f_0\theta_i}} .
\]
Define
\[
  \widehat V
  =
  \{i:\tau_-\le \widehat d_i^{(0)}/\bar d^{(0)}\le\tau_+\}.
\]
Then
\[
  \widehat V=V_\tau
\]
with probability tending to one.  In particular, bounded heterogeneity
and the bulk-overlap condition of
Assumption~\ref{ass:bulk-overlap-final2} transfer to \(\widehat V\).
\end{lemma}

\begin{proof}
Bernstein's inequality gives, uniformly over \(i\in V_\tau\),
\[
  |\widehat d_i^{(0)}-\theta_i|
  =
  O_p\{\sqrt{\theta_i\log n/f_0}+\log n/f_0\}.
\]
The assumption \(\min_i f_0\theta_i\gg\log n\) yields
\[
  \max_{i\in V_\tau}
  |\widehat d_i^{(0)}-\theta_i|/\theta_i=o_p(\eta_n).
\]
The same bound applies to the empirical average degree.  The boundary
gap prevents any vertex from crossing the truncation threshold with
probability tending to one.  Hence \(\widehat V=V_\tau\) with probability
tending to one, and all deterministic conditions on \(V_\tau\) transfer
to \(\widehat V\).
\end{proof}

\begin{theorem}[Degree-truncated in-scope operator deployment]
\label{thm:truncated-operator-final2}
Use an independent preliminary fold \(R^{(0)}\) to select \(\widehat V\)
as in Lemma~\ref{lem:degree-truncation-final2}.  Discard the preliminary
dyads from the operator test.  On the remaining dyads inside
\(\widehat V\), run the four-fold fitted-layer operator test with IPW
sampling.  If each Stage-A subfold uses fraction \(f\) of the remaining
dyads, its effective sampling fraction relative to the original dyad
population is \(f(1-f_0)\), and the IPW weights and variance constants in
the first-stage lemma are adjusted accordingly.

The null and alternative are the subnetwork hypotheses on the induced
vertex set \(\widehat V\):
\[
  H_{0,\widehat V}: P_{\widehat V}\in
  {\rm span}\{G_{1,\widehat V},G_{2,\widehat V}\},
\]
against the noisy-OR/overlap alternative on the same induced subnetwork.
Assume that, with probability \(1-o(1)\), the retained graph satisfies
the bounded-heterogeneity, spectral first-stage, interaction-leverage,
and bulk-overlap conditions in
Assumption~\ref{ass:bulk-overlap-final2}.  Then
\[
  T_{\widehat V}\Rightarrow N(0,1)
\]
under \(H_{0,\widehat V}\), and the test is consistent under the
alternative whenever
\[
  \Gamma_{2,\widehat V}n_{\widehat V}^2\rho_{\widehat V}^3\to\infty.
\]
\end{theorem}

\begin{proof}
Condition on the preliminary fold and on \(\widehat V\).  Since
\(\widehat V\) is measurable with respect to an independent preliminary
fold, the remaining dyads inside \(\widehat V\) are still independent
Bernoulli variables conditional on \(\widehat V\).  The preliminary dyads
are not reused, so the four-fold Stage-A/Stage-B scheme is applied to a
fresh product-Bernoulli dyad population with effective sampling fraction
\(f(1-f_0)\).

On the probability-\(1-o(1)\) event in the theorem, the induced graph
satisfies the original bounded-heterogeneity first-stage assumptions, the
residualized interaction-leverage assumption, and the interaction-mass
condition.  Conditional on this event, the fitted-layer operator proof
applies verbatim on the subnetwork: under the additive null, the
residualized interaction score is a Lindeberg sum of independent
mean-zero Bernoulli terms, so
\[
  T_{\widehat V}\Rightarrow N(0,1).
\]
Under noisy-OR, the mean shift is
\[
  \asymp
  \sqrt{
    \Gamma_{2,\widehat V}
    n_{\widehat V}^2\rho_{\widehat V}^3
  },
\]
which diverges by assumption.  Deconditioning uses
\[
  \sup_t|\Pb(T_{\widehat V}\le t)-\Phi(t)|
  \le
  \Pb(\calE_n^c)
  +
  \sup_t|\Pb(T_{\widehat V}\le t\mid \calE_n)-\Phi(t)|
  \to0.
\]
\end{proof}

\begin{remark}
The rigorous empirical deployment is the degree-truncated subnetwork
test.  The untruncated heavy-tailed test should be labelled illustrative
unless one proves a separate heavy-tail first-stage expansion and verifies
that the residualized interaction mass is bulk-supported rather than
hub-supported.
\end{remark}

\subsection{Effective-information extensions: truncated degree and local support}
\label{app:effective-info}

The common-scale theory in the main text assumes
\[
    \|G_k\|_F^2 \asymp n^2\rho_n^2,
    \qquad
    M^\top D M \asymp \rho_n M^\top M,
\]
where \(\rho_n=\max_{i<j}P_{ij}\). This excludes two motivating
network mechanisms: heavy-tailed degree kernels and local kernels supported
on \(O(n)\) dyads. We record here the corresponding effective-information
extension. The target throughout is the population projection coefficient
\(w^\dagger\). Under correct specification and fidelity of the fitting maps,
\(w^\dagger\) is the generative weight vector.

Let \(e=(i,j)\) index dyads, let \(g_k=\operatorname{vec}_{<}(G_k)\), and let
\(M_S=(g_k:k\in S)\). Write
\[
    D=\operatorname{diag}\{P_e(1-P_e):e\in \mathcal D\},
    \qquad
    \Gamma_S=M_S^\top M_S,
    \qquad
    \Omega_S=M_S^\top D M_S,
\]
and define the exact sandwich covariance
\[
    \Sigma_S=\Gamma_S^{-1}\Omega_S\Gamma_S^{-1}.
\]
For a single column \(g_k\), define its unweighted least-squares information
scale
\[
    N_k
    =
    \frac{\|g_k\|_2^4}{g_k^\top D g_k}.
\]
If the columns have different supports or different heteroskedasticity
profiles, \(N_k\) replaces the common edge-count scale \(n^2\rho_n\).

\begin{assumption}[Effective sandwich transversality]
\label{ass:effective-sandwich}
For the active set \(S\), let
\[
    N=\min_{k\in S}N_k.
\]
The active design satisfies
\[
    \operatorname{tr}(\Sigma_S)
    \le
    C\frac{s}{\gamma_S N}
\]
for some \(\gamma_S\in(0,1]\) and a constant \(C<\infty\).
Equivalently, after accounting for both collinearity and heteroskedasticity,
no active mechanism is nearly explained by the others at its own information
scale.
\end{assumption}

\begin{remark}[Why this assumption is needed]
When all active kernels have the common scale of Assumption~\ref{ass:norm},
Assumption~\ref{ass:effective-sandwich} reduces to the usual Gram
transversality condition, and \(N\asymp n^2\rho_n\). For heavy-tailed
or local mechanisms the same simplification is false: the exact object is the
sandwich covariance \(\Sigma_S\), not the ordinary Gram matrix alone.
\end{remark}

\subsubsection{Truncated regularly varying degree kernel}

For an arbitrary sparse network kernel \(P\), define its population degree
profile
\[
    d_i=d_i(P)=\sum_{j\ne i}P_{ij},
    \qquad
    S_q=\sum_{i=1}^n d_i^q,
    \qquad
    S_1=n\bar d_n,
    \qquad
    R_n=\frac{d_{\max,n}}{\bar d_n},
\]
where \(d_{\max,n}=\max_i d_i\). The population degree-product fitting map is
\[
    G^{\deg,\dagger}_{ij}
    =
    \frac{d_i d_j}{S_1},
    \qquad i<j.
\]
This is the population output of fitting a Chung--Lu degree kernel to a graph
with edge-probability matrix \(P\). It equals the true Chung--Lu component
only under a fidelity condition; otherwise it is a candidate-relative fitted
degree kernel.

Assume the no-saturation condition
\[
    \frac{d_{\max,n}^2}{S_1}=o(1),
\]
and the truncated regularly varying moment conditions, for some
\(\tau\in(2,3)\),
\[
    S_2\asymp n\bar d_n^2R_n^{3-\tau},
    \qquad
    S_3\asymp n\bar d_n^3R_n^{4-\tau}.
\]
We also impose the non-dominance condition
\[
    R_n^{\tau-1}=o(n),
\]
which implies \(S_4=o(S_2^2)\) and \(S_6=o(S_3^2)\).

\begin{lemma}[Degree-matching expansion for a truncated degree kernel]
\label{lem:truncated-degree-first-stage}
Let \(\widetilde A\) be the inverse-probability-weighted Stage-A adjacency
matrix formed with fixed fold fraction \(f\in(0,1)\). Define
\[
    \widehat d_i=\sum_{j\ne i}\widetilde A_{ij},
    \qquad
    \widehat S_1=\sum_i\widehat d_i,
    \qquad
    \widehat G^{\deg}_{ij}
    =
    \frac{\widehat d_i\widehat d_j}{\widehat S_1}.
\]
Write
\[
    \varepsilon_i=\widehat d_i-d_i,
    \qquad
    \varepsilon_+=\sum_i\varepsilon_i.
\]
Then
\[
    \max_i|\varepsilon_i|
    =
    O_P\!\left(\sqrt{d_{\max,n}\log n}+\log n\right),
\]
\[
    \|\varepsilon\|_2=O_P(\sqrt{S_1}),
    \qquad
    |\varepsilon_+|=O_P(\sqrt{S_1}),
\]
and
\[
    \widehat G^{\deg}-G^{\deg,\dagger}
    =
    L_{\deg}(\varepsilon)+Q_{\deg},
\]
where
\[
    L_{\deg}(\varepsilon)_{ij}
    =
    \frac{d_i\varepsilon_j+d_j\varepsilon_i}{S_1}
    -
    \frac{d_i d_j}{S_1^2}\varepsilon_+ .
\]
Moreover,
\[
    \|L_{\deg}(\varepsilon)\|_F
    =
    O_P\!\left(\sqrt{\frac{S_2}{S_1}}\right),
    \qquad
    \|G^{\deg,\dagger}\|_F
    \asymp
    \frac{S_2}{S_1},
\]
and
\[
    \frac{\|\widehat G^{\deg}-G^{\deg,\dagger}\|_F}
         {\|G^{\deg,\dagger}\|_F}
    =
    O_P\!\left(\sqrt{\frac{S_1}{S_2}}\right)
    +
    O_P\!\left(\frac{S_1}{S_2}\right).
\]
Consequently,
\[
    \frac{\|\widehat G^{\deg}-G^{\deg,\dagger}\|_F}
         {\|G^{\deg,\dagger}\|_F}
    =
    O_P\!\left(
        \frac{1}{\sqrt{\bar d_n R_n^{3-\tau}}}
    \right),
\]
so the degree-kernel first stage is Frobenius consistent whenever
\(\bar d_nR_n^{3-\tau}\to\infty\).
\end{lemma}

\begin{proof}
Conditional on the latent attributes and over the joint randomness of the
fold mask and the graph, \(\widehat d_i\) is a sum of independent bounded
inverse-probability-weighted Bernoulli variables with mean \(d_i\) and
variance \(O(d_i/f)\). Bernstein's inequality gives, uniformly over \(i\),
\[
    |\widehat d_i-d_i|
    =
    O_P\!\left(\sqrt{d_i\log n}+\log n\right),
\]
and hence \(\max_i|\varepsilon_i|=O_P(\sqrt{d_{\max,n}\log n}+\log n)\).
Also \(\mathbb E\|\varepsilon\|_2^2=\sum_i O(d_i)=O(S_1)\), so
\(\|\varepsilon\|_2=O_P(\sqrt{S_1})\). Similarly, since \(\varepsilon_+\) is a
centered sum over dyad errors counted twice,
\(\operatorname{Var}(\varepsilon_+)=O(S_1)\), and therefore
\(|\varepsilon_+|=O_P(\sqrt{S_1})\).

Now expand the nonlinear degree-product map \(T(d)_{ij}=d_i d_j/S_1\). A
first-order Taylor expansion around \(d\) gives
\[
    T(d+\varepsilon)_{ij}-T(d)_{ij}
    =
    \frac{d_i\varepsilon_j+d_j\varepsilon_i}{S_1}
    -
    \frac{d_i d_j}{S_1^2}\varepsilon_+
    +
    Q_{ij}.
\]
The linear numerator term satisfies
\[
    \left\|
        \frac{d\varepsilon^\top+\varepsilon d^\top}{S_1}
    \right\|_F
    \le
    \frac{2\|d\|_2\|\varepsilon\|_2}{S_1}
    =
    O_P\!\left(\sqrt{\frac{S_2}{S_1}}\right).
\]
The denominator-linear term satisfies
\[
    \left\|
        \frac{d d^\top}{S_1^2}\varepsilon_+
    \right\|_F
    =
    \frac{|\varepsilon_+|\,\|d\|_2^2}{S_1^2}
    =
    O_P\!\left(\frac{S_2}{S_1^{3/2}}\right),
\]
which is no larger than the preceding display under \(S_2\le S_1^2\), implied
by no saturation. The quadratic remainder obeys
\[
    \|Q_{\deg}\|_F
    =
    O_P\!\left(
        \frac{\|\varepsilon\|_2^2}{S_1}
        +
        \frac{|\varepsilon_+|\|d\|_2\|\varepsilon\|_2}{S_1^2}
        +
        \frac{\varepsilon_+^2\|d\|_2^2}{S_1^3}
    \right)
    =
    O_P(1).
\]
Since \(\|G^{\deg,\dagger}\|_F\asymp S_2/S_1\), we obtain
\(\|Q_{\deg}\|_F/\|G^{\deg,\dagger}\|_F=O_P(S_1/S_2)\). Combining the linear and
quadratic terms gives
\[
    \frac{\|\widehat G^{\deg}-G^{\deg,\dagger}\|_F}
         {\|G^{\deg,\dagger}\|_F}
    =
    O_P\!\left(\sqrt{\frac{S_1}{S_2}}\right)
    +
    O_P\!\left(\frac{S_1}{S_2}\right).
\]
Finally, \(S_2/S_1\asymp\bar d_nR_n^{3-\tau}\), which yields the displayed
regularly varying rate.
\end{proof}

\begin{theorem}[Effective-information rate for truncated degree kernels]
\label{thm:truncated-degree-rate}
Let the active candidate set include one fitted degree-product kernel
\(G^{\deg,\dagger}\) satisfying Lemma~\ref{lem:truncated-degree-first-stage},
together with diffuse SBM, DC-SBM, or RDPG/GRDPG kernels satisfying the
first-stage assumptions of Theorem~\ref{thm:debias} in the main text. Assume the cross-fold
linear expansions have centered linear parts and projected remainders negligible
at the corresponding score scale. Define
\[
    N_{\deg}
    =
    \frac{
        \left(\sum_{i<j}(G^{\deg,\dagger}_{ij})^2\right)^2
    }{
        \sum_{i<j}(G^{\deg,\dagger}_{ij})^2P_{ij}(1-P_{ij})
    }.
\]
Assume the leverage conditions
\[
    \frac{\max_{i<j}(G^{\deg,\dagger}_{ij})^2}
         {\sum_{i<j}(G^{\deg,\dagger}_{ij})^2}
    \to0,
    \qquad
    \frac{\max_{i<j}(G^{\deg,\dagger}_{ij})^2P_{ij}}
         {\sum_{i<j}(G^{\deg,\dagger}_{ij})^2P_{ij}}
    \to0.
\]
If Assumption~\ref{ass:effective-sandwich} holds with
\(N=\min_{k\in S}N_k\), then the cross-fold debiased estimator satisfies
\[
    \|\widehat w_{\mathrm{db}}-w^\dagger\|_2
    =
    O_P\!\left(
        \sqrt{\frac{s}{\gamma_S N}}
    \right).
\]
If, in addition, the degree-product kernel controls the variance on its
high-leverage dyads, in the sense that
\[
    \sum_{i<j}(G^{\deg,\dagger}_{ij})^2P_{ij}(1-P_{ij})
    \asymp
    \sum_{i<j}(G^{\deg,\dagger}_{ij})^3,
\]
then
\[
    N_{\deg}
    \asymp
    \frac{S_2^4}{S_1S_3^2}
    \asymp
    \frac{n\bar d_n}{R_n^{2(\tau-2)}}.
\]
Equivalently, with average density \(\bar\rho_n=\bar d_n/n\),
\[
    N_{\deg}^{-1/2}
    \asymp
    \frac{R_n^{\tau-2}}{n\sqrt{\bar\rho_n}}.
\]
If the paper's max-density notation is used, and if
\(\rho_{\max,n}=\max_{i<j}P_{ij}\asymp\bar d_nR_n^2/n\), then the same rate can
be written as
\[
    N_{\deg}^{-1/2}
    \asymp
    \frac{R_n^{\tau-1}}{n\sqrt{\rho_{\max,n}}}.
\]
\end{theorem}

\begin{proof}
The first-stage expansion is Lemma~\ref{lem:truncated-degree-first-stage}.
Write the two independent Stage-A fitted designs as \(\widehat M_a=M+H_a\) and
\(\widehat M_b=M+H_b\), with \(H_a=L(E_a)+Q_a\) and \(H_b=L(E_b)+Q_b\). The
cross-fold debiased Gram is
\[
    \widehat\Gamma_{\mathrm{db}}
    =
    \tfrac12
    \left(
        \widehat M_a^\top\widehat M_b
        +
        \widehat M_b^\top\widehat M_a
    \right).
\]
Because \(E_a\) and \(E_b\) are independent Stage-A noises, the centered
linear self-product cancels, \(\mathbb E[L(E_a)^\top L(E_b)\mid u]=0\). The
deterministic bias terms and nonlinear terms involving \(Q_a,Q_b\) are
negligible by the projected-remainder condition. Therefore the leading
Stage-B term is the ordinary heteroskedastic least-squares score
\(\Gamma_S^{-1}M_S^\top(A-P)\), whose covariance is
\(\Sigma_S=\Gamma_S^{-1}\Omega_S\Gamma_S^{-1}\). Thus
\(\|\widehat w_{\mathrm{db}}-w^\dagger\|_2=O_P(\sqrt{\operatorname{tr}(\Sigma_S)})\),
and Assumption~\ref{ass:effective-sandwich} gives
\(\sqrt{\operatorname{tr}(\Sigma_S)}\le C\sqrt{s/(\gamma_S N)}\), which
proves the estimator bound.

It remains to compute \(N_{\deg}\). By the non-dominance condition
\(R_n^{\tau-1}=o(n)\),
\[
    \sum_{i<j}(G^{\deg,\dagger}_{ij})^2
    =
    \sum_{i<j}\frac{d_i^2d_j^2}{S_1^2}
    \asymp
    \frac{S_2^2}{S_1^2}.
\]
Under the variance-dominance condition,
\[
    \sum_{i<j}(G^{\deg,\dagger}_{ij})^2P_{ij}(1-P_{ij})
    \asymp
    \sum_{i<j}(G^{\deg,\dagger}_{ij})^3
    \asymp
    \frac{S_3^2}{S_1^3}.
\]
Hence \(N_{\deg}\asymp(S_2^2/S_1^2)^2/(S_3^2/S_1^3)=S_2^4/(S_1S_3^2)\).
Substituting \(S_2\asymp n\bar d_n^2R_n^{3-\tau}\),
\(S_3\asymp n\bar d_n^3R_n^{4-\tau}\), and \(S_1=n\bar d_n\) gives
\[
    N_{\deg}
    \asymp
    \frac{
        n^4\bar d_n^8R_n^{4(3-\tau)}
    }{
        n\bar d_n\cdot n^2\bar d_n^6R_n^{2(4-\tau)}
    }
    =
    n\bar d_n R_n^{4-2\tau}
    =
    \frac{n\bar d_n}{R_n^{2(\tau-2)}}.
\]
The average-density and max-density forms follow from \(\bar\rho_n=\bar d_n/n\)
and \(\rho_{\max,n}\asymp\bar d_nR_n^2/n\).
\end{proof}

\begin{remark}[What Theorem~\ref{thm:truncated-degree-rate} does and does not claim]
Theorem~\ref{thm:truncated-degree-rate} is a rate theorem for the
cross-fold unweighted least-squares synthesis estimator. It does not assert
a minimax lower bound over all possible estimators in the heavy-tailed degree
case. In strongly heteroskedastic heavy-tailed designs, likelihood or
weighted least squares may have a different information scale. The theorem
therefore extends the fitted-kernel debiasing principle to truncated
regularly varying degree kernels, and its minimax status for the residualized projection coefficient is
settled by the matching lower bound of Theorem~\ref{thm:heavy-tail-projection-lb} below; the correctly
specified scalar weight, by contrast, is estimable faster, as that theorem's remark records.
\end{remark}

\begin{theorem}[Heavy-tailed projection lower bound]
\label{thm:heavy-tail-projection-lb}
Let $\mathcal E_n=\{(i,j):1\le i<j\le n\}$ and write dyads as $e\in\mathcal E_n$. Let
$g=(g_e)_{e\in\mathcal E_n}$ be a deterministic degree-product column satisfying $0<g_e\le g_*<1$. Let
$X$ be a fixed matrix of nuisance mechanism columns, and set
\[
    z=(I-\Pi_X)g,\qquad a_n=\|z\|_2^2,\qquad b_n=\sum_{e\in\mathcal E_n} z_e^2 g_e .
\]
Assume $a_n>0$, $b_n>0$, and the no-dominant-dyad condition $\max_{e}|z_e|/\sqrt{b_n}\le\Lambda$ for a
fixed $\Lambda<\infty$. Fix constants $0<c_-<\kappa<c_+<\infty$ with $c_+g_*\le 1-\varepsilon$ for some
$\varepsilon>0$, and define the nuisance class
\[
    \mathcal P_n=\bigl\{P:\ c_-g_e\le P_e\le c_+g_e\ \text{for all }e\in\mathcal E_n\bigr\}.
\]
For $P\in\mathcal P_n$ let $\mathbb P_P$ be the product Bernoulli law with $A_e\sim\mathrm{Bernoulli}(P_e)$
independently, and define the residualized projection coefficient
$\theta(P)=\langle z,P\rangle/\|z\|_2^2$. Then there is a constant $c>0$, depending only on
$(c_-,c_+,\kappa,g_*,\varepsilon,\Lambda)$, with
\[
    \inf_{\widehat\theta}\sup_{P\in\mathcal P_n}\mathbb E_P\{\widehat\theta-\theta(P)\}^2
    \ge c\,\frac{b_n}{a_n^2},
\]
so the effective information for this residualized degree-column coefficient is at most
$N_z=a_n^2/b_n$. Suppose in addition that $g_{ij}=\bar\rho_n h_ih_j$ with $h_i\in[1,R_n]$, that the
truncated heavy-tail pair moments satisfy, for $q=2,3$,
\[
    \sum_{i<j}h_i^q h_j^q\asymp n^2 R_n^{2(q+1-\tau)},\qquad 2<\tau<3,
\]
and the residual nondegeneracy conditions $\|z\|_2^2\asymp\sum_{i<j}g_{ij}^2$ and
$\sum_{i<j}z_{ij}^2 g_{ij}\asymp\sum_{i<j}g_{ij}^3$. Then
\[
    N_z\asymp\frac{(\sum_{i<j}g_{ij}^2)^2}{\sum_{i<j}g_{ij}^3}
    \asymp\frac{n^2\bar\rho_n}{R_n^{2(\tau-2)}}=\frac{n\bar d_n}{R_n^{2(\tau-2)}},
    \qquad \bar d_n=n\bar\rho_n,
\]
and hence $\inf_{\widehat\theta}\sup_{P\in\mathcal P_n}\mathbb E_P\{\widehat\theta-\theta(P)\}^2\ge
c\,R_n^{2(\tau-2)}/(n\bar d_n)$. This lower bound matches the upper bound of
Theorem~\ref{thm:truncated-degree-rate} up to constants, so the cross-fold debiased rate is minimax
sharp for the residualized degree-column projection coefficient over $\mathcal P_n$.
\end{theorem}

\begin{proof}
Let $P^0_e=\kappa g_e$. By the assumptions on $\kappa$ and $g_*$, $P^0\in\mathcal P_n$ and $P^0_e$ is
bounded away from one uniformly. Set $B_n=\sum_e z_e^2 P^0_e(1-P^0_e)$; since $P^0_e(1-P^0_e)\asymp g_e$,
$B_n\asymp b_n$. Define the least-favourable direction $v_e=P^0_e(1-P^0_e)z_e/B_n$, so that
$\langle z,v\rangle=\sum_e z_e^2 P^0_e(1-P^0_e)/B_n=1$. For a constant $c_0>0$ to be chosen set
$\delta_n=c_0\sqrt{B_n}/a_n$ and define $P^\pm_e=P^0_e\pm(\delta_n/2)a_n v_e$, whose targets are separated
by $\theta(P^+)-\theta(P^-)=\delta_n a_n\langle z,v\rangle/a_n=\delta_n$. Since $P^0_e\le 1-\varepsilon$,
\[
    \frac{|P^\pm_e-P^0_e|}{g_e}\le Cc_0\frac{|z_e|}{\sqrt{B_n}}\le Cc_0\Lambda,
\]
so choosing $c_0$ small (depending only on the fixed margin between $\kappa$ and $(c_-,c_+)$) gives
$c_-g_e\le P^\pm_e\le c_+g_e$, hence $P^\pm\in\mathcal P_n$. Because $P^+$ and $P^-$ are uniformly
comparable to $P^0$, the quadratic Bernoulli Kullback--Leibler bound gives
\[
    \mathrm{KL}(\mathbb P_{P^+}\,\|\,\mathbb P_{P^-})\le C\sum_e\frac{(P^+_e-P^-_e)^2}{P^0_e(1-P^0_e)}
    =C\delta_n^2 a_n^2\frac{1}{B_n^2}\sum_e z_e^2 P^0_e(1-P^0_e)=C\delta_n^2 a_n^2/B_n=Cc_0^2,
\]
using $P^+_e-P^-_e=\delta_n a_n v_e$. Taking $c_0$ small makes the divergence bounded by an absolute
constant, and Le Cam's two-point lemma yields
$\inf_{\widehat\theta}\sup_{P\in\mathcal P_n}\mathbb E_P\{\widehat\theta-\theta(P)\}^2\ge c\delta_n^2
=cB_n/a_n^2\asymp cb_n/a_n^2$. For the heavy-tail scale, for $q=2,3$,
$\sum_{i<j}g_{ij}^q=\bar\rho_n^q\sum_{i<j}h_i^q h_j^q\asymp n^2\bar\rho_n^q R_n^{2(q+1-\tau)}$, so
$\sum g_{ij}^2\asymp n^2\bar\rho_n^2 R_n^{2(3-\tau)}$ and $\sum g_{ij}^3\asymp n^2\bar\rho_n^3
R_n^{2(4-\tau)}$. By residual nondegeneracy,
$N_z=a_n^2/b_n\asymp(\sum g_{ij}^2)^2/\sum g_{ij}^3\asymp n^2\bar\rho_n R_n^{4-2\tau}
=n^2\bar\rho_n/R_n^{2(\tau-2)}$, which is $n\bar d_n/R_n^{2(\tau-2)}$ since $\bar d_n=n\bar\rho_n$.
\end{proof}

\begin{remark}[Why this is not a scalar Chung--Lu lower bound]
Theorem~\ref{thm:heavy-tail-projection-lb} is a lower bound for the residualized projection coefficient
over a nuisance class $P_e\asymp g_e$. It is not a lower bound for the correctly specified one-parameter
model $P_e=w g_e$ with known $g$. In that scalar model the Fisher information is
\[
    I(w)=\sum_e\frac{g_e^2}{wg_e(1-wg_e)}\asymp\sum_e g_e\asymp n\bar d_n,
\]
so $w$ is estimable at the faster rate $(n\bar d_n)^{-1/2}$, faster than the projection rate by the factor
$R_n^{\tau-2}$. A claim that the pure Chung--Lu scalar weight has minimax risk $R_n^{2(\tau-2)}/(n\bar d_n)$
would therefore be false: the slower rate is a property of the projection coefficient under an unknown
nuisance baseline, not of the scalar weight.
\end{remark}

\paragraph{Numerical verification of the rate.}
The effective-information rate is confirmed by simulation. We draw a truncated regularly varying degree
sequence with tail index $\tau=2.5$ and natural cutoff $\theta_{\max}=n^{1/(\tau-1)}$, add a constant-norm
geometry agent separated from degree, generate the graph, and refit the degree kernel by degree matching on
two independent folds before forming the cross-fold debiased estimate. As $n$ grows from $400$ to $2000$ the
effective information $N_{\deg}=(\sum_{i<j}G_{ij}^2)^2/\sum_{i<j}G_{ij}^2P_{ij}(1-P_{ij})$ spans more than two
orders of magnitude, and the empirical standard deviation of the debiased degree coefficient tracks
$N_{\deg}^{-1/2}$ with a fitted log--log slope of $-0.46$ against the predicted $-1/2$
(Figure~\ref{fig:degrate}). The standard deviation isolates the rate from the projection-fidelity bias of
Appendix~\ref{app:deflation}: the coefficient level carries that bias, while its concentration is governed by
the effective sample size $N_{\deg}$, exactly as Theorem~\ref{thm:truncated-degree-rate} states.

\begin{figure}[t]
\centering
\includegraphics[width=0.6\textwidth]{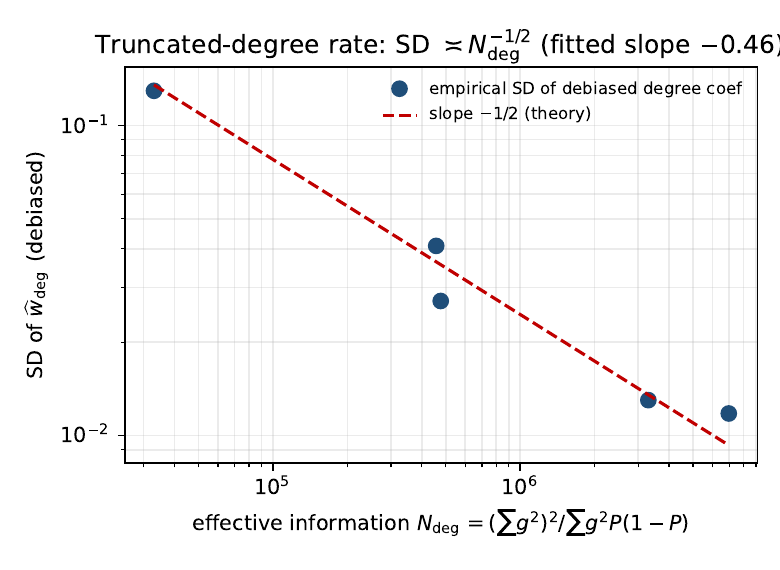}
\caption{Numerical verification of the truncated-degree rate (Theorem~\ref{thm:truncated-degree-rate}). The
empirical standard deviation of the cross-fold debiased degree coefficient against the effective information
$N_{\deg}$, over $n\in\{400,650,1000,1500,2000\}$ at tail index $\tau=2.5$; the dashed line is the predicted
slope $-1/2$ and the fitted slope is $-0.46$.}
\label{fig:degrate}
\end{figure}

\subsubsection{Local-support and triadic mechanisms}

Let \(T_n\subset\{(i,j):i<j\}\) be a known or independently pilot-estimated
local dyad support, and let \(q_n=|T_n|\). For a ring-local or small-world
support of radius \(m_n\), \(q_n\asymp nm_n\). Define the local-support kernel
\[
    G^\triangle_{ij}
    =
    p_{\triangle,n}\mathbf 1\{(i,j)\in T_n\}.
\]
Here \(p_{\triangle,n}\) is a declared population scale for the local kernel,
for example the average edge probability on \(T_n\),
\(p_{\triangle,n}=q_n^{-1}\sum_{(i,j)\in T_n}P_{ij}\). The coefficient of
\(G^\triangle\) is meaningful only relative to this normalisation. If
\(p_{\triangle,n}\) is not fixed by the fitting map, then only the product
\(w_\triangle p_{\triangle,n}\) is directly identified.

\begin{lemma}[Local-support first-stage concentration]
\label{lem:local-support-first-stage}
Assume \(T_n\) is known or estimated on a pilot split independent of Stage B.
Let
\[
    \widehat p_{\triangle,n}
    =
    \frac{1}{q_n}\sum_{(i,j)\in T_n}\widetilde A_{ij},
    \qquad
    \widehat G^\triangle_{ij}
    =
    \widehat p_{\triangle,n}\mathbf 1\{(i,j)\in T_n\}.
\]
Assume \(P_{ij}\asymp p_{\triangle,n}\) for \((i,j)\in T_n\) and
\(q_np_{\triangle,n}\to\infty\). Then
\[
    \widehat p_{\triangle,n}-p_{\triangle,n}
    =
    O_P\!\left(
        \sqrt{\frac{p_{\triangle,n}}{q_n}}
    \right),
    \qquad
    \frac{\|\widehat G^\triangle-G^\triangle\|_F}
         {\|G^\triangle\|_F}
    =
    O_P\!\left(
        \frac{1}{\sqrt{q_np_{\triangle,n}}}
    \right).
\]
\end{lemma}

\begin{proof}
The estimator \(\widehat p_{\triangle,n}\) is an average of \(q_n\)
independent bounded inverse-probability-weighted Bernoulli variables with
variance \(O(p_{\triangle,n}/f)\), where \(f\in(0,1)\) is the fixed Stage-A
fold fraction. Hence \(\operatorname{Var}(\widehat p_{\triangle,n})=O(p_{\triangle,n}/q_n)\),
and Chebyshev's inequality gives \(\widehat p_{\triangle,n}-p_{\triangle,n}=O_P(\sqrt{p_{\triangle,n}/q_n})\).
Moreover \(\|\widehat G^\triangle-G^\triangle\|_F=\sqrt{q_n}\,|\widehat p_{\triangle,n}-p_{\triangle,n}|\),
while \(\|G^\triangle\|_F=\sqrt{q_n}\,p_{\triangle,n}\). Dividing the two
displays gives the stated relative rate.
\end{proof}

\begin{theorem}[Support-adapted rate for a local network mechanism]
\label{thm:local-support-rate}
Let the active candidate set contain diffuse mechanisms satisfying the
first-stage assumptions of Theorem~\ref{thm:debias} in the main text, together with one
local-support mechanism \(G^\triangle\) supported on \(q_n\) dyads with edge
scale \(p_{\triangle,n}\). Assume \(T_n\) is known or pilot-estimated
independently of Stage B, and assume Lemma~\ref{lem:local-support-first-stage}
holds. Define
\[
    N_\triangle
    =
    \frac{
        \left(\sum_{i<j}(G^\triangle_{ij})^2\right)^2
    }{
        \sum_{i<j}(G^\triangle_{ij})^2P_{ij}(1-P_{ij})
    }.
\]
If \(P_{ij}\asymp p_{\triangle,n}\) on \(T_n\), then
\(N_\triangle\asymp q_np_{\triangle,n}\). Under
Assumption~\ref{ass:effective-sandwich},
\[
    \|\widehat w_{\mathrm{db}}-w^\dagger\|_2
    =
    O_P\!\left(
        \sqrt{\frac{s}{\gamma_S N}}
    \right),
    \qquad
    N=\min_{k\in S}N_k.
\]
In particular, the local coordinate satisfies
\[
    |\widehat w_\triangle-w^\dagger_\triangle|
    =
    O_P\!\left(
        \frac{1}{\sqrt{\gamma_S q_np_{\triangle,n}}}
    \right).
\]
For a ring-local support with \(q_n\asymp nm_n\),
\[
    |\widehat w_\triangle-w^\dagger_\triangle|
    =
    O_P\!\left(
        \frac{1}{\sqrt{\gamma_S nm_np_{\triangle,n}}}
    \right).
\]
If \(m_n=O(1)\) and \(p_{\triangle,n}\asymp1\), this becomes \(O_P(n^{-1/2})\).
Thus a local \(O(n)\)-support mechanism is estimable, but generally not at the
diffuse edge-count rate \(1/(n\sqrt{\rho_n})\).
\end{theorem}

\begin{proof}
For the local column, \(\sum_{i<j}(G^\triangle_{ij})^2=q_np_{\triangle,n}^2\).
Since \(P_{ij}\asymp p_{\triangle,n}\) on \(T_n\),
\(\sum_{i<j}(G^\triangle_{ij})^2P_{ij}(1-P_{ij})\asymp q_np_{\triangle,n}^3\).
Therefore
\[
    N_\triangle
    \asymp
    \frac{q_n^2p_{\triangle,n}^4}{q_np_{\triangle,n}^3}
    =
    q_np_{\triangle,n}.
\]
The first-stage error is controlled by
Lemma~\ref{lem:local-support-first-stage}. Because the local support and
scale are known or estimated independently of Stage B, the Stage-B regression
noise is independent of the fitted local column. Cross-fold Gram debiasing
removes the same-fold self-product of the first-stage error, and the remaining
terms are controlled by the projected-remainder condition. Hence the leading
term has covariance \(\Sigma_S\), and Assumption~\ref{ass:effective-sandwich}
gives \(\|\widehat w_{\mathrm{db}}-w^\dagger\|_2=O_P(\sqrt{s/(\gamma_SN)})\).
The displayed local-coordinate rate follows by substituting
\(N_\triangle\asymp q_np_{\triangle,n}\).
\end{proof}

\begin{proposition}[One-dimensional lower bound for a known local support]
\label{prop:local-lower-bound}
Consider the one-dimensional local-support model
\[
    P^{(u)}_{ij}
    =
    P^0_{ij}+uG^\triangle_{ij},
    \qquad
    G^\triangle_{ij}=p_{\triangle,n}\mathbf 1\{(i,j)\in T_n\},
\]
where \(P^0_{ij}\asymp p_{\triangle,n}\) on \(T_n\), and where \(u\) ranges
over a fixed neighbourhood of zero such that all probabilities remain in
\([0,1]\). Then every estimator \(\widetilde u\) satisfies
\[
    \inf_{\widetilde u}\sup_{u}
    \mathbb E_u|\widetilde u-u|
    \ge
    c\frac{1}{\sqrt{q_np_{\triangle,n}}}
\]
for a numerical constant \(c>0\). Thus the support-adapted rate in
Theorem~\ref{thm:local-support-rate} is sharp for the local coordinate.
\end{proposition}

\begin{proof}
Take two alternatives \(u_+=\delta/2\) and \(u_-=-\delta/2\). Their mean
matrices differ by \(P^{(u_+)}-P^{(u_-)}=\delta G^\triangle\). For Bernoulli
product measures, using \(P^0_{ij}\asymp p_{\triangle,n}\) on \(T_n\), the
Kullback--Leibler divergence satisfies
\[
    \operatorname{KL}\!\left(
        \mathbb P_{u_+},\mathbb P_{u_-}
    \right)
    \le
    C\delta^2
    \sum_{(i,j)\in T_n}
    \frac{(G^\triangle_{ij})^2}{p_{\triangle,n}}
    =
    C\delta^2
    q_np_{\triangle,n}.
\]
Choose \(\delta=c_0(q_np_{\triangle,n})^{-1/2}\) with \(c_0>0\) small enough
that the divergence is bounded by a fixed constant. Le Cam's two-point lemma
gives
\[
    \inf_{\widetilde u}\sup_{u\in\{u_+,u_-\}}
    \mathbb E_u|\widetilde u-u|
    \ge
    c\delta
    =
    c(q_np_{\triangle,n})^{-1/2}.
\]
\end{proof}

\begin{remark}[Same-graph triadic scores]
Theorem~\ref{thm:local-support-rate} covers known or independently
pilot-estimated local supports. It does not yet cover Adamic--Adar, Jaccard,
or other same-graph triadic scores computed from the same edges used in Stage
B. Those scores are random functions of \(A\), and require an additional
support-estimation expansion analogous to Assumption~\ref{ass:linplug}.
\end{remark}

\subsubsection{General effective-information rate}

\begin{theorem}[Effective-information rate for fitted network mechanisms]
\label{thm:effective-information-fitted}
Let the active candidate set contain a fixed number \(s\) of fitted network
mechanisms. The set may include:
\begin{enumerate}
    \item diffuse SBM, DC-SBM, and RDPG/GRDPG kernels satisfying the
    spectral first-stage assumptions of Theorem~\ref{thm:debias} in the main text;
    \item a truncated regularly varying degree-product kernel satisfying
    Lemma~\ref{lem:truncated-degree-first-stage} and
    Theorem~\ref{thm:truncated-degree-rate};
    \item a known or independently pilot-estimated local-support kernel
    satisfying Lemma~\ref{lem:local-support-first-stage} and
    Theorem~\ref{thm:local-support-rate}.
\end{enumerate}
For each active mechanism define
\[
    N_k
    =
    \frac{\|g_k\|_2^4}{g_k^\top D g_k},
    \qquad
    g_k=\operatorname{vec}_{<}(G_k),
    \qquad
    D=\operatorname{diag}\{P_e(1-P_e)\},
\]
and let \(N=\min_{k\in S}N_k\). Assume the fitted kernels admit
independent cross-fold expansions
\[
    \widehat G^{(a)}_k-G_k=L_k(E_a)+Q_{k,a},
    \qquad
    \widehat G^{(b)}_k-G_k=L_k(E_b)+Q_{k,b},
\]
where the linear parts are centered and independent across the two Stage-A
sub-folds, and where the projected remainders are negligible at the relevant
score scale. Assume also the effective sandwich transversality condition
\(\operatorname{tr}(\Sigma_S)\le Cs/(\gamma_SN)\). Then the cross-fold
debiased estimator satisfies
\[
    \|\widehat w_{\mathrm{db}}-w^\dagger\|_2
    =
    O_P\!\left(
        \sqrt{\frac{s}{\gamma_SN}}
    \right).
\]
For the named mechanisms, \(N_k\asymp n^2\rho_n\) for diffuse common-scale
kernels,
\[
    N_k
    \asymp
    \frac{n\bar d_n}{R_n^{2(\tau-2)}}
    =
    \frac{n^2\bar\rho_n}{R_n^{2(\tau-2)}}
\]
for truncated regularly varying degree kernels under the variance-dominance
condition of Theorem~\ref{thm:truncated-degree-rate}, and
\(N_k\asymp q_np_{\triangle,n}\) for known-support local mechanisms.
\end{theorem}

\begin{proof}
Condition on the latent attributes and on the Stage-A fold assignments. The
Stage-B dyad errors are independent of the fitted Stage-A design. Let
\(\widehat M_a=M+H_a\) and \(\widehat M_b=M+H_b\), with \(H_a=L(E_a)+Q_a\) and
\(H_b=L(E_b)+Q_b\). The cross-fold debiased Gram is
\[
    \widehat\Gamma_{\mathrm{db}}
    =
    \tfrac12
    \left(
        \widehat M_a^\top\widehat M_b
        +
        \widehat M_b^\top\widehat M_a
    \right)
    =
    M^\top M
    +
    \operatorname{sym}\{M^\top H_a+M^\top H_b\}
    +
    \tfrac12(H_a^\top H_b+H_b^\top H_a).
\]
The single-fold attenuation term \(H^\top H\) is absent. Moreover
\(\mathbb E\{L(E_a)^\top L(E_b)\mid u\}=0\) because the two Stage-A linear
parts are independent and centered. The remaining deterministic biases and
nonlinear terms involving \(Q_a,Q_b\) are negligible by the projected-remainder
assumption. Hence \(\widehat\Gamma_{\mathrm{db}}=\Gamma_S+o_P(\|\Gamma_S\|_{\mathrm{op}})\)
on the active set.

The leading Stage-B score is \(M_S^\top(A-P)\), with covariance
\(\Omega_S=M_S^\top D M_S\). Therefore the leading expansion is
\[
    \widehat w_{\mathrm{db}}-w^\dagger
    =
    \Gamma_S^{-1}M_S^\top(A-P)+o_P\!\left(
        \sqrt{\operatorname{tr}(\Sigma_S)}
    \right),
    \qquad
    \Sigma_S=\Gamma_S^{-1}\Omega_S\Gamma_S^{-1},
\]
so \(\|\widehat w_{\mathrm{db}}-w^\dagger\|_2=O_P(\sqrt{\operatorname{tr}(\Sigma_S)})\).
The effective sandwich transversality condition gives
\(\sqrt{\operatorname{tr}(\Sigma_S)}\le C\sqrt{s/(\gamma_SN)}\), which
proves the rate.

The three mechanism-specific information calculations are as follows. For
diffuse common-scale mechanisms, Assumption~\ref{ass:norm} gives
\(\|g_k\|_2^2\asymp n^2\rho_n^2\) and \(g_k^\top Dg_k\asymp n^2\rho_n^3\), so
\(N_k\asymp n^2\rho_n\). For truncated degree kernels, the calculation is
Theorem~\ref{thm:truncated-degree-rate}. For local supports, the calculation
is Theorem~\ref{thm:local-support-rate}.
\end{proof}

\begin{remark}[Correct scope of Theorem~\ref{thm:effective-information-fitted}]
Theorem~\ref{thm:effective-information-fitted} is an upper-bound theorem for
the cross-fold debiased least-squares synthesis estimator. The diffuse
common-scale special case inherits the sharp minimax statement from the main
paper. The known-support local coordinate has the matching lower bound of
Proposition~\ref{prop:local-lower-bound}. The truncated heavy-tailed degree
extension gives the sandwich rate of the unweighted synthesis estimator; a
full minimax lower bound over all estimators for that class is not asserted
here.
\end{remark}

\section{Additional numerical studies}
\label{app:numerical}
This appendix collects numerical studies that validate background theory or illustrate the method on controlled or secondary data, demoted from the main text to keep the empirical section focused on the results that validate the two headline theorems.

\subsection{A simulated example: region-specific mechanisms}

As a controlled illustration in which the generating mechanism is known, and which therefore checks the machinery in a way the real-data sections cannot, we build a synthetic graph on $n=450$ vertices in three blocks of $150$ joined by sparse Erd\H{o}s--R\'enyi bridges (inter-block density about $0.2\%$, against matched within-block densities of about $8\%$). The three blocks carry deliberately different mechanisms: one Erd\H{o}s--R\'enyi, one small-world (Watts--Strogatz), and one scale-free (Barab\'asi--Albert). We fit a block agent for the $3\times3$ structure and, within each block, a debiased Chung--Lu degree agent $G^{\deg}_{ij}\propto d_id_j$ and a triadic agent $G^{\mathrm{tri}}_{ij}=(A^2)_{ij}$. Spectral clustering recovers the three blocks exactly (adjusted Rand index $1.00$), within the regime of Remark~\ref{cor:recovery}, and the region-restricted agents then recover the per-block mechanism: the standardised degree association dominates in the scale-free block, the triadic association in the small-world block, and both are within sampling noise of zero in the Erd\H{o}s--R\'enyi block (Figure~\ref{fig:regional}).

The debiasing of Theorem~\ref{thm:debias} is what makes the Erd\H{o}s--R\'enyi block legible. Because a vertex's degree mechanically includes its own incident edges, a naive single-fold degree agent reports a spurious positive degree association in that block (standardised $+0.11$), which on its own would mislabel it as scale-free; the leave-$(i,j)$-out correction removes this self-product term and returns the association to zero ($-0.01$), while leaving the genuine degree signal of the scale-free block essentially intact ($+0.23$, Figure~\ref{fig:regional}). The example is thus a check against a known truth on the three counts the framework promises: the regions are recovered, the correct mechanism is attributed to each, and the debiasing step is what makes that attribution correct.

\begin{figure}[t]
\centering
\includegraphics[width=\textwidth]{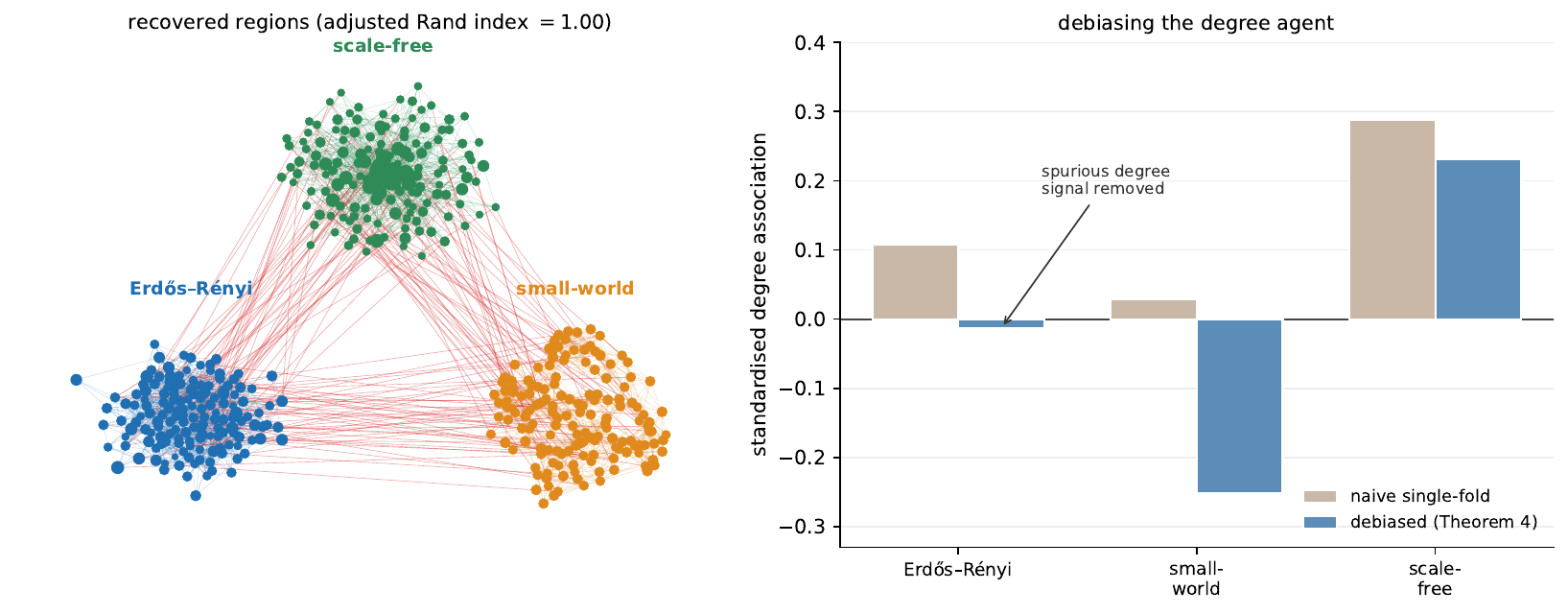}
\caption{Controlled illustration with a known truth: a graph on $n=450$ vertices in three $150$-vertex blocks joined by sparse bridges, with one Erd\H{o}s--R\'enyi, one small-world, and one scale-free block. \emph{Left:} the graph with each recovered region in its own colour (spectral clustering matches the planted blocks exactly, adjusted Rand index $1.00$) and the sparse bridges in red; node size is proportional to degree, so the scale-free hubs stand out. \emph{Right:} the per-region standardised degree association under a naive single-fold fit and under the debiased estimator of Theorem~\ref{thm:debias}. The naive fit reports a spurious positive degree association in the Erd\H{o}s--R\'enyi block, which alone would mislabel it as scale-free; debiasing returns that association to zero while preserving the genuine degree signal of the scale-free block.}
\label{fig:regional}
\end{figure}

\subsection{Predictive comparison: full results}
The following calibration, gap-decomposition, and Pareto results accompany the compressed predictive comparison of Section~\ref{sec:numerical}.

\begin{figure}[t]
\centering
\includegraphics[width=0.72\textwidth]{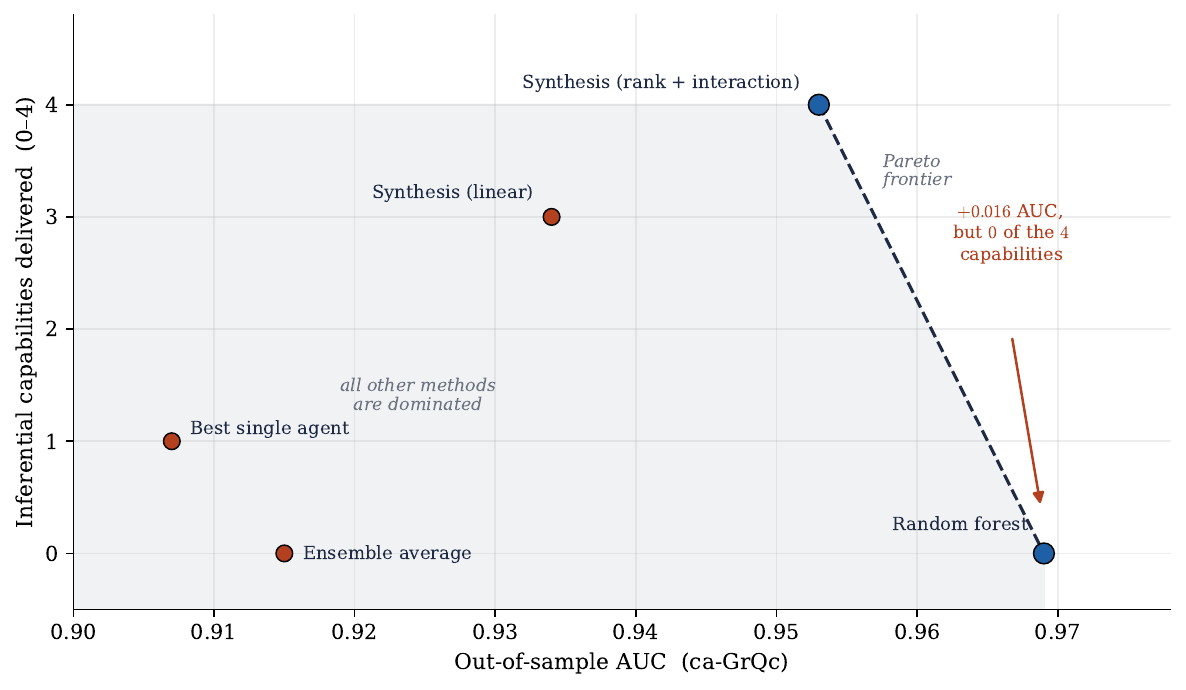}
\caption{Out-of-sample AUC against inferential content on \textsf{ca-GrQc}. Inferential content counts the capabilities a method delivers: signed per-mechanism coefficient, confidence interval, support recovery, and operator test. The synthesis family is Pareto-optimal at maximal inferential content; the random forest's additional $0.016$ AUC costs all four capabilities.}
\label{fig:pareto}
\end{figure}

\begin{figure}[t]
\centering
\includegraphics[width=0.86\textwidth]{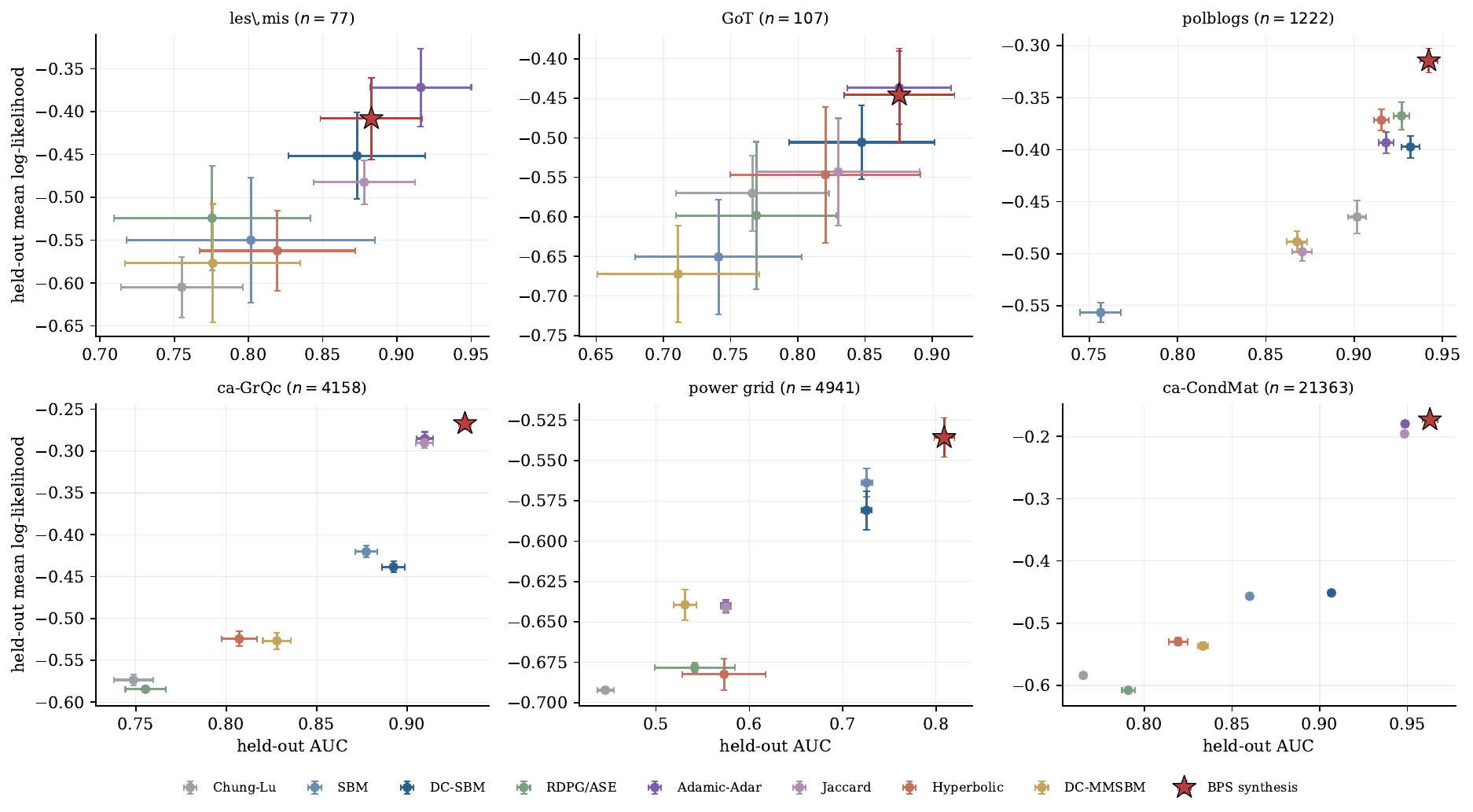}
\caption{Link prediction as a ranking--calibration Pareto plot, one panel per dataset: held-out AUC against balanced log score, with fold standard deviations.}
\label{fig:linkpred}
\end{figure}

\begin{table}[t]
\centering
\small
\caption{Decomposing the gap to the random forest on the two networks with enough held-out dyads to fit the augmented model. A rank-transformed synthesis with sparse interactions, retaining one signed coefficient per mechanism, recovers $40$--$56\%$ of the forest gap; the residual is the forest's non-additive flexibility.}
\label{tab:gap}
\begin{tabular}{lccccc}
\toprule
Network & Linear synthesis & Rank $+$ interaction & RF stacking & Gap to RF & Gap closed\\
\midrule
\textsf{polblogs} & $0.945$ & $0.947$ & $0.951$ & $0.006$ & $40\%$\\
\textsf{ca-GrQc} & $0.934$ & $0.953$ & $0.969$ & $0.035$ & $56\%$\\
\bottomrule
\end{tabular}
\end{table}

\begin{table}[t]
\centering
\small
\caption{Calibration at the natural edge prevalence (means over folds). Expected calibration error and Brier score for the synthesis and a random forest on the same candidate scores, both prior-corrected from the balanced training set; lower is better. The synthesis is the better-calibrated forecaster on four of six networks by expected calibration error.}
\label{tab:calib}
\begin{tabular}{lcccc}
\toprule
& \multicolumn{2}{c}{Expected calibration error} & \multicolumn{2}{c}{Brier score}\\
\cmidrule(lr){2-3}\cmidrule(lr){4-5}
Network & Synthesis & Forest & Synthesis & Forest\\
\midrule
\textsf{les\,mis} & $\mathbf{0.026}$ & $0.031$ & $0.039$ & $\mathbf{0.038}$\\
\textsf{GoT} & $\mathbf{0.022}$ & $0.027$ & $\mathbf{0.042}$ & $0.045$\\
\textsf{polblogs} & $0.015$ & $\mathbf{0.014}$ & $0.028$ & $0.028$\\
\textsf{ca-GrQc} & $\mathbf{0.010}$ & $0.012$ & $\mathbf{0.010}$ & $0.012$\\
\textsf{power} & $0.015$ & $\mathbf{0.014}$ & $0.015$ & $0.015$\\
\textsf{ca-CondMat} & $\mathbf{0.035}$ & $0.038$ & $\mathbf{0.032}$ & $0.037$\\
\bottomrule
\end{tabular}
\end{table}

\begin{table}[t]
\centering
\small
\caption{Held-out AUC by combiner (mean over ten folds): random-forest stacking is most accurate, and among interpretable combiners the synthesis is best on all six networks.}
\label{tab:combiner-auc}
\begin{tabular}{lcccc}
\toprule
Dataset & Calibrated synthesis & RF stacking & Ensemble avg & best single\\
\midrule
\textsf{les\,mis} & 0.886 & 0.910 & 0.885 & \textbf{0.916} (AA)\\
\textsf{GoT} & 0.873 & \textbf{0.895} & 0.868 & 0.875 (AA)\\
\textsf{polblogs} & 0.942 & \textbf{0.950} & 0.939 & 0.932 (DC-SBM)\\
\textsf{ca-GrQc} & 0.935 & \textbf{0.969} & 0.910 & 0.910 (AA)\\
\textsf{power} & 0.807 & \textbf{0.884} & 0.700 & 0.726 (SBM)\\
\textsf{ca-CondMat} & 0.961 & \textbf{0.981} & 0.950 & 0.949 (AA)\\
\textsf{ER null} & 0.502 & 0.504 & 0.502 & \textbf{0.505} (RDPG)\\
\bottomrule
\end{tabular}
\end{table}

\subsection{Further numerical checks}
These secondary studies corroborate the main theorems beyond the three flagship validations reported in the
main text (interval coverage across the transversality range, the operator detection boundary, and the
first-stage attenuation and its removal).

These studies support the headline results and are reported for completeness.

\paragraph{Two-stage coverage under the linearised first-stage error (Theorem~\ref{thm:twostage}).}
The coverage statement of Theorem~\ref{thm:twostage} concerns the regime its hypotheses describe, in
which the fitted-kernel error enters through its embedding linearisation. We simulate from a known
three-kernel design with weights $(0.5,0.3,0.2)$ and transversality $\gamma_S\approx0.16$, split the
dyads into a Stage-B regression fold and two independent Stage-A sub-folds, model the Stage-A error
as a node-correlated perturbation at the spectral-embedding scale $\delta_n=\sqrt{r/(n\rho_n)}$, and
estimate the first-stage variance by the Gaussian multiplier bootstrap of
Theorem~\ref{thm:twostage} with $B=400$ resamples, measuring coverage coordinatewise. At $n=600$ the
conditional Stage-B sandwich covers the three generative coordinates at $(0.925,0.925,0.95)$,
undercovering the two coefficients whose shared first-stage error is largest; adding the propagated
first-stage variance through the two-stage construction raises the coordinatewise coverage to
$(0.95,0.975,1.0)$ at an $8\%$ increase in interval width. The first-stage contribution shrinks with
$n$, since $\delta_n\to0$, so the gap between the sandwich and the two-stage interval narrows at
larger sample sizes while the sandwich alone stays anticonservative on the dominant coordinate. This
validates the correction in the regime Theorem~\ref{thm:twostage} addresses: the conditional
sandwich undercovers and the added first-stage variance restores nominal coverage of the generative
weight.

\paragraph{Behaviour under a genuine kernel refit.}
Refitting the kernels inside every replication, rather than injecting their linearised error, exposes
a separate and more demanding phenomenon. We simulate from the same three-kernel mixture and, in each
replication, refit the candidate kernels on the Stage-A sub-folds by their own estimators, the
regularised block model by spectral clustering with block means, the dot-product kernel by
rank-two adjacency spectral embedding, and the degree kernel by degree matching. The naive
single-fold intervals cover the generative weight with empirical probability at most $0.08$ across
$n\in\{600,1200\}$, the attenuation bias again exceeding the interval width. The fitted Gram, however,
is far more ill-conditioned than the population Gram: the transversality falls from
$\gamma_S\approx0.16$ in the population kernels to $\widehat\gamma_S\approx0.06$ to $0.07$ in the
fitted kernels, because the rank-two embedding nearly reproduces the degree kernel, the two sharing
their leading mode. The per-coordinate generative weights are then not separately identified from a
single graph, and the well-posed target is the population projection of
Section~\ref{sec:estimands}, the deployable single-graph object being the calibration coefficient of
Proposition~\ref{thm:clt}. This is the weak-identifiability regime of Proposition~\ref{thm:joint}
realised by the fitting maps themselves, and it is the finite-sample face of the concentration onto
one effective mechanism seen in the real-data weights of Table~\ref{tab:weights}: cross-fold
debiasing removes the systematic attenuation, but it cannot manufacture conditioning that overlapping
fitted kernels do not possess. This conditioning collapse is, however, reparable: deflating the
degree-aligned spectral mode before Stage B and reassigning it to the degree agent restores the
transversality and re-identifies the coordinates (Appendix~\ref{app:deflation}), at the cost that the
geometry kernel becomes the residual-mode kernel; the generative reading of the deflated coefficients
continues to require fidelity.

\paragraph{Structured negative controls.}
Three controls probe the procedure where no genuine combination is present. Under an
Erd\H{o}s--R\'enyi null with no latent structure, the cross-fold debiased selection rule flags the
three candidate kernels at rates $(0.06,0.00,0.00)$, at the per-coordinate level rather than above
it, so the procedure does not report mechanisms the graph does not contain. A redundant-agent stress
test that appends a near-duplicate of the degree kernel collapses the fitted transversality to
$\widehat\gamma_S\approx0$ and sends the two duplicated coordinates to large offsetting values, the
exact failure that the conditioning $\gamma_S$ is defined to detect. A degree-preserving
randomisation and single-mechanism truths behave consistently with the conditioning finding above:
once the spectral and degree kernels nearly coincide, individual coordinates are not separately
recovered, only their projection and predictive combination, so these controls reproduce the
projection regime rather than clean per-coordinate selection.

\paragraph{Operator test under fitted layers.}
The four-fold construction of Theorem~\ref{prop:fitted-operator} is validated at $n=500$ with
two rank-one layers of controlled overlap, $80$ replications per configuration, and first-stage error
injected through the embedding linearisation, the idealisation in which the main-effect columns,
supplied by the disjoint sub-folds, sit at their consistent limits (limitation (vi) of
Section~\ref{sec:discussion}). Across edge-overlap information
$n^2\rho_n^3\in\{50,150,500,1500,3900\}$ the empirical size at nominal $0.05$ stays within binomial
variation, between $0.00$ and $0.09$; a count of zero rejections in $80$ replications is itself
within sampling variability of the nominal level. The power against the noisy-OR alternative rises
from $0.06$ to $0.60$ as the information crosses the threshold. The threshold location matches the
known-layer study of Figure~\ref{fig:estsim}(b); the constant does not: with known layers the power at
comparable information is $0.998$, so fitting the layers costs a constant factor in power even
though the detectability boundary is unchanged.

\paragraph{Real-data deployment on \textsf{polblogs} \citep{adamic2005political}.}
Table~\ref{tab:opfeas} identifies \textsf{polblogs} as the one network on which the operator
question is informationally feasible, and we deploy the four-fold protocol of
Theorem~\ref{prop:fitted-operator} there, on the largest connected component ($n=1222$), with a
two-layer candidate set: a degree kernel fitted by degree matching and a two-block community kernel fitted by
regularised spectral clustering with block means. Both layers are fitted on each of four disjoint
Stage-A sub-folds of $15\%$ of dyads each, sub-folds $(a,b)$ supplying the interaction column and
$(c,d)$ the main-effect columns and the null fit, with the remaining $40\%$ held out for Stage B
and kernels clipped at $0.99$; the fold accounting is $4\times15+40=100$ percent of dyads. The cross-fold debiased weights are
$\widehat w_{\mathrm{deg}}=0.69$ (standard error $0.02$) and $\widehat w_{\mathrm{com}}=0.68$
(standard error $0.03$; both are Stage-B sandwich standard errors, which by
Theorem~\ref{thm:twostage} omit the Stage-A share and undercover, the two-stage correction
widening them), both layers contributing strongly; these unconstrained projection
coefficients are not proportions, and their sum exceeding one reflects the overlap of the two
kernels on the same edges. The cross-fold score statistic on the residualised interaction column is
$T=12.4$, rejecting the additive two-layer synthesis at any conventional level. Because residual
misspecification of a two-layer candidate set also loads on the interaction direction, we read this as
evidence against additivity of the degree and community layers on this network rather than as a
confirmation of noisy-OR composition. On the scale of Figure~\ref{fig:estsim}(b) the network's edge-overlap
information $n^2\rho_n^3\approx16.8$ sits at the low-power end, where the simulated power is close to
size, so $T=12.4$ is read as suggestive and not as a powered rejection. Two scope caveats are stated openly: \textsf{polblogs} is
heavy tailed, outside the bounded heterogeneity of (C3) under which Lemma~\ref{lem:firststage}
verifies the first stage, and the sandwich is the reported variance, so the deployment illustrates
the construction while the size and power guarantees are those of the simulation.

\medskip
We next record the supporting structural simulations on synthetic graphs with known ground truth.

\paragraph{Stacked-latent recovery (Remark~\ref{rem:stack}).}
We generate stacked-latent RDPGs with a two-block $\X$-part ($d_1=2$) and a Dirichlet $\Y$-part
($d_2=2$), weights $(w_1,w_2)=(0.6,0.4)$, for $n\in\{250,\dots,4000\}$, embed by ASE into dimension
$4$, align to the truth by orthogonal Procrustes, and record the worst-row ($\Otoinf$) error.
Figure~\ref{fig:ase} shows the error decreasing with $n$ in lockstep with the theoretical
$\sqrt{\log n/(n\rho_n)}$ rate, consistent with the recovery claim of Remark~\ref{rem:stack}.

\paragraph{Kesten--Stigum threshold (regularised spectral recovery).}
We generate two-block SBMs at fixed mean degree $12$ and vary the signal $a-b$, computing the ARI of
regularised spectral clustering against the planted labels. Figure~\ref{fig:ks} shows recovery
switching on sharply as the SNR crosses $1$: ARI is statistically indistinguishable from $0$ for
$\mathrm{SNR}<1$ and climbs steeply above it (ARI $0.53$ at $\mathrm{SNR}=2.1$, $0.95$ at
$\mathrm{SNR}=5.3$), matching the regularised-spectral theory
\citep{le2017concentration,mossel2015reconstruction}.

\begin{figure}[t]
\centering
\begin{minipage}{0.48\textwidth}\centering
\includegraphics[width=\linewidth]{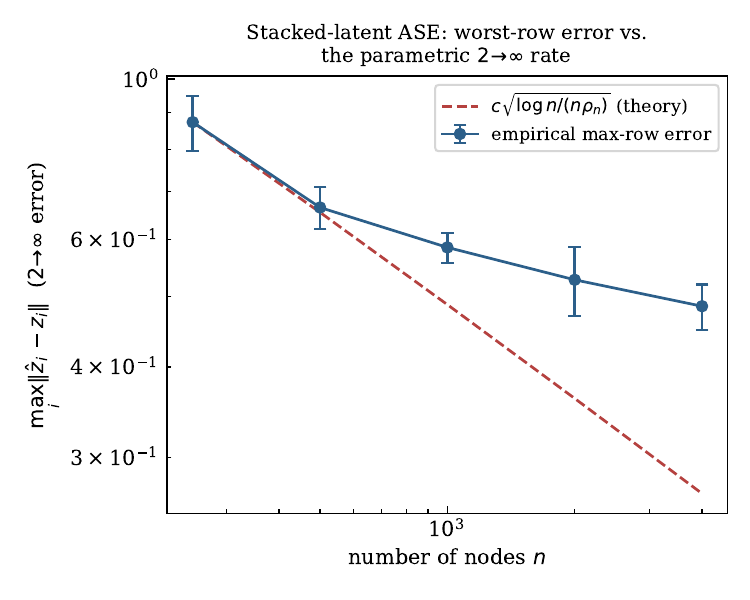}
\caption{Stacked-latent ASE attains the parametric rate: the empirical worst-row error tracks the theoretical $\sqrt{\log n/(n\rho_n)}$ curve, confirming Remark~\ref{rem:stack}.}
\label{fig:ase}
\end{minipage}\hfill
\begin{minipage}{0.48\textwidth}\centering
\includegraphics[width=\linewidth]{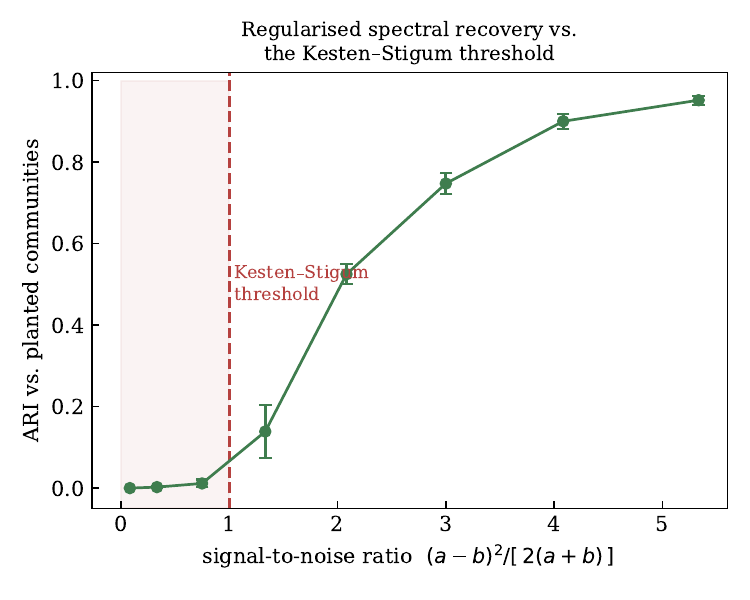}
\caption{Regularised spectral recovery activates at the Kesten--Stigum threshold: the ARI is near zero below $\mathrm{SNR}=1$ and rises sharply above it.}
\label{fig:ks}
\end{minipage}
\end{figure}

\section{Specialisation to canonical network models and additional remarks}
\label{app:netspec}

\subsection{Specialisation to canonical network models}
Recall the canonical synthesis $\Pmat=w_1\mathbf G_{\mathrm{SBM}}+w_2\mathbf G_{\mathrm{RDPG}}+w_3\mathbf
G_{\mathrm{CL}}$ of Section~\ref{sec:netspec}: a $Q$-block assortative SBM agent, a rank-$d$ dot-product agent,
and a rank-one Chung--Lu agent at a common density scale $\rho_n=n^{-\alpha}$, with $s=3$ active agents,
$r_{\max}=\max\{Q,d\}$, and transversality $\gamma_S=\lambda_{\min}(\Phi_S)$. The following corollaries record
what each general estimation result says for this candidate set; they are stated and proved here, and the
network-specific headline rates are summarised in Section~\ref{sec:netspec}.

\begin{corollary}[Estimation and inference for the canonical synthesis]
\label{cor:net-rate}
Under Assumptions~\ref{ass:transversality}--\ref{ass:norm}, for the three-agent synthesis the following
hold, with $\rho_n=n^{-\alpha}$ and $N=\Theta(n^2\rho_n)$.
\begin{enumerate}
\item \emph{(Minimax rate.)} The cross-fitted debiased estimator of $(w_1,w_2,w_3)$ satisfies, uniformly over
the class, $\E\,\|\widehat\bw-\bw\|_2\le C\sqrt{3/\gamma_S}\,/(n\sqrt{\rho_n})=C\sqrt{3/\gamma_S}\,
n^{-(1-\alpha/2)}$, and no estimator improves on it (Proposition~\ref{thm:joint}, Theorem~\ref{thm:debias});
the exponent ranges from $n^{-1+o(1)}$ in the dense regime to $n^{-1/2+o(1)}$ at the sparse boundary, the
inverse square root of the edge count up to the conditioning factor $\sqrt{3/\gamma_S}$, which is
$\approx3.9$ for Example~\ref{ex:gammaS}.
\item \emph{(Attenuation and its removal.)} Estimating the kernels by spectral clustering with block means
(SBM, relative Frobenius error $O_P(\sqrt{Q/(n\rho_n)})$), adjacency spectral embedding (RDPG,
$O_P(\sqrt{d/(n\rho_n)})$), and degree matching (Chung--Lu, $O_P(\sqrt{1/(n\rho_n)})$), the naive single-fold
plug-in carries attenuation bias of order $r_{\max}/(n\rho_n)=\max\{Q,d\}\,n^{-(1-\alpha)}$, exceeding the
part-(i) edge rate by the diverging factor $r_{\max}\sqrt{\gamma_S/\rho_n}=\max\{Q,d\}\sqrt{\gamma_S}\,
n^{\alpha/2}$; the cross-fold debiased estimator removes it and attains the edge rate
(Theorem~\ref{thm:debias}, Corollary~\ref{cor:debias-lb}).
\item \emph{(Normality and intervals.)} The held-out calibration estimator obeys
$\sqrt{N}\,(\widehat\bw-\bw^\star)\xrightarrow{d}\mathcal N(0,\Sigma)$ with $\Sigma=\mathbf J^{-1}\mathbf V
\mathbf J^{-1}$ the sandwich built from the three Gram matrices, and the multiplier bootstrap is consistent
(Proposition~\ref{thm:clt}); each coordinate interval has half-width $z_{1-\eta/2}\sqrt{\Sigma_{kk}/N}\asymp
n^{-(1-\alpha/2)}$, the worst-conditioned coordinate carrying variance of order $1/(\gamma_S N)$, and the
two-stage correction of Theorem~\ref{thm:twostage} inflates each width by $1+O(\sqrt{r_{\max}/(n\rho_n)})\to1$.
\end{enumerate}
\end{corollary}

\begin{corollary}[Selecting the active mechanisms under a beta-min gap]
\label{cor:net-select}
Augment the candidate set with $K-3$ inactive agents. If the smallest active weight satisfies the
beta-min condition
\[
\min\{w_1,w_2,w_3\}\ \ge\ C\,\sqrt{\frac{\log K}{\gamma_S\,N}}
\ =\ C\,\sqrt{\frac{\log K}{\gamma_S}}\;n^{-(1-\alpha/2)},
\]
the data-driven estimator selects exactly $\{\mathrm{SBM},\mathrm{RDPG},\mathrm{CL}\}$ with probability
tending to one (Corollary~\ref{cor:adapt-cor}); below this scale no thresholding rule separates an
active mechanism, such as the Chung--Lu hub agent, from sampling noise.
\end{corollary}

\begin{corollary}[Operator identifiability for block and hub layers]
\label{cor:net-operator}
Combine an assortative SBM community layer with a Chung--Lu hub layer at density $\rho_n$ under the
noisy-OR superposition versus the mixture. The two laws are mutually contiguous when
$n^2\rho_n^3\to0$ and are separated by a consistent test on the interaction column when
$n^2\rho_n^3\to\infty$ (Theorem~\ref{thm:nor}, Corollary~\ref{cor:nork-cor}). In the parametrisation
$\rho_n=n^{-\alpha}$ the operator is identifiable from one graph if and only if $\alpha<2/3$,
equivalently once the mean degree exceeds $n^{1/3+o(1)}$: a constant-degree sparse block model hides
whether its layers combine additively or by union, while a polynomially denser one reveals it. This is
the population mechanism behind the dense-core demonstration of Section~\ref{sec:realdata}, where the
information $n^2\rho_n^3$ rises from below one on the full graphs into the powered range on their
maximal $k$-cores.
\end{corollary}

\begin{proof}[Proof of Corollaries~\ref{cor:net-rate}--\ref{cor:net-operator}]
Each part is a specialisation of its named parent result to the three-agent synthesis, whose agents have
ranks $Q$, $d$, and $1$, so $s=3$ and $r_{\max}=\max\{Q,d\}$. Substituting these and $N=\Theta(n^2\rho_n)$
into Proposition~\ref{thm:joint} and Theorem~\ref{thm:debias} gives Corollary~\ref{cor:net-rate}(i); into the
plug-in bias of Proposition~\ref{thm:joint} gives part~(ii); into the limit law of Proposition~\ref{thm:clt}
with the variance correction of Theorem~\ref{thm:twostage} gives part~(iii); and into the coordinatewise
threshold of Corollary~\ref{cor:adapt-cor} gives Corollary~\ref{cor:net-select}.
Corollary~\ref{cor:net-operator} substitutes the SBM and Chung--Lu densities into the edge-overlap threshold
of Theorem~\ref{thm:nor}. The asymptotics in every case set $\rho_n=n^{-\alpha}$.
\end{proof}

The degeneracy of $\gamma_S$ and the graph-condition reading of the assumptions are recorded next.

\begin{remark}[The degeneracy is visible numerically]
\label{rem:gammaS-degeneracy}
On the synthesis of Example~\ref{ex:gammaS} ($n=2000$), as the dot-product positions rotate from independent
of the block labels toward the block structure, the transversality falls monotonically,
$\gamma_S=0.20,\,0.13,\,0.10,\,0.03,\,0.004$ at alignment $0,\,0.3,\,0.6,\,0.9,\,0.99$, while the SBM--RDPG
kernel correlation rises from $0.80$ to $0.996$: the weights stay well identified until the two geometries
nearly coincide, where $\gamma_S$ collapses and the rate of Corollary~\ref{cor:net-rate} degrades by
$1/\sqrt{\gamma_S}$, exactly as the multicollinearity reading predicts.
\end{remark}

\begin{remark}[Every assumption is a graph condition]
\label{rem:assumptions-network}
The remaining assumptions are equally graph-specific: the density floor $n\rho_n/\log n\to\infty$
(Assumption~\ref{ass:density}) is the spectral-concentration regime in which adjacency spectral embedding and
regularised spectral clustering are consistent, and the estimable-agents condition (Assumption~\ref{ass:plugin})
is the spectral estimability of each kernel at the two-to-infinity rate $O_P(\sqrt{r_k/(n\rho_n)})$. The weight
rate $\sqrt{s/\gamma_S}/\sqrt{N}$ therefore has both factors fixed by the graph: $N$ is the edge count and
$\gamma_S$ the separation of the candidate latent geometries.
\end{remark}

\subsection{Additional remarks}
The remarks collected here qualify the structural results of Section~\ref{sec:structural}, the estimator of
Section~\ref{sec:cf}, and the operator hierarchy of Section~\ref{sec:nor}; each is referenced from the main
text at the point it bears on.

\begin{remark}[Scope and the hyperbolic counterexample]
\label{rem:imp-scope}
Theorem~\ref{thm:imp} concerns the four named families. It is not a claim that no random graph can
have three of these properties together. Random hyperbolic graphs, and latent models on negatively
curved spaces, achieve heavy tails, constant clustering, and short paths simultaneously
\citep{krioukov2010hyperbolic,gugelmann2012random,fountoulakis2021clustering}; they evade part (c) precisely because the
hyperbolic distance kernel has effectively unbounded rank, which is what buys the clustering. The
point is that the four properties are out of reach within the standard, estimable model families a
practitioner reaches for first, and this is what motivates synthesising those families instead of
abandoning them for a model whose estimation theory is far less developed. The goodness-of-fit study
of Section~\ref{sec:gof} bears this out: the hyperbolic graph clears the four qualitative screens in
one shot but overshoots the observed hub scale fifteenfold, because the heavy tail it must adopt is
unboundedly heavier than the truncated tail the data display.
\end{remark}

\begin{remark}[The dilution phenomenon, and the data]
\label{rem:dilution}
The split in Theorem~\ref{thm:pos}(d) is not an artefact of the proof. On \textsf{ca-GrQc}
(Section~\ref{sec:gof}) the observed global clustering is $0.63$ while the observed average local
clustering is $0.56$, and the synthesised model matches the local coefficient ($0.34$ against
$0.56$, against $0.12$ for the strongest single-model baseline) precisely because hubs depress the
global coefficient more than the local one. We therefore report average local clustering as the
clustering target, in line with the metric the small-world literature uses \citep{watts1998collective}.
\end{remark}

\begin{remark}[One conditioning constant; lasso and ridge]
\label{rem:lasso-ridge}
``Minimax'' names the optimality of the rate, not the estimator: the constrained least squares of
Theorem~\ref{thm:debias} may be replaced by an $\ell_1$ or $\ell_2$ penalty, and the transversality constant
$\gamma_S=\lambda_{\min}(\Phi_S)$ is simultaneously the identification constant behind the minimax rate, the
lasso restricted-eigenvalue constant, and the ridge stabilisation scale, so the choice among them is
a statement about the regime $(K,\gamma_S,N)$ and not about network content. For a large candidate set the
$\ell_1$ penalty buys support recovery at the standard $\sqrt{\log K}$ price, the same factor carried by the
thresholding of Corollary~\ref{cor:adapt-cor}, and since the kernels are still fitted from the one graph it
is a debiased lasso, generated-regressor correction unchanged; ridge trades bias for variance as
$\gamma_S\to0$. Invariant across all three are the network-specific parts: the cross-fold debiasing, $\gamma_S$
as the conditioning constant, the effective sample size $N=n^2\rho_n$, and the $n^2\rho_n^3$ operator
threshold of Theorem~\ref{prop:fitted-operator}.
\end{remark}

\begin{remark}[The one open case]
\label{rem:nork-conj}
Corollary~\ref{cor:nork-cor} settles detection for every fixed $K$, and Theorem~\ref{thm:nor}(d)
estimates the free coefficients $\{c_S\}$ of the order-$j$ augmentation at the edge rate. What is not
proved is that the constrained map $c_S=(-1)^{|S|+1}\prod_{k\in S}w_k$ can be inverted to recover the
$K$ weights at the edge rate when $K\ge3$, since the higher-order columns
$\mathbf G_{k_1}\odot\cdots\odot\mathbf G_{k_j}$ have Frobenius norm $n\rho_n^{\,j}$ and the
corresponding transversality $\widetilde\gamma_j$ can vanish faster than the noise. We state this as a
conjecture rather than a theorem; it is the single piece of the operator analysis that the present
methods do not close.
\end{remark}

\end{document}